\UseRawInputEncoding
\documentclass[12pt, reqno]{amsart}
\usepackage{amssymb,amsthm,amsmath}
\usepackage[numbers,sort&compress]{natbib}
\usepackage{amsfonts}
\usepackage{mathrsfs}
\usepackage{latexsym}
\usepackage{tikz}
\usepackage{xcolor}
\usepackage{tkz-euclide}
\usepackage{float}
\usepackage{pdfsync}
\usepackage{indentfirst}
\usepackage{esint}
\usepackage{microtype} 
\usepackage[colorlinks=true, linkcolor=blue, citecolor=red]{hyperref}

\newtheorem{theorem}{Theorem}[section]

\renewcommand{\theequation}{\arabic{section}.\arabic{equation}}

\newtheorem{corollary}{Corollary}[section]
\newtheorem{lemma}{Lemma}[section]
\newtheorem{proposition}{Proposition}[section]
\newtheorem{maincorollary}[theorem]{Corollary}
\newtheorem{mainproposition}[theorem]{Proposition}
\newtheorem{definition}{Definition}[section]

\newtheorem{remark}{Remark}[section]

\newcommand{\beq}{\begin{equation}}
	\newcommand{\eeq}{\end{equation}}
\newcommand{\ben}{\begin{eqnarray}}
	\newcommand{\een}{\end{eqnarray}}
\newcommand{\beno}{\begin{eqnarray*}}
	\newcommand{\eeno}{\end{eqnarray*}}
\numberwithin{equation}{section} 

\numberwithin{equation}{section}

\newcommand{\alabel}{\stepcounter{equation}\tag{\theequation}\label}  
\allowdisplaybreaks

\begin{document}
	\title[Global Existence and Partial Regularity of cns equations]{\bf Global Existence and Partial Regularity for the
Three-Dimensional Chemotaxis--Navier--Stokes System
in a Half-Space}
	\author{Yubo~Chen$^{1}$}
	\address[Yubo~Chen]{School of Mathematical Sciences, Dalian University of Technology, Dalian, 116024,  China}
	\email{1220823215@mail.dlut.edu.cn}

	\author{Jianbo~Yu$^{1,*}$}
	\address[Jianbo~Yu]{School of Mathematical Sciences, Dalian University of Technology, Dalian, 116024,  China}
	\email{yujb@mail.dlut.edu.cn}
	\thanks{$^*$Author to whom correspondence should be addressed.}
	
	\date{\today}

\begin{abstract}
We establish endpoint partial regularity for suitable weak solutions to the three-dimensional chemotaxis--Navier--Stokes system with constant sensitivity and linear consumption in a half-space under no-slip and no-flux boundary conditions. The full space--time singular set, including its boundary part, has vanishing one-dimensional parabolic Hausdorff measure. This follows from a Caffarelli--Kohn--Nirenberg type criterion at the natural $r^{-1}$ dissipation scaling, valid at interior and boundary points without additional pressure smallness.

The key estimate combines weighted mass control with the localized entropy--energy inequality to close the iteration without a separate $n\log n$ quantity and accounts for the mass term generated by rescaling. Boundary-compatible pressure and caloric decompositions yield local H\"older regularity. We also prove a lower-order criterion requiring smallness only of the velocity, signal gradient, and pressure oscillation, with a threshold depending on an arbitrary finite density mass bound.

We construct global suitable weak solutions under natural entropy and energy assumptions without smallness restrictions on the initial data. The potential is time-independent with bounded gradient; boundedness of the potential itself is assumed only for the large-time results. If its gradient is also H\"older continuous, a solution can be selected to become bounded and classical throughout the closed half-space after a finite waiting time. For the same solution, we obtain decay and Neumann Gaussian asymptotics for both scalar components without initial spatial moment assumptions. The profile coefficients are the conserved bacterial mass and the limiting signal mass, respectively, and the latter is strictly positive whenever the initial signal is nonzero.
\end{abstract}
\maketitle

	 {\small\noindent {\bf Keywords:} Boundary $\varepsilon$-regularity; Hausdorff measure;
	 Suitable weak solutions;\\
	 Chemotaxis--Navier--Stokes equations.}


\clearpage
\tableofcontents
\clearpage

\section{Introduction and main results}
\label{section1}
Chemotaxis describes the directed movement of cells in response to
chemical gradients. In suspensions of aerobic bacteria, this mechanism
interacts with oxygen consumption and the motion of the surrounding
fluid: bacteria migrate towards oxygen-rich regions, consume oxygen,
and affect the fluid through buoyancy, while the fluid transports both
bacteria and oxygen. The experiments of Dombrowski et al.~\cite{DCCGK2004} and the
modeling work of Tuval et al.~\cite{ILCJR 2005} illustrate how these interactions
can produce collective motion and concentrated bacterial structures.
The associated chemotaxis--fluid equations couple cross-diffusion to
incompressible fluid dynamics and raise questions about global
existence, local regularity, and the possible concentration of weak
solutions near physical boundaries.

In this paper, we study the three-dimensional
chemotaxis--Navier--Stokes system
\begin{equation}\label{che}
\left\{
\begin{aligned}
&\partial_t n+u\cdot\nabla n-\Delta n
  =-\nabla\cdot(n\nabla c),\\
&\partial_t c+u\cdot\nabla c-\Delta c=-nc,\\
&\partial_t u+u\cdot\nabla u-\Delta u+\nabla p
  =-n\nabla\phi,\\
&\nabla\cdot u=0
\end{aligned}
\right.
\end{equation}
in the upper half-space
\[
\mathbb H^3:=\{x=(x',x_3)\in\mathbb R^3:x_3>0\},
\qquad
\Gamma:=\partial\mathbb H^3=\{x_3=0\}.
\]
Here $n$, $c$, $u$, and $p$ denote the cell density, oxygen
concentration, fluid velocity, and pressure, respectively.
The potential $\phi$ is prescribed and independent of time.
Our standing assumption is
\[
\phi\in W^{1,\infty}_{\rm loc}(\overline{\mathbb H^3}),
\qquad
M_\phi:=\|\nabla\phi\|_{L^\infty(\mathbb H^3)}<\infty.
\]
The local Sobolev condition is understood on bounded half-balls.
The global weak-solution construction and the local regularity
criteria require only this gradient bound; neither boundedness of
$\phi$ nor a bound on its second derivatives is imposed here.
For the large-time results we additionally assume
$\phi\in L^\infty(\mathbb H^3)$, so that
$\phi\in W^{1,\infty}(\mathbb H^3)$, to control the effective force
$\phi\nabla n$. H\"older continuity of $\nabla\phi$ is imposed separately
for the final classical regularity conclusion in
Theorem~\ref{thm:eventual-regularity}.
Thus the local results include the linear potential $\phi(x)=x_3$,
whereas the large-time results below do not.
We impose the boundary conditions
\begin{equation}\label{bound}
u=0,\qquad \partial_\nu n=\partial_\nu c=0
\qquad\text{on }\Gamma\times(0,\infty).
\end{equation}
Here $\nu=-e_3$ is the outward unit normal to $\mathbb H^3$.
The scalar variables $n$ and $c$ satisfy homogeneous Neumann
boundary conditions, while $u$ satisfies the no-slip boundary
condition. For weak solutions, the Neumann conditions are
understood through the integral identities in Definition~\ref{def:suitable-half-space}(ii),
and the no-slip condition is understood in the trace sense.

The existence theory for chemotaxis--fluid systems has developed
through estimates that exploit the interaction between density
entropy, concentration dissipation, and fluid energy.
Early contributions include the construction of global solutions
near constant equilibria by Duan, Lorz, and Markowich~
\cite{DLM2010}, and the analysis of a coupled chemotaxis--Stokes
model by Lorz~\cite{AL2010}. Chae, Kang, and Lee~\cite{CKL2013}
established local regular solvability and continuation criteria,
together with global results under structural assumptions.
Their subsequent work~\cite{CKL2014} investigated global existence
and temporal decay in regimes involving smallness conditions,
including the two-dimensional Navier--Stokes and three-dimensional
Stokes cases. These results identify several mechanisms through
which diffusion and consumption can control chemotactic aggregation.

Dai and Liu~\cite{DL2020LowModes} obtained a regularity criterion
involving only low modes of the oxygen concentration and fluid
velocity through a wavenumber splitting approach.
The lower-order criterion proved here uses local space--time
norms of suitable weak solutions,
including at a no-slip boundary, and allows an arbitrary finite
density mass bound through a mass-dependent smallness threshold.

In two space dimensions, global solvability is understood in a
substantially wider range of settings. Liu and Lorz~\cite{LL2011}
obtained global solutions for large data under suitable assumptions,
while Winkler~\cite{Winkler2012} established global classical
solvability for a two-dimensional chemotaxis--Navier--Stokes system;
the three-dimensional result in~\cite{Winkler2012} concerns the
Stokes case. Zhang and Zheng~\cite{ZZ2014} further developed the
two-dimensional well-posedness theory. Wang, Winkler, and Xiang
studied global classical solutions with subcritical sensitivity~
\cite{WWX2018} and the small-convection limit~\cite{WWX2018-2}.
Related questions concerning control of cross-diffusion in a
two-dimensional coral-fertilization model were studied by
Wang, Zhang, and Zheng~\cite{WZZ2021}. The differences in dimension,
sensitivity, and fluid convection in these works are relevant when
comparing their conclusions with the three-dimensional system~\eqref{che}.

For three-dimensional chemotaxis--Navier--Stokes systems, global
weak solutions have been constructed in several settings.
Winkler~\cite{Winkler2016} treated bounded convex domains,
He and Zhang~\cite{HZ2017} established further global weak
existence results, and Kang, Lee, and Winkler~\cite{KLW2022}
considered the Cauchy problem in $\mathbb R^3$. Such existence
results provide solutions beyond the possible breakdown of
classical regularity. A further issue is to identify a local
admissibility condition and use it to quantify the exceptional
set on which regularity may fail.

For the Navier--Stokes equations, this program originates in the
work of Scheffer~\cite{SV1,SV2,SV4} and the partial regularity
theorem of Caffarelli, Kohn, and Nirenberg~\cite{CKN}.
The latter proves that the interior singular set of a suitable
weak solution has zero one-dimensional parabolic Hausdorff
measure. Lin~\cite{Lin} gave a subsequent proof of this theorem,
and Tsai~\cite{Tsai-2018} provides a systematic account of the
surrounding regularity theory. Iteration and decay arguments
also appear in the four-dimensional first-blow-up-time analysis
of Dong and Du~\cite{DD}.

For suitable weak solutions of the Navier--Stokes equations,
Seregin~\cite{Seregin2002Boundary} established a
Caffarelli--Kohn--Nirenberg type condition for local
H\"older regularity near a plane boundary.

For the three-dimensional boundary
problem, Dong and Wang~\cite{DW2020} established
$\varepsilon$-regularity criteria in a half-cylinder with the
no-slip condition, using pressure decompositions, interpolation,
and iteration. Their criteria provide a useful starting point
for analyzing the fluid component of~\eqref{che}.

Partial regularity for chemotaxis--Navier--Stokes equations
must also account for the entropy of the density.
Chen, Li, and Wang~\cite{CLW2025} studied singularities at the
first blow-up time and obtained a vanishing $5/3$-dimensional
Hausdorff measure result. Chen, Li, Wang, and Wang~
\cite{CLWW2023} constructed global suitable weak solutions
in the whole space by deriving a local inequality for the
coupled entropy and energy. Their work~
\cite{CLWW2023Hausdorff} subsequently proved
$\mathcal P^{1+\delta}(\mathcal S)=0$ for arbitrarily small
$\delta>0$. The density-dissipation condition there involves
an additional factor $r^{-\delta}$ beyond the natural
normalization $r^{-1}$. The resulting dimension bound
$\dim_{\mathcal P}\mathcal S\leq 1$ leaves open the endpoint
measure conclusion $\mathcal P^1(\mathcal S)=0$ within that
argument.

A related line of work concerns exceptional times.
Fuest~\cite{Fuest2024Temporal} proved vanishing
$1/2$-dimensional Hausdorff measure of a temporal exceptional
set for weak energy solutions in a bounded convex domain.
Chen and Wang~\cite{ChenWang2026Temporal} also investigated
partial regularity in time for a class of three-dimensional
chemotaxis--fluid models. These results describe the size of
a subset of the time axis. The present work addresses local
regularity at individual space-time points and the parabolic
Hausdorff measure of the full space-time singular set,
including points on the physical boundary.

Large-time regularity and asymptotic behavior provide another
perspective on the interaction between chemotaxis and fluid
motion. In bounded convex domains,
Winkler~\cite{Winkler2017} established eventual smoothness
and stabilization for an appropriate class of
three-dimensional chemotaxis--Navier--Stokes solutions.
For the two-dimensional Cauchy problem, Chae, Kang, and
Lee~\cite{CKL2016Asymptotics} obtained temporal decay and
asymptotic profiles under smallness assumptions on the
initial data.

The finite-mass half-space problem considered here requires
a different large-time analysis. The absence of a
Poincar\'e inequality of the bounded-domain type prevents
a direct use of the corresponding stabilization argument.
Moreover, decay of the concentration in the uniform norm
does not determine whether its total mass vanishes.
These issues motivate our study of eventual regularity
and of the scalar diffusion profiles for the same global
solution constructed below.

The chemotactic coupling introduces difficulties that are
already present in the interior and are not resolved by
the kinetic-energy estimates for the Navier--Stokes
equations alone. The density is controlled through an
entropy involving $n\log n$, whose dissipation contains
$|\nabla\sqrt n|^2$, and its estimates are coupled to the
gradient energy of the concentration. Unlike the kinetic
energy density, $n\log n$ changes sign and does not scale
homogeneously: parabolic rescaling produces an additional
term proportional to the local mass. Consequently, the
localized entropy--energy inequality cannot be inserted
directly into the usual scale-invariant energy iteration.
Chen, Li, Wang, and Wang~\cite{CLWW2023Hausdorff} already used
weighted local mass estimates to control the negative entropy;
see in particular their equations~(4.13)--(4.15).
Building on this mechanism, we combine the entropy--energy
inequality and the mass identity with the same weight to obtain
a closed estimate involving only the local energy, dissipation,
and pressure quantities. The resulting iteration removes the
additional density-dissipation radius loss and gives the endpoint
measure conclusion in our suitable weak solution framework,
including at the boundary.

The boundary creates a further compatibility problem in
the existence construction. The scalar variables satisfy
Neumann conditions, whereas the velocity satisfies the
full no-slip condition and must remain divergence-free.
A straightforward reflection followed by convolution
does not generally preserve both velocity constraints.
At the same time, regularizing the chemotactic flux and
the consumption term independently may destroy the
cancellation underlying the entropy estimates.
We therefore combine half-space semigroup regularization
with linked truncations of these two nonlinearities.
The construction also preserves the localized inequality
for test functions satisfying the required Neumann
compatibility condition on the flat boundary.

Boundary regularity requires additional information beyond
these global bounds. The local pressure is coupled both
to the no-slip condition and to the density-dependent
force $n\nabla\phi$, so its estimates must be combined
with the scalar estimates in the iteration.
The boundary decay argument must also respect the
different symmetries of the Dirichlet velocity and the
Neumann scalar variables. We use pressure decompositions
and caloric remainders with compatible reflection
parities. The velocity reflection is applied componentwise
after the pressure and force have been included in the
right-hand side of the heat equation; it is not used
as a solenoidal extension of the coupled system.
These estimates connect regularity at boundary points
with uniform control at nearby interior points.

A separate difficulty arises in the large-time analysis
on the unbounded half-space. Finite-mass dispersion must
replace the bounded-domain stabilization mechanism.
We first obtain signal decay independently of the
transport and then derive weighted density estimates.
Since $\phi$ is bounded, the identity
\[
-n\nabla\phi=-\nabla(n\phi)+\phi\nabla n
\]
allows the gradient part to be absorbed into the pressure.
The estimate
\[
\|\phi\nabla n(t)\|_{L^2(\mathbb H^3)}
\leq\|\phi\|_{L^\infty(\mathbb H^3)}
      \|\nabla n(t)\|_{L^2(\mathbb H^3)}
\]
makes explicit where boundedness of the potential itself enters.
The remaining force belongs to $L^1_tL^2_x\cap L^2_tL^2_x$
on a time tail, which permits a strong restart of the
velocity. This yields eventual regularity for the selected
global solution and provides the starting point for the
subsequent decay and Gaussian asymptotics.

We develop this treatment using the shifted concentration
\[
    \tilde c=1+c,
    \qquad
    M_c=1+\|c_0\|_{L^\infty(\mathbb H^3)}.
\]
The shift permits local estimates involving
$\sqrt{\tilde c}$ without assuming a strictly positive
lower bound for $c$. The initial concentration still satisfies
the hypotheses of Theorem~\ref{thm:half-space-existence}, including
$\nabla\sqrt{c_0}\in L^2$. The three dissipation densities
used throughout the regularity analysis are
\[
    |\nabla\sqrt n|^2,
    \qquad
    |\nabla^2\sqrt{\tilde c}|^2,
    \qquad
    |\nabla u|^2.
\]

Our main result, Theorem~\ref{main4}, gives a dissipation criterion
at the natural scaling: an interior or boundary point $z_0$ is
regular whenever
\[
\limsup_{r\downarrow0}\frac1r\int_{\mathcal Q_r(z_0)}
\left(|\nabla\sqrt n|^2+|\nabla^2\sqrt{\tilde c}|^2
      +|\nabla u|^2\right)\,dx\,dt\leq\varepsilon_3,
\]
where $\varepsilon_3$ depends only on $M_c$ and $M_\phi$.
The condition uses the common normalization $r^{-1}$ for all
three dissipation terms and requires no additional pressure
smallness. Its immediate consequence, Corollary~\ref{cor:full-singular-set},
is the endpoint estimate $\mathcal P^1(\mathcal S)=0$ for the full
singular set in $\overline{\mathbb H^3}\times(0,\infty)$.
The comparison with previous interior results and the relevance
of including the boundary are discussed in
Remarks~\ref{rem:interior-comparison}--\ref{rem:boundary-concentration}.

A second contribution concerns the large-time behavior
of the global solution selected in our construction.
Under the standing bounded-potential assumption and without
smallness assumptions on the initial data,
this solution becomes bounded and classical throughout
the closed half-space after a finite waiting time.
For the same solution, we obtain quantitative decay
and identify the Neumann Gaussian profiles of both
scalar components, including the $L^1$ and $L^\infty$
endpoint norms and without initial spatial moment
assumptions. In particular, every nonzero initial signal
leaves a strictly positive limiting signal mass,
although its concentration converges uniformly to zero.

The existence framework is provided by a global
construction of suitable weak solutions satisfying a
localized entropy--energy inequality up to the flat
boundary. The approximation preserves incompressibility,
the boundary conditions, and the cancellation between
density entropy and concentration energy. It also yields
local $L^{5/3}$ pressure integrability on time intervals
compactly contained in $(0,\infty)$.

Two auxiliary regularity criteria support the endpoint
partial regularity theorem. The first uses local energy,
dissipation, and pressure oscillation; the second involves
lower-order space--time norms of the velocity,
the concentration gradient, and pressure oscillation.
In the second criterion, the density is only required to satisfy
an arbitrary prescribed $L^1$ bound on the unit cylinder;
the smallness threshold depends on this bound.
A caloric decomposition converts this bound into smallness
of the rescaled density mass at a suitably chosen smaller scale.
For mass-conserving global solutions, the bound on a physical
unit cylinder is supplied by the initial density mass.
Both criteria give local H\"older
regularity of $(n,\nabla c,u)$ without a separate smallness
assumption on $n|\log n|$. The proof mechanisms are
described in the subsection on key ingredients below.

The paper is organized as follows.
Section~\ref{subsec:solution-framework} defines suitable weak solutions
and states the global existence theorem. Section~\ref{subsec:main-regularity-results}
introduces regular and singular points and states the main regularity
theorem, its corollary, and the auxiliary criteria.
Section~\ref{subsec:key-ingredients} describes the proof strategy.
Section~\ref{section2} collects notation and preliminary estimates.
Section~\ref{sec:existence} proves Theorem~\ref{thm:half-space-existence}
by a global approximation and compactness argument. Section~\ref{section4} proves
Proposition~\ref{main1}, and Section~\ref{sec:lower-order-criteria}
establishes the lower-order criteria, including Proposition~\ref{main2}.
Section~\ref{sec:dissipation-singular-set} proves Theorem~\ref{main4}
and Corollary~\ref{cor:full-singular-set}.
Section~\ref{sec:eventual-regularity} proves
Theorem~\ref{thm:eventual-regularity}. Finally,
Section~\ref{sec:large-time-asymptotics} proves the decay and
Gaussian asymptotics in Theorem~\ref{thm:large-time-asymptotics}.

\subsection{Suitable weak solutions and global existence}
\label{subsec:solution-framework}
We first specify the solution class and state its global existence
theorem. The following notation will be used in the initial-data
assumptions and in the construction in Section~\ref{sec:existence}.
For $1<q<\infty$, let
\[
L_\sigma^q(\mathbb H^3)
:=
\overline{
\left\{
v\in C_c^\infty(\mathbb H^3;\mathbb R^3):
\nabla\cdot v=0
\right\}
}^{\,L^q(\mathbb H^3;\mathbb R^3)},
\]
and set $H(s):=(s+1)\log(s+1)$ for $s\geq0$.

\begin{definition}[Suitable weak solutions in the half-space] \label{def:suitable-half-space} 
Let $0<T<\infty$ and let the initial data $(n_0,c_0,u_0)$
satisfy \eqref{eq:initial-data-half-space} below.
A quadruple $(n,c,u,p)$ is called a \emph{suitable weak solution}
to \eqref{che}--\eqref{bound} in $\mathbb H^3\times(0,T)$ if the
following properties hold:
\begin{itemize} 
\item[\rm (i)] 
The functions \(n\) and \(c\) are nonnegative almost everywhere in \(\mathbb H^3\times(0,T)\). Throughout the paper, we use the convention \(0\log0=0\). Moreover, writing 

\[ \tilde c:=c+1, \] we assume that $ M_c:= \|\tilde c_0\|_{L^\infty(\mathbb H^3)} =\|1+c_0\|_{L^\infty(\mathbb H^3)}<\infty$, where $c_0$ is the given nonnegative initial concentration.
We require, in addition, that
\[
0\leq c(x,t)\leq M_c-1
\quad\text{for almost every }(x,t)\in\mathbb H^3\times(0,T).
\]
In particular, $1\leq\sqrt{\tilde{c}}\leq M_c^{1/2}$ almost everywhere.

We also require the following global bounds. For almost every $0<t<T$,
\begin{equation}\label{eq:limit-global-scalar-bounds}
\|n(t)\|_{L^1(\mathbb H^3)}\leq\|n_0\|_{L^1(\mathbb H^3)},
\qquad
\|c(t)\|_{L^1(\mathbb H^3)}\leq\|c_0\|_{L^1(\mathbb H^3)}.
\end{equation}
In the following estimates, all spatial norms and integrals are over
$\mathbb H^3$. There is a time-independent data constant
\begin{equation}\label{eq:data-constant}
\begin{aligned}
C_{\rm in}:=C\bigl(&1+\|\nabla\phi\|_{L^\infty},
\|n_0\|_{L^1},\int_{\mathbb H^3}H(n_0)\,dx,\\
&\|c_0\|_{L^1},\|c_0\|_{L^\infty},
\|\nabla\sqrt{c_0}\|_{L^2},\|u_0\|_{L^2}\bigr)
\end{aligned}
\end{equation}
such that the global entropy--energy estimate
\begin{equation}\label{eq:limit-global-energy}
\begin{aligned}
&\operatorname*{ess\,sup}_{0<t<T}
\left\{\int_{\mathbb H^3}\bigl(n(t)+H(n(t))\bigr)\,dx
+\|\nabla\sqrt{1+c(t)}\|_{L^2}^2+\|u(t)\|_{L^2}^2\right\}\\
&\quad+\int_0^T\int_{\mathbb H^3}
\frac{|\nabla n|^2}{1+n}\,dx\,dt
+\int_0^T\left(\|\nabla^2\sqrt{1+c}\|_{L^2}^2
+\|\nabla u\|_{L^2}^2\right)\,dt\\
&\quad+\int_0^T\int_{\mathbb H^3}
\left(\frac{|\nabla\sqrt{1+c}|^4}{1+c}
+n|\nabla\sqrt{1+c}|^2\right)\,dx\,dt
\leq C_{\rm in}(1+T).
\end{aligned}
\end{equation}
holds. Here $H(n)=(1+n)\log(1+n)$ and
$|\nabla n|^2/(1+n)=4|\nabla\sqrt{1+n}|^2$.
In particular,
\[
\begin{aligned}
&n\in L^{5/3}((0,T)\times\mathbb H^3),\\
&u,\ \nabla\sqrt{1+c}\in L^{10/3}((0,T)\times\mathbb H^3),
\qquad
\nabla c\in L^4((0,T)\times\mathbb H^3).
\end{aligned}
\]

\item[\rm (ii)] 
 The quadruple $(n,c,u,p)$ satisfies \eqref{che} in the sense of
distributions in $\mathbb H^3\times(0,T)$, with the boundary conditions
\[
u=0,\qquad \partial_\nu n=\partial_\nu c=0
\quad\text{on }\Gamma\times(0,T).
\]
The condition $u=0$ is understood in the trace sense.
The homogeneous Neumann conditions for $n$ and $c$ are understood
in the following weak sense: for every
$\zeta\in C_c^\infty(\overline{\mathbb H^3}\times(0,T))$,
\[
\int_0^T\int_{\mathbb H^3}
\bigl[-n\partial_t\zeta-nu\cdot\nabla\zeta
      +(\nabla n-n\nabla c)\cdot\nabla\zeta\bigr]\,dx\,dt=0
\]
and
\[
\int_0^T\int_{\mathbb H^3}
\bigl[-c\partial_t\zeta-cu\cdot\nabla\zeta
      +\nabla c\cdot\nabla\zeta+nc\zeta\bigr]\,dx\,dt=0.
\]
The test functions may be nonzero on $\Gamma$.

\item[\rm (iii)]
For every bounded domain $D\subset\mathbb H^3$ whose closure is compact in $\overline{\mathbb H^3}$,
the following local regularity properties hold:
\begin{align*}
&n,\ n\log n \in L^\infty(0,T;L^1(D)),
\qquad
\nabla\sqrt{\tilde c}\in L^\infty(0,T;L^2(D)),\\
&u\in L^\infty(0,T;L^2(D)),
\qquad
\nabla\sqrt n,\ \nabla^2\sqrt{\tilde c},\ \nabla u
\in L^2((0,T)\times D).
\end{align*}
In addition, for every compact interval $I\Subset(0,T]$,
\[
p\in L^{5/3}(I\times D).
\]

\item[\rm (iv)] \textbf{(Local boundary energy inequality).} 
\par Let $0<t_1<\tau<T$ and $0\leq\psi\in C_c^\infty\bigl( \overline{\mathbb H^3}\times(t_1,T) \bigr)$ satisfy the boundary compatibility condition \[ \partial_3\psi=0 \qquad\text{on }\Gamma\times(t_1,T). \] 
Then, for almost every \(\tau\in(t_1,T)\), 
\begin{align}
&\int_{\mathbb H^3} \left(n\log n+2|\nabla\sqrt{\tilde c}|^2+112M_{c}|u(x,\tau)|^2\right)\,\psi(x,\tau)\,dx\nonumber\\
&+\int_0^\tau\!\!\int_{\mathbb H^3}
      \left(2|\nabla\sqrt n|^2+\frac47|\nabla^2\sqrt{\tilde c}|^2+112M_c|\nabla u|^2\right)\psi\,dx\,dt
 \nonumber\\
&+2\int_0^\tau\!\!\int_{\mathbb H^3}
      n|\nabla\sqrt{\tilde c}|^2\psi\,dx\,dt
+\frac14\int_0^\tau\!\!\int_{\mathbb H^3}
      (\sqrt{\tilde c})^{-2}|\nabla\sqrt{\tilde c}|^4
      \psi\,dx\,dt
 \nonumber\\
\leq&
 \int_0^\tau\!\!\int_{\mathbb H^3}
      \left(n\log n+2|\nabla\sqrt{\tilde c}|^2+112M_c|u|^2\right)
      (\partial_t\psi+\Delta\psi+u\cdot\nabla\psi)
      \,dx\,dt
 \nonumber\\
&+\int_0^\tau\!\!\int_{\mathbb H^3}
      n\log n\,\nabla\tilde c\cdot\nabla\psi\,dx\,dt
 +\int_0^\tau\!\!\int_{\mathbb H^3}
      n\nabla\tilde c\cdot\nabla\psi\,dx\,dt
 \nonumber\\
&-\frac47\int_0^\tau\!\!\int_{\mathbb H^3}
      (\sqrt{\tilde c})^{-1}|\nabla\sqrt{\tilde c}|^2
      \nabla\sqrt{\tilde c}\cdot\nabla\psi\,dx\,dt
 -\frac{16}{7}\sum_i\int_0^\tau\!\!\int_{\mathbb H^3}
      (\sqrt{\tilde c})^{-1}
      (\partial_i\sqrt{\tilde c})^3
      \partial_i\psi\,dx\,dt
 \nonumber\\
&+224M_c\int_0^\tau\!\!\int_{\mathbb H^3}
      (p-\bar p(t))u\cdot\nabla\psi\,dx\,dt
 -224M_c\int_0^\tau\!\!\int_{\mathbb H^3}
      n\nabla\phi\cdot u\,\psi\,dx\,dt,
\label{le}
\end{align}
where \(\bar p\in L^1_{\rm loc}(0,T)\) is an arbitrary locally integrable function of time; in applications it will be chosen as a suitable local spatial average of the pressure. 
\end{itemize} 
\end{definition}

The following theorem provides global solutions in this class.
For a global solution, suitability is understood on every finite
time interval. The proof is given in Section~\ref{sec:existence}.

\begin{theorem}[Global half-space existence]
\label{thm:half-space-existence}
Assume that $\phi$ satisfies the standing bounded-gradient assumption
and that the initial data satisfy
\begin{equation}
\label{eq:initial-data-half-space}
\begin{aligned}
&n_0\geq0,\qquad
n_0\in L^1(\mathbb H^3),\qquad
H(n_0)\in L^1(\mathbb H^3),\\
&c_0\geq0,\qquad
c_0\in L^1(\mathbb H^3)\cap L^\infty(\mathbb H^3),\qquad
\nabla\sqrt{c_0}\in L^2(\mathbb H^3),\\
&u_0\in L^2_\sigma(\mathbb H^3).
\end{aligned}
\end{equation}
Then there exists a global quadruple $(n,c,u,p)$ solving
\begin{equation}
\label{eq:half-space-CNS}
\left\{
\begin{aligned}
&\partial_t n+u\cdot\nabla n-\Delta n
=-\nabla\cdot(n\nabla c),\\
&\partial_t c+u\cdot\nabla c-\Delta c
=-nc,\\
&\partial_t u+u\cdot\nabla u-\Delta u+\nabla p
=-n\nabla\phi,\\
&\nabla\cdot u=0
\end{aligned}
\right.
\qquad
\text{in }\mathbb H^3\times(0,\infty).
\end{equation}
The boundary conditions are understood in the sense specified in
Definition~\ref{def:suitable-half-space}(ii).

For every $0<T<\infty$ and every bounded domain
$D\subset\mathbb H^3$ whose closure is compact in
$\overline{\mathbb H^3}$,
\begin{equation}
\label{eq:limit-regularity}
\begin{aligned}
&n,\ n\log n\in L^\infty(0,T;L^1(D)),
\qquad
\nabla\sqrt n\in L^2((0,T)\times D),\\
&\nabla\sqrt{\tilde c}\in L^\infty(0,T;L^2(D)),
\qquad
\nabla^2\sqrt{\tilde c}\in L^2((0,T)\times D),\\
&u\in L^\infty(0,T;L^2(D))
\cap L^2(0,T;H^1(D)).
\end{aligned}
\end{equation}
Moreover, for every compact interval $I\Subset(0,T]$,
\[
p\in L^{5/3}(I\times D).
\]

The solution attains its initial data in the weak sense specified in
Section~\ref{subsec:compactness-local-energy} and satisfies Definition~\ref{def:suitable-half-space} on every finite
time interval. In particular, the pressure is locally $L^{5/3}$ up to
the flat boundary on every time interval compactly contained in
$(0,\infty)$; no pressure estimate at the genuine initial time $t=0$
is asserted.
The approximation may be selected as specified in
Section~\ref{sec:eventual-regularity}; this selection also has the
eventual regularity asserted in Theorem~\ref{thm:eventual-regularity}
and the large-time behavior in Theorem~\ref{thm:large-time-asymptotics}.
\end{theorem}

\begin{remark}\label{rem:local-nature-suitable}

Definition~\ref{def:suitable-half-space}(i) includes the global mass
and entropy--energy bounds
\eqref{eq:limit-global-scalar-bounds}--\eqref{eq:limit-global-energy}.
The global density dissipation controls $\nabla\sqrt{1+n}$;
item~(iii) additionally requires the local Fisher-information bound
for $\nabla\sqrt n$.

The energy quantities in items~(i) and~(iii) are controlled on
$(0,T)$, including near the initial time. In contrast, the pressure is
only required to be locally integrable on time intervals compactly
contained in $(0,T)$. In particular, no pressure integrability up to
the genuine initial time $t=0$ is required. This is sufficient for the
local analysis below, since every cylinder under consideration is chosen
so that its time interval is compactly contained in $(0,T)$.
\end{remark}

\subsection{Main regularity results}
\label{subsec:main-regularity-results}
Having established the solution framework and stated its existence
theorem, we turn to local regularity at interior and boundary points.
We first fix the cylinders and the parabolic measure used below,
then define regular points and state the main regularity theorem.

For
\[
z_0=(x_0,t_0)\in\overline{\mathbb H^3}\times(0,T),
\]
we call $r>0$ an \emph{admissible radius at $z_0$} if
\[
r^2<t_0,
\]
and, whenever $x_0\in\mathbb H^3$,
\[
B_r(x_0)\Subset\mathbb H^3.
\]
For such an admissible radius, we set
\[
\mathcal B_r(x_0)
:=
\begin{cases}
B_r(x_0),
& x_0\in\mathbb H^3,\\[1mm]
B_r(x_0)\cap\mathbb H^3,
& x_0\in\Gamma,
\end{cases}
\]
and
\[
\mathcal Q_r(z_0)
:=
\mathcal B_r(x_0)\times(t_0-r^2,t_0).
\]
Thus $\mathcal Q_r(z_0)$ is a full parabolic cylinder at an interior
point and a half-cylinder at a boundary point. Notice that every
\[
z_0\in\overline{\mathbb H^3}\times(0,T)
\]
admits sufficiently small admissible radii.

For later use, define the dissipation density
\[
e
:=
|\nabla\sqrt n|^2
+
|\nabla^2\sqrt{\tilde c}|^2
+
|\nabla u|^2.
\]
We use the standard parabolic metric
\[
d_P\bigl((x,t),(y,s)\bigr)
:=
\max\left\{
|x-y|,\ |t-s|^{1/2}
\right\}
\]
on $\mathbb R^3\times\mathbb R$. For a set
$U\subset\mathbb R^3\times\mathbb R$, define
\[
\operatorname{diam}_P U
:=
\sup\{
d_P(z_1,z_2):z_1,z_2\in U
\}.
\]
For $\alpha\geq0$, $\delta>0$, and
$E\subset\mathbb R^3\times\mathbb R$, define
\[
\mathcal P_\delta^\alpha(E)
:=
\inf
\left\{
\sum_{j=1}^\infty
(\operatorname{diam}_P U_j)^\alpha:
E\subset\bigcup_{j=1}^\infty U_j,\ 
\operatorname{diam}_P U_j<\delta
\right\},
\]
and
\[
\mathcal P^\alpha(E)
:=
\lim_{\delta\downarrow0}
\mathcal P_\delta^\alpha(E).
\]

\begin{definition}[Regular and singular points]
\label{def:boundary-regular-point}
Let $(n,c,u,p)$ be a suitable weak solution in
$\mathbb H^3\times(0,T)$, and let
\[
z_0=(x_0,t_0)\in\overline{\mathbb H^3}\times(0,T).
\]
We call $z_0$ a \emph{regular point} if there exist an admissible
radius $r>0$ at $z_0$ and some $\alpha\in(0,1)$ such that
\[
n,\ u,\ \nabla c
\in C^{\alpha,\alpha/2}
\bigl(\overline{\mathcal Q_r(z_0)}\bigr).
\]
Otherwise, $z_0$ is called a \emph{singular point}.

The singular set on the time interval $(0,T)$ is defined by
\[
\mathcal S(T)
:=
\left\{
z_0\in\overline{\mathbb H^3}\times(0,T):
z_0\text{ is a singular point}
\right\}.
\]
For a global suitable weak solution on
$\mathbb H^3\times(0,\infty)$, we write
\[
\mathcal S
:=
\left\{
z_0\in\overline{\mathbb H^3}\times(0,\infty):
z_0\text{ is a singular point}
\right\}.
\]

\end{definition}

With regularity understood in the preceding sense, our main theorem
gives a common dissipation criterion for interior and boundary points.

\begin{theorem}[Main theorem: dissipation criterion]\label{main4}
Let $(n,c,u,p)$ be a suitable weak solution to
\eqref{che}--\eqref{bound} in $\mathbb H^3\times(0,\infty)$.
There exists $\varepsilon_3=\varepsilon_3(M_c,M_\phi)>0$ such that
every point
\[
z_0=(x_0,t_0)\in\overline{\mathbb H^3}\times(0,\infty)
\]
is regular provided
\begin{equation}\label{eq:main-dissipation-criterion}
\limsup_{r\downarrow0}\frac1r\int_{\mathcal Q_r(z_0)}
\left(|\nabla\sqrt n|^2+|\nabla^2\sqrt{\tilde c}|^2
      +|\nabla u|^2\right)\,dx\,dt\leq\varepsilon_3.
\end{equation}
Here the radii are admissible at $z_0$, so $\mathcal Q_r(z_0)$
is an interior cylinder or a boundary half-cylinder as appropriate.
\end{theorem}

\begin{maincorollary}[Size of the full singular set]
\label{cor:full-singular-set}
Let $(n,c,u,p)$ be a global suitable weak solution to
\eqref{che}--\eqref{bound} in $\mathbb H^3\times(0,\infty)$,
and let $\mathcal S\subset\overline{\mathbb H^3}\times(0,\infty)$
be its singular set. Then
\[
\mathcal P^1(\mathcal S)=0.
\]
\end{maincorollary}

\begin{remark}[Endpoint improvement and inclusion of boundary singularities]
\label{rem:interior-comparison}

The earlier interior theory includes the vanishing
$5/3$-dimensional Hausdorff measure result at the first blow-up
time in~\cite{CLW2025}, and the subsequent space-time estimate
$\mathcal P^{1+\delta}(\mathcal S)=0$ for every sufficiently small
$\delta>0$ in~\cite{CLWW2023Hausdorff}. The latter already implies
$\dim_{\mathcal P}\mathcal S\leq1$.

For the constant-sensitivity, linear-consumption system~\eqref{che}
and the suitable weak solutions of
Definition~\ref{def:suitable-half-space},
Corollary~\ref{cor:full-singular-set} gives the endpoint conclusion
$\mathcal P^1(\mathcal S)=0$ for the full singular set, including
its boundary part. Theorem~\ref{main4} uses the natural factor
$r^{-1}$ for the density dissipation, in place of the
$r^{-1-\delta}$ factor in the earlier criterion.

The improvement is the endpoint measure conclusion and the
inclusion of boundary points, rather than a dimension bound
strictly below one. The comparison is made within the respective
solution frameworks: the entropy variables ($\sqrt{1+c}$ here
and $\sqrt c$ in~\cite{CLWW2023Hausdorff}) and pressure
integrability requirements differ, so no inclusion of the two
classes of suitable weak solutions is asserted.

\end{remark}

\begin{remark}[Boundary concentration and analytic difficulties]
\label{rem:boundary-concentration}
The experiments of Dombrowski et al.~\cite{DCCGK2004} show
concentration of bacteria near the edge of sessile drops,
while Tuval et al.~\cite{ILCJR 2005} describe accumulation near
solid--air--water contact lines. Including physical boundary
points in the regularity theory is therefore qualitatively
consistent with the observed importance of boundary regions.
The result bounds possible loss of regularity; it does not
identify the experimentally observed concentration regions
with mathematical singularities or assert that singularities occur.

The principal difficulties are the nonhomogeneous scaling and
negative part of $n\log n$, the interaction of Dirichlet and
Neumann conditions in boundary decay estimates, and the
coupling of pressure to $n\nabla\phi$. Our main innovations are
a local energy estimate without a separate entropy quantity,
parity-compatible caloric decompositions, and a boundary-compatible
existence and pressure framework. Their roles are detailed in
Section~\ref{subsec:key-ingredients}.
\end{remark}

We next state the two auxiliary $\varepsilon$-regularity criteria.
Given $z_*\in\overline{\mathbb H^3}\times(0,T)$ and an admissible
radius $\rho$, a translation and parabolic rescaling reduce
$\mathcal Q_\rho(z_*)$ to a unit cylinder. The radius $1$ in
Propositions~\ref{main1}--\ref{main2} is only a normalization.
We retain the notation $z_0=(x_0,t_0)$ for the normalized center.

The rescaled potential is
$\phi_\rho(y)=\phi(x_*+\rho y)$, so that
$\|\nabla_y\phi_\rho\|_{L^{\infty}}\le\rho M_\phi$
on the rescaled domain. The criteria use this rescaled gradient
bound; for $\rho\le1$ it is uniformly controlled by the original
$M_\phi$.

\begin{mainproposition}[Energy and pressure criterion]\label{main1}
Let $0<T<\infty$, and let $(n,c,u,p)$ be a normalized local
restriction of a suitable weak solution to \eqref{che}--\eqref{bound}
on the unit cylinder $\mathcal Q_1(z_0)$ described above.
For any $\tilde q\in(1,\infty)$, assume further that

$$
p\in L_t^{\tilde q}L_x^1(\mathcal Q_1(z_0)).
$$
There exists
\[
\varepsilon_0
=
\varepsilon_0(\tilde q,M_c,M_\phi)>0
\]
such that, if
\begin{align}
&\operatorname*{ess\,sup}_{t\in(t_0-1,t_0)}
\int_{\mathcal B_1(x_0)}
\left(
n+|u|^2+|\nabla\sqrt{\tilde c}|^2
\right)(x,t)\,dx
\nonumber\\
&\quad+
\int_{\mathcal Q_1(z_0)}
\left(
|\nabla\sqrt n|^2
+|\nabla^2\sqrt{\tilde c}|^2
+|\nabla u|^2
\right)\,dx\,dt
\nonumber\\
&\quad+
\int_{t_0-1}^{t_0}
\left\|
p-[p]_{x_0,1}(t)
\right\|_{L_x^1(\mathcal B_1(x_0))}^{\tilde q}\,dt
\leq\varepsilon_0,
\label{14}
\end{align}
where
$$
[p]_{x_0,r}(t)
:=
\fint_{\mathcal B_r(x_0)}p(x,t)\,dx,
$$
then $(n,\nabla c,u)$ is regular in
$\overline{\mathcal Q_{1/2}(z_0)}$.
\end{mainproposition}

Proposition~\ref{main1} involves the dissipation
$|\nabla^2\sqrt{\tilde c}|^2$. The next criterion replaces the
energy and dissipation smallness by lower-order space-time norms.
The density mass need not be small; its prescribed
upper bound enters the threshold for the other quantities.

\begin{mainproposition}[Lower-order criterion with bounded density mass]\label{main2}
Let $0<T<\infty$, and let $(n,c,u,p)$ be a normalized local
restriction of a suitable weak solution to \eqref{che}--\eqref{bound}
on the unit cylinder $\mathcal Q_1(z_0)$ described above. For each pair
\[
q>\frac52,
\qquad
\tilde q>1,
\]
and every $M>0$, there exists
\[
\varepsilon_1
=
\varepsilon_1(M,q,\tilde q,M_c,M_\phi)>0
\]
such that, if
\begin{equation}
\|n\|_{L^1_{t,x}(\mathcal Q_1(z_0))}\leq M
\label{eq:main2-mass-bound}
\end{equation}
and
\begin{align}
&\|u\|_{L_{t,x}^q(\mathcal{Q}_1(z_0))}
+\|\nabla\sqrt{\tilde c}\|_
{L_{t,x}^q(\mathcal{Q}_1(z_0))}
\nonumber\\
&\qquad
+\left\|
p-[p]_{x_0,1}(t)
\right\|_
{L_t^{\tilde q} L_x^1(\mathcal{Q}_1(z_0))}
\leq\varepsilon_1,
\label{eq:main2-smallness}
\end{align}
then $(n,\nabla c,u)$ is regular in
$\overline{\mathcal Q_{1/2}(z_0)}$.
\end{mainproposition}

\begin{remark}[Dependence on the initial data and the observation scale]
\label{rem:mass-threshold-data}
The bound in~\eqref{eq:main2-mass-bound} is a space--time
$L^1$ bound, not a pointwise bound on $n$. The threshold is
uniform over all suitable weak solutions obeying the same
mass bound and can be chosen nonincreasing in $M$.
For a global solution satisfying
\[
\int_{\mathbb H^3}n(x,t)\,dx
=\|n_0\|_{L^1(\mathbb H^3)}
\quad\text{for a.e. }t>0,
\]
every physical unit cylinder contained in the positive-time
domain satisfies~\eqref{eq:main2-mass-bound} with
$M=1+\|n_0\|_{L^1(\mathbb H^3)}$. Hence its threshold may be
chosen as
\[
\varepsilon_1\!\left(
1+\|n_0\|_{L^1},q,\tilde q,
1+\|c_0\|_{L^\infty},\|\nabla\phi\|_{L^\infty}
\right).
\]
No dependence on $\|u_0\|_{L^2}$ or on the initial entropy
is needed for this local threshold; the velocity, signal-gradient,
and pressure smallness assumptions are still required.

For a physical cylinder of radius $R$, the natural rescaling gives
\[
\|n_R\|_{L^1(\mathcal Q_1)}
=R^{-3}\int_{\mathcal Q_R}n\,dx\,dt
\leq R^{-1}\|n_0\|_{L^1},
\qquad
\|\nabla\phi_R\|_{L^\infty}
\leq R\|\nabla\phi\|_{L^\infty}.
\]
Thus one can use $M=1+R^{-1}\|n_0\|_{L^1}$ after normalization.
The assertion does not provide a mass-independent threshold
uniform over all observation scales. The small-mass criterion
with a threshold independent of $M$ is retained in the proof
of the dissipation criterion.
\end{remark}

The next theorem concerns the large-time behavior of a solution
selected in the global existence construction. It supplements the
partial regularity results with eventual regularity throughout the
closed half-space.

\begin{theorem}[Eventual regularity in the half-space]
\label{thm:eventual-regularity}
In addition to the assumptions of Theorem~\ref{thm:half-space-existence},
assume that
\begin{equation}\label{eq:eventual-potential-assumptions}
\phi\in L^\infty(\mathbb H^3),\qquad
\nabla\phi\in C^{0,\beta_0}(\overline{\mathbb H^3})
\quad\text{for some }\beta_0\in(0,1),
\end{equation}
where the H\"older norm is global.
The global suitable weak solution $(n,c,u,p)$ in that theorem can be
chosen so that there exist $T_*>0$ and $\alpha\in(0,\beta_0)$ for which
\[
n,c,u\in C^{2+\alpha,1+\alpha/2}_{\rm loc}
\bigl(\overline{\mathbb H^3}\times[T_*,\infty)\bigr)
\]
and
\[
\sup_{t\geq T_*}
\left(\|n(t)\|_{L^\infty(\mathbb H^3)}
+\|c(t)\|_{W^{1,\infty}(\mathbb H^3)}
+\|u(t)\|_{L^\infty(\mathbb H^3)}\right)<\infty.
\]
The equations and boundary conditions hold classically on this
time interval, and $\nabla p$ has the corresponding local
$C^{\alpha,\alpha/2}$ regularity up to the flat boundary.
In particular, the full singular set satisfies
\[
\mathcal S\cap
\bigl(\overline{\mathbb H^3}\times[T_*,\infty)\bigr)=\varnothing.
\]
No smallness of the initial data is imposed.
The eventual uniform bounds require only the additional assumption
$\phi\in L^\infty$; H\"older continuity of $\nabla\phi$ is used solely
to obtain the displayed classical regularity and the corresponding
pressure regularity. 
\end{theorem}

\begin{remark}[Comparison with Winkler's eventual regularity argument]
\label{rem:eventual-winkler}
Winkler~\cite{Winkler2017} proves eventual regularity for an
appropriate class of weak solutions in bounded convex domains.
Our weighted density test is inspired by his joint scalar tests;
we do not claim that this scalar device is new. The argument on
$\mathbb H^3$, however, has a different large-time mechanism.
Finite signal mass and the half-space Nash inequality give
$\|c(t)\|_{L^{\infty}}\leq Ct^{-3/2}\|c_0\|_{L^1}$, independently of the
transport. A weighted density estimate then yields
\[
\int_{T_0+R}^\infty\|\nabla n\|_{L^2}^2\leq CR^{-3/2},
\qquad
\int_{T_0+R}^\infty\|\nabla n\|_{L^2}\leq CR^{-1/4}.
\]
These estimates hold for $R\geq1$.
Since $\phi$ is bounded, the decomposition
$-n\nabla\phi=-\nabla(n\phi)+\phi\nabla n$ gives an effective
solenoidal force in $L^1_tL^2_x\cap L^2_tL^2_x$ on the time tail.
This permits a strong restart of the velocity without a Poincar\'e
inequality. Half-space Stokes estimates and the Neumann reflection
argument then give regularity at every interior and boundary point
after one common waiting time. The proof in
Section~\ref{sec:eventual-regularity} also specifies an iterated
approximation limit that retains the global kinetic-energy inequality
and makes the joint scalar tests legitimate. Accordingly, the theorem
asserts eventual regularity for this selection of the global solution;
it does not infer it from the local solution definition or from the
Hausdorff-measure estimate alone.
\end{remark}

For the same selected solution, eventual regularity can be strengthened
to quantitative decay and a precise description of the scalar diffusion
profiles. We state this separately to distinguish regularity after a
waiting time from the asymptotic behavior as $t\to\infty$.

\begin{theorem}[Large-time decay and Gaussian asymptotics]
\label{thm:large-time-asymptotics}
Let $(n,c,u,p)$ be the solution selected in the proof of
Theorem~\ref{thm:eventual-regularity}. Set
\[
 M=\int_{\mathbb H^3}n_0,\qquad L=\int_{\mathbb H^3}c_0.
\]
All spatial norms below are over $\mathbb H^3$.
There is $T_1>0$ such that, for every $1\leq r\leq\infty$,
\begin{equation}\label{eq:main-asym-scalar-rates}
 \|n(t)\|_{L^r}+\|c(t)\|_{L^r}\leq C_r t^{-\frac32(1-\frac1r)},
 \qquad \|\nabla c(t)\|_{L^{\infty}}\leq Ct^{-2},
 \quad t\geq T_1.
\end{equation}
The velocity satisfies
\begin{equation}\label{eq:main-asym-velocity}
\begin{aligned}
 &\lim_{t\to\infty}
   \left(\|u(t)\|_{L^2}+t^{1/2}\|\nabla u(t)\|_{L^2}\right)=0,\\
 &\lim_{t\to\infty}
   t^{\frac32(\frac12-\frac1q)}\|u(t)\|_{L^q}=0,
   \qquad 2<q\leq\infty.
\end{aligned}
\end{equation}
In particular,
\[
 \|n(t)\|_{L^{\infty}}+\|c(t)\|_{W^{1,\infty}}+\|u(t)\|_{L^{\infty}}
 \longrightarrow0.
\]

The density has conserved mass $M$, and the remaining signal mass is
\begin{equation}\label{eq:main-asym-signal-mass}
 L_\infty:=\lim_{t\to\infty}\int_{\mathbb H^3}c(t)
 =L-\int_0^\infty\int_{\mathbb H^3}nc.
\end{equation}
If $L>0$, then $0<L_\infty\leq L$; if $L=0$, then $c=0$.
Define the unit-mass Neumann Gaussian on the half-space by
\[
 G_N(x,t)=2(4\pi t)^{-3/2}e^{-|x|^2/(4t)},
 \qquad x\in\mathbb H^3.
\]
For every $1\leq r\leq\infty$,
\begin{equation}\label{eq:main-asym-gaussian}
 \lim_{t\to\infty}t^{\frac32(1-\frac1r)}
 \left(\|n(t)-MG_N(t)\|_{L^r}
       +\|c(t)-L_\infty G_N(t)\|_{L^r}\right)=0.
\end{equation}
No initial spatial moment assumption is required.
\end{theorem}

\begin{remark}[Zero equilibrium and diffusion of the conserved mass]
\label{rem:large-time-interpretation}
The uniform limit of $(n,c,u)$ is the zero equilibrium. When $M>0$,
however, $\|n(t)\|_{L^1}=M$ at every time, so this convergence does not
hold in $L^1$ for the density. Similarly, if $c_0\not\equiv0$, the
signal converges uniformly to zero while its total mass tends to the
strictly positive number $L_\infty$. Formula~\eqref{eq:main-asym-gaussian}
describes both scalars on the diffusive spatial scale $|x|\sim\sqrt t$,
including convergence in the normalized $L^1$ and $L^\infty$ norms.
The positive remaining signal mass is a feature of the present
finite-mass half-space problem; it contrasts with complete signal
consumption in the bounded-domain setting of \cite{Winkler2017}
with nonzero bacterial mass. The velocity conclusion follows from an energy comparison with the
homogeneous Stokes evolution. No algebraic rate for
$\|u(t)\|_{L^2}$ is asserted under the sole assumption
$u_0\in L^2_\sigma$. The theorem concerns the same approximation
selection as Theorem~\ref{thm:eventual-regularity}.
\end{remark}

\subsection{Key ingredients in the proof}
\label{subsec:key-ingredients}

The proof follows the local-energy, decay, and covering strategy
of partial regularity theory, with modifications required by
chemotaxis and the boundary. We first construct suitable weak
solutions and obtain pressure integrability. We then derive
scale-local energy and pressure estimates and iterate them to
prove Proposition~\ref{main1}; parabolic interpolation and a
second pressure decomposition first give a small-mass
lower-order criterion. To prove Proposition~\ref{main2}, we
derive an energy bound depending only on the prescribed mass
bound, split the density into a caloric part and a small heat
potential, and recover small rescaled density mass on a smaller
cylinder. The small-mass criterion then applies.
Finally, smallness of the full dissipation at all sufficiently
small scales drives the lower-order quantities below the
regularity threshold. This proves Theorem~\ref{main4}, and a
parabolic covering argument yields Corollary~\ref{cor:full-singular-set}.
The first three ingredients below concern partial regularity; the
fourth gives the additional large-time argument.

\medskip\noindent
\textbf{1. A new local energy estimate without separate control of
$n\log n$.}
The density entropy changes sign, and under the natural scaling
$n_\rho(y,s)=\rho^2n(x_0+\rho y,t_0+\rho^2s)$ one has
\[
n_\rho\log n_\rho
=\rho^2(n\log n)(x_0+\rho y,t_0+\rho^2s)
 +2\log\rho\,n_\rho.
\]
Consequently, the localized entropy term cannot be treated as
a nonnegative, scale-invariant energy.

Weighted mass control of the negative entropy was already used
in~\cite{CLWW2023Hausdorff}, equations~(4.13)--(4.15).
Our contribution here is the closed local energy estimate and
its use in an iteration at the natural dissipation scaling.
For this purpose, we combine the entropy--energy inequality
with the weak mass identity using the same nonnegative
backward-heat-kernel weight $\psi$.

The mass identity accounts for the additional
term $2\log\rho\,n_\rho$ and bounds the weighted mass
$M_\psi(t)=\int n\psi$. It also controls the negative entropy:
for $0<\alpha<1/6$,
\[
\int_{\{0<n<1\}}(-n\log n)\psi
\leq C_\alpha M_\psi(t)^{1-\alpha}
                \left(\int\psi\right)^\alpha.
\]
Together with interpolation and
$|n\log n|\leq C_\alpha(n^{1+\alpha}+n^{1-\alpha})$,
this yields Lemma~\ref{AE}. In particular, with the
scale-invariant quantities of Section~\ref{subsec:notation-scaling},
write $X(r)=A(r)+E(r)$, where $A$ is the total local energy,
$E$ the full dissipation, and $D$ the mixed-norm pressure oscillation.
Corollary~\ref{corollary1} gives, when $X(\rho)\leq1$,
$0<\rho\leq1$, $0<\gamma\leq1/2$, and $0<\delta<1/4$,
\[
\begin{aligned}
X(\gamma\rho)\leq{}C_\delta\gamma^2X(\rho)^{1-\delta}
 +C_\delta\gamma^{-2}X(\rho)^{3/2-\delta}
+C_\delta\gamma^{-2}X(\rho)^{1/2}D(\rho).
\end{aligned}
\]
The right-hand side has no independent $n\log n$ quantity.
This is the sense in which the new local energy estimate and
the resulting smallness criteria do not require separate
entropy control. The entropy structure remains part of
suitability and of the existence argument.

At the final dissipation step, the same weighted argument gives
$A(r)+E(r)\leq C(1+G(2r))$, where
\[
G(r)=r^{-5/3}\int_{\mathcal Q_r}
\left(|u|^{10/3}+n^{5/3}
 +|\nabla\sqrt{\tilde c}|^{10/3}
 +|p-[p]_r(t)|^{5/3}\right)\,dx\,dt.
\]
The pressure estimate and interpolation produce a contracting
iteration for $G$ once the full dissipation is small.
\textbf{Local entropy bounds imply uniform integrability of $n$, hence
$rA_{\sqrt n}(r)\to0$}; this removes the remaining local-mass
contribution of the force. Equation~\eqref{eq:corrected-G-iteration}
makes this mechanism precise and leads to the natural
$r^{-1}$ criterion without a radius loss in the density term.

\medskip\noindent
\textbf{2. Odd and even extensions in the boundary H\"older estimates.}
The decay estimates must respect the different boundary conditions
of the velocity and the scalar variables. After separating off
explicit heat potentials, we extend each Dirichlet velocity
component oddly and the Neumann scalar variables $n$ and
$\tilde c$ evenly across $\Gamma$. For a scalar divergence-form
equation, an odd extension of the unknown requires odd tangential
fluxes and an even normal flux; an even extension requires the
opposite parities. These choices ensure that the reflected
equations hold distributionally, without an artificial source
on the reflecting plane. The componentwise odd velocity
reflection is used for the heat equation after the pressure
and force have been included in its right-hand side.

For the concentration, the even caloric remainder $c_2$ has
even tangential derivatives and an odd normal derivative.
Thus the decay estimate for $\nabla c_2$ subtracts only constant
vectors $a$ with $a_3=0$, preserving the Neumann symmetry.
For the density, the tangential components of
$-n(u+\nabla c)$ are extended evenly and the normal component
oddly, so the caloric remainder is even and ordinary scalar
mean oscillations can be used. Lemmas~\ref{lem:heat-potentials}
and~\ref{lem:caloric-parity-decay} quantify these constructions.
Uniform decay at boundary centers and neighboring interior
centers then yields the Campanato estimates for H\"older
continuity of $(n,\nabla c,u)$ up to $\Gamma$.

\medskip\noindent
\textbf{3. Suitable weak solutions and pressure estimates in the half-space.}
The approximation must preserve incompressibility, the no-slip
condition, and the scalar entropy cancellations simultaneously.
We regularize the velocity drift with the half-space Stokes
semigroup, use Neumann heat smoothing for the scalar initial
data, and regularize the force by the Dirichlet heat semigroup.
The paired truncations
\[
a_\tau(n)=\frac{n}{1+\tau n},\qquad
h_\tau(n)=\frac{\log(1+\tau n)}{\tau},\qquad
a_\tau(n)=n h_\tau'(n)
\]
preserve the cancellation between the density entropy and
concentration energy. The shift $\tilde c=1+c$ avoids a positive
lower-bound assumption on $c$. Uniform entropy and dissipation
bounds, local compactness, and lower semicontinuity then yield
Theorem~\ref{thm:half-space-existence}, including the localized
inequality for test functions satisfying $\partial_3\psi=0$
on $\Gamma$.

Pressure control is needed both in this limit and in the
regularity iteration. For existence, the global half-space
Stokes decomposition separates the zero-initial-data forced
part from the homogeneous evolution of the initial velocity.
Maximal regularity, spatial Sobolev estimates, and Stokes
semigroup smoothing give the local $L^{5/3}$ pressure bound
on every time interval compactly contained in $(0,\infty)$
(Proposition~\ref{prop:uniform-pressure}). For local regularity,
an inhomogeneous Stokes problem on a smooth auxiliary domain
is combined with estimates for the homogeneous remainder.
The force carries the factor $\rho M_\phi$ after scaling,
which is retained in Lemma~\ref{p} and the ensuing iteration.
These pressure estimates connect the energy criterion,
the lower-order criterion, and the final dissipation criterion.

\medskip\noindent
\textbf{4. Eventual regularity without a Poincar\'e inequality.}
On the half-space the signal has a drift-independent $t^{-3/2}$
decay estimate, obtained from mass, the Nash inequality, and duality.
Once the signal is small, a joint test with a weighted power of the
density gives heat-like $L^2$ and $L^6$ decay and an integrable tail
for $\|\nabla n(t)\|_{L^2}$. To justify this test, we first remove the
velocity smoothing at fixed density truncation, so that both scalar
equations have the same drift, and then remove the truncation.
Global spatial tightness retains the kinetic-energy inequality at
almost every starting time. The identity
$-n\nabla\phi=-\nabla(n\phi)+\phi\nabla n$ and boundedness of
$\phi$ then turn the density-gradient estimate into a time-integrable
effective force. A small-enstrophy restart, weak--strong uniqueness
for the velocity equation with this fixed force, and flat-boundary
scalar--Stokes estimates complete the proof of
Theorem~\ref{thm:eventual-regularity}.

For the large-time behavior, an energy comparison with the homogeneous
Stokes evolution first gives decay of the kinetic energy.
Time averaging of the eventual enstrophy inequality yields
$\|\nabla u(t)\|_{L^2}=o(t^{-1/2})$. Long time intervals in the scalar
mild formulas give the heat-scale density decay and
$\|\nabla c(t)\|_{L^{\infty}}=O(t^{-2})$. A relative-entropy comparison
with the passive signal prevents complete consumption at any finite
time; integrability of the eventual density maximum preserves a
positive limiting signal mass. Finally, the scalar divergence fluxes
are asymptotically negligible compared with diffusion, while the
integrable absorption changes only the mass of the signal Gaussian.
This proves Theorem~\ref{thm:large-time-asymptotics}, including the
$L^1$ and $L^\infty$ endpoint profile statements without moment
assumptions.

For background on the linear analytic tools used in
this framework, we refer to
Krylov~\cite{K2001} and Dong and Kim~\cite{DK2011}
for parabolic $L^p$ theory,
Seeley~\cite{Seeley1972} for interpolation with boundary
conditions, and
Denk, Hieber, and Pr\"uss~\cite{DHP2003}
for maximal regularity.

	\section{Preliminaries for the main results}\label{section2}
    \subsection{Notation and scale-invariant quantities}
\label{subsec:notation-scaling}

We first introduce the notation used throughout this section. Set
\[
B_r:=B(0,r),\qquad
B_r^+:=B_r\cap\mathbb H^3,
\qquad
B:=B_1,\qquad
B^+:=B_1^+,
\]
and
\[
Q_r:=(-r^2,0)\times B_r,
\qquad
Q_r^+:=(-r^2,0)\times B_r^+,
\qquad
Q:=Q_1,\qquad
Q^+:=Q_1^+.
\]

For $q,r\in[1,\infty]$, we use the mixed-norm spaces
\[
L_t^qL_x^r
=
L^q(0,T;L^r(\Omega)),
\]
together with
\[
W^{1,2}_{q,r}
=
\left\{
u\in L_t^qL_x^r:
\partial_tu,\ D_xu,\ D_x^2u\in L_t^qL_x^r
\right\},
\]
and
\[
W^{0,1}_{q,r}
=
\left\{
u\in L_t^qL_x^r:
D_xu\in L_t^qL_x^r
\right\},
\]
equipped with their natural norms.

Throughout this section, admissible radii and the corresponding physical
cylinders
\[
\mathcal B_\rho(x_0),
\qquad
\mathcal Q_\rho(z_0)
\]
are understood in the sense introduced in Section~\ref{section1}.

For the iterative arguments, it is convenient to distinguish these
physical cylinders from the normalized reference cylinders. After fixing
$z_0$ and rescaling an admissible physical cylinder to unit size, we write
\[
\mathcal Q_r:=(-r^2,0)\times\mathcal B_r,
\qquad 0<r\leq1,
\]
where
\[
\mathcal B_r=
\begin{cases}
B_r, & \text{in the interior case},\\[1mm]
B_r^+, & \text{in the boundary case}.
\end{cases}
\]
We also set
\[
\mathcal B:=\mathcal B_1,
\qquad
\mathcal Q:=\mathcal Q_1.
\]
Thus $\mathcal Q_r$ denotes a full cylinder in the normalized interior
case and a half-cylinder in the normalized boundary case.

For a space--time point $z_0=(x_0,t_0)$ and an admissible physical radius
$\rho>0$, define
\[
A_u(z_0,\rho)
=
\operatorname*{ess\,sup}_{t_0-\rho^2<t<t_0}
\frac1\rho
\int_{\mathcal B_\rho(x_0)}|u(x,t)|^2\,dx,
\quad
E_u(z_0,\rho)
=
\frac1\rho
\int_{\mathcal Q_\rho(z_0)}
|\nabla u|^2\,dx\,dt,
\]
\[
A_{\sqrt n}(z_0,\rho)
=
\operatorname*{ess\,sup}_{t_0-\rho^2<t<t_0}
\frac1\rho
\int_{\mathcal B_\rho(x_0)}n(x,t)\,dx,
\quad
E_{\sqrt n}(z_0,\rho)
=
\frac1\rho
\int_{\mathcal Q_\rho(z_0)}
|\nabla\sqrt n|^2\,dx\,dt,
\]
\[
A_{\nabla\sqrt{\tilde c}}(z_0,\rho)
=
\operatorname*{ess\,sup}_{t_0-\rho^2<t<t_0}
\frac1\rho
\int_{\mathcal B_\rho(x_0)}
|\nabla\sqrt{\tilde c}|^2\,dx,
\quad
E_{\nabla\sqrt{\tilde c}}(z_0,\rho)
=
\frac1\rho
\int_{\mathcal Q_\rho(z_0)}
|\nabla^2\sqrt{\tilde c}|^2\,dx\,dt.
\]

For the pressure, let
\[
D_{q,r}(z_0,\rho)
:=
\rho^{\,2-\frac2q-\frac3r}
\left\|
p-[p]_{x_0,\rho}(t)
\right\|_{L_t^qL_x^r(\mathcal Q_\rho(z_0))}, 
\]
where $[p]_{x_0,\rho}(t)$ is defined in Proposition \ref{main1}. For $f\in\{u,\sqrt n,\nabla\sqrt{\tilde c}\}$,
we further set
\[
C_f^{q,r}(z_0,\rho)
:=
\rho^{\,1-\frac2q-\frac3r}
\|f\|_{L_t^qL_x^r(\mathcal Q_\rho(z_0))}.
\]

When the center $z_0$ is fixed and no confusion can arise, we suppress it
from the notation and write
\[
A_u(\rho),\quad
E_u(\rho),\quad
D_{q,r}(\rho),\quad
C_f^{q,r}(\rho),
\]
and similarly for the other quantities. In normalized coordinates, we
also write
\[
[p]_\rho(t)
:=
\fint_{\mathcal B_\rho}p(x,t)\,dx.
\]
Thus the average is taken over $B_\rho$ in the interior case and over
$B_\rho^+$ in the boundary case.

For convenience, define
\begin{align*}
A_{u,\sqrt n,\nabla\sqrt{\tilde c}}(\rho)
&=
A_u(\rho)
+A_{\sqrt n}(\rho)
+A_{\nabla\sqrt{\tilde c}}(\rho),\\
E_{u,\sqrt n,\nabla\sqrt{\tilde c}}(\rho)
&=
E_u(\rho)
+E_{\sqrt n}(\rho)
+E_{\nabla\sqrt{\tilde c}}(\rho),\\
C_{u,\sqrt n,\nabla\sqrt{\tilde c}}^{q,r}(\rho)
&=
C_u^{q,r}(\rho)
+C_{\sqrt n}^{q,r}(\rho)
+C_{\nabla\sqrt{\tilde c}}^{q,r}(\rho).
\end{align*}
More generally, comma-separated subscripts on $A$, $E$,
and $C^{q,r}$ denote the sum of the corresponding individual
quantities; for example,
\[
A_{u,\sqrt n}(\rho)=A_u(\rho)+A_{\sqrt n}(\rho),
\qquad
E_{u,\nabla\sqrt{\tilde c}}(\rho)
=E_u(\rho)+E_{\nabla\sqrt{\tilde c}}(\rho).
\]
These quantities are invariant under the natural scaling of
\eqref{che},
\[
u_\lambda(t,x)
=
\lambda u(\lambda^2t,\lambda x),
\qquad
n_\lambda(t,x)
=
\lambda^2 n(\lambda^2t,\lambda x),
\]
\[
c_\lambda(t,x)
=
c(\lambda^2t,\lambda x),
\qquad
p_\lambda(t,x)
=
\lambda^2p(\lambda^2t,\lambda x),
\qquad
\phi_\lambda(x)
=
\phi(\lambda x).
\]

The boundary case will be written explicitly in the iterative estimates
below. Set
\[
\rho_k:=1-2^{-k},
\qquad k\geq1,
\]
and define
\[
B_k^+:=B_{\rho_k}^+,
\qquad
Q_k^+:=Q_{\rho_k}^+.
\]
We also write
\[
A_k
:=
A_{u,\sqrt n,\nabla\sqrt{\tilde c}}(\rho_k),
\qquad
E_k
:=
E_{u,\sqrt n,\nabla\sqrt{\tilde c}}(\rho_k).
\]

For each integer $k\geq0$, choose a domain
$\tilde B_k\subset\mathbb R^3$ with smooth boundary such that
\[
B_{k+1}^+
\subset
\tilde B_k
\subset
B_{k+2}^+,
\]
and such that the $C^2$ norm of $\partial\tilde B_k$ is bounded by
$C\,2^k$. We further set
\[
\tilde Q_k
:=
(-\rho_{k+2}^2,0)\times\tilde B_k.
\]
For a space--time cylinder $I\times\Omega$, where $I=(a,b)$,
we denote its parabolic boundary by
\[
\partial_p(I\times\Omega)
:=
\bigl(\{a\}\times\overline{\Omega}\bigr)
\cup
\bigl((a,b)\times\partial\Omega\bigr).
\]
Thus $\partial_p(I\times\Omega)$ consists of the initial
time slice and the lateral spatial boundary, and does not
include the terminal time slice.

   	\begin{lemma}(See \cite[Theorem 1.1]{MS1995} and \cite[Lemma 4.4]{DW2020})\label{stokes}
   		Let $q_0$, $r_0\in(1,\infty)$ be fixed numbers. Assume that $g\in L_t^{q_0}L_x^{r_0}(\tilde{Q}_k)$. Then there exists a unique function pair $(v,p)$, which satisfies the following equations:
   		\begin{equation}
   			\begin{aligned}
   				\begin{cases}
   					\partial_tv-\Delta v+\nabla p=g \quad &\text{in}~\tilde{Q}_k,\\
   					\nabla\cdot v=0 & \text{in}~ \tilde{Q}_k,\\
   					[p]_{\tilde{B}_k}=0 &\text{for ~a.e.}~t\in[-\rho_{k+2}^2,0],\\
   					v=0 &\text{on}~\partial_p\tilde{Q}_k.
   				\end{cases}
   			\end{aligned}
   		\end{equation}
   		Moreover, $v$ and $p$ satisfy the following estimate:
   		\begin{align}
   			\|v\|_{W^{1,2}_{q_0,r_0}(\tilde{Q}_k)}+\|p\|_{W^{0,1}_{q_0,r_0}(\tilde{Q}_k)}\leq C2^{bk}\|g\|_{L_t^{q_0}L_x^{r_0}(\tilde{Q}_k)},
   		\end{align}
   		where the constants $C$ and $b$ depend only on $q_0$, $r_0$.

The factor $2^{bk}$ follows by tracking the spatial localization
constants, as recorded in \cite[Lemma 4.4]{DW2020}.
The time lengths $\rho_{k+2}^2\in[9/16,1)$ introduce no further
$k$-dependence: use the zero-initial-data estimate on $(-1,0)$,
extend the forcing by zero before $-\rho_{k+2}^2$, and restrict
the resulting solution to $\tilde Q_k$.

   	\end{lemma}
   	\begin{lemma}(See \cite[Lemma 4.5]{DW2020})\label{Seregin}
   		Let $q_0$, $r_0\in(1,\infty)$ be constants and $g\in L_t^{q_0}L_x^{r_0}(Q_{k+1}^+)$. Assume that the functions $v\in W^{0,1}_{q_0,1}(Q_{k+1}^+)$ and $p\in L_t^{q_0}L_x^1(Q_{k+1}^+)$ satisfy the equations
   		\begin{equation}
   			\begin{aligned}
   				\begin{cases}
   					\partial_tv-\Delta v+\nabla p=g\quad &\text{in}~Q_{k+1}^+,\\
   					\nabla\cdot v=0,&\text{in}~Q_{k+1}^+.
   				\end{cases}
   			\end{aligned}
   		\end{equation}
   		and the boundary condition
   		\begin{align*}
   			v=0\quad on~[-\rho_{k+1}^2,0)\times \{y|y=(y',0),|y'|<\rho_{k+1}\}.
   		\end{align*}
   		Then we have $v\in W^{1,2}_{q_0,r_0}(Q_k^+)$, $p\in W^{0,1}_{q_0,r_0}(Q_k^+)$, and 
   		\begin{align*}
   			&\|v\|_{W^{1,2}_{q_0,r_0}(Q_k^+)}+\|p\|_{W^{0,1}_{q_0,r_0}(Q_k^+)}\\&\quad\leq C2^{bk}\left(\|g\|_{L_t^{q_0}L_x^{r_0}(Q_{k+1}^+)}+\|v\|_{L_t^{q_0}L_x^1(Q_{k+1}^+)}+\|p\|_{L_t^{q_0}L_x^1(Q_{k+1}^+)}\right),
   		\end{align*}
   		where the constants $C$ and $b$ only depend on $q_0$ and $r_0$.
   	\end{lemma}
   	\begin{lemma}(See \cite[Chapter III, Lemma 2.1]{G1983})\label{it1}
   		Let $f:(0,\rho_0]\to[0,\infty)$ and suppose $f(\rho_0)\leq C_0$. If there exist $\alpha>\beta>0$ and $C_1$, $C_2>0$ such that for any $0<\gamma<\rho\leq \rho_0$, it holds that
   		\begin{align*}
   			f(\gamma)\leq C_1\left(\frac{\gamma}{\rho}\right)^\alpha f(\rho)+C_2\rho^\beta,
   		\end{align*}
   		then there exist constants $C_3$, $C_4>0$ depending on $C_0$, $C_1$, $C_2$, $\alpha$, and $\beta$ such that 
   		\begin{align*}
   			f(\gamma)\leq C_3\left(\frac{\gamma}{\rho_0}\right)^\beta f(\rho_0)+C_4\gamma^\beta
   		\end{align*}
   		for any $\gamma\in (0,\rho_0]$.
   	\end{lemma}

The following elementary heat-kernel estimates specify the auxiliary
problems used in Section~\ref{subsec:holder-regularity}. They also make
the dependence on the radius explicit.

\begin{lemma}[Heat potentials in mixed norms]
\label{lem:heat-potentials}
Let $1<q,r<\infty$, $I_R=(-R^2,0)$, and
$H_t(x)=(4\pi t)^{-3/2}\exp(-|x|^2/(4t))$ for $t>0$.
For $F=(F_1,F_2,F_3)$ and $g$ in
$L^q(I_R;L^r(\mathbb R^3))$, define
\[
\begin{split}
\mathcal T_R F(t)
&=\sum_{j=1}^3\int_{-R^2}^t
       (\partial_j H_{t-a})*F_j(a)\,da,\\
\mathcal S_R g(t)
&=\int_{-R^2}^t H_{t-a}*g(a)\,da.
\end{split}
\]
These are the zero-initial-data distributional solutions with
right-hand sides $\operatorname{div}F$ and $g$, respectively, and
\begin{align}
\|\mathcal T_RF\|_{L^q(I_R;L^r)}
&\leq CR\|F\|_{L^q(I_R;L^r)},
\label{eq:heat-potential-div}\\
\|\mathcal S_Rg\|_{L^q(I_R;L^r)}
+R\|\nabla\mathcal S_Rg\|_{L^q(I_R;L^r)}
&\leq CR^2\|g\|_{L^q(I_R;L^r)}.
\label{eq:heat-potential-scalar}
\end{align}
The constants are independent of $R$. Under reflection in $x_3=0$,
$\mathcal T_RF$ is odd if $F_1,F_2$ are odd and $F_3$ is even;
it is even if these parities are reversed. If $g$ is even, then
$\mathcal S_Rg$ is even. Data given on $B_R$ may be extended by
zero outside $B_R$ in these statements.
\end{lemma}
\begin{proof}
The identities for the equations follow first for smooth data from
Duhamel's formula and then by approximation in the indicated mixed
norms. Spatial Young's inequality and
\[
\|H_t\|_{L^1(\mathbb R^3)}=1,
\qquad \|\nabla H_t\|_{L^1(\mathbb R^3)}\leq Ct^{-1/2}
\]
reduce the estimates to one-dimensional Young's inequality in time.
Indeed, the $L^1(0,R^2)$ norms of $1$ and $t^{-1/2}$ are $R^2$
and $2R$, respectively. Reflection commutes with convolution by
$H_t$, whereas a derivative in the normal direction reverses parity.
This proves all the symmetry assertions. No boundary condition is
imposed on the artificial sphere: these are whole-space auxiliary
solutions, subsequently restricted to $Q_R$ or $Q_R^+$.
\end{proof}

\begin{lemma}[Caloric decay compatible with the flat boundary]
\label{lem:caloric-parity-decay}
Let $1<q,r<\infty$.
For $0<\gamma\leq R/2$, the following assertions hold.
\begin{enumerate}
\renewcommand{\labelenumi}{(\roman{enumi})}
\item If the odd extension of $v$ is caloric in $Q_R$, then
\begin{equation}
\|v\|_{L_t^qL_x^r(Q_\gamma^+)}
\leq C(\gamma/R)^{2/q+3/r+1}\|v\|_{L_t^qL_x^r(Q_R^+)}.
\label{eq:caloric-zero-trace}
\end{equation}
\item If the even extension of $z$ is caloric in $Q_R$, put
\[
\Phi_z(a)=\inf_{b\in\mathbb R^3,\ b_3=0}
             \|\nabla z-b\|_{L_t^qL_x^r(Q_a^+)}.
\]
Then
\begin{equation}
\Phi_z(\gamma)\leq C(\gamma/R)^{2/q+3/r+1}\Phi_z(R).
\label{eq:caloric-neumann-gradient}
\end{equation}
\item For a caloric function on a full cylinder, or a function on
a half-cylinder whose even extension is caloric, its ordinary
mean oscillation satisfies
\begin{equation}
\|w-[w]_\gamma\|_{L_t^qL_x^r(Q_\gamma)}
\leq C(\gamma/R)^{2/q+3/r+1}\|w-b\|_{L_t^qL_x^r(Q_R)}
\label{eq:caloric-ordinary-oscillation}
\end{equation}
for every constant $b$. Here the norms and the space--time means
are taken on the corresponding full or half cylinders.
\end{enumerate}
All assertions hold componentwise and with constants depending only
on $q,r$. They extend to $0<\gamma<R$ after increasing $C$.
\end{lemma}
\begin{proof}
The interior derivative estimate for a caloric function $W$ is
\begin{equation}
\|\nabla W\|_{L^\infty(Q_{R/2})}
+R\|\partial_tW\|_{L^\infty(Q_{R/2})}
\leq C R^{-1-2/q-3/r}\|W\|_{L_t^qL_x^r(Q_R)}.
\label{eq:caloric-interior-derivatives}
\end{equation}
For completeness, at unit scale take a spatial cutoff
$\chi\in C_c^\infty(B_1)$ equal to one on $B_{3/4}$ and a time
cutoff $\zeta$ vanishing near $-1$ and equal to one on $[-9/16,0]$.
Since $W$ is caloric, $V=\chi\zeta W$ has zero initial value and
\[
(\partial_t-\Delta)V
=(\chi\zeta'+\zeta\Delta\chi)W
 -2\operatorname{div}(\zeta W\nabla\chi).
\]
Duhamel's formula and integration by parts place all spatial
derivatives on the heat kernel and the cutoffs. At evaluation
points in $Q_{1/2}$, the term containing $\zeta'$ is separated
in time, while all other source terms are separated in space.
The heat kernel and its required derivatives are therefore
uniformly bounded on these source regions, even as the time lag
tends to zero. This bounds the indicated derivatives by
$C\|W\|_{L^1(Q_1)}$. H\"older's inequality gives the mixed-norm
bound. The argument applies to distributional caloric functions
by local smoothing. Parabolic scaling yields
\eqref{eq:caloric-interior-derivatives}. The estimates, initially
obtained below time zero, also give the one-sided continuous limits
needed at that time.

For (i), apply this estimate to the odd extension $v^*$.
Since $v^*(x',0,t)=0$, integration in $x_3$ gives
$|v(x,t)|\leq C\gamma R^{-1-2/q-3/r}\|v\|_{L_t^qL_x^r(Q_R^+)}$
on $Q_\gamma^+$. The mixed norm of a constant there is a constant
multiple of $\gamma^{2/q+3/r}$, proving \eqref{eq:caloric-zero-trace}.

For (ii), let $z^*$ be the even extension and fix $b_3=0$.
Every component of $W=\nabla z^*-b$ is caloric. The first two
components are even and the third is odd, so
\[
\|W\|_{L_t^qL_x^r(Q_R)}
=2^{1/r}\|\nabla z-b\|_{L_t^qL_x^r(Q_R^+)}.
\]
Moreover, $\partial_3z^*(0,0)=0$. With
$b_\gamma=\nabla z^*(0,0)$, the derivative estimate and
$|t|\leq\gamma^2$ imply
\[
\|\nabla z-b_\gamma\|_{L_t^qL_x^r(Q_\gamma^+)}
\leq C(\gamma/R)^{2/q+3/r+1}\|\nabla z-b\|_{L_t^qL_x^r(Q_R^+)}.
\]
Taking the infimum over the admissible $b$ proves (ii).
For (iii), apply the same argument to $w-b$ and use
\[
\|f-[f]_a\|_{L_t^qL_x^r(Q_a)}\leq2\|f-b\|_{L_t^qL_x^r(Q_a)}.
\]
This last inequality follows from H\"older's inequality for the
space--time average. In the even case every scalar constant
preserves parity. Finally, restriction of norms and the same
mean estimate give the assertions for $R/2<\gamma<R$.
\end{proof}

\section{Proof of Theorem~\ref{thm:half-space-existence}}
\label{sec:existence}

This section proves the global existence theorem stated in
Section~\ref{subsec:solution-framework}. We work on the whole
half-space and use boundary-compatible approximation, uniform
entropy and pressure estimates, and compactness. The local
energy inequality is obtained with test functions supported up
to $\Gamma$, without introducing an artificial boundary condition
on the spherical part of a half-ball.

Throughout this section, the initial data satisfy
\eqref{eq:initial-data-half-space}, and $L_\sigma^q(\mathbb H^3)$
and $H(s)=(s+1)\log(s+1)$ have the meanings fixed in
Section~\ref{subsec:solution-framework}.
In the notation of \cite{CLWW2023}, the chemotactic sensitivity
is identically one and the consumption rate is linear, with coefficient one.
Unless a subscript is displayed, every generic constant $C$
is independent of $\varepsilon$ and $\tau$. Dependence on the
regularization parameters is recorded explicitly, for example
by $C_\varepsilon$, $C_\tau$, or $C_{\varepsilon,\tau}$.

\begin{remark}[Compatibility of the boundary test functions]
\label{rem:test-function-compatibility}

The condition $\partial_3\psi=0$ on $\Gamma$ in
Definition~\ref{def:suitable-half-space}(iv) is built into the
local energy calculation below. For the scalar Neumann equations,
the integrations by parts that convert
$\nabla(n\log n)\cdot\nabla\psi$ and the corresponding
concentration terms into Laplacian terms would otherwise
produce additional integrals on $\Gamma$. The prescribed
compatibility condition removes these boundary-flux terms.

\end{remark}

\subsection{Boundary-compatible approximation}
\label{subsec:boundary-compatible-approximation}

We use the half-space heat semigroups. Let
\[
G(x,t):=(4\pi t)^{-3/2}\exp\left(-\frac{|x|^2}{4t}\right).
\]
The Neumann and Dirichlet semigroups on $\mathbb H^3$ are given by
\begin{equation}\label{eq:half-space-heat-kernels}
S_{N,D}(t)f(x)
=\int_{\mathbb H^3}
\bigl[G(x-y,t)\pm G(x-(y',-y_3),t)\bigr]f(y)\,dy,
\end{equation}
where the plus sign corresponds to $S_N$ and the minus sign to $S_D$.
For \(1\leq r\leq q\leq\infty\) and every integer \(m\geq0\),
\begin{equation}
\label{eq:half-space-heat-estimates}
\|\nabla^mS_{N,D}(t)f\|_{L^q(\mathbb H^3)}\leq C t^{-\frac m2-\frac32(\frac1r-\frac1q)}\|f\|_{L^r(\mathbb H^3)}.
\end{equation}
If \(F\cdot\nu=0\) on \(\Gamma\), the solution operator for \(z_t-\Delta z=-\nabla\cdot F\), \(\partial_3z=0\), is given by
\begin{equation}
\label{eq:neumann-divergence-operator}
\mathcal T_N(t)F(x)
:=\int_{\mathbb H^3}\nabla_y
\bigl[G(x-y,t)+G(x-(y',-y_3),t)\bigr]\cdot F(y)\,dy,
\end{equation}
and it satisfies the estimate obtained from \eqref{eq:half-space-heat-estimates} with one spatial derivative.

The velocity cannot be regularized by a standard reflection followed by convolution. Such an operation does not, in general, preserve both the full no-slip trace and the solenoidal constraint. We instead use the Stokes semigroup. Let \(\mathbb P_+\) be the Helmholtz projection in \(\mathbb H^3\), let \(A=-\mathbb P_+\Delta\) be the Dirichlet Stokes operator, and define
\begin{equation}
\label{eq:stokes-smoother}
\mathcal J_\varepsilon v:=e^{-\varepsilon A}v.
\end{equation}
The following properties are the reason for this choice.

\begin{lemma}[Boundary-compatible velocity smoothing]
\label{lem:stokes-smoothing}
Let \(1<q<\infty\), \(m\in\mathbb N_0\), and \(v\in L^q_\sigma(\mathbb H^3)\). Then
\begin{equation}
\label{eq:stokes-smoothing-properties}
\nabla\cdot\mathcal J_\varepsilon v=0\quad\text{in }\mathbb H^3,
\qquad
\mathcal J_\varepsilon v=0\quad\text{on }\Gamma,
\end{equation}
and
\begin{equation}
\label{eq:stokes-smoothing-estimates}
\|\mathcal J_\varepsilon v\|_{L^q}\leq C_q\|v\|_{L^q},
\qquad
\|\nabla^m\mathcal J_\varepsilon v\|_{L^q}\leq C_{m,q}\varepsilon^{-m/2}\|v\|_{L^q}.
\end{equation}
Moreover, \(\mathcal J_\varepsilon v\to v\) strongly in \(L^q_\sigma(\mathbb H^3)\) as \(\varepsilon\to0\). If \(v\in H^1_0(\mathbb H^3)\cap L^2_\sigma(\mathbb H^3)\), then
\begin{equation}
\label{eq:stokes-smoothing-difference}
\|\mathcal J_\varepsilon v-v\|_{L^2}\leq C\varepsilon^{1/2}\|\nabla v\|_{L^2}.
\end{equation}
\end{lemma}

\begin{proof}
For clarity, let \(A_q\) denote the Dirichlet Stokes operator on \(L^q_\sigma(\mathbb H^3)\). The semigroups generated by \(-A_q\) are consistent on the intersections of the spaces \(L^q_\sigma\), so we continue to write \(A\). If \(w(s):=e^{-sA_q}v\), then \((w,\pi)\) is the unique solution of
\begin{equation}
\label{eq:homogeneous-stokes-for-smoother}
\partial_sw-\Delta w+\nabla\pi=0,\qquad \nabla\cdot w=0\quad\text{in }\mathbb H^3,\qquad w|_\Gamma=0,
\end{equation}
with initial value \(v\). Hence \(e^{-sA_q}L^q_\sigma\subset L^q_\sigma\) for every \(s\geq0\), and, for \(s>0\), analyticity gives \(w(s)\in D(A_q^k)\) for every integer \(k\geq1\). Since every element of \(D(A_q)\) has zero Dirichlet trace and is solenoidal, this proves both identities in \eqref{eq:stokes-smoothing-properties}.

The analytic-semigroup estimate reads
\begin{equation}
\label{eq:analytic-stokes-power}
\|A_q^\theta e^{-sA_q}\|_{\mathcal L(L^q_\sigma)}\leq C_{q,\theta}s^{-\theta},\qquad s>0,\quad \theta\geq0.
\end{equation}
For integer \(k\geq1\), the homogeneous half-space Stokes estimate gives
\begin{equation}
\label{eq:stokes-elliptic-graph}
\|\nabla^{2k}z\|_{L^q}\leq C_{k,q}\|A_q^kz\|_{L^q},\qquad z\in D(A_q^k).
\end{equation}
Interpolation of the homogeneous graph norms yields, for every integer \(m\geq1\),
\[
\|\nabla^m z\|_{L^q}\leq C_{m,q}\|A_q^{m/2}z\|_{L^q},\qquad z\in D(A_q^{m/2}).
\]
Applying this inequality to \(z=e^{-\varepsilon A_q}v\) and then using \eqref{eq:analytic-stokes-power} with \(\theta=m/2\) gives
\[
\|\nabla^m e^{-\varepsilon A_q}v\|_{L^q}\leq C_{m,q}\varepsilon^{-m/2}\|v\|_{L^q},\qquad \varepsilon>0.
\]
The case \(m=0\) follows from boundedness of the analytic semigroup. This proves the full family of estimates in \eqref{eq:stokes-smoothing-estimates}.

Because \(-A_q\) generates a strongly continuous semigroup on \(L^q_\sigma\), we have \(e^{-\varepsilon A_q}v\to v\) strongly in \(L^q_\sigma\) as \(\varepsilon\downarrow0\). It remains to prove the quantitative \(L^2\) difference estimate. The operator \(A=A_2\) is positive and self-adjoint on \(L^2_\sigma\), and its closed quadratic form is
\[
(Az,z)_{L^2}=\|\nabla z\|_{L^2}^2,
\qquad D(A^{1/2})=H^1_0(\mathbb H^3;\mathbb R^3)\cap L^2_\sigma(\mathbb H^3).
\]
For \(v\in D(A^{1/2})\), the fundamental theorem of calculus and commutation of the spectral powers of \(A\) give
\begin{align*}
(I-e^{-\varepsilon A})v
&=\int_0^\varepsilon Ae^{-sA}v\,ds
=\int_0^\varepsilon A^{1/2}e^{-sA}A^{1/2}v\,ds.
\end{align*}
Using \(\|A^{1/2}e^{-sA}\|_{\mathcal L(L^2,L^2)}\leq Cs^{-1/2}\), we conclude that
\[
\|(I-e^{-\varepsilon A})v\|_{L^2}
\leq C\int_0^\varepsilon s^{-1/2}\,ds\,\|A^{1/2}v\|_{L^2}
\leq C\varepsilon^{1/2}\|\nabla v\|_{L^2}.
\]
This is \eqref{eq:stokes-smoothing-difference}.
\end{proof}

\begin{remark}
\label{rem:no-reflection-mollifier}
The normal component is the obstruction to the reflection-convolution construction. Even if an even extension of \(u_3\) vanishes at \(x_3=0\), its convolution with a symmetric kernel need not vanish there. The operator \(\mathcal J_\varepsilon\) avoids this problem and gives both statements in \eqref{eq:stokes-smoothing-properties} exactly, not merely in the limit.
\end{remark}

For \(\tau>0\), define
\begin{equation}
\label{eq:regularizing-functions}
a_\tau(s):=\frac{s}{1+\tau s},\qquad h_\tau(s):=\frac1\tau\log(1+\tau s),\qquad s\geq0.
\end{equation}
The two functions are linked by
\begin{equation}
\label{eq:regularizing-identity}
a_\tau(s)=s h_\tau'(s),\qquad 0\leq a_\tau(s)\leq s,\qquad 0\leq h_\tau(s)\leq s,
\qquad s\bigl(h_\tau'(s)\bigr)^2\leq h_\tau(s).
\end{equation}
The first identity is essential for the cancellation in the local energy inequality. For the fixed-point argument only, choose smooth extensions \(\widehat a_\tau,\widehat h_\tau\in C^3(\mathbb R)\) which agree with \(a_\tau,h_\tau\) on \([0,\infty)\), satisfy \(\widehat a_\tau(0)=\widehat h_\tau(0)=0\), and obey
\begin{equation}
\label{eq:regularizing-extensions}
|\widehat a_\tau(s)|+|\widehat h_\tau(s)|\leq C_\tau|s|\quad(s\leq0),
\qquad
\sum_{j=1}^3\bigl(\|\widehat a_\tau^{(j)}\|_{L^\infty}+\|\widehat h_\tau^{(j)}\|_{L^\infty}\bigr)\leq C_\tau.
\end{equation}
For example, choose $\chi\in C^\infty(\mathbb R)$ with $0\leq\chi\leq1$,
$\chi=0$ on $(-\infty,-1/2]$ and $\chi=1$ on $[-1/4,\infty)$, and set
\[
\bigl(\widehat a_\tau(s),\widehat h_\tau(s)\bigr)
:=\begin{cases}
\displaystyle\chi(\tau s)\left(\frac{s}{1+\tau s},
\frac{\log(1+\tau s)}{\tau}\right),&s>-\dfrac1{2\tau},\\[6pt]
(0,0),&s\leq-\dfrac1{2\tau}.
\end{cases}
\]
The cutoff and all its derivatives vanish at $s=-1/(2\tau)$,
where the rational and logarithmic factors are smooth.
Thus the two branches join smoothly. Near $s=0$ the cutoff is identically
one, so the extensions agree with the same analytic formulas on both
sides of zero. In particular, they belong to $C^\infty(\mathbb R)$,
hence to $C^3(\mathbb R)$.
Their first three derivatives are bounded: on the negative half-line
they vanish outside $[-1/(2\tau),0]$, and on $[0,\infty)$ they are
the bounded derivatives of the original formulas. Their vanishing at zero and the mean value theorem then
give the first bound in \eqref{eq:regularizing-extensions}.
Once nonnegativity of \(n^{\varepsilon,\tau}\) is proved, the extensions
coincide with the original regularizing functions and disappear from the equations.

We also use the componentwise Dirichlet heat smoother
\begin{equation}
\label{eq:force-smoother}
\mathcal K_\varepsilon F:=S_D(\varepsilon)F.
\end{equation}
It is bounded on every \(L^q\), is strongly continuous as \(\varepsilon\to0\) for $1\leq q<\infty$, vanishes on \(\Gamma\), and supplies the half-space analogue of the mollification of the buoyancy force used in \cite{CLWW2023}.

Let
\begin{equation}
\label{eq:regularized-initial-data}
n_0^\varepsilon:=S_N(\varepsilon)n_0,
\qquad
c_0^\varepsilon:=\bigl(S_N(\varepsilon)\sqrt{c_0}\bigr)^2,
\qquad
u_0^\varepsilon:=\mathcal J_\varepsilon u_0.
\end{equation}
These data are smooth, preserve positivity, satisfy the required boundary conditions, and converge to the original data in their natural spaces. The Markov property of \(S_N\), the \(L^2\)-contractivity of \(S_N\), and the contractivity of the Stokes semigroup imply
\begin{equation}
\label{eq:initial-data-uniform-control}
\begin{aligned}
&\|n_0^\varepsilon\|_{L^1}=\|n_0\|_{L^1},
\qquad \int_{\mathbb H^3}H(n_0^\varepsilon)\,dx\leq\int_{\mathbb H^3}H(n_0)\,dx,\\
&\|c_0^\varepsilon\|_{L^1}\leq\|c_0\|_{L^1},
\qquad \|c_0^\varepsilon\|_{L^\infty}\leq\|c_0\|_{L^\infty},
\qquad \|\nabla\sqrt{c_0^\varepsilon}\|_{L^2}\leq\|\nabla\sqrt{c_0}\|_{L^2},\\
&\|u_0^\varepsilon\|_{L^2}\leq\|u_0\|_{L^2}.
\end{aligned}
\end{equation}

The proof of existence of suitable weak solutions is based on the following global well-posedness result of the following regularized system:
\begin{equation}
\label{eq:regularized-half-space-system}
\left\{
\begin{aligned}
&\partial_t n^{\varepsilon,\tau}+(\mathcal J_\varepsilon u^{\varepsilon,\tau})\cdot\nabla n^{\varepsilon,\tau}-\Delta n^{\varepsilon,\tau}
=-\nabla\cdot\left(a_\tau(n^{\varepsilon,\tau})\nabla c^{\varepsilon,\tau}\right),\\
&\partial_t c^{\varepsilon,\tau}+u^{\varepsilon,\tau}\cdot\nabla c^{\varepsilon,\tau}-\Delta c^{\varepsilon,\tau}
=-h_\tau(n^{\varepsilon,\tau})c^{\varepsilon,\tau},\\
&\partial_t u^{\varepsilon,\tau}+(\mathcal J_\varepsilon u^{\varepsilon,\tau})\cdot\nabla u^{\varepsilon,\tau}-\Delta u^{\varepsilon,\tau}+\nabla p^{\varepsilon,\tau}
=-\mathcal K_\varepsilon(n^{\varepsilon,\tau}\nabla\phi),\\
&\nabla\cdot u^{\varepsilon,\tau}=0,
\end{aligned}
\right.
\end{equation}
The boundary and initial conditions are
\begin{equation}
\label{eq:regularized-boundary-initial}
u^{\varepsilon,\tau}=0,
\qquad
\partial_3n^{\varepsilon,\tau}=\partial_3c^{\varepsilon,\tau}=0
\quad\text{on }\Gamma,
\qquad
(n^{\varepsilon,\tau},c^{\varepsilon,\tau},u^{\varepsilon,\tau})|_{t=0}
=(n_0^\varepsilon,c_0^\varepsilon,u_0^\varepsilon).
\end{equation}
Notice that the transport term in the concentration equation contains \(u^{\varepsilon,\tau}\), not \(\mathcal J_\varepsilon u^{\varepsilon,\tau}\). This is the same choice as in \cite{CLWW2023}; it prevents an additional local-energy term involving \(\nabla\mathcal J_\varepsilon u^{\varepsilon,\tau}\), which cannot be absorbed uniformly by the dissipation of \(u^{\varepsilon,\tau}\).

\paragraph{Cutoff at spatial infinity.}
Whenever an integration by parts is performed on the whole half-space,
we first insert a smooth radial cutoff
\[
\chi_R\in C_c^\infty(\overline{\mathbb H^3}),
\qquad
0\leq \chi_R\leq1,
\]
such that
\[
\chi_R\equiv1\quad\text{on }B_R^+,
\qquad
\chi_R\equiv0\quad\text{outside }B_{2R}^+,
\]
and
\[
|\nabla\chi_R|\leq \frac{C}{R},
\qquad
|\Delta\chi_R|\leq \frac{C}{R^2},
\qquad
\partial_\nu\chi_R=0
\quad\text{on }\Gamma.
\]
All integrations by parts are first carried out with $\chi_R$
present.  The flat-boundary terms are treated by the boundary
conditions, whereas the terms supported in
$B_{2R}^+\setminus B_R^+$ are sent to zero as $R\to\infty$
using the integrability estimates available at the corresponding
stage of the proof.

For a smooth scalar function \(z\), a vector field \(X\), a solenoidal vector field \(v\), and a smooth scalar function \(F\), the precise formulas used below are
\[
\int_{\mathbb H^3}\nabla\cdot X\,\zeta\,dx
=-\int_{\mathbb H^3}X\cdot\nabla\zeta\,dx
+\int_\Gamma (X\cdot\nu)\zeta\,dS,
\]
\[
\int_{\mathbb H^3}v\cdot\nabla F\,\zeta\,dx
=-\int_{\mathbb H^3}Fv\cdot\nabla\zeta\,dx
+\int_\Gamma F\zeta\,v\cdot\nu\,dS,
\qquad \nabla\cdot v=0,
\]
and
\[
\int_{\mathbb H^3}\bigl((\Delta z)^2-|\nabla^2z|^2\bigr)\,dx
=\int_\Gamma\left(\Delta z\,\partial_\nu z-\nabla z\cdot(\nabla^2z\,\nu)\right)dS.
\]
If \(\partial_\nu z=0\) on \(\Gamma\), then \(\partial_{i3}z=\partial_i(\partial_3z)=0\) on \(\Gamma\) for \(i=1,2\). Consequently, both terms in the last boundary integral vanish: the first contains \(\partial_\nu z\), while the second is
\[
-\partial_1z\,\partial_{13}z-\partial_2z\,\partial_{23}z-\partial_3z\,\partial_{33}z=0.
\]
Hence \(\|\nabla^2z\|_{L^2}=\|\Delta z\|_{L^2}\) for the smooth Neumann functions used below.

For a no-slip vector field \(v\), the diffusion and pressure integrations are
\[
-\int_{\mathbb H^3}\Delta v\cdot(v\zeta)\,dx
=\int_{\mathbb H^3}|\nabla v|^2\zeta\,dx
-\frac12\int_{\mathbb H^3}|v|^2\Delta\zeta\,dx
-\int_\Gamma\partial_\nu v\cdot v\,\zeta\,dS
+\frac12\int_\Gamma|v|^2\partial_\nu\zeta\,dS,
\]
\[
\int_{\mathbb H^3}\nabla p\cdot(v\zeta)\,dx
=-\int_{\mathbb H^3}p\,v\cdot\nabla\zeta\,dx
+\int_\Gamma p\zeta\,v\cdot\nu\,dS.
\]
Every displayed boundary integral is zero when \(v=0\) on \(\Gamma\). These formulas, rather than an implicit reflection argument, are used each time a boundary term occurs below.

\subsection{Global solvability of the regularized system}
\label{subsec:regularized-solvability}
We now define the fixed-point map. Fix $\alpha=7/8$ and let
$H_N^2:=\{z\in H^2(\mathbb H^3):\partial_3z|_\Gamma=0\}$.
For $T>0$, set
\[
\mathcal X_T
:=C\!\left([0,T];
(H^1\cap L^\infty)\times(H_N^2\cap W^{1,\infty})\times D(A^\alpha)
\right),
\]
where all spatial spaces are over $\mathbb H^3$,
with the norm
\begin{equation}
\label{eq:fixed-point-norm}
\|(n,c,u)\|_{\mathcal X_T}:=\sup_{0\leq t\leq T}
\left(\|n(t)\|_{H^1}+\|n(t)\|_{L^\infty}+\|c(t)\|_{H^2}+\|\nabla c(t)\|_{L^\infty}+\|u(t)\|_{D(A^\alpha)}\right).
\end{equation}
The fractional-domain embedding \(D(A^{7/8})\hookrightarrow H^{7/4}(\mathbb H^3)\) yields
\begin{equation}
\label{eq:fixed-point-embeddings}
\|u\|_{L^\infty}+\|\nabla u\|_{L^4}\leq C\|u\|_{D(A^\alpha)},
\qquad
\|c\|_{L^\infty}+\|\nabla c\|_{L^4}\leq C\|c\|_{H^2\cap W^{1,\infty}}.
\end{equation}
Moreover, Lemma~\ref{lem:stokes-smoothing} and Sobolev embedding imply, for fixed \(\varepsilon>0\),
\begin{equation}
\label{eq:fixed-point-smoothed-bounds}
\|\mathcal J_\varepsilon u\|_{W^{1,\infty}}\leq C_\varepsilon\|u\|_{L^2},
\qquad
\|\mathcal K_\varepsilon(n\nabla\phi)\|_{L^2}\leq C\|\nabla\phi\|_{L^\infty}\|n\|_{L^2}.
\end{equation}

For \(U=(n^{\varepsilon,\tau},c^{\varepsilon,\tau},u^{\varepsilon,\tau})\in\mathcal X_T\), define
\begin{equation}
\label{eq:fixed-point-map}
\Phi(U):=\bigl(\Phi_1(U),\Phi_2(U),\Phi_3(U)\bigr),
\end{equation}
where
\begin{equation}
\label{eq:fixed-point-components}
\begin{aligned}
\Phi_1(U)(t)&:=S_N(t)n_0^\varepsilon+\int_0^t\mathcal T_N(t-s)
\left(\widehat a_\tau(n^{\varepsilon,\tau})\nabla c^{\varepsilon,\tau}+(\mathcal J_\varepsilon u^{\varepsilon,\tau})n^{\varepsilon,\tau}\right)(s)\,ds,\\
\Phi_2(U)(t)&:=S_N(t)c_0^\varepsilon-\int_0^tS_N(t-s)
\left(u^{\varepsilon,\tau}\cdot\nabla c^{\varepsilon,\tau}+\widehat h_\tau(n^{\varepsilon,\tau})c^{\varepsilon,\tau}\right)(s)\,ds,\\
\Phi_3(U)(t)&:=e^{-tA}u_0^\varepsilon-\int_0^te^{-(t-s)A}\mathbb P_+
\left((\mathcal J_\varepsilon u^{\varepsilon,\tau})\cdot\nabla u^{\varepsilon,\tau}+\mathcal K_\varepsilon(n^{\varepsilon,\tau}\nabla\phi)\right)(s)\,ds.
\end{aligned}
\end{equation}
Thus the mild solution is not postulated in advance: it is, by definition, a fixed point of \(\Phi\), namely a triple \(U\) satisfying \(\Phi(U)=U\).

\begin{proposition}[Global strong solvability of the regularized problem]
\label{prop:regularized-global}
Let $\nabla\phi\in L^\infty(\mathbb H^3)$, and let
$(n_0^\varepsilon,c_0^\varepsilon,u_0^\varepsilon)$
be the boundary-compatible regularized initial data constructed above.
In particular,
\[
n_0^\varepsilon\in H^2(\mathbb H^3)\cap L^1(\mathbb H^3),
\qquad
c_0^\varepsilon\in H^3(\mathbb H^3)\cap L^1(\mathbb H^3),
\qquad
u_0^\varepsilon\in H^2(\mathbb H^3)\cap L_\sigma^2(\mathbb H^3),
\]
with $n_0^\varepsilon,c_0^\varepsilon\ge0$ and
\[
\partial_\nu n_0^\varepsilon
=\partial_\nu c_0^\varepsilon=0,
\qquad
u_0^\varepsilon=0
\quad\text{on }\Gamma.
\]
Then, for every $\varepsilon,\tau>0$, problem
\eqref{eq:regularized-half-space-system}--%
\eqref{eq:regularized-boundary-initial},
admits a unique global strong solution
$(n^{\varepsilon,\tau},c^{\varepsilon,\tau},u^{\varepsilon,\tau})$.

For every $0<T<\infty$, suppressing the approximation indices, the solution satisfies
\[
\begin{aligned}
n&\in C([0,T];H^1)\cap L^\infty(0,T;L^\infty)
       \cap L^2(0,T;H^2), & n_t&\in L^2(0,T;L^2),\\
c&\in C([0,T];H_N^2)\cap L^\infty(0,T;W^{1,\infty})
       \cap L^2(0,T;H^3), & c_t&\in L^2(0,T;H^1),\\
u&\in C([0,T];D(A^{7/8}))\cap L^2(0,T;D(A)),
       & u_t&\in L^2(0,T;L_\sigma^2).
\end{aligned}
\]
All spatial spaces are over $\mathbb H^3$. Uniqueness holds in the
fixed-point class \eqref{eq:fixed-point-norm}. The boundary conditions
hold in the trace sense and the initial data are attained in the displayed
continuous norms. The pressure, determined modulo functions of time,
satisfies $\nabla p\in L^2(0,T;L^2)$.
For every $0<T<\infty$,
\begin{equation}\label{eq:fixed-parameter-high-order}
\begin{aligned}
&\sup_{0<t<T}\bigl(\|n\|_{H^1}^2+\|n\|_{L^\infty}^2
+\|c\|_{H^2}^2+\|\nabla c\|_{L^\infty}^2+\|u\|_{D(A^{7/8})}^2\bigr)\\
&\quad+\int_0^T\bigl(\|n\|_{H^2}^2+\|c\|_{H^3}^2+\|u\|_{H^2}^2
+\|n_t\|_{L^2}^2+\|c_t\|_{H^1}^2+\|u_t\|_{L^2}^2\bigr)\,dt\\
&\leq C_{\varepsilon,\tau,T,\phi,n_0^\varepsilon,c_0^\varepsilon,u_0^\varepsilon}
<\infty.
\end{aligned}
\end{equation}
The dependence on the initial data is through
$\|n_0^\varepsilon\|_{H^2\cap L^1}$,
$\|c_0^\varepsilon\|_{H^3\cap L^1}$, and
$\|u_0^\varepsilon\|_{H^2}$.
Here the constant may depend on $T$; no polynomial growth rate is asserted.
These bounds suffice for continuation and for the entropy and compactness
arguments below.

Moreover, for every $t\ge0$,
\begin{equation}
\label{eq:positivity-mass-maximum}
\begin{aligned}
&n^{\varepsilon,\tau}(t)\ge0,
\qquad
0\le c^{\varepsilon,\tau}(t)
\le\|c_0^\varepsilon\|_{L^\infty},\\
&\|n^{\varepsilon,\tau}(t)\|_{L^1}
=\|n_0^\varepsilon\|_{L^1},\\
&\|c^{\varepsilon,\tau}(t)\|_{L^1}
+\int_0^t\int_{\mathbb H^3}
h_\tau(n^{\varepsilon,\tau}(x,s))
c^{\varepsilon,\tau}(x,s)\,dx\,ds
=\|c_0^\varepsilon\|_{L^1}.
\end{aligned}
\end{equation}
\end{proposition}

\begin{proof}
We divide the proof into six steps.

\emph{Step 1: the closed ball and the elementary nonlinear bounds.} Set
\[
R:=2C_0\bigl(\|n_0^\varepsilon\|_{H^1}+\|n_0^\varepsilon\|_{L^\infty}
+\|c_0^\varepsilon\|_{H^2}+\|\nabla c_0^\varepsilon\|_{L^\infty}
+\|u_0^\varepsilon\|_{D(A^\alpha)}\bigr)+1,
\]
where $C_0\geq1$ is a fixed constant large enough for the semigroup
and embedding bounds below. Let $T\in(0,1)$ and let \(\mathbb B_R\) be the closed radius-\(R\) ball in \(\mathcal X_T\). For \(U=(n^{\varepsilon,\tau},c^{\varepsilon,\tau},u^{\varepsilon,\tau})\in\mathbb B_R\), the embeddings of the fixed-point space give
\begin{equation}
\label{eq:ball-embeddings}
\|n^{\varepsilon,\tau}(t)\|_{L^\infty}+\|n^{\varepsilon,\tau}(t)\|_{H^1}+\|\nabla c^{\varepsilon,\tau}(t)\|_{L^\infty}+\| c^{\varepsilon,\tau}(t)\|_{H^2}+\|u^{\varepsilon,\tau}(t)\|_{D(A^\alpha)}
\leq C_{\varepsilon,R}.
\end{equation}
Define
\[
\begin{aligned}
F_U&:=\widehat a_\tau(n^{\varepsilon,\tau})\nabla c^{\varepsilon,\tau}+(\mathcal J_\varepsilon u^{\varepsilon,\tau})n^{\varepsilon,\tau},\\
G_U&:=u^{\varepsilon,\tau}\cdot\nabla c^{\varepsilon,\tau}+\widehat h_\tau(n^{\varepsilon,\tau})c^{\varepsilon,\tau},\\
K_U&:=(\mathcal J_\varepsilon u^{\varepsilon,\tau})\cdot\nabla u^{\varepsilon,\tau}+\mathcal K_\varepsilon(n^{\varepsilon,\tau}\nabla\phi).
\end{aligned}
\]
Using \eqref{eq:regularizing-extensions} and \eqref{eq:ball-embeddings}, we obtain
\begin{equation}
\label{eq:nonlinear-source-bounds}
\|F_U\|_{L^4}+\|F_U\|_{L^2}+\|\nabla\cdot F_U\|_{L^2}
+\|G_U\|_{L^4}+\|G_U\|_{L^2}+\|\nabla G_U\|_{L^2}
+\|K_U\|_{L^2}
\leq C_{\varepsilon,\tau,R,\phi}.
\end{equation}
For example,
\begin{align*}
\nabla\cdot F_U
&=\widehat a_\tau'(n^{\varepsilon,\tau})\nabla n^{\varepsilon,\tau}\cdot\nabla c^{\varepsilon,\tau}+\widehat a_\tau(n^{\varepsilon,\tau})\Delta c^{\varepsilon,\tau}+(\mathcal J_\varepsilon u^{\varepsilon,\tau})\cdot\nabla n^{\varepsilon,\tau},\\
\nabla G_U
&=(\nabla u^{\varepsilon,\tau})^{\mathsf T}\nabla c^{\varepsilon,\tau}+u^{\varepsilon,\tau}\cdot\nabla^2c^{\varepsilon,\tau}
+\widehat h_\tau'(n^{\varepsilon,\tau})c\nabla n^{\varepsilon,\tau}+\widehat h_\tau(n^{\varepsilon,\tau})\nabla c^{\varepsilon,\tau},
\end{align*}
and every product on the right belongs to \(L^2\) by \eqref{eq:ball-embeddings}. This proves \eqref{eq:nonlinear-source-bounds}.

\emph{Step 2: self-mapping.} The half-space kernel estimates imply
\begin{equation}
\label{eq:Phi1-self-map}
\begin{aligned}
\|\Phi_1(U)(t)\|_{L^\infty}
&\leq\|n_0^\varepsilon\|_{L^\infty}+C\int_0^t(t-s)^{-7/8}\|F_U(s)\|_{L^4}\,ds
\leq\|n_0^\varepsilon\|_{L^\infty}+C_{\varepsilon,\tau,R}T^{1/8},\\
\|\Phi_1(U)(t)\|_{L^2}
&\leq\|n_0^\varepsilon\|_{L^2}+C\int_0^t(t-s)^{-1/2}\|F_U(s)\|_{L^2}\,ds
\leq\|n_0^\varepsilon\|_{L^2}+C_{\varepsilon,\tau,R}T^{1/2},\\
\|\nabla\Phi_1(U)(t)\|_{L^2}
&\leq\|\nabla n_0^\varepsilon\|_{L^2}+C\int_0^t(t-s)^{-1/2}\|\nabla\cdot F_U(s)\|_{L^2}\,ds\\
\qquad&\leq\|\nabla n_0^\varepsilon\|_{L^2}+C_{\varepsilon,\tau,R}T^{1/2}.
\end{aligned}
\end{equation}
In the last line we use \(\Phi_1=S_N(t)n_0^\varepsilon-\int_0^tS_N(t-s)\nabla\cdot F_U(s)\,ds\); this avoids a nonintegrable two-derivative kernel.\\
For $\Phi_2$ we use
\[
\|\nabla S_N(t)g\|_{L^\infty}\leq Ct^{-7/8}\|g\|_{L^4},
\qquad
\|\nabla^2S_N(t)g\|_{L^2}\leq Ct^{-1/2}\|\nabla g\|_{L^2}.
\]
These estimates give
\begin{equation}
\label{eq:Phi2-self-map}
\begin{aligned}
\|\Phi_2(U)(t)\|_{L^2}
&\leq \|S_N(t)c_0^\varepsilon\|_{L^2}+\int_0^t\|S_N(t-s)G_U(s)\|_{L^2}\,ds\\
&\leq \|c_0^\varepsilon\|_{L^2}+\int_0^t\|G_U(s)\|_{L^2}\,ds
\leq \|c_0^\varepsilon\|_{L^2}+C_{\varepsilon,\tau,R}T,\\
\|\nabla\Phi_2(U)(t)\|_{L^\infty}
&\leq \|\nabla S_N(t)c_0^\varepsilon\|_{L^\infty}
+C\int_0^t(t-s)^{-7/8}\|G_U(s)\|_{L^4}\,ds\\
&\leq \|\nabla c_0^\varepsilon\|_{L^\infty}+C_{\varepsilon,\tau,R}\int_0^t(t-s)^{-7/8}\,ds
\leq \|\nabla c_0^\varepsilon\|_{L^\infty}+C_{\varepsilon,\tau,R}T^{1/8},\\
\|\nabla^2\Phi_2(U)(t)\|_{L^2}
&\leq \|\nabla^2S_N(t)c_0^\varepsilon\|_{L^2}
+C\int_0^t(t-s)^{-1/2}\|\nabla G_U(s)\|_{L^2}\,ds\\
&\leq \|\nabla^2c_0^\varepsilon\|_{L^2}+C_{\varepsilon,\tau,R}\int_0^t(t-s)^{-1/2}\,ds
\leq \|\nabla^2c_0^\varepsilon\|_{L^2}+C_{\varepsilon,\tau,R}T^{1/2}.
\end{aligned}
\end{equation}
For the velocity component, analyticity of the Stokes semigroup yields
\[
\|A^\alpha e^{-(t-s)A}\mathbb P_+g\|_{L^2}\leq C(t-s)^{-\alpha}\|g\|_{L^2},
\qquad \alpha=\frac78.
\]
Since the graph norm of \(D(A^\alpha)\) is equivalent to \(\|z\|_{L^2}+\|A^\alpha z\|_{L^2}\), we have
\begin{equation}
\label{eq:Phi3-self-map}
\begin{aligned}
\|\Phi_3(U)(t)\|_{D(A^\alpha)}
&\leq \|e^{-tA}u_0^\varepsilon\|_{D(A^\alpha)}
+C\int_0^t\bigl(1+(t-s)^{-\alpha}\bigr)\|K_U(s)\|_{L^2}\,ds\\
&\leq \|u_0^\varepsilon\|_{D(A^\alpha)}
+C_{\varepsilon,R,\phi}\left(\int_0^t1\,ds+\int_0^t(t-s)^{-7/8}\,ds\right)\\
&\leq \|u_0^\varepsilon\|_{D(A^\alpha)}+C_{\varepsilon,R,\phi}\bigl(T+T^{1/8}\bigr).
\end{aligned}
\end{equation}
Combining \eqref{eq:Phi1-self-map}--\eqref{eq:Phi3-self-map}, we obtain
\begin{equation}
\label{eq:self-map-summary}
\|\Phi(U)\|_{\mathcal X_T}
\leq \frac{R-1}{2}+C_{\varepsilon,\tau,R,\phi}\bigl(T^{1/8}+T^{1/2}+T\bigr).
\end{equation}
Choose \(T_0\in(0,1]\) so that the second term on the right is at most \((R+1)/2\). Then \(\Phi(\mathbb B_R)\subset\mathbb B_R\).

\emph{Step 3: contraction.} Let \(U_j=(n_j,c_j,u_j)\in\mathbb B_R\), \(j=1,2\). Since \(\widehat a_\tau,\widehat h_\tau\in C^3\) and the functions \(n_j\) are uniformly bounded, the Moser estimates give
\begin{equation}
\label{eq:composition-difference}
\begin{aligned}
&\|\widehat a_\tau(n_1)-\widehat a_\tau(n_2)\|_{L^\infty}
+\|\nabla(\widehat a_\tau(n_1)-\widehat a_\tau(n_2))\|_{L^2}\\
&\quad
+\|\widehat h_\tau(n_1)-\widehat h_\tau(n_2)\|_{L^\infty}
+\|\nabla(\widehat h_\tau(n_1)-\widehat h_\tau(n_2))\|_{L^2}
\leq C_{\tau,R}\|n_1-n_2\|_{H^1\cap L^\infty}.
\end{aligned}
\end{equation}
For instance,
\[
\widehat a_\tau(n_1)\nabla c_1-\widehat a_\tau(n_2)\nabla c_2
=\bigl(\widehat a_\tau(n_1)-\widehat a_\tau(n_2)\bigr)\nabla c_1
+\widehat a_\tau(n_2)\nabla(c_1-c_2),
\]
and
\[
u_1\cdot\nabla c_1-u_2\cdot\nabla c_2
=(u_1-u_2)\cdot\nabla c_1+u_2\cdot\nabla(c_1-c_2).
\]
To make the Lipschitz estimate explicit, we expand the differences directly.
Then we have
\begin{align*}
F_{U_1}-F_{U_2}
&=\bigl(\widehat a_\tau(n_1)-\widehat a_\tau(n_2)\bigr)\,\nabla c_1+\widehat a_\tau(n_2)\nabla(c_1-c_2)
\\ &\quad+\mathcal J_\varepsilon(u_1-u_2)\,n_1+(\mathcal J_\varepsilon u_2)(n_1-n_2),\\
G_{U_1}-G_{U_2}
&=(u_1-u_2)\cdot\nabla c_1+u_2\cdot\nabla(c_1-c_2)
\\ &\quad+\bigl(\widehat h_\tau(n_1)-\widehat h_\tau(n_2)\bigr)\,c_1+\widehat h_\tau(n_2)(c_1-c_2),\\
K_{U_1}-K_{U_2}
&=\mathcal J_\varepsilon(u_1-u_2)\cdot\nabla u_1+(\mathcal J_\varepsilon u_2)\cdot\nabla(u_1-u_2)\\ &\quad+\mathcal K_\varepsilon((n_1-n_2)\nabla\phi).
\end{align*}
Moreover, because \(\nabla\cdot (\mathcal J_\varepsilon u_j)=0\),
\begin{align*}
\nabla\cdot(F_{U_1}-F_{U_2})={}&
\bigl(\widehat a_\tau'(n_1)-\widehat a_\tau'(n_2)\bigr)\nabla n_1\cdot\nabla c_1
+\widehat a_\tau'(n_2)\nabla(n_1-n_2)\cdot\nabla c_1\\
&+\widehat a_\tau'(n_2)\nabla n_2\cdot\nabla(c_1-c_2)
\\ &+\bigl(\widehat a_\tau(n_1)-\widehat a_\tau(n_2)\bigr)\,\Delta c_1+\widehat a_\tau(n_2)\Delta(c_1-c_2)\\
&+\mathcal J_\varepsilon(u_1-u_2)\cdot\nabla n_1+(\mathcal J_\varepsilon u_2)\cdot\nabla(n_1-n_2).
\end{align*}
The corresponding differentiated identity for the reaction part is
\begin{align*}
\nabla\bigl(\widehat h_\tau(n_1)c_1-\widehat h_\tau(n_2)c_2\bigr)
=&\nabla\bigl(\widehat h_\tau(n_1)-\widehat h_\tau(n_2)\bigr)c_1+\bigl(\widehat h_\tau(n_1)-\widehat h_\tau(n_2)\bigr)\,\nabla c_1\\
&+\widehat h_\tau'(n_2)\nabla n_2(c_1-c_2)
+\widehat h_\tau(n_2)\nabla(c_1-c_2),
\end{align*}
where
\[
\nabla\bigl(\widehat h_\tau(n_1)-\widehat h_\tau(n_2)\bigr)
=\bigl(\widehat h_\tau'(n_1)-\widehat h_\tau'(n_2)\bigr)\nabla n_1
+\widehat h_\tau'(n_2)\nabla(n_1-n_2).
\]
Using these formulas, \eqref{eq:composition-difference}, the embeddings in \eqref{eq:fixed-point-embeddings}, and boundedness of \(\mathcal J_\varepsilon\) and \(\mathcal K_\varepsilon\), each product is estimated in the same space as its counterpart in \eqref{eq:nonlinear-source-bounds}. We obtain
\begin{equation}
\label{eq:source-difference-bounds}
\begin{aligned}
&\|F_{U_1}-F_{U_2}\|_{L^4}+\|F_{U_1}-F_{U_2}\|_{L^2}
+\|\nabla\cdot(F_{U_1}-F_{U_2})\|_{L^2}\\
&\quad
+\|G_{U_1}-G_{U_2}\|_{L^4}+\|G_{U_1}-G_{U_2}\|_{L^2}
+\|\nabla(G_{U_1}-G_{U_2})\|_{L^2}
+\|K_{U_1}-K_{U_2}\|_{L^2}\\
&\leq C_{\varepsilon,\tau,R,\phi}\|U_1-U_2\|_{\mathcal X_T}.
\end{aligned}
\end{equation}
Applying the kernel estimates component by component gives
\begin{align*}
&\|\Phi_1(U_1)-\Phi_1(U_2)\|_{L^\infty(0,T;L^\infty)}
\leq C_{\varepsilon,\tau,R}T^{1/8}\|U_1-U_2\|_{\mathcal X_T},\\
&\|\Phi_1(U_1)-\Phi_1(U_2)\|_{L^\infty(0,T;H^1)}
\leq C_{\varepsilon,\tau,R}T^{1/2}\|U_1-U_2\|_{\mathcal X_T},\\
&\|\Phi_2(U_1)-\Phi_2(U_2)\|_{L^\infty(0,T;L^2)}
\leq C_{\varepsilon,\tau,R}T\|U_1-U_2\|_{\mathcal X_T},\\
&\|\nabla(\Phi_2(U_1)-\Phi_2(U_2))\|_{L^\infty(0,T;L^\infty)}
\leq C_{\varepsilon,\tau,R}T^{1/8}\|U_1-U_2\|_{\mathcal X_T},\\
&\|\nabla^2(\Phi_2(U_1)-\Phi_2(U_2))\|_{L^\infty(0,T;L^2)}
\leq C_{\varepsilon,\tau,R}T^{1/2}\|U_1-U_2\|_{\mathcal X_T},\\
&\|\Phi_3(U_1)-\Phi_3(U_2)\|_{L^\infty(0,T;D(A^\alpha))}
\leq C_{\varepsilon,R,\phi}(T^{1/8}+T)\|U_1-U_2\|_{\mathcal X_T}.
\end{align*}
Consequently,
\begin{equation}
\label{eq:contraction-summary}
\|\Phi(U_1)-\Phi(U_2)\|_{\mathcal X_T}
\leq C_{\varepsilon,\tau,R,\phi}\bigl(T^{1/8}+T^{1/2}+T\bigr)
\|U_1-U_2\|_{\mathcal X_T}.
\end{equation}
We choose \(T_0\in(0,1]\) so that simultaneously
\[
C_{\varepsilon,\tau,R,\phi}\bigl(T_0^{1/8}+T_0^{1/2}+T_0\bigr)
\leq \min\{(R+1)/2,\tfrac12\}.
\]
Then \(\Phi\) maps \(\mathbb B_R\) into itself and has Lipschitz constant at most \(1/2\) there. Banach's fixed-point theorem therefore gives a unique \(U=(n,c,u)\in\mathbb B_R\) satisfying
\begin{equation}
\label{eq:fixed-point-identity}
\Phi(n^{\varepsilon,\tau},c^{\varepsilon,\tau},u^{\varepsilon,\tau})=(n^{\varepsilon,\tau},c^{\varepsilon,\tau},u^{\varepsilon,\tau}).
\end{equation}
By \eqref{eq:fixed-point-components}, this fixed point is the mild solution of the auxiliary system with \(\widehat a_\tau,\widehat h_\tau\). Standard parabolic and Stokes bootstrap, using the smooth initial data and the fixed positive parameters \(\varepsilon,\tau\), makes the solution classical on \((0,T_0]\).

\emph{Step 4: nonnegativity.} Let \(n^{\varepsilon,\tau}_-:=\max\{-n^{\varepsilon,\tau},0\}\). We multiply the auxiliary cell equation by \(-n_-\). With all boundary terms on the left, the transport, diffusion, and chemotactic integrations produce the flat-boundary contribution
\[
\frac12\int_\Gamma (n^{\varepsilon,\tau}_-)^2 (\mathcal J_\varepsilon u^{\varepsilon,\tau})\cdot\nu\,dS
+\int_\Gamma \partial_\nu n^{\varepsilon,\tau}\,n^{\varepsilon,\tau}_-\,dS
-\int_\Gamma \widehat a_\tau(n^{\varepsilon,\tau})\partial_\nu c^{\varepsilon,\tau}\,n^{\varepsilon,\tau}_-\,dS.
\]
By Lemma~\ref{lem:stokes-smoothing}, \((\mathcal J_\varepsilon u^{\varepsilon,\tau})=0\) on \(\Gamma\), while \eqref{eq:regularized-boundary-initial} gives \(\partial_\nu n^{\varepsilon,\tau}=\partial_\nu c^{\varepsilon,\tau}=0\). Hence all three boundary integrals vanish. Since \(\nabla\cdot (\mathcal J_\varepsilon u^{\varepsilon,\tau})=0\), the interior transport integral is also zero, and we obtain
\[
\frac12\frac{d}{dt}\|n^{\varepsilon,\tau}_-\|_{L^2}^2+\|\nabla n^{\varepsilon,\tau}_-\|_{L^2}^2
=-\int_{\{n<0\}}\widehat a_\tau(n^{\varepsilon,\tau})\nabla c^{\varepsilon,\tau}\cdot\nabla n^{\varepsilon,\tau}_-\,dx.
\]
By \eqref{eq:regularizing-extensions}, \(|\widehat a_\tau(n^{\varepsilon,\tau})|\leq C_\tau n^{\varepsilon,\tau}_-\) on \(\{n^{\varepsilon,\tau}<0\}\). Hence
\[
\frac{d}{dt}\|n^{\varepsilon,\tau}_-\|_{L^2}^2+\|\nabla n^{\varepsilon,\tau}_-\|_{L^2}^2
\leq C_\tau\|\nabla c^{\varepsilon,\tau}\|_{L^\infty}^2\|n^{\varepsilon,\tau}_-\|_{L^2}^2.
\]
Since \(n_0^\varepsilon\geq0\), Gronwall's inequality gives \(n^{\varepsilon,\tau}_-=0\). Thus \(n\geq0\), and the auxiliary functions coincide with \(a_\tau,h_\tau\).

Next let \(c^{\varepsilon,\tau}_-:=\max\{-c^{\varepsilon,\tau},0\}\) and test the concentration equation by \(-c^{\varepsilon,\tau}_-\). With the same convention, the only flat-boundary contribution is
\[
\frac12\int_\Gamma (c^{\varepsilon,\tau}_-)^2u\cdot\nu\,dS
+\int_\Gamma\partial_\nu c^{\varepsilon,\tau}\,c^{\varepsilon,\tau}_-\,dS=0,
\]
because \(u^{\varepsilon,\tau}=0\) and \(\partial_\nu c^{\varepsilon,\tau}=0\) on \(\Gamma\). The interior transport term is zero because \(\nabla\cdot u^{\varepsilon,\tau}=0\). Therefore
\[
\frac12\frac{d}{dt}\|c^{\varepsilon,\tau}_-\|_{L^2}^2+\|\nabla c^{\varepsilon,\tau}_-\|_{L^2}^2
+\int_{\mathbb H^3}h_\tau(n^{\varepsilon,\tau})(c^{\varepsilon,\tau}_-)^2\,dx=0.
\]
Since \(c_0^\varepsilon\geq0\), it follows that \(c^{\varepsilon,\tau}_-=0\), and hence \(c^{\varepsilon,\tau}\geq0\).

\par\medskip
\noindent\textbf{Step 5. Mass conservation for $n$, mass balance
for $c$, and the maximum principle.}

We first justify the propagation of spatial integrability.
Fix a compact time interval $[0,S]$ contained in the interval
of existence, and suppress the approximation indices.
By nonnegativity, the cell equation can be written as
\[
\partial_t n-\Delta n
+\nabla\cdot\left[n\left(\mathcal J_\varepsilon u
+\frac{\nabla c}{1+\tau n}\right)\right]=0.
\]
The fixed-point bounds imply that $\mathcal J_\varepsilon u+\nabla c/(1+\tau n)$ and $u$ are bounded
on $\mathbb H^3\times(0,S)$.

For $R\geq1$, set
\[
w_R(x)=\exp\!\left(-\frac{\sqrt{1+|x|^2}}{R}\right).
\]
Then $\partial_\nu w_R=0$ on $\Gamma$ and
\[
|\nabla w_R|+|\Delta w_R|\leq \frac{C}{R}w_R.
\]
Testing the two scalar equations with $w_R$, first inserting
an additional compact spatial cutoff, gives
\[
\begin{aligned}
\frac{d}{dt}\int_{\mathbb H^3}n\,w_R\,dx
&\leq\frac{C_S}{R}\int_{\mathbb H^3}n\,w_R\,dx,\\
\frac{d}{dt}\int_{\mathbb H^3}c\,w_R\,dx
&\leq\frac{C_S}{R}\int_{\mathbb H^3}c\,w_R\,dx.
\end{aligned}
\]
In the concentration estimate, the nonnegative reaction term is discarded.
Gronwall's inequality and monotone convergence as $R\to\infty$
therefore imply
\[
\sup_{0\leq t\leq S}\|n(t)\|_{L^1}
\leq \|n_0^\varepsilon\|_{L^1},
\qquad
\sup_{0\leq t\leq S}\|c(t)\|_{L^1}
\leq \|c_0^\varepsilon\|_{L^1}.
\]
The compact-cutoff calculations below are consequently
justified.

\par Let $\chi_R$ be the cutoff introduced above.  Testing the cell
equation by $\chi_R$ and using
\[
\nabla\cdot (\mathcal J_\varepsilon u^{\varepsilon,\tau})=0,
\qquad
(\mathcal J_\varepsilon u^{\varepsilon,\tau})=0,
\qquad
\partial_\nu n^{\varepsilon,\tau}
=
\partial_\nu c^{\varepsilon,\tau}=0
\quad\text{on }\Gamma,
\]
we obtain
\begin{align}
\frac{d}{dt}
\int_{\mathbb H^3} n^{\varepsilon,\tau}\chi_R\,dx
={}&
\int_{\mathbb H^3}
n^{\varepsilon,\tau}\Delta\chi_R\,dx
+
\int_{\mathbb H^3}
n^{\varepsilon,\tau}
(\mathcal J_\varepsilon u^{\varepsilon,\tau})\cdot\nabla\chi_R\,dx
\nonumber\\
&+
\int_{\mathbb H^3}
a_\tau(n^{\varepsilon,\tau})
\nabla c^{\varepsilon,\tau}\cdot\nabla\chi_R\,dx.
\label{eq:mass-cutoff-n}
\end{align}
For fixed $\varepsilon,\tau>0$,
\[
n^{\varepsilon,\tau}\in L^1\cap L^2,
\qquad
(\mathcal J_\varepsilon u^{\varepsilon,\tau})\in L^\infty,
\qquad
a_\tau(n^{\varepsilon,\tau})\nabla c^{\varepsilon,\tau}
\in L^1,
\]
on every finite time interval.  Hence, by the bounds on
$\chi_R$,
the three terms on the right-hand side of
\eqref{eq:mass-cutoff-n} converge to zero as $R\to\infty$.
Therefore
\[
\frac{d}{dt}
\int_{\mathbb H^3}n^{\varepsilon,\tau}(x,t)\,dx=0,
\]
and consequently
\begin{equation}
\|n^{\varepsilon,\tau}(t)\|_{L^1(\mathbb H^3)}
=
\|n_0^\varepsilon\|_{L^1(\mathbb H^3)}.
\label{eq:n-mass-conservation}
\end{equation}

Similarly, testing the concentration equation by $\chi_R$ gives
\begin{align*}
\frac{d}{dt}
\int_{\mathbb H^3}c^{\varepsilon,\tau}\chi_R\,dx
+
\int_{\mathbb H^3}
h_\tau(n^{\varepsilon,\tau})
c^{\varepsilon,\tau}\chi_R\,dx
={}&
\int_{\mathbb H^3}
c^{\varepsilon,\tau}\Delta\chi_R\,dx\\
&+
\int_{\mathbb H^3}
c^{\varepsilon,\tau}
u^{\varepsilon,\tau}\cdot\nabla\chi_R\,dx.
\end{align*}
Letting $R\to\infty$ yields
\[
\|c^{\varepsilon,\tau}(t)\|_{L^1}
+
\int_0^t\int_{\mathbb H^3}
h_\tau(n^{\varepsilon,\tau})c^{\varepsilon,\tau}\,dxds
=
\|c_0^\varepsilon\|_{L^1}.
\]
The $L^\infty$ estimate follows from the maximum principle:
\[
0\leq c^{\varepsilon,\tau}(x,t)
\leq \|c_0^\varepsilon\|_{L^\infty}.
\]

\emph{Step 6: continuation in the fixed-point spaces.}
Fix $\varepsilon,\tau>0$ and suppress these indices. Let $T_*$ be the
maximal existence time. Fix $0<T<\infty$ and
$0<S<\min\{T,T_*\}$. Throughout this step, all time norms are taken
over $(0,S)$, and $K_T$ denotes a positive constant, which may change
from line to line, depending only on
$\varepsilon,\tau,T,\|\nabla\phi\|_{L^\infty}$ and the regularized
initial data, but not on $S$. The energy calculations below are
justified by the local classical regularity and the cutoff procedure
described above. The flat-boundary terms vanish by
$u=(\mathcal J_\varepsilon u)=0$ and $\partial_\nu n=\partial_\nu c=0$.

Testing the concentration equation by $c$, using
$\nabla\cdot u=0$ and $h_\tau(n)\geq0$, gives
\begin{equation}\label{eq:fixed-c-L2}
\sup_{0<t<S}\|c(t)\|_{L^2}^2
+\int_0^S\|\nabla c\|_{L^2}^2\,dt
\leq C\|c_0^\varepsilon\|_{L^2}^2.
\end{equation}
Similarly, testing the cell equation by $n$ and using
$0\leq a_\tau(n)\leq\tau^{-1}$ yields
\[
\frac12\frac{d}{dt}\|n\|_{L^2}^2+\|\nabla n\|_{L^2}^2
=\int_{\mathbb H^3}a_\tau(n)\nabla c\cdot\nabla n\,dx
\leq\frac12\|\nabla n\|_{L^2}^2
+\frac{1}{2\tau^2}\|\nabla c\|_{L^2}^2.
\]
Together with \eqref{eq:fixed-c-L2}, this implies
\begin{equation}\label{eq:fixed-n-L2}
\sup_{0<t<S}\|n(t)\|_{L^2}^2
+\int_0^S\|\nabla n\|_{L^2}^2\,dt
\leq C\bigl(\|n_0^\varepsilon\|_{L^2}^2
+\tau^{-2}\|c_0^\varepsilon\|_{L^2}^2\bigr).
\end{equation}
For the velocity equation, the $L^{6/5}$-boundedness of
$\mathcal K_\varepsilon$, Sobolev embedding, and interpolation give
\[
\begin{aligned}
\left|\int_{\mathbb H^3}\mathcal K_\varepsilon(n\nabla\phi)\cdot u\,dx\right|
&\leq C\|\nabla\phi\|_{L^\infty}
\|n\|_{L^{6/5}}\|\nabla u\|_{L^2}\\
&\leq\frac12\|\nabla u\|_{L^2}^2
+C\|\nabla\phi\|_{L^\infty}^2
\|n\|_{L^1}^{4/3}\|n\|_{L^2}^{2/3}.
\end{aligned}
\]
Testing by $u$ and using mass conservation and
\eqref{eq:fixed-n-L2}, we therefore obtain
\begin{equation}\label{eq:fixed-u-L2}
\sup_{0<t<S}\|u(t)\|_{L^2}^2
+\int_0^S\|\nabla u\|_{L^2}^2\,dt
\leq K_T.
\end{equation}
The smoothing properties of $\mathcal J_\varepsilon$ and the
$L^2$-boundedness of $\mathcal K_\varepsilon$ now imply
\begin{equation}\label{eq:fixed-epsilon-Linfty}
\sup_{0<t<S}\bigl(\|(\mathcal J_\varepsilon u)(t)\|_{W^{1,\infty}}
+\|\mathcal K_\varepsilon(n\nabla\phi)(t)\|_{L^2}\bigr)\leq K_T.
\end{equation}

We use the standard linear regularity estimates for the Neumann
heat equation and the no-slip Stokes system; see
\cite{DHP2003,DHP2001}. For $z_t-\Delta z=F$ with homogeneous
Neumann boundary condition, and $m=0,1$, these estimates read
\begin{equation}\label{eq:Neumann-high-order-estimate}
\begin{aligned}
&\|z\|_{L^\infty H^{m+1}}^2
+\|z\|_{L^2 H^{m+2}}^2+\|z_t\|_{L^2 H^m}^2\\
&\qquad\leq C\bigl(\|z(0)\|_{H^{m+1}}^2
+\|F\|_{L^2 H^m}^2+\|z\|_{L^2L^2}^2\bigr),
\end{aligned}
\end{equation}
where $z(0)\in H_N^2$ is required when $m=1$.
For the Stokes system, we have
\begin{equation}\label{eq:Stokes-high-order-estimate}
\begin{aligned}
&\|u\|_{L^\infty H^1}^2+\|u\|_{L^2H^2}^2
+\|u_t\|_{L^2L^2}^2+\|\nabla p\|_{L^2L^2}^2\\
&\qquad\leq C\bigl(\|u_0^\varepsilon\|_{H^1}^2
+\|\bigl((\mathcal J_\varepsilon u)\cdot\nabla u+\mathcal K_\varepsilon(n\nabla\phi)\bigr)\|_{L^2L^2}^2+\|u\|_{L^2L^2}^2\bigr).
\end{aligned}
\end{equation}
The constants in these estimates can be chosen independently of $S$.

First, \eqref{eq:fixed-u-L2} and
\eqref{eq:fixed-epsilon-Linfty} give
\[
\|\bigl((\mathcal J_\varepsilon u)\cdot\nabla u+\mathcal K_\varepsilon(n\nabla\phi)\bigr)\|_{L^2L^2}
\leq\|(\mathcal J_\varepsilon u)\|_{L^\infty L^\infty}\|\nabla u\|_{L^2L^2}
+S^{1/2}\|\mathcal K_\varepsilon(n\nabla\phi)\|_{L^\infty L^2}\leq K_T.
\]
Applying \eqref{eq:Stokes-high-order-estimate}, we obtain
\[
\|u\|_{L^\infty H^1}+\|u\|_{L^2H^2}
+\|u_t\|_{L^2L^2}+\|\nabla p\|_{L^2L^2}\leq K_T.
\]
Consequently,
\[
\|\bigl((\mathcal J_\varepsilon u)\cdot\nabla u+\mathcal K_\varepsilon(n\nabla\phi)\bigr)\|_{L^\infty L^2}
\leq\|(\mathcal J_\varepsilon u)\|_{L^\infty L^\infty}
\|\nabla u\|_{L^\infty L^2}
+\|\mathcal K_\varepsilon(n\nabla\phi)\|_{L^\infty L^2}\leq K_T.
\]
The Stokes mild formula, the boundedness of the Helmholtz projection
on $L^2$, and $\alpha=7/8<1$ therefore yield
\begin{equation}\label{eq:fixed-u-hierarchy-explicit}
\begin{aligned}
\sup_{0<t<S}\|u(t)\|_{D(A^\alpha)}
&\leq C\|u_0^\varepsilon\|_{D(A^\alpha)}
+K_T\int_0^T(1+s^{-\alpha})\,ds\\
&\leq K_T.
\end{aligned}
\end{equation}
Since $D(A^{7/8})\hookrightarrow H^{7/4}
\hookrightarrow L^\infty\cap W^{1,4}$, it follows that
\[
\|u\|_{L^\infty L^\infty}
+\|\nabla u\|_{L^\infty L^4}\leq K_T.
\]

For the concentration equation,
\eqref{eq:fixed-c-L2}--\eqref{eq:fixed-n-L2} and the preceding
velocity bound give
\[
\begin{aligned}
\|\bigl(u\cdot\nabla c+h_\tau(n)c\bigr)\|_{L^2L^2}
&\leq\|u\|_{L^\infty L^\infty}\|\nabla c\|_{L^2L^2}
+S^{1/2}\|c_0^\varepsilon\|_{L^\infty}\|n\|_{L^\infty L^2}\\
&\leq K_T.
\end{aligned}
\]
Applying \eqref{eq:Neumann-high-order-estimate} with $m=0$ to $c$
therefore gives
\[
\|c\|_{L^\infty H^1}+\|c\|_{L^2H^2}
+\|c_t\|_{L^2L^2}\leq K_T.
\]
Next, use the concentration mild formula and
$0\leq h_\tau(n)c\leq\|c_0^\varepsilon\|_{L^\infty}n$.
The $L^4$ gradient estimate for the Neumann heat semigroup gives,
for $0<t<S$,
\[
\|\nabla c(t)\|_{L^4}
\leq K_T
+K_T\int_0^t(t-s)^{-1/2}\|\nabla c(s)\|_{L^4}\,ds
+K_T\int_0^t(t-s)^{-7/8}\,ds.
\]
Here the transport term is estimated in $L^4$, whereas the reaction
term is estimated in $L^2$ using \eqref{eq:fixed-n-L2}.
The fractional Gronwall inequality gives a uniform $L^4$ bound
for $\nabla c$. Combining this with the preceding $H^1$ bound yields
\begin{equation}\label{eq:fixed-c-gradient-powers}
\sup_{0<t<S}\bigl(\|\nabla c(t)\|_{L^2}
+\|\nabla c(t)\|_{L^4}\bigr)\leq K_T.
\end{equation}

The cell flux satisfies
\[
\sup_{0<t<S}\|a_\tau(n)\nabla c+(\mathcal J_\varepsilon u)n\|_{L^2}
\leq\tau^{-1}\|\nabla c\|_{L^\infty L^2}
+\|(\mathcal J_\varepsilon u)\|_{L^\infty L^\infty}\|n\|_{L^\infty L^2}
\leq K_T.
\]
Its normal component vanishes on $\Gamma$. Hence the cell mild
formula and the $L^2$ to $L^4$ estimate for $\mathcal T_N$ imply
\[
\|n(t)\|_{L^4}
\leq C\|n_0^\varepsilon\|_{L^4}
+K_T\int_0^t(t-s)^{-7/8}\,ds\leq K_T.
\]
Together with \eqref{eq:fixed-c-gradient-powers}, this bounds the
same flux in $L^\infty(0,S;L^4)$. The $L^4$ to $L^\infty$
estimate for $\mathcal T_N$, again with singularity
$(t-s)^{-7/8}$, then gives $\|n\|_{L^\infty L^\infty}\leq K_T$.
Moreover,
\[
\sup_{0<t<S}\|u\cdot\nabla c+h_\tau(n)c\|_{L^4}
\leq\|u\|_{L^\infty L^\infty}\|\nabla c\|_{L^\infty L^4}
+\|c_0^\varepsilon\|_{L^\infty}\|n\|_{L^\infty L^4}
\leq K_T.
\]
Applying the $L^4$ to $L^\infty$ gradient heat estimate to the
concentration mild formula, we conclude that
\begin{equation}\label{eq:fixed-scalar-Linfty}
\sup_{0<t<S}\bigl(\|n(t)\|_{L^\infty}
+\|\nabla c(t)\|_{L^\infty}\bigr)\leq K_T.
\end{equation}

It remains to control the other Sobolev norms in the fixed-point space.
The cell equation reads
\[
n_t-\Delta n=-(\mathcal J_\varepsilon u)\cdot\nabla n-a_\tau'(n)\nabla n\cdot\nabla c
-a_\tau(n)\Delta c.
\]
Since $0\leq a_\tau'(n)\leq1$, we have
\[
\begin{aligned}
&\|(\mathcal J_\varepsilon u)\cdot\nabla n+a_\tau'(n)\nabla n\cdot\nabla c+a_\tau(n)\Delta c\|_{L^2L^2}\\
&\leq\bigl(\|(\mathcal J_\varepsilon u)\|_{L^\infty L^\infty}
+\|\nabla c\|_{L^\infty L^\infty}\bigr)
\|\nabla n\|_{L^2L^2}
+\tau^{-1}\|\Delta c\|_{L^2L^2}\\
&\leq K_T.
\end{aligned}
\]
Applying \eqref{eq:Neumann-high-order-estimate} with $m=0$ to $n$ gives
\[
\|n\|_{L^\infty H^1}+\|n\|_{L^2H^2}
+\|n_t\|_{L^2L^2}\leq K_T.
\]

Finally, $0\leq h_\tau'(n)\leq1$ and $h_\tau(n)\leq n$ give
\[
\begin{aligned}
\|\nabla \bigl(u\cdot\nabla c+h_\tau(n)c\bigr)\|_{L^2L^2}
\leq{}&S^{1/2}\|\nabla u\|_{L^\infty L^4}
\|\nabla c\|_{L^\infty L^4}
+\|u\|_{L^\infty L^\infty}\|D^2c\|_{L^2L^2}\\
&+\|c\|_{L^\infty L^\infty}\|\nabla n\|_{L^2L^2}
+\|n\|_{L^\infty L^\infty}\|\nabla c\|_{L^2L^2}
\leq K_T.
\end{aligned}
\]
Thus $\|\bigl(u\cdot\nabla c+h_\tau(n)c\bigr)\|_{L^2H^1}\leq K_T$. Since
$c_0^\varepsilon\in H_N^2$, applying
\eqref{eq:Neumann-high-order-estimate} with $m=1$ yields
\[
\|c\|_{L^\infty H^2}+\|c\|_{L^2H^3}
+\|c_t\|_{L^2H^1}\leq K_T.
\]
The standard time-continuity part of the linear regularity estimates
gives $n\in C([0,S];H^1)$ and $c\in C([0,S];H_N^2)$.
The mild formulas give continuity in the remaining fixed-point
norms: their time singularities are integrable, and the regularized
initial data are smooth and boundary compatible.

We have therefore bounded every component of
\eqref{eq:fixed-point-norm}, as well as all the norms in
\eqref{eq:fixed-parameter-high-order}, by constants independent of
$S<\min\{T,T_*\}$. If $T_*<\infty$, choose $T>T_*$ and let
$S\uparrow T_*$. The local contraction argument can then be
restarted at any $t_0<T_*$ with an existence time bounded below by
a positive constant depending only on the preceding uniform bounds
and the fixed parameters. Choosing $t_0$ sufficiently close to
$T_*$ extends the solution beyond $T_*$; uniqueness ensures agreement
with the original solution on their common interval. This
contradicts the maximality of $T_*$. Hence $T_*=\infty$, and
\eqref{eq:fixed-parameter-high-order} follows for every finite $T$.
These strong-solution bounds, together with spatial cutoffs and
time regularization, also justify the subsequent energy calculations.

\end{proof}

\subsection{Uniform entropy and integrability estimates}
\label{subsec:uniform-estimates-pressure}

The constants in this subsection are independent of
$\varepsilon$ and $\tau$. We use the time-independent data constant
$C_{\rm in}$ introduced in \eqref{eq:data-constant}; its value may
increase from line to line, with the same dependence on the data.

\begin{proposition}[Uniform entropy-energy estimate with explicit time dependence]
\label{prop:uniform-energy}
Let $(n_0^\varepsilon,c_0^\varepsilon,u_0^\varepsilon)$ satisfy
\begin{equation}
\left\{
\begin{aligned}
& n_0^\varepsilon \in L^1,\qquad
(n_0^\varepsilon+1)\log(n_0^\varepsilon+1)\in L^1,
\qquad u_0^\varepsilon\in L_\sigma^2,\\
& \nabla\sqrt{\tilde{c_0}^\varepsilon}\in L^2,\qquad
c_0^\varepsilon\in L^1\cap L^\infty,\\
& n_0^\varepsilon\geq 0,\qquad c_0^\varepsilon\geq 0.
\end{aligned}
\right.
\end{equation}

Assume that $(n^{\varepsilon,\tau},c^{\varepsilon,\tau},
u^{\varepsilon,\tau})$ is a solution of \eqref{eq:regularized-half-space-system}.
Then there exists a constant $C>0$ independent of
$\varepsilon$ and $\tau$ such that for every \(T<\infty\),
\begin{equation}
\label{eq:uniform-energy}
\begin{aligned}
&\sup_{0<t<T}\left\{\int_{\mathbb H^3}n^{\varepsilon,\tau}(t)+H(n^{\varepsilon,\tau}(t))\,dx
+\|\nabla \sqrt{\tilde{c}^{\varepsilon,\tau}}(t)\|_{L^2}^2
+\|u^{\varepsilon,\tau}(t)\|_{L^2}^2\right\}\\
&\quad+\int_0^T\int_{\mathbb H^3}\frac{|\nabla n^{\varepsilon,\tau}|^2}{1+n^{\varepsilon,\tau}}\,dxdt
+\int_0^T\|\nabla^2\sqrt{\tilde{c}^{\varepsilon,\tau}}\|_{L^2}^2\,dt
+\int_0^T\|\nabla u^{\varepsilon,\tau}\|_{L^2}^2\,dt\\
&\quad+\int_0^T\int_{\mathbb H^3}(\sqrt{\tilde{c}^{\varepsilon,\tau}})^{-2}|\nabla \sqrt{\tilde{c}^{\varepsilon,\tau}}|^4\,dxdt
+\int_0^T\int_{\mathbb H^3}h_\tau(n^{\varepsilon,\tau})|\nabla \sqrt{\tilde{c}^{\varepsilon,\tau}}|^2\,dxdt\\
&\leq C_{\rm in}(1+T),
\end{aligned}
\end{equation}
where the constant is independent of \(\varepsilon\) and \(\tau\). Since \(\sqrt{\tilde{c}^{\varepsilon,\tau}}=\sqrt{1+c^{\varepsilon,\tau}}\geq1\), every negative power in \eqref{eq:uniform-energy} is bounded and no \(c+\delta\) regularization is required.
\end{proposition}

\begin{proof}

We omit the approximation indices and recall that
$H(n)=(1+n)\log(1+n)$. The entropy test gives
\[
\frac{d}{dt}\int H(n)
+4\int|\nabla\sqrt{1+n}|^2
=\int\frac{n h_\tau'(n)}{1+n}\nabla n\cdot\nabla c.
\]
Use \eqref{eq:approx-c-combined}, proved below by integrations by parts,
with cutoffs tending to one, multiply it by $4/7$, and estimate the
fluid coupling by \eqref{eq:approx-fluid-coupling}. Adding the velocity
energy identity multiplied by $112M_c$, we obtain
\begin{align*}
&\frac{d}{dt}\int
\bigl(H(n)+2|\nabla\sqrt{1+c}|^2+112M_c|u|^2\bigr)
\\ &\qquad+4\|\nabla\sqrt{1+n}\|_{L^2}^2
+\frac47\|D^2\sqrt{1+c}\|_{L^2}^2
+2\int h_\tau(n)|\nabla\sqrt{1+c}|^2\\
&\qquad
+\frac14\int\frac{|\nabla\sqrt{1+c}|^4}{1+c}
+112M_c\|\nabla u\|_{L^2}^2
\\
&\leq\int\frac{4h_\tau'(n)(n-c)}{\sqrt{1+n}\sqrt{1+c}}
\nabla\sqrt{1+n}\cdot\nabla\sqrt{1+c}\\
&\qquad
-2\int\frac{h_\tau(n)}{1+c}|\nabla\sqrt{1+c}|^2
-224M_c\int \mathcal K_\varepsilon(n\nabla\phi)\cdot u.
\end{align*}
In particular, the factor $n/(1+n)$ in the entropy equation is retained.
Since $0\leq h_\tau'(n)\leq1$,
$n(h_\tau'(n))^2\leq h_\tau(n)$ and
$(n-c)^2/(1+n)\leq n+c^2$, Young's inequality gives
\begin{align*}
&\int\frac{4h_\tau'(n)(n-c)}{\sqrt{1+n}\sqrt{1+c}}
\nabla\sqrt{1+n}\cdot\nabla\sqrt{1+c}
-2\int\frac{h_\tau(n)}{1+c}|\nabla\sqrt{1+c}|^2\\
&\leq2\|\nabla\sqrt{1+n}\|_{L^2}^2\\
&\quad+2\int
\frac{(h_\tau'(n))^2(n-c)^2/(1+n)-h_\tau(n)}{1+c}
|\nabla\sqrt{1+c}|^2\\
&\leq2\|\nabla\sqrt{1+n}\|_{L^2}^2
+2\int\frac{c^2}{1+c}|\nabla\sqrt{1+c}|^2\\
&\leq2\|\nabla\sqrt{1+n}\|_{L^2}^2
+\tfrac12\|\nabla c\|_{L^2}^2.
\end{align*}

For the force, mass conservation gives
\[
\|\sqrt{1+n}-1\|_{L^2}^2\leq\|n_0^\varepsilon\|_{L^1},
\qquad
\nabla(\sqrt{1+n}-1)=\nabla\sqrt{1+n}.
\]
Splitting at $n=1$, applying the half-space Gagliardo--Nirenberg
inequality to $\sqrt{1+n}-1$, and interpolating between $L^1$
and $L^{5/3}$ yield
\[
\begin{aligned}
\int n^{5/3}
&\leq C\|n_0^\varepsilon\|_{L^1}^{2/3}
\|\nabla\sqrt{1+n}\|_{L^2}^2
+C\|n_0^\varepsilon\|_{L^1},\\
\|n\|_{L^{6/5}}^2
&\leq\|n_0^\varepsilon\|_{L^1}^{7/6}
\left(\int n^{5/3}\right)^{1/2}.
\end{aligned}
\]
Consequently, for each $\eta>0$,
\[
\|n\|_{L^{6/5}}^2
\leq\eta\|\nabla\sqrt{1+n}\|_{L^2}^2
+C_{\eta,\|n_0^\varepsilon\|_{L^1}}.
\]
Sobolev embedding and the $L^{6/5}$-boundedness of
$\mathcal K_\varepsilon$ therefore imply
\[
224M_c\left|\int \mathcal K_\varepsilon(n\nabla\phi)\cdot u\right|
\leq56M_c\|\nabla u\|_{L^2}^2
+\|\nabla\sqrt{1+n}\|_{L^2}^2+C_{\rm in}.
\]
Finally, the concentration $L^2$ identity gives
\[
\int_0^T\|\nabla c\|_{L^2}^2\,dt
\leq\frac12\|c_0^\varepsilon\|_{L^2}^2.
\]
Integration, mass conservation, and
\eqref{eq:initial-data-uniform-control} prove \eqref{eq:uniform-energy},
after enlarging $C_{\rm in}$. Every constant in this argument is
independent of $\varepsilon$ and $\tau$.

\end{proof}

\paragraph{Explicit spacetime consequences.}
For every \(T<\infty\), uniformly in \(\varepsilon,\tau\),
\begin{equation}
\label{eq:uniform-integrability}
\begin{aligned}
&\|n^{\varepsilon,\tau}\|_{L^{5/3}((0,T)\times\mathbb H^3)}
\leq C_{\rm in}(1+T)^{3/5},\\
&\|u^{\varepsilon,\tau}\|_{L^{10/3}((0,T)\times\mathbb H^3)}
+\|\mathcal J_\varepsilon u^{\varepsilon,\tau}\|_{L^{10/3}((0,T)\times\mathbb H^3)}
\leq C_{\rm in}(1+T)^{1/2},\\
&\|u^{\varepsilon,\tau}\|_{L^{10}(0,T;L^{30/13})}
+\|\mathcal J_\varepsilon u^{\varepsilon,\tau}\|_{L^{10}(0,T;L^{30/13})}
\leq C_{\rm in}(1+T)^{1/2}.
\end{aligned}
\end{equation}

To prove the global $L^{5/3}$ estimate for
$n^{\varepsilon,\tau}$, set
\[
w^{\varepsilon,\tau}
:=
\sqrt{1+n^{\varepsilon,\tau}}-1.
\]
Then
\[
|w^{\varepsilon,\tau}|^2
\leq n^{\varepsilon,\tau},
\qquad
\nabla w^{\varepsilon,\tau}
=
\nabla\sqrt{1+n^{\varepsilon,\tau}},
\]
and therefore, by mass conservation,
\[
\|w^{\varepsilon,\tau}(t)\|_{L^2(\mathbb H^3)}^2
\leq
\|n_0^\varepsilon\|_{L^1(\mathbb H^3)}
\leq C_{\rm in}.
\]
The Gagliardo--Nirenberg inequality on the half-space gives
\[
\|w^{\varepsilon,\tau}\|_{L^{10/3}}^{10/3}
\leq
C
\|w^{\varepsilon,\tau}\|_{L^2}^{4/3}
\|\nabla w^{\varepsilon,\tau}\|_{L^2}^{2}.
\]
Since
\[
(n^{\varepsilon,\tau})^{5/3}
\leq
C|w^{\varepsilon,\tau}|^{10/3}
+
Cn^{\varepsilon,\tau},
\]
we obtain
\[
\int_{\mathbb H^3}
(n^{\varepsilon,\tau})^{5/3}\,dx
\leq
C_{\rm in}
\int_{\mathbb H^3}
\left|
\nabla\sqrt{1+n^{\varepsilon,\tau}}
\right|^2dx
+
C_{\rm in}.
\]
Integrating in time and using
\eqref{eq:uniform-energy} yields
\[
\int_0^T\int_{\mathbb H^3}
(n^{\varepsilon,\tau})^{5/3}\,dxdt
\leq
C_{\rm in}(1+T).
\]

This proves the first estimate in \eqref{eq:uniform-integrability}.
We claim that, for every \(T<\infty\),
\begin{equation}
\label{eq:global-pressure-source-bounds}
\begin{aligned}
&\left\|
(\mathcal J_\varepsilon u^{\varepsilon,\tau})\cdot\nabla u^{\varepsilon,\tau}
\right\|_
{L^{5/3}(0,T;L^{15/14}(\mathbb H^3))}
\leq
C_{\rm in}(1+T),
\\
&\left\|
\mathcal K_\varepsilon(n^{\varepsilon,\tau}\nabla\phi)
\right\|_
{L^{5/3}(0,T;L^{15/14}(\mathbb H^3))}
\leq
C_{\rm in}(1+T)^{3/5}.
\end{aligned}
\end{equation}
In particular,
\begin{equation}
\label{eq:global-G-bound}
\left\|
\bigl((\mathcal J_\varepsilon u^{\varepsilon,\tau})\cdot\nabla u^{\varepsilon,\tau}+\mathcal K_\varepsilon(n^{\varepsilon,\tau}\nabla\phi)\bigr)
\right\|_
{L^{5/3}(0,T;L^{15/14}(\mathbb H^3))}
\leq
C_{\rm in}(1+T),
\end{equation}
uniformly with respect to \(\varepsilon\) and \(\tau\).

Indeed, by global
energy estimate,
\[
\begin{aligned}
&
\left\|
(\mathcal J_\varepsilon u^{\varepsilon,\tau})\cdot\nabla u^{\varepsilon,\tau}
\right\|_
{L^{5/3}(0,T;L^{15/14}(\mathbb H^3))}
\\
&\qquad\leq
\left\|
(\mathcal J_\varepsilon u^{\varepsilon,\tau})
\right\|_
{L^{10}(0,T;L^{30/13}(\mathbb H^3))}
\left\|
\nabla u^{\varepsilon,\tau}
\right\|_
{L^2(0,T;L^2(\mathbb H^3))}
\\
&\qquad\leq
C_{\rm in}(1+T).
\end{aligned}
\]

For the force term, mass conservation and interpolation between
\(L^1(\mathbb H^3)\) and \(L^{5/3}(\mathbb H^3)\) give, for almost
every \(t\),
\[
\left\|n^{\varepsilon,\tau}(t)\right\|_{L^{15/14}}
\leq
\left\|n^{\varepsilon,\tau}(t)\right\|_{L^1}^{5/6}
\left\|n^{\varepsilon,\tau}(t)\right\|_{L^{5/3}}^{1/6}.
\]
Consequently,
\begin{align*}
&
\left\|
n^{\varepsilon,\tau}
\right\|_
{L^{5/3}(0,T;L^{15/14}(\mathbb H^3))}
\\
&\qquad\leq
\left(
\sup_{0<t<T}
\left\|n^{\varepsilon,\tau}(t)\right\|_{L^1}
\right)^{5/6}
T^{1/2}
\left\|
n^{\varepsilon,\tau}
\right\|_
{L^{5/3}((0,T)\times\mathbb H^3)}^{1/6}
\\
&\qquad\leq
C_{\rm in}(1+T)^{3/5}.
\end{align*}
Since $\mathcal K_\varepsilon$ is bounded on
$L^{15/14}(\mathbb H^3)$,
we obtain
\[
\left\|
\mathcal K_\varepsilon(n^{\varepsilon,\tau}\nabla\phi)
\right\|_
{L^{5/3}(0,T;L^{15/14}(\mathbb H^3))}
\leq
C\|\nabla\phi\|_{L^\infty}
\left\|
n^{\varepsilon,\tau}
\right\|_
{L^{5/3}(0,T;L^{15/14}(\mathbb H^3))},
\]
which proves \eqref{eq:global-pressure-source-bounds}.

\subsection{Local pressure estimates away from the initial time}
\label{subsec:local-pressure-estimates}

\begin{proposition}[Uniform local pressure estimate away from the initial time]
\label{prop:uniform-pressure}
Let
\[
0<\sigma<T<\infty,
\]
and let $D\subset\mathbb H^3$ be a bounded Lipschitz domain whose closure
is compact in $\overline{\mathbb H^3}$. If $\overline D\cap\Gamma\neq\varnothing$,
we assume that the corresponding portion of $\partial D$ agrees locally
with the flat boundary $\Gamma$. Then
\begin{equation}
\label{eq:pressure-five-thirds}
\left\|
p^{\varepsilon,\tau}
-
[p^{\varepsilon,\tau}]_D(t)
\right\|_
{L^{5/3}((\sigma,T)\times D)}
\leq
C_{\sigma,D,C_{\rm in}}(1+T),
\end{equation}
where
\[
[p^{\varepsilon,\tau}]_D(t)
:=
\fint_D p^{\varepsilon,\tau}(x,t)\,dx,
\]
and the constant is independent of $\varepsilon$ and $\tau$.
No uniform $L^{5/3}$ estimate for the pressure at the genuine initial
time $t=0$ is asserted.
\end{proposition}

\begin{proof}
By \eqref{eq:global-G-bound},
\begin{equation}
\label{eq:G-global-pressure}
\left\|
\bigl((\mathcal J_\varepsilon u^{\varepsilon,\tau})\cdot\nabla u^{\varepsilon,\tau}+\mathcal K_\varepsilon(n^{\varepsilon,\tau}\nabla\phi)\bigr)
\right\|_
{L^{5/3}(0,T;L^{15/14}(\mathbb H^3))}
\leq
C_{\rm in}(1+T).
\end{equation}

The crucial point is that the Stokes decomposition is performed on
the whole half-space, before any spatial localization.  This avoids
using the nonlocal Helmholtz projection on data known only in a
bounded subdomain.

Let
\((v^{\varepsilon,\tau},p_1^{\varepsilon,\tau})\)
be the solution of the zero-initial-data Stokes problem
\begin{equation}
\label{eq:global-pressure-stokes-forced}
\left\{
\begin{aligned}
\partial_t v^{\varepsilon,\tau}
-\Delta v^{\varepsilon,\tau}
+\nabla p_1^{\varepsilon,\tau}
&=
-\bigl((\mathcal J_\varepsilon u^{\varepsilon,\tau})\cdot\nabla u^{\varepsilon,\tau}+\mathcal K_\varepsilon(n^{\varepsilon,\tau}\nabla\phi)\bigr)
&&\text{in }(0,T)\times\mathbb H^3,
\\
\nabla\cdot v^{\varepsilon,\tau}&=0
&&\text{in }(0,T)\times\mathbb H^3,
\\
v^{\varepsilon,\tau}&=0
&&\text{on }(0,T)\times\Gamma,
\\
v^{\varepsilon,\tau}(\cdot,0)&=0
&&\text{in }\mathbb H^3.
\end{aligned}
\right.
\end{equation}
By the maximal $L^{5/3}_tL^{15/14}_x$-regularity of the
Dirichlet Stokes system in the half-space
(see \cite[Theorem~6.3 and the subsequent discussion,
p.~136]{DHP2001}), we obtain
\begin{equation}
\label{eq:global-stokes-maximal-pressure}
\begin{aligned}
&
\left\|
\partial_t v^{\varepsilon,\tau}
\right\|_
{L^{5/3}(0,T;L^{15/14})}
+
\left\|
\nabla^2 v^{\varepsilon,\tau}
\right\|_
{L^{5/3}(0,T;L^{15/14})}
\\
&\qquad
+
\left\|
\nabla p_1^{\varepsilon,\tau}
\right\|_
{L^{5/3}(0,T;L^{15/14})}\\
&\leq
C
\left\|
\bigl((\mathcal J_\varepsilon u^{\varepsilon,\tau})\cdot\nabla u^{\varepsilon,\tau}+\mathcal K_\varepsilon(n^{\varepsilon,\tau}\nabla\phi)\bigr)
\right\|_
{L^{5/3}(0,T;L^{15/14})}
\leq
C_{\rm in}(1+T),
\end{aligned}
\end{equation}
where the spatial norms in this estimate are over
\(\mathbb H^3\).

Since $14/15-1/3=3/5$,
the Sobolev--Poincar\'e inequality on the bounded domain \(D\) gives
\begin{equation}
\label{eq:p1-local-five-thirds}
\left\|
p_1^{\varepsilon,\tau}
-
[p_1^{\varepsilon,\tau}]_D(t)
\right\|_
{L^{5/3}(0,T;L^{5/3}(D))}
\leq
C_D
\left\|
\nabla p_1^{\varepsilon,\tau}
\right\|_
{L^{5/3}(0,T;L^{15/14}(\mathbb H^3))}
\leq
C_{D,C_{\rm in}}(1+T).
\end{equation}

For fixed $\varepsilon,\tau>0$, the strong-solution estimates
also give $\bigl((\mathcal J_\varepsilon u^{\varepsilon,\tau})\cdot\nabla u^{\varepsilon,\tau}+\mathcal K_\varepsilon(n^{\varepsilon,\tau}\nabla\phi)\bigr)\in L^2(0,T;L^2(\mathbb H^3))$ for the forcing
in the Stokes problem above. The Stokes semigroups and Helmholtz
projections are consistent on intersections of their finite-$L^r$
spaces \cite{DHP2001}. Consequently the zero-initial-data solution
$v^{\varepsilon,\tau}$ constructed by mixed-norm maximal regularity
coincides with the $L^2$ Stokes solution. This permits the following
identification in the global energy class.

We next isolate the part generated by the initial velocity.  Define
\[
w^{\varepsilon,\tau}
:=
u^{\varepsilon,\tau}
-
v^{\varepsilon,\tau},
\qquad
p_2^{\varepsilon,\tau}
:=
p^{\varepsilon,\tau}
-
p_1^{\varepsilon,\tau}.
\]
Then
\((w^{\varepsilon,\tau},p_2^{\varepsilon,\tau})\)
solves the homogeneous Stokes system
\begin{equation}
\label{eq:global-homogeneous-stokes-pressure}
\left\{
\begin{aligned}
\partial_t w^{\varepsilon,\tau}
-\Delta w^{\varepsilon,\tau}
+\nabla p_2^{\varepsilon,\tau}
&=0
&&\text{in }(0,T)\times\mathbb H^3,
\\
\nabla\cdot w^{\varepsilon,\tau}&=0,
\\
w^{\varepsilon,\tau}&=0
&&\text{on }(0,T)\times\Gamma,
\\
w^{\varepsilon,\tau}(\cdot,0)
&=
u_0^\varepsilon.
\end{aligned}
\right.
\end{equation}
Thus
\[
w^{\varepsilon,\tau}(t)
=
e^{-tA}u_0^\varepsilon,
\]
with the same Dirichlet Stokes operator $A$ on $L_\sigma^2(\mathbb H^3)$.

Analyticity of the Stokes semigroup gives
\[
\left\|
Ae^{-tA}u_0^\varepsilon
\right\|_{L^2}
\leq
Ct^{-1}
\left\|u_0^\varepsilon\right\|_{L^2},
\qquad t>0.
\]
Together with the half-space stationary Stokes estimate, this yields
\begin{equation}
\label{eq:homogeneous-pressure-decay}
\left\|
\nabla p_2^{\varepsilon,\tau}(t)
\right\|_{L^2(\mathbb H^3)}
\leq
Ct^{-1}
\left\|u_0^\varepsilon\right\|_{L^2},
\qquad t>0.
\end{equation}
Therefore, since \(D\) is bounded,
\begin{align}
&
\left\|
p_2^{\varepsilon,\tau}
-
[p_2^{\varepsilon,\tau}]_D(t)
\right\|_
{L^{5/3}(\sigma,T;L^{5/3}(D))}
\nonumber\\
&\qquad\leq
C_D
\left\|
\nabla p_2^{\varepsilon,\tau}
\right\|_
{L^{5/3}(\sigma,T;L^{15/14}(D))}
\nonumber\\
&\qquad\leq
C_D
\left\|
\nabla p_2^{\varepsilon,\tau}
\right\|_
{L^{5/3}(\sigma,T;L^2(\mathbb H^3))}
\nonumber\\
&\qquad\leq
C_D
\left(
\int_\sigma^T t^{-5/3}\,dt
\right)^{3/5}
\left\|u_0^\varepsilon\right\|_{L^2}
\nonumber\\
&\qquad\leq
C_{\sigma,D,C_{\rm in}}.
\label{eq:p2-local-five-thirds}
\end{align}

Finally,
\[
p^{\varepsilon,\tau}
-
[p^{\varepsilon,\tau}]_D(t)
=
\left(
p_1^{\varepsilon,\tau}
-
[p_1^{\varepsilon,\tau}]_D(t)
\right)
+
\left(
p_2^{\varepsilon,\tau}
-
[p_2^{\varepsilon,\tau}]_D(t)
\right).
\]
Combining
\eqref{eq:p1-local-five-thirds}
and
\eqref{eq:p2-local-five-thirds}
proves
\eqref{eq:pressure-five-thirds}.
\end{proof}

\begin{remark}
\label{rem:pressure-away-initial}
For every $0<\sigma<T<\infty$, Proposition~\ref{prop:uniform-pressure}
provides a uniform local pressure estimate on the time interval
$(\sigma,T)$, with a constant that may depend on $\sigma$ but is
independent of the regularization parameters $\varepsilon$ and $\tau$.
Consequently, for every compact interval
$I\Subset(0,\infty)$ and every bounded domain
$D\subset\mathbb H^3$ whose closure is compact in
$\overline{\mathbb H^3}$, the normalized pressures $p^{\varepsilon,\tau}
-
[p^{\varepsilon,\tau}]_D(t)$ are weakly
compact in $L^{5/3}(I\times D)$.
This local-in-time pressure control is sufficient for the subsequent
compactness and regularity arguments, since all distributional test
functions and all test functions entering the local energy inequality
are supported away from the genuine initial time $t=0$.
\end{remark}


\subsection{The local energy inequality for approximate solutions}
\label{subsec:approximate-local-energy}

We first record the localized energy inequality for the regularized solutions.

Let \(\psi\geq0\) be smooth on the closure of a bounded half-cylinder, vanish at the initial time and in a neighborhood of the nonflat lateral boundary, and satisfy \(\partial_3\psi=0\) on the flat boundary.
Here \(\overline p(s)\) denotes any spatially constant normalization on a ball containing the support of \(\psi(\cdot,s)\).

\begin{lemma}[Local energy inequality for the approximation]
\label{lem:approximate-local-energy}
To keep the displayed inequality compact, only in this lemma and its proof we omit the superscript \((\varepsilon,\tau)\) and write
\[
\begin{aligned}
(n,c,u,p)&=(n^{\varepsilon,\tau},c^{\varepsilon,\tau},u^{\varepsilon,\tau},p^{\varepsilon,\tau}),
&\qquad \tilde c&=1+c,\\
a&=a_\tau(n),\qquad h=h_\tau(n).
\end{aligned}
\]
For almost every \(t>0\),
\begin{align}
&\int_{\mathbb H^3}\!\big[n\log n+2|\nabla\sqrt{\tilde c}|^2+112M_c|u|^2\big]\psi(t)\,dx \nonumber\\
&\quad+\int_{0}^{t}\int_{\mathbb H^3}\!\Big[2|\nabla\sqrt n|^2+\frac47|\nabla^2\sqrt{\tilde c}|^2
+2h|\nabla\sqrt{\tilde c}|^2
+\frac14\tilde c^{-1}|\nabla\sqrt{\tilde c}|^4
+112M_c|\nabla u|^2\Big]\psi\,dxds \nonumber\\
&\leq\int_{0}^{t}\int_{\mathbb H^3}\!\Big[n\log n\big(\partial_s\psi+\Delta\psi+(\mathcal J_\varepsilon u)\cdot\nabla\psi\big)
+a(1+\log n)\nabla c\cdot\nabla\psi \nonumber\\
&\quad+2|\nabla\sqrt{\tilde c}|^2
\big(\partial_s\psi+\Delta\psi+u\cdot\nabla\psi\big) \nonumber\\
&\quad-\frac47\tilde c^{-1/2}|\nabla\sqrt{\tilde c}|^2
\nabla\sqrt{\tilde c}\cdot\nabla\psi
-\frac{16}{7}\tilde c^{-1/2}\sum_{i=1}^3
(\partial_i\sqrt{\tilde c})^3\partial_i\psi \nonumber\\
&\quad+112M_c\Big\{|u|^2
\big(\partial_s\psi+\Delta\psi+(\mathcal J_\varepsilon u)\cdot\nabla\psi\big)
\nonumber\\
&\qquad+2(p-\overline p(s))u\cdot\nabla\psi-2\mathcal K_\varepsilon(n\nabla\phi)\cdot u\,\psi\Big\}\Big]\,dxds.
\label{eq:approximate-local-energy}
  \end{align}
\end{lemma}

\begin{proof}
We keep every modification caused by the regularization. In this proof we suppress \((\varepsilon,\tau)\) and write
\[
a=a_\tau(n),\qquad h=h_\tau(n),\qquad \tilde c=1+c.
\]
We use \(\iint\) for integration over \((0,t)\times\mathbb H^3\) and \(\iint_\Gamma\) for integration over \((0,t)\times\Gamma\). The calculation is first performed with a compact spatial cutoff, so the only boundary that remains is \(\Gamma\). Its outward normal is \(\nu=-e_3\). On \(\Gamma\),
\[
(\mathcal J_\varepsilon u)=u=0,\qquad \partial_\nu n=\partial_\nu c=\partial_\nu \sqrt{\tilde{c}}=0,
\qquad \partial_\nu\psi=0.
\]
Since the solution is smooth, \(\partial_{i3}\sqrt{\tilde{c}}=\partial_i(\partial_3\sqrt{\tilde{c}})=0\) on \(\Gamma\) for \(i=1,2\). We do not suppress any boundary integral below: each one is displayed and then evaluated by these conditions. Notice also that \(\sqrt{\tilde{c}}\geq1\). Therefore the factors \(\sqrt{\tilde{c}}^{-1}\) and \(\sqrt{\tilde{c}}^{-2}\) are bounded and the concentration calculation contains no singular negative power.

\emph{The cell equation.} Multiply
\[
\partial_tn+(\mathcal J_\varepsilon u)\cdot\nabla n-\Delta n+\nabla\cdot(a\nabla c)=0
\]
by \((1+\log n)\psi\). If \(n\) can vanish, replace \(\log n\) by \(\log(n+\delta)\) and let \(\delta\downarrow0\) at the end. At the $\delta$-level the chemotactic term contains the factor
\[
\frac{n}{n+\delta}h_\tau'(n),
\]
rather than $h_\tau'(n)$.  Since
\[
0\leq\frac{n}{n+\delta}\leq1
\]
and
\[
\frac{n}{n+\delta}\longrightarrow1
\qquad\text{a.e. on }\{n>0\},
\]
all identities used below are first derived at fixed $\delta>0$,
and the displayed cancellation is recovered after letting
$\delta\downarrow0$ by dominated convergence.

Since \(\psi(\cdot,0)=0\), the time term is
\[
\iint\partial_sn\,(1+\log n)\psi=\int_{\mathbb H^3}n\log n\,\psi(t)\,dx-\iint n\log n\,\partial_s\psi.
\]
For the transport term,
\begin{align*}
&\iint (\mathcal J_\varepsilon u)\cdot\nabla(n\log n)\,\psi\\
&=\iint_\Gamma n\log n\,\psi\,(\mathcal J_\varepsilon u)\cdot\nu
-\iint n\log n\,(\mathcal J_\varepsilon u)\cdot\nabla\psi\\
&=-\iint n\log n\,(\mathcal J_\varepsilon u)\cdot\nabla\psi,
\end{align*}
because \((\mathcal J_\varepsilon u)=0\) on \(\Gamma\). For the diffusion term,
\begin{align*}
&-\iint\Delta n\,(1+\log n)\psi\\
&={}-\iint_\Gamma\partial_\nu n(1+\log n)\psi
+\iint\frac{|\nabla n|^2}{n}\psi
+\iint(1+\log n)\nabla n\cdot\nabla\psi\\
&={}4\iint|\nabla\sqrt n|^2\psi
-\iint n\log n\,\Delta\psi
+\iint_\Gamma n\log n\,\partial_\nu\psi\\
&={}4\iint|\nabla\sqrt n|^2\psi
-\iint n\log n\,\Delta\psi.
\end{align*}
The first boundary integral is zero because \(\partial_\nu n=0\), and the second is zero because \(\partial_\nu\psi=0\). Finally,
\begin{align*}
&\iint\nabla\cdot(a\nabla c)(1+\log n)\psi\\
&={}\iint_\Gamma a\partial_\nu c(1+\log n)\psi
-\iint\frac{a}{n}\nabla n\cdot\nabla c\,\psi
-\iint a(1+\log n)\nabla c\cdot\nabla\psi\\
&={}-\iint\nabla h\cdot\nabla c\,\psi
-\iint a(1+\log n)\nabla c\cdot\nabla\psi,
\end{align*}
because \(\partial_\nu c=0\) on \(\Gamma\).
Here we used \(a=nh_\tau'(n)\), hence \(a\nabla n/n=\nabla h\). Adding the four contributions gives
\begin{align}
&\int_{\mathbb H^3}n\log n\,\psi(t)\,dx+4\iint|\nabla\sqrt n|^2\psi\notag\\
&\quad=\iint n\log n(\partial_s\psi+\Delta\psi)+\iint n\log n\,(\mathcal J_\varepsilon u)\cdot\nabla\psi
+\iint\nabla h\cdot\nabla c\,\psi\notag\\&\quad
+\iint a(1+\log n)\nabla c\cdot\nabla\psi.
\label{eq:approx-cell-local}
\end{align}

\emph{The concentration equation.} Because \(\tilde c=1+c\),
\[
\partial_t\tilde c+u\cdot\nabla\tilde c-\Delta\tilde c=-h\tilde c+h.
\]
Using \(\Delta\tilde c=2|\nabla\sqrt{\tilde{c}}|^2+2\sqrt{\tilde{c}}\Delta \sqrt{\tilde{c}}\), we obtain
\begin{equation}
\label{eq:approx-s-equation}
\partial_t\sqrt{\tilde{c}}+u\cdot\nabla \sqrt{\tilde{c}}-(\sqrt{\tilde{c}})^{-1}|\nabla \sqrt{\tilde{c}}|^2-\Delta \sqrt{\tilde{c}}=-\frac12h\sqrt{\tilde{c}}+\frac12h(\sqrt{\tilde{c}})^{-1}.
\end{equation}
Multiply \eqref{eq:approx-s-equation} by \(-\partial_i(\partial_i\sqrt{\tilde{c}}\,\psi)\), sum over \(i\), and integrate.

For the time derivative,
\begin{align*}
&-\iint\partial_s\sqrt{\tilde c}\,\nabla\cdot(\psi\nabla\sqrt{\tilde c})\\
&=-\iint_\Gamma\partial_s\sqrt{\tilde{c}}\,\partial_\nu \sqrt{\tilde{c}}\,\psi
+\frac12\int_{\mathbb H^3}|\nabla \sqrt{\tilde{c}}|^2\psi(t)\,dx
-\frac12\iint|\nabla \sqrt{\tilde{c}}|^2\partial_s\psi\\
&=\frac12\int_{\mathbb H^3}|\nabla \sqrt{\tilde{c}}|^2\psi(t)\,dx
-\frac12\iint|\nabla \sqrt{\tilde{c}}|^2\partial_s\psi,
\end{align*}
because \(\partial_\nu \sqrt{\tilde{c}}=0\). For the transport term, two spatial integrations give
\begin{align*}
-\iint(u\cdot\nabla\sqrt{\tilde c})\nabla\cdot(\psi\nabla\sqrt{\tilde c})
={}&-\iint_\Gamma(u\cdot\nabla \sqrt{\tilde{c}})\partial_\nu \sqrt{\tilde{c}}\,\psi
+\iint\nabla u:(\nabla \sqrt{\tilde{c}}\otimes\nabla \sqrt{\tilde{c}})\psi\\
&-\frac12\iint|\nabla \sqrt{\tilde{c}}|^2u\cdot\nabla\psi
+\frac12\iint_\Gamma|\nabla \sqrt{\tilde{c}}|^2\psi\,u\cdot\nu\\
={}&\iint\nabla u:(\nabla \sqrt{\tilde{c}}\otimes\nabla \sqrt{\tilde{c}})\psi
-\frac12\iint|\nabla \sqrt{\tilde{c}}|^2u\cdot\nabla\psi.
\end{align*}
Both boundary terms vanish because \(u=0\) on \(\Gamma\); the first also contains \(\partial_\nu \sqrt{\tilde{c}}\).

The nonlinear diffusion term is expanded algebraically and creates no boundary integral:
\[
\iint(\sqrt{\tilde c})^{-1}|\nabla\sqrt{\tilde c}|^2\nabla\cdot(\psi\nabla\sqrt{\tilde c})=\iint (\sqrt{\tilde{c}})^{-1}|\nabla \sqrt{\tilde{c}}|^2\Delta \sqrt{\tilde{c}}\,\psi
+\iint (\sqrt{\tilde{c}})^{-1}|\nabla \sqrt{\tilde{c}}|^2\nabla \sqrt{\tilde{c}}\cdot\nabla\psi.
\]
For the Laplacian term,
\begin{align*}
&\iint\Delta\sqrt{\tilde c}\,\nabla\cdot(\psi\nabla\sqrt{\tilde c})\\
&={}\iint_\Gamma\Delta \sqrt{\tilde{c}}\,\partial_\nu \sqrt{\tilde{c}}\,\psi
-\sum_{i,j}\iint\partial_{ijj}\sqrt{\tilde{c}}\,\partial_i\sqrt{\tilde{c}}\,\psi\\
&={}\iint_\Gamma\Delta \sqrt{\tilde{c}}\,\partial_\nu \sqrt{\tilde{c}}\,\psi
-\sum_{i,j}\iint_\Gamma\partial_{ij}\sqrt{\tilde{c}}\,\nu_j\partial_i\sqrt{\tilde{c}}\,\psi
+\sum_{i,j}\iint\partial_{ij}\sqrt{\tilde{c}}\,\partial_j(\partial_i\sqrt{\tilde{c}}\,\psi)\\
&={}\iint|\nabla^2\sqrt{\tilde{c}}|^2\psi-\frac12\iint|\nabla \sqrt{\tilde{c}}|^2\Delta\psi.
\end{align*}
The first boundary term is zero because \(\partial_\nu \sqrt{\tilde{c}}=0\). In the second, \(\nu_j=0\) unless \(j=3\); for \(i=1,2\), \(\partial_{i3}\sqrt{\tilde c}=0\), while for \(i=3\), the factor \(\partial_3\sqrt{\tilde{c}}\) is zero. The final integration of \(\partial_{ij}\sqrt{\tilde{c}}\,\partial_i\sqrt{\tilde{c}}\,\partial_j\psi\) produces the boundary term \(\frac12\iint_\Gamma|\nabla \sqrt{\tilde{c}}|^2\partial_\nu\psi\), which is zero because \(\partial_\nu\psi=0\).

The two reaction terms satisfy
\begin{align*}
-\frac12\iint h\sqrt{\tilde c}\,\nabla\cdot(\psi\nabla\sqrt{\tilde c})
&=-\frac12\iint_\Gamma h\sqrt{\tilde{c}}\,\partial_\nu \sqrt{\tilde{c}}\,\psi
+\frac12\iint\nabla(h\sqrt{\tilde{c}})\cdot\nabla \sqrt{\tilde{c}}\,\psi\\
&=\frac12\iint h|\nabla \sqrt{\tilde{c}}|^2\psi+\frac14\iint\nabla h\cdot\nabla c\,\psi,
\end{align*}
and
\begin{align*}
\frac12\iint h(\sqrt{\tilde c})^{-1}\nabla\cdot(\psi\nabla\sqrt{\tilde c})
&=\frac12\iint_\Gamma h(\sqrt{\tilde{c}})^{-1}\partial_\nu \sqrt{\tilde{c}}\,\psi
-\frac12\iint\nabla(h(\sqrt{\tilde{c}})^{-1})\cdot\nabla \sqrt{\tilde{c}}\,\psi\\
&=-\frac12\iint\nabla(h(\sqrt{\tilde{c}})^{-1})\cdot\nabla \sqrt{\tilde{c}}\,\psi.
\end{align*}
Both boundary integrals vanish because \(\partial_\nu \sqrt{\tilde{c}}=0\). Adding the six contributions therefore gives
\begin{align}
&\frac12\int|\nabla \sqrt{\tilde{c}}|^2\psi(t)+\iint|\nabla^2\sqrt{\tilde{c}}|^2\psi
\notag\\&=\frac12\iint|\nabla \sqrt{\tilde{c}}|^2(\partial_s\psi+\Delta\psi)
+\frac12\iint|\nabla \sqrt{\tilde{c}}|^2u\cdot\nabla\psi\notag\\
&-\iint\nabla u:(\nabla \sqrt{\tilde{c}}\otimes\nabla \sqrt{\tilde{c}})\psi
-\iint (\sqrt{\tilde{c}})^{-1}|\nabla \sqrt{\tilde{c}}|^2\Delta \sqrt{\tilde{c}}\,\psi\notag\\
&-\iint (\sqrt{\tilde{c}})^{-1}|\nabla \sqrt{\tilde{c}}|^2\nabla \sqrt{\tilde{c}}\cdot\nabla\psi
-\frac12\iint h|\nabla \sqrt{\tilde{c}}|^2\psi\notag\\
&-\frac14\iint\nabla h\cdot\nabla c\,\psi
+\frac12\iint\nabla(h(\sqrt{\tilde{c}})^{-1})\cdot\nabla \sqrt{\tilde{c}}\,\psi.
\label{eq:approx-c-basic}
\end{align}

We now treat
\[
I:=-\iint (\sqrt{\tilde{c}})^{-1}|\nabla \sqrt{\tilde{c}}|^2\Delta \sqrt{\tilde{c}}\,\psi
=-\sum_{i,j}\iint (\sqrt{\tilde{c}})^{-1}(\partial_j\sqrt{\tilde{c}})^2\partial_{ii}\sqrt{\tilde{c}}\,\psi.
\]
An integration by parts in \(x_i\) yields
\begin{align*}
I={}&-\sum_{i,j}\iint_\Gamma (\sqrt{\tilde{c}})^{-1}(\partial_j\sqrt{\tilde{c}})^2\partial_i\sqrt{\tilde{c}}\,\nu_i\psi
-\sum_{i,j}\iint (\sqrt{\tilde{c}})^{-2}(\partial_i\sqrt{\tilde{c}})^2(\partial_j\sqrt{\tilde{c}})^2\psi\\
&+\iint (\sqrt{\tilde{c}})^{-1}|\nabla \sqrt{\tilde{c}}|^2\nabla \sqrt{\tilde{c}}\cdot\nabla\psi
+2\left(\sum_{i=j}+\sum_{i\neq j}\right)\iint (\sqrt{\tilde{c}})^{-1}\partial_{ij}\sqrt{\tilde{c}}\,\partial_i\sqrt{\tilde{c}}\,\partial_j\sqrt{\tilde{c}}\,\psi.
\end{align*}
The displayed boundary term is zero: \(\nu_i=0\) unless \(i=3\), and then it contains \(\partial_3\sqrt{\tilde{c}}=0\).
Using the diagonal part of this identity to eliminate \(\sum_i (\sqrt{\tilde{c}})^{-1}\partial_{ii}\sqrt{\tilde{c}}(\partial_i\sqrt{\tilde{c}})^2\), we obtain
\begin{align*}
I={}&-\frac13\sum_{i,j}\iint (\sqrt{\tilde{c}})^{-2}(\partial_i\sqrt{\tilde{c}})^2(\partial_j\sqrt{\tilde{c}})^2\psi
+\frac13\iint (\sqrt{\tilde{c}})^{-1}|\nabla \sqrt{\tilde{c}}|^2\nabla \sqrt{\tilde{c}}\cdot\nabla\psi\\
&+\frac23\sum_{i\neq j}\iint (\sqrt{\tilde{c}})^{-1}\left(\partial_{ij}\sqrt{\tilde{c}}\,\partial_i\sqrt{\tilde{c}}\,\partial_j\sqrt{\tilde{c}}-\partial_{ii}\sqrt{\tilde{c}}(\partial_j\sqrt{\tilde{c}})^2\right)\psi.
\end{align*}
For the first off-diagonal term, Young's inequality gives
\begin{align*}
\frac23\sum_{i\neq j}\iint (\sqrt{\tilde{c}})^{-1}\partial_{ij}\sqrt{\tilde{c}}\,\partial_i\sqrt{\tilde{c}}\,\partial_j\sqrt{\tilde{c}}\,\psi&
\leq\frac13\sum_{i\neq j}\iint (\sqrt{\tilde{c}})^{-2}|\partial_i\sqrt{\tilde{c}}|^2|\partial_j\sqrt{\tilde{c}}|^2\psi\\&\quad
+\frac13\sum_{i\neq j}\iint|\partial_{ij}\sqrt{\tilde{c}}|^2\psi.
\end{align*}
Integrating the second off-diagonal term by parts once more gives
\begin{align*}
&-\frac23\sum_{i\neq j}\iint (\sqrt{\tilde{c}})^{-1}\partial_{ii}\sqrt{\tilde{c}}(\partial_j\sqrt{\tilde{c}})^2\psi\\
&=-\frac23\sum_{i\neq j}\iint_\Gamma (\sqrt{\tilde{c}})^{-1}\partial_i\sqrt{\tilde{c}}(\partial_j\sqrt{\tilde{c}})^2\nu_i\psi-\frac23\sum_{i\neq j}\iint (\sqrt{\tilde{c}})^{-2}(\partial_i\sqrt{\tilde{c}})^2(\partial_j\sqrt{\tilde{c}})^2\psi\\
&\quad+\frac43\sum_{i\neq j}\iint (\sqrt{\tilde{c}})^{-1}\partial_{ij}\sqrt{\tilde{c}}\,\partial_i\sqrt{\tilde{c}}\,\partial_j\sqrt{\tilde{c}}\,\psi+\frac23\sum_{i\neq j}\iint (\sqrt{\tilde{c}})^{-1}|\partial_j\sqrt{\tilde{c}}|^2\partial_i\sqrt{\tilde{c}}\,\partial_i\psi\\
&\leq\frac23\sum_{i\neq j}\iint|\partial_{ij}\sqrt{\tilde{c}}|^2\psi
+\frac23\sum_{i\neq j}\iint (\sqrt{\tilde{c}})^{-1}|\partial_j\sqrt{\tilde{c}}|^2\partial_i\sqrt{\tilde{c}}\,\partial_i\psi.
\end{align*}

The boundary term in the first equality is zero because \(\nu_i\neq0\) only for \(i=3\), in which case \(\partial_i\sqrt{\tilde{c}}=\partial_3\sqrt{\tilde{c}}=0\). Combining the last three formulas yields
\begin{align}
I\leq{}&-\frac13\sum_i\iint (\sqrt{\tilde{c}})^{-2}(\partial_i\sqrt{\tilde{c}})^4\psi
+\sum_{i\neq j}\iint|\partial_{ij}\sqrt{\tilde{c}}|^2\psi\notag\\
&+\frac23\sum_{i\neq j}\iint (\sqrt{\tilde{c}})^{-1}|\partial_j\sqrt{\tilde{c}}|^2\partial_i\sqrt{\tilde{c}}\,\partial_i\psi
+\frac13\iint (\sqrt{\tilde{c}})^{-1}|\nabla \sqrt{\tilde{c}}|^2\nabla \sqrt{\tilde{c}}\cdot\nabla\psi.
\label{eq:approx-I-refined}
\end{align}
Substitution into \eqref{eq:approx-c-basic}, together with
\(\frac13|\Delta \sqrt{\tilde{c}}|^2\leq\sum_i|\partial_{ii}\sqrt{\tilde{c}}|^2\), gives
\begin{align}
&\frac12\int|\nabla \sqrt{\tilde{c}}|^2\psi(t)
+\frac13\iint|\Delta \sqrt{\tilde{c}}|^2\psi\notag\\
&\quad+\frac12\iint h|\nabla \sqrt{\tilde{c}}|^2\psi
+\frac13\sum_i\iint (\sqrt{\tilde{c}})^{-2}
(\partial_i\sqrt{\tilde{c}})^4\psi\notag\\
&\leq\frac12\iint|\nabla \sqrt{\tilde{c}}|^2(\partial_s\psi+\Delta\psi)
+\frac12\iint|\nabla \sqrt{\tilde{c}}|^2u\cdot\nabla\psi\notag\\
&\quad-\iint\nabla u:
(\nabla \sqrt{\tilde{c}}\otimes\nabla \sqrt{\tilde{c}})\psi
+\frac12\iint\nabla(h(\sqrt{\tilde{c}})^{-1})
\cdot\nabla \sqrt{\tilde{c}}\,\psi\notag\\
&\quad-\frac14\iint\nabla h\cdot\nabla c\,\psi
-\frac23\sum_{i=1}^3\iint
(\sqrt{\tilde{c}})^{-1}|\partial_i\sqrt{\tilde{c}}|^2
\partial_i\sqrt{\tilde{c}}\,\partial_i\psi.
\label{eq:approx-c-refined}
\end{align}
We also need the rough estimate
\begin{equation}
\label{eq:approx-I-rough}
I\leq\frac32\iint|\Delta \sqrt{\tilde{c}}|^2\psi
+\frac12\sum_{i=1}^3\iint (\sqrt{\tilde{c}})^{-2}(\partial_i\sqrt{\tilde{c}})^4\psi,
\end{equation}
which follows directly from Young's inequality and
\(|\nabla\sqrt{\tilde c}|^4\leq3\sum_i(\partial_i\sqrt{\tilde c})^4\).
The off-diagonal quartic term in the first Young estimate cancels
the corresponding negative term; hence only the diagonal sum remains
in \eqref{eq:approx-I-refined}--\eqref{eq:approx-c-refined}.
Adding \eqref{eq:approx-c-basic} to six times \eqref{eq:approx-c-refined}
and using \eqref{eq:approx-I-rough} leaves, in addition to the full Hessian,
\[
\frac12\iint|\Delta\sqrt{\tilde c}|^2\psi
+\frac32\sum_i\iint(\sqrt{\tilde c})^{-2}(\partial_i\sqrt{\tilde c})^4\psi
\geq\frac12\iint(\sqrt{\tilde c})^{-2}|\nabla\sqrt{\tilde c}|^4\psi.
\]
Thus the following combined inequality, with its original coefficients,
is still valid:

\begin{align}
&\frac72\int|\nabla \sqrt{\tilde{c}}|^2\psi(t)+\iint|\nabla^2\sqrt{\tilde{c}}|^2\psi+\frac72\iint h|\nabla \sqrt{\tilde{c}}|^2\psi
+\frac12\iint (\sqrt{\tilde{c}})^{-2}|\nabla \sqrt{\tilde{c}}|^4\psi\notag\\
&\leq\frac72\iint|\nabla \sqrt{\tilde{c}}|^2(\partial_s\psi+\Delta\psi)
+\frac72\iint|\nabla \sqrt{\tilde{c}}|^2u\cdot\nabla\psi-7\iint\nabla u:(\nabla \sqrt{\tilde{c}}\otimes\nabla \sqrt{\tilde{c}})\psi\notag\\
&\quad-\iint (\sqrt{\tilde{c}})^{-1}|\nabla \sqrt{\tilde{c}}|^2\nabla \sqrt{\tilde{c}}\cdot\nabla\psi
-\frac74\iint\nabla h\cdot\nabla c\,\psi+\frac72\iint\nabla(h(\sqrt{\tilde{c}})^{-1})\cdot\nabla \sqrt{\tilde{c}}\,\psi\notag\\
&\quad-4\sum_{i=1}^3\iint (\sqrt{\tilde{c}})^{-1}(\partial_i\sqrt{\tilde{c}})^3\partial_i\psi.
\label{eq:approx-c-combined}
\end{align}

The two coupling terms are estimated exactly at this point. Since \((\sqrt{\tilde{c}})^2=1+c\leq M_c\),
\begin{equation}
\label{eq:approx-fluid-coupling}
\left|\iint\nabla u:(\nabla \sqrt{\tilde{c}}\otimes\nabla \sqrt{\tilde{c}})\psi\right|
\leq\frac1{112}\iint (\sqrt{\tilde{c}})^{-2}|\nabla \sqrt{\tilde{c}}|^4\psi
+28M_c\iint|\nabla u|^2\psi.
\end{equation}
Moreover, using \(n(h_\tau'(n))^2\leq h_\tau(n)=h\),
\begin{align}
&\frac72\iint\nabla(h(\sqrt{\tilde{c}})^{-1})\cdot\nabla \sqrt{\tilde{c}}\,\psi\notag\\
&=-\frac72\iint h(\sqrt{\tilde{c}})^{-2}|\nabla \sqrt{\tilde{c}}|^2\psi
+\frac72\iint (\sqrt{\tilde{c}})^{-1}h_\tau'(n)\nabla n\cdot\nabla \sqrt{\tilde{c}}\,\psi\notag\\
&\leq\frac72\iint|\nabla\sqrt n|^2\psi.
\label{eq:approx-reaction-coupling}
\end{align}
Indeed, \(\nabla n=2\sqrt n\nabla\sqrt n\) and \(\sqrt n\,h_\tau'(n)\leq\sqrt h\), so the second integrand is bounded by twice the product of \(\sqrt h\,(\sqrt{\tilde{c}})^{-1}|\nabla \sqrt{\tilde{c}}|\) and \(|\nabla\sqrt n|\). Substituting \eqref{eq:approx-fluid-coupling}--\eqref{eq:approx-reaction-coupling} into \eqref{eq:approx-c-combined} yields
\begin{align}
&\frac72\int|\nabla \sqrt{\tilde{c}}|^2\psi(t)+\iint|\nabla^2\sqrt{\tilde{c}}|^2\psi+\frac72\iint h|\nabla \sqrt{\tilde{c}}|^2\psi
+\frac7{16}\iint (\sqrt{\tilde{c}})^{-2}|\nabla \sqrt{\tilde{c}}|^4\psi\notag\\
&\leq\frac72\iint|\nabla \sqrt{\tilde{c}}|^2(\partial_s\psi+\Delta\psi)
+\frac72\iint|\nabla \sqrt{\tilde{c}}|^2u\cdot\nabla\psi-\frac74\iint\nabla h\cdot\nabla c\,\psi\notag\\
&\quad
-\iint (\sqrt{\tilde{c}})^{-1}|\nabla \sqrt{\tilde{c}}|^2\nabla \sqrt{\tilde{c}}\cdot\nabla\psi\notag+196M_c\iint|\nabla u|^2\psi
+\frac72\iint|\nabla\sqrt n|^2\psi\notag\\
&\quad-4\sum_{i=1}^3\iint (\sqrt{\tilde{c}})^{-1}(\partial_i\sqrt{\tilde{c}})^3\partial_i\psi.
\label{eq:approx-c-final}
\end{align}

Multiply \eqref{eq:approx-c-final} by \(4/7\) and add \eqref{eq:approx-cell-local}. The mixed terms \(\iint\nabla h\cdot\nabla c\,\psi\) cancel exactly. Moving the resulting \(2\iint|\nabla\sqrt n|^2\psi\) from the right to the left gives
\begin{align}
&\int n\log n\,\psi(t)+2\int|\nabla \sqrt{\tilde{c}}|^2\psi(t)
+2\iint|\nabla\sqrt n|^2\psi\notag\\
&\quad+\frac47\iint|\nabla^2\sqrt{\tilde{c}}|^2\psi
+2\iint h|\nabla \sqrt{\tilde{c}}|^2\psi
+\frac14\iint (\sqrt{\tilde{c}})^{-2}|\nabla \sqrt{\tilde{c}}|^4\psi\notag\\
&\leq\iint n\log n(\partial_s\psi+\Delta\psi)
+\iint n\log n\,(\mathcal J_\varepsilon u)\cdot\nabla\psi+\iint a(1+\log n)\nabla c\cdot\nabla\psi\notag\\
&\quad+2\iint|\nabla \sqrt{\tilde{c}}|^2(\partial_s\psi+\Delta\psi)
+2\iint|\nabla \sqrt{\tilde{c}}|^2u\cdot\nabla\psi-\frac47\iint (\sqrt{\tilde c})^{-1}|\nabla \sqrt{\tilde{c}}|^2\nabla \sqrt{\tilde{c}}\cdot\nabla\psi\notag\\
&\quad-\frac{16}{7}\sum_{i=1}^3\iint (\sqrt{\tilde{c}})^{-1}(\partial_i\sqrt{\tilde{c}})^3\partial_i\psi+112M_c\iint|\nabla u|^2\psi.
\label{eq:approx-scalar-local-energy}
\end{align}
This is the regularized analogue of the scalar part of the local energy inequality. The only changes are \((\mathcal J_\varepsilon u)\) in the cell transport, \(a\) in the chemotactic flux, and \(h\) in the concentration reaction; the identity \(a=nh'\) is what preserves the cancellation.

\emph{The momentum equation.} Multiply
\[
\partial_tu+(\mathcal J_\varepsilon u)\cdot\nabla u-\Delta u+\nabla p=-\mathcal K_\varepsilon(n\nabla\phi)
\]
by \(2u\psi\). The convective term satisfies
\[
2\iint((\mathcal J_\varepsilon u)\cdot\nabla u)\cdot u\psi
=\iint_\Gamma |u|^2\psi\,(\mathcal J_\varepsilon u)\cdot\nu
-\iint|u|^2(\mathcal J_\varepsilon u)\cdot\nabla\psi
=-\iint|u|^2(\mathcal J_\varepsilon u)\cdot\nabla\psi,
\]
because \((\mathcal J_\varepsilon u)=0\) on \(\Gamma\). The diffusion term is
\begin{align*}
-2\iint\Delta u\cdot u\psi
={}&2\iint|\nabla u|^2\psi-\iint|u|^2\Delta\psi\\
&-2\iint_\Gamma\partial_\nu u\cdot u\psi
+\iint_\Gamma|u|^2\partial_\nu\psi.
\end{align*}
Both boundary integrals vanish because \(u=0\) on \(\Gamma\). For the pressure, after subtracting an arbitrary spatial constant \(\overline p(s)\),
\begin{align*}
2\iint\nabla(p-\overline p(s))\cdot u\psi
={}&-2\iint(p-\overline p(s))u\cdot\nabla\psi\\
&+2\iint_\Gamma(p-\overline p(s))\psi\,u\cdot\nu
=-2\iint(p-\overline p(s))u\cdot\nabla\psi,
\end{align*}
because \(\nabla\cdot u=0\) and \(u=0\) on \(\Gamma\). The time integration has no initial term because \(\psi(\cdot,0)=0\). Consequently,
\begin{align}
\int|u|^2\psi(t)+2\iint|\nabla u|^2\psi
={}&\iint|u|^2(\partial_s\psi+\Delta\psi)
+\iint|u|^2(\mathcal J_\varepsilon u)\cdot\nabla\psi\notag\\
&+2\iint(p-\overline p(s))u\cdot\nabla\psi
-2\iint \mathcal K_\varepsilon(n\nabla\phi)\cdot u\,\psi.
\label{eq:approx-momentum-local}
\end{align}
Multiply \eqref{eq:approx-momentum-local} by \(112M_c\) and add it to \eqref{eq:approx-scalar-local-energy}. The inequality has $ 112M_c\iint|\nabla u|^2\psi$ on its right, whereas the multiplied momentum identity contributes \(224M_c\iint|\nabla u|^2\psi\) on its left. After moving the former term to the left, the remaining coefficient is \(112M_c\). The resulting inequality is exactly \eqref{eq:approximate-local-energy}.
\end{proof}
\begin{remark}
\label{rem:approx-lei-test}
Lemma~\ref{lem:approximate-local-energy} is proved for the larger class of
boundary-compatible test functions that merely vanish at the initial
time. In the compactness argument below, however, we only pass to the
limit for the class of test functions appearing in
Definition~\ref{def:suitable-half-space} , whose temporal supports are compactly
contained in $(0,T)$. This restriction is needed only for the pressure
term, since the uniform pressure estimate of
Proposition~\ref{prop:uniform-pressure} is available locally in time away from
the genuine initial time.
\end{remark}

\subsection{Compactness and passage to the limit}
\label{subsec:compactness-local-energy}

The compactness argument below applies to subsequences with both
parameters tending to zero. For the eventual regularity conclusion,
Section~\ref{sec:eventual-regularity} refines this choice by first
taking $\varepsilon\to0$ at fixed $\tau$, then taking $\tau\to0$,
and selecting a diagonal that retains the additional global estimates.

We now choose a sequence \((\varepsilon_k,\tau_k)\to(0,0)\).  Set
\[
(n_k,c_k,u_k,p_k)(t):=(n^{\varepsilon_k,\tau_k},c^{\varepsilon_k,\tau_k},u^{\varepsilon_k,\tau_k},p^{\varepsilon_k,\tau_k})(t),
\qquad 0<t<\infty,
\]
and write
\[
a_k:=a_{\tau_k}(n_k),\qquad h_k:=h_{\tau_k}(n_k).
\]
For the velocity equation, we introduce the solenoidal test space
\[
V_\sigma(D)
:=
\left\{
\varphi\in W_0^{1,2}(D;\mathbb R^3):
\nabla\cdot\varphi=0
\right\},
\]
equipped with the $W^{1,2}$ norm, and denote its dual by $V_\sigma(D)^*$.

Fix $T>0$ and a bounded Lipschitz domain $D\subset\mathbb H^3$;
its closure may meet $\Gamma$. All bounds in this paragraph are
uniform in $k$, and may depend on $D,T,C_{\rm in}$. It suffices
to establish them on a larger half-ball and then restrict to $D$.
The equations and the estimates already proved give
\begin{equation}
\label{eq:time-derivative-bounds}
\begin{aligned}
&\|\partial_tn_k\|_{L^{10/9}(0,T;W^{-1,10/9}(D))}
 \leq C_{D,T,C_{\rm in}},\\
&\|\partial_t\nabla c_k\|_{L^{4/3}(0,T;W^{-1,2}(D))}
 \leq C_{D,T,C_{\rm in}},\\
&\|\partial_t\nabla\sqrt{1+c_k}\|_{L^{4/3}(0,T;W^{-1,2}(D))}
 \leq C_{D,T,C_{\rm in}},\\
&\|\partial_tu_k\|_{L^{4/3}(0,T;V_\sigma(D)^*)}
 \leq C_{D,T,C_{\rm in}}.
\end{aligned}
\end{equation}
Here no prescribed growth rate in $T$ is needed for compactness.
To justify the density estimate, write
\[
\partial_tn_k=\nabla\cdot
  (\nabla n_k-a_k\nabla c_k-(\mathcal J_{\varepsilon_k}u_k)n_k).
\]
By \eqref{eq:uniform-energy}--\eqref{eq:uniform-integrability},
$n_k\in L^{5/3}$ and $(\mathcal J_{\varepsilon_k}u_k),\nabla c_k\in L^{10/3}$ locally,
while
\[
\nabla n_k=2\sqrt{1+n_k}\,\nabla\sqrt{1+n_k}
\quad\hbox{is bounded in }L^{5/4}((0,T)\times D).
\]
Since $a_k\leq n_k$, the flux is bounded in $L^{10/9}$
on this finite cylinder, proving the first estimate.

For the concentration estimates, set $s_k=\sqrt{1+c_k}$.
The same bounds imply
\[
\sqrt{1+n_k},\ u_k,\ \nabla s_k
\quad\hbox{bounded in }
L^\infty(0,T;L^2(D))\cap L^2(0,T;H^1(D)).
\]
The Gagliardo--Nirenberg inequality
\[
\|w\|_{L^4(D)}
\leq C_D\|w\|_{L^2(D)}^{1/4}\|w\|_{H^1(D)}^{3/4}
\]
therefore bounds all three quantities in $L^{8/3}(0,T;L^4(D))$.
Consequently,
\[
n_k,\quad u_k\cdot\nabla c_k,\quad
u_k\cdot\nabla s_k,\quad s_k^{-1}|\nabla s_k|^2
\quad\hbox{are bounded in }L^{4/3}(0,T;L^2(D)).
\]
We used $1\leq s_k\leq\sqrt{M_c}$ and
$\nabla c_k=2s_k\nabla s_k$ here. Moreover,
\[
\nabla^2c_k=2\nabla s_k\otimes\nabla s_k+2s_k\nabla^2s_k
\]
is bounded in $L^2((0,T)\times D)$ by the quartic and Hessian
terms in \eqref{eq:uniform-energy}. The equations
\[
\begin{split}
\partial_tc_k-\Delta c_k&=-u_k\cdot\nabla c_k-h_kc_k,\\
\partial_ts_k-\Delta s_k&=-u_k\cdot\nabla s_k
 +s_k^{-1}|\nabla s_k|^2-\tfrac12h_k(s_k-s_k^{-1})
\end{split}
\]
and $0\leq h_k\leq n_k$ now bound $\partial_tc_k$ and
$\partial_ts_k$ in $L^{4/3}(0,T;L^2(D))$.
Applying the distributional spatial gradient, which maps $L^2(D)$
continuously into $W^{-1,2}(D)$, proves the second and third
estimates in \eqref{eq:time-derivative-bounds}.

For the velocity equation, it is important to work in the dual of the
solenoidal test space. Indeed, for every
$\varphi\in V_\sigma(D)$, the pressure term vanishes and the weak
form of the momentum equation gives
\begin{align}\label{3.99a}
\left\langle\partial_tu_k,\varphi\right\rangle
&=
-\int_D\nabla u_k:\nabla\varphi\,dx
+
\int_D((\mathcal J_{\varepsilon_k}u_k)\otimes u_k):\nabla\varphi\,dx
-
\int_D \mathcal K_{\varepsilon_k}(n_k\nabla\phi)\cdot\varphi\,dx.
\end{align}

Here we used
\[
\nabla\cdot (\mathcal J_{\varepsilon_k}u_k)=0
\]
and
\[
\int_D\nabla p_k\cdot\varphi\,dx
=
-\int_Dp_k\nabla\cdot\varphi\,dx
=0.
\]

By interpolation between
\[
u_k\in L^\infty(0,T;L^2(D))
\]
and
\[
u_k\in L^2(0,T;L^6(D)),
\]
we have
\[
u_k\in L^4(0,T;L^3(D)).
\]
The same estimate holds for $\mathcal J_{\varepsilon_k}u_k$. Hence
\[
\|(\mathcal J_{\varepsilon_k}u_k)\otimes u_k\|_{L^{4/3}(0,T;L^2(D))}
\le
\|(\mathcal J_{\varepsilon_k}u_k)\|_{L^4(0,T;L^3(D))}
\|u_k\|_{L^2(0,T;L^6(D))}
\le C.
\]
Moreover,
\[
\nabla u_k\in L^2(0,T;L^2(D))
\subset
L^{4/3}(0,T;L^2(D)),
\]
and, by \eqref{eq:uniform-integrability},
\[
\mathcal K_{\varepsilon_k}(n_k\nabla\phi)\in L^{5/3}((0,T)\times D)
\subset
L^{4/3}(0,T;L^{6/5}(D)).
\]
Since
\[
W_0^{1,2}(D)\hookrightarrow L^6(D),
\]
\eqref{3.99a} yields
\[
\|\partial_tu_k\|_
{L^{4/3}(0,T;V_\sigma(D)^*)}
\le C_{D,C_{\rm in},T},
\]
as asserted in \eqref{eq:time-derivative-bounds}.
We next establish the compactness of the approximate solutions.
We begin with the cell density. From \eqref{eq:uniform-energy} and \eqref{eq:uniform-integrability},
for every bounded $D\Subset\mathbb H^3$,
\[
\sqrt{1+n_k}
\quad\text{is bounded in}\quad
L^{10/3}((0,T)\times D),
\]
while
\[
\nabla\sqrt{1+n_k}
\quad\text{is bounded in}\quad
L^2((0,T)\times D).
\]
Since
\[
\nabla n_k
=
2\sqrt{1+n_k}\,
\nabla\sqrt{1+n_k},
\]
H\"older's inequality gives
\begin{equation}
\label{eq:compactness-density-gradient-bound}
\|\nabla n_k\|_{L^{5/4}((0,T)\times D)}
\le
2
\|\sqrt{1+n_k}\|_{L^{10/3}((0,T)\times D)}
\|\nabla\sqrt{1+n_k}\|_{L^2((0,T)\times D)}
\le C_{D,T}.
\end{equation}
Together with the $L^{5/3}$ bound for $n_k$, this implies
\[
n_k
\quad\text{is bounded in}\quad
L^{5/4}(0,T;W^{1,5/4}(D)).
\]
On the other hand, \eqref{eq:time-derivative-bounds} gives
\[
\partial_tn_k
\quad\text{bounded in}\quad
L^{10/9}(0,T;W^{-1,10/9}(D)).
\]
Since
\[
W^{1,5/4}(D)
\Subset
L^1(D)
\hookrightarrow
W^{-1,10/9}(D),
\]
the Aubin--Lions lemma yields
\begin{align}\label{n-strong}
n_k\to n
\qquad\text{strongly in }L^1((0,T)\times D).
\end{align}
Interpolating the strong convergence in \eqref{n-strong} with the uniform
$L^{5/3}$ bound in \eqref{eq:uniform-integrability} gives
\begin{equation}
\label{eq:compactness-density-strong-lq}
n_k\to n
\qquad\text{strongly in }
L^q_{\rm loc}\bigl((0,T)\times\mathbb H^3\bigr),
\qquad
1\le q<\frac53.
\end{equation}

We next consider the velocity. Throughout the remainder of this passage,
$D$ denotes a bounded Lipschitz subdomain of $\mathbb H^3$;
its closure is allowed to meet $\Gamma$. All the estimates below can
first be established on the half-balls $B_R^+$ and then restricted to
$D\subset B_R^+$. Thus no separation of $D$ from the flat boundary
is imposed. We fix $T>0$ and take a common subsequence by a diagonal
argument over an increasing sequence of half-balls and time intervals.

Fix $0<\sigma<T$. Proposition~\ref{prop:uniform-pressure} and the full
momentum equation imply that
\[
\partial_tu_k\quad\hbox{is bounded in}\quad
L^{5/3}(\sigma,T;W^{-1,5/3}(D)).
\]
Indeed, $(\mathcal J_{\varepsilon_k}u_k)\otimes u_k$ and $p_k-[p_k]_D(t)$ are bounded in
$L^{5/3}((\sigma,T)\times D)$, whereas $\nabla u_k$ is bounded in
$L^2((0,T)\times D)$ and $\mathcal K_{\varepsilon_k}(n_k\nabla\phi)$ in $L^{5/3}((0,T)\times D)$.
For the last assertion, use the $L^{5/3}$-boundedness of
$\mathcal K_{\varepsilon_k}$ and the global $L^{5/3}$ bound for $n_k$.
Here the pressure estimate on $D$ follows, if necessary, by restriction
from a larger half-ball and subtraction of the mean over $D$.
Since
\[
H^1(D)\Subset L^2(D)\hookrightarrow W^{-1,5/3}(D),
\]
the Aubin--Lions lemma gives
\[
u_k\longrightarrow u
\qquad\text{strongly in }L^2((\sigma,T)\times D).
\]
The uniform $L^\infty(0,T;L^2)$ bound, inherited by the weak limit,
also gives
\[
\sup_k\int_0^\sigma\int_D|u_k-u|^2\,dx\,dt
\le C_T\sigma.
\]
Letting first $k\to\infty$ and then $\sigma\downarrow0$, and
interpolating with the uniform $L^{10/3}$ bound, we obtain
\[
u_k\to u
\qquad\text{strongly in }L^q((0,T)\times D),
\qquad 2\le q<\frac{10}{3}.
\]

We make explicit how the interior compactness established above
extends to the spatial boundary and to the whole time interval.
Let $E_j$ be an increasing sequence of compact subsets of
$(0,T)\times D$ exhausting that open set. If $z_k\to z$ strongly
in $L^q(E_j)$ for every $j$, and $z_k$ is uniformly bounded in
$L^p((0,T)\times D)$ for some $q<p<\infty$, then $z$ has the same
$L^p$ bound and
\[
\int_{((0,T)\times D)\setminus E_j}|z_k-z|^q
\le C\bigl|((0,T)\times D)\setminus E_j\bigr|^{1-q/p}.
\]
The right-hand side tends to zero uniformly in $k$ as $j\to\infty$.
For $p=\infty$, the corresponding bound is
$C|((0,T)\times D)\setminus E_j|$.
Consequently, interior local strong convergence and these uniform
bounds imply strong convergence on the full cylinder $(0,T)\times D$.
This observation applies to the density convergence in
\eqref{n-strong} and its interpolation with
\eqref{eq:uniform-integrability}; in particular,
\[
n_k\to n
\qquad\text{strongly in }L^q((0,T)\times D),
\qquad 1\le q<\frac53.
\]

The previously established time-derivative estimates and the
Aubin--Lions lemma give interior local strong compactness for
$\nabla c_k$ and $\nabla\sqrt{\tilde c_k}$ in $L^2$.
Their spatial $H^1$ bounds follow from the dissipation estimates;
for $\nabla c_k$, use
\[
\nabla^2c_k
=2\nabla\sqrt{\tilde c_k}\otimes\nabla\sqrt{\tilde c_k}
 +2\sqrt{\tilde c_k}\,\nabla^2\sqrt{\tilde c_k},
\]
together with $1\le\tilde c_k\le M_c$ and the uniform bound for
$\iint(\tilde c_k)^{-1}|\nabla\sqrt{\tilde c_k}|^4$.
For $c_k$ itself, the equation gives a local bound for
$\partial_tc_k$ in $L^{5/3}(0,T;W^{-1,5/3})$:
write the transport as $\nabla\cdot(u_kc_k)$ and use
$0\le c_k\le\|c_0\|_{L^\infty}$ and $0\le h_k\le n_k$.
Thus $c_k$ is also strongly compact in interior local $L^2$.
The strong limits of the gradients are identified by testing against
smooth functions compactly supported in the interior.
Interpolation and the exhaustion argument above now yield
\[
c_k\to c
\qquad\text{strongly in }L^q((0,T)\times D),
\qquad 1\le q<\infty,
\]
and
\[
\nabla c_k\to\nabla c,
\qquad
\nabla\sqrt{\tilde c_k}\to\nabla\sqrt{\tilde c}
\qquad\text{strongly in }L^q((0,T)\times D),
\qquad 2\le q<\frac{10}{3}.
\]
After a further common subsequence, these strong convergences also
hold almost everywhere in $\mathbb H^3\times(0,T)$.
In particular, $0\le c\le\|c_0\|_{L^\infty}$ and $\tilde c=1+c$.

We also fix a common set of admissible terminal times. By the
strong space--time convergences just proved, one may choose a
further subsequence, diagonally over positive integer radii and positive integer
time horizons, such that for every integer $m\ge1$ and almost
every $t>0$,
\begin{equation}\label{eq:compactness-common-time-slices}
\begin{aligned}
n_k(t)&\longrightarrow n(t)
 &&\text{in }L^1(B_m^+),\\
u_k(t)&\longrightarrow u(t)
 &&\text{in }L^2(B_m^+),\\
\nabla\sqrt{\tilde c_k}(t)&\longrightarrow
 \nabla\sqrt{\tilde c}(t)
 &&\text{in }L^2(B_m^+).
\end{aligned}
\end{equation}
For example, choose the subsequence so that the corresponding
space--time errors, with the $L^2$ errors squared, are summable
on each such cylinder, and apply Tonelli's theorem. Fubini's
theorem applied to the almost-everywhere space--time convergence
also gives $n_k(t,x)\to n(t,x)$ for almost every $x$ at these
times. The intersection of the resulting countably many
full-measure sets is independent of the later test function
$\psi$.

Furthermore,
\[
n_k\log n_k\to n\log n
\qquad\text{strongly in }L^q((0,T)\times D),
\qquad 1\le q<\frac32.
\]
Indeed, for each such $q$, choose $q<r<3/2$. The inequality
$|z\log z|^r\le C_r(1+z^{5/3})$ for $z\ge0$ gives a uniform
$L^r((0,T)\times D)$ bound. Almost-everywhere convergence and
Vitali's theorem then give the asserted strong $L^q$ convergence.
We use $0\log0=0$ throughout.

The pressure is treated separately and only away from the genuine
initial time. To obtain a single pressure on the whole half-space,
normalize $p_k$ by subtracting $[p_k]_{B_1^+}(t)$.
For each integer $m\ge1$, the oscillation estimate on $B_m^+$
also bounds this normalized pressure on
$(\sigma,T)\times B_m^+$: indeed,
\[
[p_k]_{B_1^+}-[p_k]_{B_m^+}
=\fint_{B_1^+}(p_k-[p_k]_{B_m^+})\,dx.
\]
H\"older's inequality controls this difference of means by the
oscillation on $B_m^+$. A diagonal extraction therefore gives a
single limit pressure $p$, normalized by $[p]_{B_1^+}=0$.
Restriction to $D$ and continuity of spatial averaging imply
\begin{equation}
\label{eq:compactness-pressure-weak}
p_k-[p_k]_D(t)\rightharpoonup p-[p]_D(t)
\qquad\text{weakly in }L^{5/3}((\sigma,T)\times D),
\qquad 0<\sigma<T.
\end{equation}
No pressure bound on $(0,T)\times D$ is asserted.

We now identify the weak limits of the cell dissipation.
By \eqref{eq:uniform-energy} and
$|\sqrt{1+a}-\sqrt{1+b}|^2\le |a-b|$ for $a,b\ge0$,
\[
\sqrt{1+n_k}\to\sqrt{1+n}
\qquad\text{strongly in }L^2((0,T)\times D).
\]
Testing weak derivatives in the open set $(0,T)\times D$ identifies
the weak limit and gives
\begin{equation}
\label{eq:compactness-shifted-fisher-weak}
\nabla\sqrt{1+n_k}\rightharpoonup\nabla\sqrt{1+n}
\qquad\text{weakly in }L^2((0,T)\times D).
\end{equation}

For $\sqrt{n_k}$, we first establish the local Fisher-information
bound on the entire interval $(0,T)$. Choose a nonnegative radial
function $\chi\in C_c^\infty(\mathbb R^3)$ which equals one in a
neighborhood of $\overline D$, and restrict it to the half-space.
Then $\partial_\nu\chi=0$ on $\Gamma$.
Repeat the scalar calculation leading to
\eqref{eq:approx-scalar-local-energy} with the time-independent
test function $\psi(x,s)=\chi(x)$, retaining the initial term
\[
\int_{\mathbb H^3}
\left(n_0^{\varepsilon_k}\log n_0^{\varepsilon_k}
+2|\nabla\sqrt{1+c_0^{\varepsilon_k}}|^2\right)\chi\,dx.
\]
This is the scalar energy inequality, and it contains no pressure.
The initial term is uniformly bounded by the initial entropy and
concentration-energy estimates. Here
$|z\log z|\le H(z)+e^{-1}$, and
$|\nabla\sqrt{1+c_0^{\varepsilon_k}}|
\le |\nabla\sqrt{c_0^{\varepsilon_k}}|$.
On the bounded support of $\chi$, the entropy transport and
chemotactic terms are integrable by the uniform $L^{3/2}$ bound
for $n_k\log n_k$ and the $L^{10/3}$ bounds for
$(\mathcal J_{\varepsilon_k}u_k)$ and $\nabla c_k$; indeed $2/3+3/10<1$.
The cubic concentration terms and the concentration transport term
are controlled by the $L^{10/3}$ bounds for
$\nabla\sqrt{\tilde c_k}$ and $u_k$, while
$\iint|\nabla u_k|^2\chi$ is uniformly bounded.
The negative part of the terminal entropy is at most
$e^{-1}\int_{\mathbb H^3}\chi\,dx$.
Dropping the other nonnegative terms and letting the terminal time
increase to $T$, we obtain
\begin{equation}
\label{eq:compactness-local-fisher-bound}
\sup_k\int_0^T\int_D|\nabla\sqrt{n_k}|^2\,dx\,dt
\le C_{D,T,C_{\rm in}}.
\end{equation}
This argument uses the scalar calculation with its initial term,
rather than applying an inequality whose test functions vanish
at the initial time directly to $\chi$.

On the other hand,
\begin{equation}
\label{eq:compactness-sqrt-density-strong}
\|\sqrt{n_k}-\sqrt n\|_{L^2((0,T)\times D)}^2
\le\|n_k-n\|_{L^1((0,T)\times D)}\longrightarrow0.
\end{equation}
By the uniform $L^2$ bound, the gradients have a weakly convergent
subsequence. For every $\varphi\in C_c^\infty((0,T)\times D)$ and
$i=1,2,3$, strong convergence of $\sqrt{n_k}$ gives
\[
\begin{aligned}
\lim_{k\to\infty}\int_0^T\int_D
\partial_i\sqrt{n_k}\,\varphi\,dx\,dt
&=-\lim_{k\to\infty}\int_0^T\int_D
\sqrt{n_k}\,\partial_i\varphi\,dx\,dt\\
&=-\int_0^T\int_D\sqrt n\,\partial_i\varphi\,dx\,dt.
\end{aligned}
\]
There is no boundary term because these tests are compactly supported
in the open domain $D$. Thus every weak limit is the distributional
gradient of $\sqrt n$, and
\begin{equation}
\label{eq:compactness-fisher-weak}
\nabla\sqrt{n_k}\rightharpoonup\nabla\sqrt n
\qquad\text{weakly in }L^2((0,T)\times D).
\end{equation}
In particular, this proves the Fisher-information requirement
up to the genuine initial time in
Definition~\ref{def:suitable-half-space}(iii), also on domains
whose closures meet $\Gamma$.
The uniform estimates likewise imply
\[
\nabla u_k\rightharpoonup\nabla u,
\qquad
\nabla^2\sqrt{\tilde c_k}\rightharpoonup\nabla^2\sqrt{\tilde c}
\qquad\text{weakly in }L^2((0,T)\times D).
\]

The velocity smoother converges to the same limit.
Lemma~\ref{lem:stokes-smoothing} and \eqref{eq:uniform-energy} give
\begin{equation}
\label{eq:smoothed-velocity-convergence}
\|\mathcal{J}_{\varepsilon_k}u_k-u_k\|_{L^2(0,T;L^2(\mathbb H^3))}
\le C\varepsilon_k^{1/2}
\|\nabla u_k\|_{L^2(0,T;L^2(\mathbb H^3))}\longrightarrow0.
\end{equation}
Combining this with the strong convergence of $u_k$, and then
interpolating with the uniform $L^{10/3}$ bounds, yields
\begin{equation}
\label{eq:b-convergence}
(\mathcal J_{\varepsilon_k}u_k)\to u\qquad\text{strongly in }L^q((0,T)\times D),
\qquad 2\le q<\frac{10}{3}.
\end{equation}
Since $a_\tau(z)$ and $h_\tau(z)$ converge to $z$ uniformly on
bounded subsets of $[0,\infty)$, the almost-everywhere convergence
of $n_k$, together with $0\le a_k,h_k\le n_k$ and the uniform
$L^{5/3}$ bound, gives
\begin{equation}
\label{eq:regularizing-functions-convergence}
a_k\to n,\qquad h_k\to n
\qquad\text{strongly in }L^q((0,T)\times D),
\qquad 1\le q<\frac53.
\end{equation}
Similarly, using $|a_k\log n_k|\le n_k|\log n_k|$ and the
same uniform-integrability argument as for $n_k\log n_k$, we obtain
\begin{equation}
\label{eq:regularized-entropy-flux-convergence}
a_k\log n_k\to n\log n
\qquad\text{strongly in }L^q((0,T)\times D),
\qquad 1\le q<\frac32.
\end{equation}
The product $a_k\log n_k$ is set to zero where $n_k=0$.

The regularized force satisfies
\begin{equation}
\label{eq:force-convergence}
\mathcal K_{\varepsilon_k}(n_k\nabla\phi)\to n\nabla\phi
\qquad\text{strongly in }L^q((0,T)\times D),
\qquad 1\le q<\frac53.
\end{equation}
Here local convergence of the input must be combined with a
far-field estimate, since the heat smoother is nonlocal.
Choose $0\le\chi\le1$ in $C_c^\infty(\mathbb R^3)$ which equals
one in a neighborhood of $\overline D$, and restrict it to
$\overline{\mathbb H^3}$, so that
\[
\operatorname{dist}\bigl(\overline D,
\operatorname{supp}_{\overline{\mathbb H^3}}(1-\chi)\bigr)>0.
\]
On $(0,T)\times D$, decompose
\begin{equation}
\label{eq:compactness-force-decomposition}
\begin{aligned}
\mathcal K_{\varepsilon_k}(n_k\nabla\phi)-n\nabla\phi
={}&\mathcal{K}_{\varepsilon_k}\bigl(\chi(n_k-n)\nabla\phi\bigr)
 +(\mathcal{K}_{\varepsilon_k}-I)\bigl(\chi n\nabla\phi\bigr)\\
&+\mathcal{K}_{\varepsilon_k}\bigl((1-\chi)n_k\nabla\phi\bigr)
.
\end{aligned}
\end{equation}
Fix $1\le q<5/3$ and choose $R$ so that
$\operatorname{supp}\chi\cap\mathbb H^3\subset B_R^+$.
The strong density convergence on $(0,T)\times B_R^+$ gives
\begin{align}
\label{3.105b}
\|\mathcal K_{\varepsilon_k}\bigl(\chi(n_k-n)\nabla\phi\bigr)\|_{L^q((0,T)\times D)}
&\le C_q\|\chi(n_k-n)\nabla\phi\|_{L^q((0,T)\times\mathbb H^3)}
\nonumber\\
&\le C_q\|\nabla\phi\|_{L^\infty}
\|n_k-n\|_{L^q((0,T)\times B_R^+)}\longrightarrow0.
\end{align}
Since $\chi n\nabla\phi\in L^q((0,T)\times\mathbb H^3)$,
strong continuity and uniform boundedness of $S_D(s)$ on
$L^q(\mathbb H^3)$, followed by dominated convergence in time, give
\begin{equation}
\label{eq:compactness-force-semigroup-limit}
\|(\mathcal K_{\varepsilon_k}-I)\bigl(\chi n\nabla\phi\bigr)\|_{L^q((0,T)\times\mathbb H^3)}\longrightarrow0.
\end{equation}
The Dirichlet heat kernel satisfies
\[
|G(x-y,s)-G(x-(y',-y_3),s)|\le Cs^{-3/2}
\exp\left(-\frac{|x-y|^2}{Cs}\right).
\]
For $x\in D$ and $y\in\operatorname{supp}_{\overline{\mathbb H^3}}(1-\chi)$,
the separation is bounded below by a positive constant depending only
on $D$ and $\chi$. Mass conservation therefore yields
\[
|\mathcal K_{\varepsilon_k}\bigl((1-\chi)n_k\nabla\phi\bigr)(x,t)|
\le C\|\nabla\phi\|_{L^\infty}\|n_0\|_{L^1}
\varepsilon_k^{-3/2}
\exp\left(-\frac{1}{C_{D,\chi}\varepsilon_k}\right).
\]
Consequently,
\begin{equation}
\label{3.105d}
\|\mathcal K_{\varepsilon_k}\bigl((1-\chi)n_k\nabla\phi\bigr)\|_{L^q((0,T)\times D)}\longrightarrow0.
\end{equation}
Combining \eqref{3.105b},
\eqref{eq:compactness-force-semigroup-limit}, and \eqref{3.105d}
proves \eqref{eq:force-convergence}.

We now pass to the equations and to the right-hand side of the
local energy inequality. All strong convergences used below hold
on $(0,T)\times D$, including when $\overline D$ meets $\Gamma$.
Pressure convergence is used only on $(\sigma,T)\times D$ with
$\sigma>0$.

\emph{The cell equation.}
Using strong convergence with exponents $3<10/3$ and $8/5<5/3$,
together with $1/3+5/8=23/24<1$, H\"older's inequality gives
\[
(\mathcal J_{\varepsilon_k}u_k)n_k\to un,\qquad a_k\nabla c_k\to n\nabla c
\qquad\text{strongly in }L^1((0,T)\times D).
\]
The linear diffusion term passes distributionally; in the
boundary weak formulation, one may also use the weak
$L^{5/4}((0,T)\times D)$ convergence of $\nabla n_k$.
Indeed, its uniform bound follows from
$\nabla n_k=2\sqrt{1+n_k}\nabla\sqrt{1+n_k}$,
and its weak limit is identified by the strong convergence of $n_k$.
Thus
\[
\partial_tn+\nabla\cdot(nu)-\Delta n
=-\nabla\cdot(n\nabla c)
\]
in distributions.

\emph{The concentration equation.}
The strong $L^2$ convergences of $u_k$ and $\nabla c_k$ imply
$u_k\cdot\nabla c_k\to u\cdot\nabla c$ in
$L^1((0,T)\times D)$. Moreover, the strong convergence of $h_k$
in $L^{3/2}$ and that of $c_k$ in $L^3$ give
\[
h_kc_k\to nc\qquad\text{strongly in }L^1((0,T)\times D).
\]
Hence $\partial_tc+u\cdot\nabla c-\Delta c=-nc$.

\emph{The momentum equation.}
The strong convergences of $(\mathcal J_{\varepsilon_k}u_k)$ and $u_k$ imply
\[
(\mathcal J_{\varepsilon_k}u_k)\otimes u_k\to u\otimes u
\qquad\text{strongly in }L^q((0,T)\times D),
\qquad 1\le q<\frac53.
\]
Let $\varphi\in C_c^\infty(\mathbb H^3\times(0,T);\mathbb R^3)$.
Choose $D=B_R^+$ and $0<\sigma<T$ so that the spatial support of
$\varphi$ is contained in $D$ and its temporal support in $(\sigma,T)$.
By \eqref{eq:compactness-pressure-weak}, the normalized pressures
converge weakly in $L^{5/3}((\sigma,T)\times D)$.
Subtraction of $[p_k]_D(t)$ does not change the pressure term,
since $\int_D\nabla\cdot\varphi(x,t)\,dx=0$.
Together with \eqref{eq:force-convergence}, this proves
\[
\partial_tu+\nabla\cdot(u\otimes u)-\Delta u+\nabla p
=-n\nabla\phi,\qquad \nabla\cdot u=0
\]
in distributions in $\mathbb H^3\times(0,T)$.
The zero trace of $u$ follows from the weak convergence
\[
u_k\rightharpoonup u
\quad\text{in }L^2(0,T;H_0^1(\mathbb H^3)).
\]
For the scalar equations, the same convergence statements allow
tests supported up to $\Gamma$, since the spatial supports of
such tests lie in the closure of a sufficiently large half-ball.
The resulting weak identities retain the scalar no-flux boundary
conditions of Definition~\ref{def:suitable-half-space}(ii).
They do not assert the existence of a classical normal derivative
of $n$ on $\Gamma$.

\emph{Terms on the right-hand side of the local energy inequality.}
We now pass to the limit in \eqref{eq:approximate-local-energy} to verify the local energy
inequality in Definition~\ref{def:suitable-half-space}(iv).
Fix $0<t_1<T$ and a nonnegative test function
\[
\psi\in C_c^\infty
\bigl(\overline{\mathbb H^3}\times(t_1,T)\bigr),
\qquad
\partial_\nu\psi=0 \quad\text{on }\Gamma.
\]
Choose $D=B_R^+$ so that $\psi$ vanishes near the artificial
spherical boundary, and fix $0<\sigma<t_1$.

For each $t\in(t_1,T)$, we consider the right-hand side of
\eqref{eq:approximate-local-energy} integrated up to time $t$.
Here $T$ is the fixed time horizon, whereas $t$ is the terminal
time in the energy inequality; in particular, $\psi(\cdot,t)$
need not vanish.
Since $\psi$, $\nabla\psi$, and $\partial_s\psi$ vanish for
$s\le t_1$, all these integrals reduce to $(t_1,t)\times D$.
The convergence statements established above remain valid
on this smaller cylinder.

The terms involving $n_k\log n_k$ without transport converge in
$L^1$. For the entropy transport, use the strong convergences with
exponents $16/11<3/2$ and $16/5<10/3$;
their reciprocals sum to one. Thus
\[
n_k\log n_k\,(\mathcal J_{\varepsilon_k}u_k)\to n\log n\,u
\qquad\text{strongly in }L^1((t_1,t)\times D).
\]
Split the chemotactic flux as
\[
a_k(1+\log n_k)\nabla c_k
=a_k\nabla c_k+(a_k\log n_k)\nabla c_k.
\]
The first term converges strongly in $L^1$ by
\eqref{eq:regularizing-functions-convergence}; the second does so
by \eqref{eq:regularized-entropy-flux-convergence} and the same
pair of exponents $16/11$ and $16/5$.

The gradient convergence gives
\[
|\nabla\sqrt{\tilde c_k}|^2\to|\nabla\sqrt{\tilde c}|^2
\qquad\text{strongly in }L^r((t_1,t)\times D),
\qquad 1\le r<\frac53.
\]
Consequently, the terms containing
$|\nabla\sqrt{\tilde c_k}|^2(\partial_t\psi+\Delta\psi)$ converge
in $L^1$. For their transport terms use strong $L^3$ convergence
of both $\nabla\sqrt{\tilde c_k}$ and $u_k$.
Since $(\sqrt{\tilde c_k})^{-1}\le1$ and these factors converge
almost everywhere, strong $L^3$ convergence of the gradients and
dominated convergence against the limiting cubic terms give
\[
\begin{aligned}
(\sqrt{\tilde c_k})^{-1}|\nabla\sqrt{\tilde c_k}|^2
\nabla\sqrt{\tilde c_k}
&\to(\sqrt{\tilde c})^{-1}|\nabla\sqrt{\tilde c}|^2
\nabla\sqrt{\tilde c},\\
(\sqrt{\tilde c_k})^{-1}(\partial_i\sqrt{\tilde c_k})^3
&\to(\sqrt{\tilde c})^{-1}(\partial_i\sqrt{\tilde c})^3
\end{aligned}
\qquad\text{strongly in }L^1((t_1,t)\times D).
\]
Similarly, $|u_k|^2\to|u|^2$ strongly in
$L^r((t_1,t)\times D)$ for $1\le r<5/3$, and
$|u_k|^2(\mathcal J_{\varepsilon_k}u_k)\to|u|^2u$ strongly in $L^1$.

For the pressure term, its integral over $(0,t)$ equals that over
$(t_1,t)$ because $\nabla\psi=0$ for $s\le t_1$.
The support of $\nabla\psi$ need not be contained in $(t_1,t)$.
The subtraction of the spatial mean is legitimate because
\[
\int_Du_k\cdot\nabla\psi\,dx
=\int_Du\cdot\nabla\psi\,dx=0
\]
for almost every time: the velocities are divergence free with
zero trace on $\Gamma$, and $\psi$ vanishes near the artificial
spherical boundary. The following weak and strong convergences
have conjugate exponents $5/3$ and $5/2$, respectively, and therefore
allow passage to the pressure term by weak--strong pairing:
\[
p_k-[p_k]_D(s)\rightharpoonup p-[p]_D(s)
\quad\text{weakly in }L^{5/3}((\sigma,t)\times D),
\]
and
\[
u_k\cdot\nabla\psi\to u\cdot\nabla\psi
\quad\text{strongly in }L^{5/2}((\sigma,t)\times D).
\]

Therefore,
\[
\int_0^t\!\!\int_D
\bigl(p_k-[p_k]_D(s)\bigr)
u_k\cdot\nabla\psi\,dx\,ds
\longrightarrow
\int_0^t\!\!\int_D
\bigl(p-[p]_D(s)\bigr)
u\cdot\nabla\psi\,dx\,ds .
\]
Hence
\[
\int_0^t\int_D(p_k-[p_k]_D)u_k\cdot\nabla\psi
\longrightarrow
\int_0^t\int_D(p-[p]_D)u\cdot\nabla\psi.
\]
Finally, \(\mathcal K_{\varepsilon_k}(n_k\nabla\phi)\to n\nabla\phi\) in \(L^{3/2}_{\rm loc}\) and \(u_k\to u\) in \(L^3_{\rm loc}\), so the force term also converges.

\emph{Terms on the left-hand side.} Choose $t$ in the common
full-measure set from \eqref{eq:compactness-common-time-slices}.
Since $z\log z\ge-e^{-1}$ and $\psi(t)\ge0$ has compact spatial
support, Fatou's lemma applied after adding $e^{-1}\psi(t)$ gives
\[
\int n\log n\,\psi(t)
\le\liminf_{k\to\infty}\int n_k\log n_k\,\psi(t).
\]
The strong $L^2$ convergence on this time slice gives convergence
of the two weighted quadratic terms involving
$\nabla\sqrt{\tilde c_k}(t)$ and $u_k(t)$. This use of a common
time-slice subsequence is essential; space--time weak convergence
alone would not suffice. The space--time weak convergences of
$\nabla\sqrt{n_k}$, $\nabla^2\sqrt{\tilde c_k}$, and $\nabla u_k$
give the three quadratic dissipation terms by weak lower
semicontinuity on $(0,t)$.

For the reaction-weighted term, the local energy inequality bounds
\(\sqrt{h_k}\nabla \sqrt{\tilde{c_k}}\) in \(L^2_{\rm loc}\). Moreover,
\[
\|\sqrt{h_k}-\sqrt n\|_{L^2(K)}^2\leq\|h_k-n\|_{L^1(K)}\to0
\]
for every compact \(K\). Testing against smooth vector fields and using the weak convergence of \(\nabla \sqrt{\tilde{c_k}}\), we identify the weak limit as \(\sqrt n\nabla \sqrt{\tilde{c}}\). Therefore
\[
\int_0^t\int n|\nabla \sqrt{\tilde{c}}|^2\psi
\leq\liminf_{k\to\infty}\int_0^t\int h_k|\nabla \sqrt{\tilde{c_k}}|^2\psi.
\]
For the quartic term, the local energy bound makes
$ (\sqrt{\tilde{c_k}})^{-1}|\nabla\sqrt{\tilde{c_k}}|^2 $
weakly compact in $L^2_{\rm loc}$, while the strong \(L^2\) convergence of \(\nabla \sqrt{\tilde{c_k}}\) and the bounded convergence of \((\sqrt{\tilde{c_k}})^{-1}\) identify its weak limit as \((\sqrt{\tilde{c}})^{-1}|\nabla (\sqrt{\tilde{c}})|^2\). Hence
\[
\int_0^t\int (\sqrt{\tilde{c}})^{-2}|\nabla \sqrt{\tilde{c}}|^4\psi
\leq\liminf_{k\to\infty}\int_0^t\int (\sqrt{\tilde{c_k}})^{-2}|\nabla \sqrt{\tilde{c_k}}|^4\psi.
\]
Taking the lower limit in \eqref{eq:approximate-local-energy}, using the convergences established
above, we obtain precisely the local
boundary energy inequality \eqref{le} for the limiting quadruple
\((n,c,u,p)\), for almost every \(t\in(t_1,T)\). In the pressure term
we may take
\[
\bar p(s)=[p]_D(s),
\]
since the local energy inequality is unchanged if the pressure is
modified by an arbitrary spatial constant depending only on time.

Moreover, since
\[
\operatorname{supp}\psi
   \Subset \overline{\mathbb H^3}\times(t_1,T),
\]
all terms containing \(\psi\), \(\nabla\psi\), or
\(\partial_s\psi\) vanish on \((0,t_1)\). Hence, after choosing
\(0<\sigma<t_1\), the local pressure compactness on
\((\sigma,t)\times D\) is sufficient for the above passage to the
limit.

\medskip
\noindent
\emph{Attainment of the initial data.}
Let
$\zeta\in C_c^\infty(\overline{\mathbb H^3})$
and let
$\varphi\in C_c^\infty(\mathbb H^3;\mathbb R^3)$
satisfy $\nabla\cdot\varphi=0$.
The approximate equations and the uniform estimates imply,
for $0\leq s<t\leq T$ with $t-s\leq1$,
\[
\begin{aligned}
&|\langle n_k(t)-n_k(s),\zeta\rangle|
 +|\langle c_k(t)-c_k(s),\zeta\rangle|\\
&\qquad
 +|\langle u_k(t)-u_k(s),\varphi\rangle|
 \leq C_{\zeta,\varphi,T}|t-s|^{1/10},
\end{aligned}
\]
where the constant is independent of $k$.
Indeed, the cell fluxes are locally bounded in $L^{10/9}$;
the concentration and momentum terms have stronger
time-integrability. The pressure term vanishes against
$\varphi$.

Together with the convergence of the regularized initial data,
these estimates identify representatives of the limits
satisfying
\[
\begin{aligned}
\lim_{t\downarrow0}\int_{\mathbb H^3}n(t)\zeta\,dx
 &=\int_{\mathbb H^3}n_0\zeta\,dx,\\
\lim_{t\downarrow0}\int_{\mathbb H^3}c(t)\zeta\,dx
 &=\int_{\mathbb H^3}c_0\zeta\,dx,\\
\lim_{t\downarrow0}\int_{\mathbb H^3}u(t)\cdot\varphi\,dx
 &=\int_{\mathbb H^3}u_0\cdot\varphi\,dx.
\end{aligned}
\]
The global $L^\infty(0,T;L^2)$ bound and density of smooth
solenoidal test fields further give
$u(t)\rightharpoonup u_0$ in $L^2_\sigma(\mathbb H^3)$
as $t\downarrow0$.

\begin{proof}[Proof of Theorem~\ref{thm:half-space-existence}]
Proposition~\ref{prop:regularized-global} supplies global smooth approximate solutions, while Proposition~\ref{prop:uniform-energy} gives the explicit uniform bound \(C_{\rm in}(1+T)\).
A diagonal compactness argument on time intervals separated from the
genuine initial time yields a global distributional solution satisfying
\eqref{eq:half-space-CNS}, attaining the initial data and the boundary conditions in the usual
weak sense.

The global estimates also pass to this limit. The approximate mass
bounds and maximum principle give
\eqref{eq:limit-global-scalar-bounds} by Fatou's lemma on the common
time slices chosen above, shrinking this full-measure set if necessary
so that $c_k(t,x)\to c(t,x)$ for almost every $x$ as well.
To pass to \eqref{eq:uniform-energy}, first
restrict the nonnegative integrands to $(0,T)\times B_R^+$.
The convergences in \eqref{eq:compactness-common-time-slices}
give the entropy term by Fatou's lemma and the quadratic energy
terms by strong $L^2$ convergence for almost every time.
Weak lower semicontinuity gives the dissipation terms involving
$\nabla\sqrt{1+n}$, $\nabla^2\sqrt{1+c}$, and $\nabla u$.
Moreover, $h_k\to n$, $c_k\to c$, and
$\nabla\sqrt{1+c_k}\to\nabla\sqrt{1+c}$ almost everywhere, so Fatou's
lemma gives the quartic and reaction-weighted terms. Taking the
essential supremum in time and then letting $R\to\infty$ by monotone
convergence proves \eqref{eq:limit-global-energy}, with the same
dependence $C_{\rm in}(1+T)$. Thus the limit satisfies the global
requirements in Definition~\ref{def:suitable-half-space}(i).
The $L^{5/3}$ bound for $n$ follows from
\eqref{eq:uniform-integrability} by Fatou's lemma. Applying the
half-space Gagliardo--Nirenberg inequality to $u$ and
$\nabla\sqrt{1+c}$ gives their $L^{10/3}$ bounds. Finally,
$\nabla c=2\sqrt{1+c}\,\nabla\sqrt{1+c}$ and $1\leq1+c\leq M_c$
give
\[
|\nabla c|^4
\leq16M_c^3\frac{|\nabla\sqrt{1+c}|^4}{1+c},
\]
which proves the asserted global $L^4$ integrability of $\nabla c$.

Moreover, Proposition~\ref{prop:uniform-pressure} implies that, for
every $0<\sigma<T<\infty$ and every bounded spatial domain $D$,
\[
p_k-[p_k]_D(t)
\rightharpoonup
p-[p]_D(t)
\quad\text{weakly in }
L^{5/3}\bigl((\sigma,T)\times D\bigr),
\]
after passing to a diagonal subsequence. Consequently, for every bounded domain
$D\subset\mathbb H^3$ whose closure is compact in
$\overline{\mathbb H^3}$ and every compact interval
$I\Subset(0,\infty)$,
\[
p\in L^{5/3}(I\times D).
\]
Since all distributional test functions and all test functions in the
local energy inequality are supported away from the genuine initial
time, this local pressure compactness is sufficient for the passage to
the limit carried out above. Therefore the limiting quadruple satisfies
all the requirements of Definition \ref{def:suitable-half-space} on
$\mathbb H^3\times(0,\infty)$.

The refinement of this construction used in the eventual regularity
assertion is proved in Lemma~\ref{lem:ev-selection}, with its
additional scalar estimates in Lemma~\ref{lem:ev-scalar-decay}.

This completes the existence argument.
\end{proof}

   	\section{Proof of Proposition \ref{main1}}\label{section4}

  This section is devoted to the proof of Proposition \ref{main1}. 
  We introduce the abbreviation
  \begin{align*}
  	D(\rho)=D_{\tilde{q},\tilde{r}}(\rho)=\frac1\rho\|p-[p]_\rho(t)\|_{L_t^{\tilde{q}}L_x^{\tilde{r}}(\mathcal{Q}_\rho(z_0))},
  \end{align*}
  where $\tilde{q}\in(1,2)$, $\tilde{r}\in (\frac32,3)$, and $\frac2{\tilde{q}}+\frac3{\tilde{r}}=3$. We first establish several preliminary lemmas.

  \subsection{Interpolation and local energy estimates}
\label{subsec:interpolation-local-energy}
  \begin{lemma}\label{interpolation}
  	For any $\rho>0$ and any pair of exponents $(q,r)$ satisfying
  	\begin{align*}
  		\frac2q+\frac3r=\frac32,\quad 2\leq q\leq \infty,\quad 2\leq r\leq 6,
  	\end{align*}
  	we have
  	\begin{align*}
  		\|f\|^2_{L_t^qL_x^r(\mathcal{Q})}\leq C\left(\|f\|^2_{L_t^\infty L_x^2(\mathcal{Q})}+\|\nabla f\|^2_{L_t^2 L_x^2(\mathcal{Q})}\right).
  	\end{align*}
  	In particular, the following scale-invariant estimates hold:
  	\begin{align*}
  		&\rho^{-\frac12}\|u\|_{L_t^qL_x^r(\mathcal{Q}_\rho(z_0))}\leq C(A_u(\rho)+E_u(\rho))^\frac12,\\
  		&\rho^{-\frac12}\|\sqrt{n}\|_{L_t^qL_x^r(\mathcal{Q}_\rho(z_0))}\leq C(A_{\sqrt{n}}(\rho)+E_{\sqrt{n}}(\rho))^\frac12,\\
	  		&\rho^{-\frac12}\|\nabla\sqrt{\tilde{c}}\|_{L_t^qL_x^r(\mathcal{Q}_\rho(z_0))}\leq C\bigl(A_{\nabla\sqrt{\tilde{c}}}(\rho)+E_{\nabla\sqrt{\tilde{c}}}(\rho)\bigr)^\frac12,
  	\end{align*}
  	where $C>0$ is a universal constant.
  	\begin{proof}
  		The results follow from standard interpolation, the Sobolev embedding theorem, and H\"{o}lder's inequality.
  	\end{proof}
  \end{lemma}

\begin{lemma}\label{AE} 
Let $(n,c,u,p)$ be a suitable weak solution. Write 
\[
A(r):=A_{u,\sqrt n,\nabla\sqrt{\tilde c}}(r),
\qquad
E(r):=E_{u,\sqrt n,\nabla\sqrt{\tilde c}}(r),
\]
and
\[
X(r):=A(r)+E(r),
\qquad
U(r):=A_u(r)+E_u(r).
\]
For $0<\gamma\leq1/2$, every admissible $0<\rho\leq1$, and 
$0<\alpha<1/6$, we have, with $X=X(\rho)$, 
\begin{equation}\label{eq:energy-decay-lemma} 
\begin{aligned} 
X(\gamma\rho)\leq{}&
C_\alpha\gamma^2 
\bigl(X+X^{1+\alpha}+X^{1-\alpha}\bigr)\\ 
&+C_\alpha\gamma^{-2} 
\bigl(X^{3/2}+X^{3/2+\alpha}+X^{3/2-\alpha} 
+U(\rho)^{1/2}D(\rho)\bigr)\\ 
&+C_\alpha\gamma^2
\bigl(X+\gamma^{-4}X^{3/2}\bigr)^{1-\alpha}. 
\end{aligned} 
\end{equation} 
The constant depends on $M_c,M_\phi,\alpha$ and the fixed pressure 
exponents, and is independent of $\rho$ and $\gamma$. 
\end{lemma} 

\begin{proof} 
Normalize the cylinder by translation and parabolic scaling, so that 
$\rho=1$ and the terminal point is the origin. The rescaled density is 
\[
n_\rho(y,s)=\rho^2n(x_0+\rho y,t_0+\rho^2s).
\]
The additional term $2\log\rho\,n_\rho$ in $n_\rho\log n_\rho$ is 
accounted for by adding $2\log\rho$ times the weak mass identity to 
the rescaled entropy inequality. Thus the normalized suitable 
inequality has the same form. The rescaled force contains 
$\nabla\phi_\rho=\rho\nabla\phi$, which is bounded by $M_\phi$ for 
$\rho\leq1$.

We give the boundary proof. Set 
\[
\Gamma_\gamma(x,t)
=[4\pi(\gamma^2-t)]^{-3/2}
\exp\bigl(-|x|^2/[4(\gamma^2-t)]\bigr),
\qquad 
\psi=\Gamma_\gamma\varphi, 
\]
where $0\leq\varphi\leq1$ is an even-in-$x_3$ smooth cutoff supported 
in $B_1\times(-1,1)$ and equal to one on $Q_{1/2}$. 
It is chosen to vanish near $t=-1$. Then $\partial_3\psi=0$ on 
the flat boundary, and 
\[
\begin{gathered} 
\psi\geq c\gamma^{-3}\quad\hbox{on }Q_\gamma^+, 
\qquad 
\|\psi(t)\|_{L^1(B^+)}\leq1,\\ 
|\psi|\leq C\gamma^{-3},
\qquad 
|\nabla\psi|\leq C\gamma^{-4},
\qquad 
|\psi_t+\Delta\psi|\leq C
\quad\text{in }(-1,0)\times\mathbb R^3.
\end{gathered} 
\]
The last bound follows because the derivatives of the cutoff are 
supported away from the spatial and temporal pole of the kernel.

Put $X=X(1)$ and $U=U(1)$. For almost every terminal time 
$t\in(-\gamma^2,0)$, all weak identities and inequalities below are 
first integrated on $(-1,t)$. By interpolation, 
\[
\|n\|_{L^{5/3}(Q^+)}\leq CX,
\qquad 
\|u\|_{L^{10/3}(Q^+)}
+\|\nabla\sqrt{\tilde c}\|_{L^{10/3}(Q^+)}
\leq CX^{1/2}.
\]
Since
\[
|n\log n|
\leq C_\alpha\bigl(n^{1+\alpha}+n^{1-\alpha}\bigr),
\]
H\"older's inequality gives 
\[
\begin{aligned} 
\int_{Q^+}|n\log n|\,dxdt
&\leq C_\alpha
\bigl(X^{1+\alpha}+X^{1-\alpha}\bigr),\\ 
\int_{Q^+}|n\log n|(|u|+|\nabla c|)\,dxdt 
&\leq C_\alpha
\bigl(X^{3/2+\alpha}+X^{3/2-\alpha}\bigr).
\end{aligned} 
\]
Here $\alpha<1/6$ ensures that the H\"older exponents are admissible. 
All other scalar and velocity cubic terms are bounded by $CX^{3/2}$. 
For the pressure term, 
\[
\int |p-[p]_1(t)||u|\,dxdt 
\leq
D(1)
\|u\|_{L_t^{\tilde q'}L_x^{\tilde r'}(Q^+)}
\leq
CD(1)U^{1/2}.
\]
The last embedding follows from the energy interpolation line 
$2/q+3/r=3/2$ and the finite measure of $Q_1^+$. 
Thus the absolute sum of the right-hand terms in the tested local 
energy inequality is at most 
\begin{equation}\label{jbn} 
\begin{aligned} 
\mathcal R_\gamma:={}&
C_\alpha\bigl(X+X^{1+\alpha}+X^{1-\alpha}\bigr)\\ 
&+C_\alpha\gamma^{-4} 
\bigl(
X^{3/2}+X^{3/2+\alpha}+X^{3/2-\alpha}
+U^{1/2}D(1)
\bigr). 
\end{aligned} 
\end{equation} 

The negative part of the entropy must be estimated with the same 
weight. The weak mass identity, with its upper limit still $t$, is 
\[
M_\psi(t)
:=\int_{B^+}n(t)\psi(t)\,dx
=
\int_{-1}^t\int_{B^+}
\bigl[
n(\psi_s+\Delta\psi)
+n(u+\nabla c)\cdot\nabla\psi
\bigr]\,dxds.
\]
Consequently, 
\begin{equation}\label{eq:weighted-mass-repair} 
0\leq M_\psi(t)
\leq C\bigl(X+\gamma^{-4}X^{3/2}\bigr). 
\end{equation} 
H\"older's inequality with respect to the measure $\psi(t,x)\,dx$ gives 
\begin{equation}\label{eq:weighted-negative-entropy} 
\begin{aligned} 
\int_{B^+\cap\{0<n<1\}}
(-n\log n)\psi\,dx 
&\leq
C_\alpha
\int_{B^+\cap\{0<n<1\}}
n^{1-\alpha}\psi\,dx\\
&\leq
C_\alpha M_\psi(t)^{1-\alpha}
\left(\int_{B^+}\psi\,dx\right)^\alpha\\
&\leq
C_\alpha
\bigl(X+\gamma^{-4}X^{3/2}\bigr)^{1-\alpha}.
\end{aligned} 
\end{equation} 
In particular, the weight cannot be discarded in this step.

Add the bound for the negative entropy to the tested energy inequality, 
and use \eqref{eq:weighted-mass-repair} for the density contribution 
to $A$. First take the essential supremum over $t\in(-\gamma^2,0)$ 
for the terminal energies; then let admissible terminal times increase 
to zero for the nonnegative dissipation integrals. Multiplication 
by $\gamma^2$ gives 
\begin{equation}\label{11} 
\begin{aligned} 
&X(\gamma)
+\gamma^{-1}
\operatorname*{ess\,sup}_{-\gamma^2<t<0}
\int_{B_\gamma^+}n|\log n|(t)\,dx\\
&\qquad
+\gamma^{-1}
\iint_{Q_\gamma^+}
\frac{|\nabla\sqrt{\tilde c}|^4}{\tilde c}\,dxdt\\
&\qquad\leq
C\gamma^2\mathcal R_\gamma
+C_\alpha\gamma^2
\bigl(X+\gamma^{-4}X^{3/2}\bigr)^{1-\alpha}. 
\end{aligned} 
\end{equation} 
Discard the last two nonnegative terms and rescale to obtain 
\eqref{eq:energy-decay-lemma}. The interior argument uses a compactly 
supported cutoff in $Q_1$, with no symmetry restriction, and is identical. 
The cubic terms have critical scaling and acquire no extra factor $\rho$; 
only the gravitational force has the additional factor described above. 
\end{proof}

\begin{corollary}\label{corollary1} 
Assume that $A(\rho)+E(\rho)$ is sufficiently small. Then
\begin{align*} 
A(\gamma\rho)+E(\gamma\rho)
\leq{}&
C_1\gamma^2
\bigl(A(\rho)+E(\rho)\bigr)^{1-\delta}\\
&+C_1\gamma^{-2}
\bigl(A(\rho)+E(\rho)\bigr)^{\frac32-\delta}\\
&+C_1\gamma^{-2}
\bigl(A(\rho)+E(\rho)\bigr)^{1/2}D(\rho),
\end{align*} 
where $C_1=C(M_c,M_\phi,\delta)$,
$0<\delta<\frac14$, and $0<\gamma\leq\frac12$, with the pressure 
exponents fixed.

\begin{proof} 
It suffices to assume
\[
A(\rho)+E(\rho)\leq1.
\]
Apply Lemma~\ref{AE} with
\[
\alpha=\frac{2\delta}{3},
\qquad 0<\delta<\frac14.
\]
Since $\alpha<1/6$, this choice is admissible. Writing
\[
X=X(\rho)=A(\rho)+E(\rho),
\]
concavity gives
\[
\begin{aligned}
\gamma^2
\bigl(X+\gamma^{-4}X^{3/2}\bigr)^{1-\alpha}
&\leq
\gamma^2X^{1-\alpha}
+\gamma^{-2+4\alpha}
X^{3/2-3\alpha/2}\\
&\leq
\gamma^2X^{1-\delta}
+\gamma^{-2}X^{3/2-\delta}.
\end{aligned}
\]
Indeed, $\alpha=2\delta/3<\delta$, while
\[
\frac32-\frac{3\alpha}{2}
=\frac32-\delta,
\]
and $0<\gamma\leq1/2$ implies
\[
\gamma^{-2+4\alpha}\leq\gamma^{-2}.
\]

Since $X\leq1$ and $\alpha\leq\delta$, we also have
\[
X+X^{1+\alpha}+X^{1-\alpha}
\leq C X^{1-\delta},
\]
and
\[
X^{3/2}+X^{3/2+\alpha}+X^{3/2-\alpha}
\leq C X^{3/2-\delta}.
\]
Finally,
\[
U(\rho)\leq X(\rho)=A(\rho)+E(\rho),
\]
and hence
\[
U(\rho)^{1/2}D(\rho)
\leq
\bigl(A(\rho)+E(\rho)\bigr)^{1/2}D(\rho).
\]
Substituting these estimates into
\eqref{eq:energy-decay-lemma} gives the asserted inequality.
\end{proof} 
 
\end{corollary}

  \subsection{Pressure estimates}
\label{subsec:regularity-pressure-estimates}
  \begin{lemma}\label{p}
  	Let $(n,c,u,p)$ be a suitable weak solution of \eqref{che}--\eqref{bound}. For any $\gamma\in(0,\frac12]$, $\rho>0$, and $\epsilon_1\in(0,\frac{3}{\tilde{r}}-1)$ is sufficiently small, there exist a constant $C=C(\tilde{q},\tilde{r},\epsilon_1)$ such that
  	\begin{align}
  		D(\gamma\rho)\leq& C\gamma^{1+\epsilon_1}\Big[D_{\tilde{q},1}(\rho)+A_u(\rho)+E_u(\rho)+\rho\left(A_{\sqrt{n}}(\rho)+E_{\sqrt{n}}(\rho)\right)\nonumber\\
  		& \qquad\qquad+\left(A_u(\rho)+E_u(\rho)\right)^\frac12\Big]\nonumber\\
  		&+C\gamma^{-1}\left[A_u(\rho)+E_u(\rho)+\rho M_\phi\left(A_{\sqrt{n}}(\rho)+E_{\sqrt{n}}(\rho)\right)\right], \label{D_1}
  	\end{align}
  	and
  	\begin{align}
  		D(\gamma\rho)\leq& C\gamma^{1+\epsilon_1}\Big[D(\rho)+A_u(\rho)+E_u(\rho)+\rho\left(A_{\sqrt{n}}(\rho)+E_{\sqrt{n}}(\rho)\right)+\left(A_u(\rho)+E_u(\rho)\right)^\frac12\Big]\nonumber\\
  		&+C\gamma^{-1}\left[A_u(\rho)+E_u(\rho)+\rho M_\phi\left(A_{\sqrt{n}}(\rho)+E_{\sqrt{n}}(\rho)\right)\right]. \label{D_2}
  	\end{align}
  	\begin{proof}
  		We prove the boundary case first. Without loss of generality, by scaling we assume $\rho=1$. Fix once and for all a smooth domain
$\tilde B$ satisfying $B_{3/4}^+\subset\tilde B\subset B_1^+$,
and set $\tilde Q=(-1,0)\times\tilde B$. This leaves a fixed
geometric gap for the homogeneous Stokes estimate from
$Q_{3/4}^+$ to $Q_{1/2}^+$. Define $r^*$ via $\frac{1}{r^*}=\frac{1}{\tilde{r}}+\frac13<1$. H\"{o}lder's inequality implies
  		\begin{align*}
  			\|u\cdot\nabla u\|_{L_t^{\tilde{q}}L_x^{r^*}(Q^+)}\leq \|\nabla u\|_{L_t^2L_x^2(Q^+)}\|u\|_{L_t^{q_1}L_x^{r_1}(Q^+)},
  		\end{align*}
  		where $\frac{1}{\tilde{q}}=\frac12+\frac{1}{q_1}$ and $\frac{1}{r^*}=\frac{1}{2}+\frac{1}{r_1}$.
  		Under the conditions on $(\tilde{q},\tilde{r})$, we have
  		\begin{align*}
  			&\frac{2}{\tilde{q}}+\frac{3}{r^*}=\frac{3}{\tilde{r}}+1+\frac{2}{\tilde{q}}=4,\\
  			&\frac{2}{\tilde{q}}+\frac{3}{r^*}=\frac{2}{q_1}+1+\frac{3}{r_1}+\frac32=4.
  		\end{align*}
  		Since $q_1\in(2,\infty)$, $r_1\in(2,6)$, and $\frac{2}{q_1}+\frac{3}{r_1}=\frac32$, Lemma \ref{interpolation} yields
  		\begin{align*}
  			\|u\cdot\nabla u\|_{L_t^{\tilde{q}}L_x^{r^*}(Q^+)}\leq C(A_u(1)+E_u(1)).
  		\end{align*}
  		Similarly,
  		\begin{align*}
  			\|n\|_{L_t^{\tilde{q}}L_x^{r^*}(Q^+)}\leq C\|n\|_{L_t^{\tilde{q}}L_x^{\tilde{r}}(Q^+)}= C\|\sqrt{n}\|^2_{L_t^{2\tilde{q}}L_x^{2\tilde{r}}(Q^+)}.
  		\end{align*}
  		Since $2\tilde{q}\in (2,4)$ and $2\tilde{r}\in (3,6)$, the scaling relation $\frac{2}{2\tilde{q}}+\frac{3}{2\tilde{r}}=\frac32$ holds. Thus,
  		\begin{align*}
  			\|n\|_{L_t^{\tilde{q}}L_x^{r^*}(Q^+)}\leq C(A_{\sqrt{n}}(1)+E_{\sqrt{n}}(1)).
  		\end{align*}
  		By Lemma \ref{stokes}, there exists a unique pair
  		\begin{align*}
  			v\in W^{1,2}_{\tilde{q},r^*}(\tilde{Q})\quad \text{and}\quad p_1\in W^{0,1}_{\tilde{q},r^*}(\tilde{Q})
  		\end{align*}
  		solving the initial boundary value problem:
  		\begin{align*}
  			\begin{cases}
  				\partial_tv-\Delta v+\nabla p_1=-u\cdot\nabla u-n\nabla\phi \quad&\text{in}~\tilde{Q},\\
  				\nabla\cdot v=0   &\text{in}~\tilde{Q},\\
  				[p_1]_{\tilde{B}}(t)=0 &\text{for almost everywhere (a.e.)}~t\in[-1,0],\\
  				v=0 &\text{on}~ \partial_p\tilde{Q}.
  			\end{cases}
  		\end{align*}
  		Moreover, the solution satisfies
  		\begin{align}
  			&\||v|+|\nabla v|+|p_1|+|\nabla p_1|\|_{L_t^{\tilde{q}}L_x^{r^*}(\tilde{Q})}\nonumber\\
  			\leq& C\left(\|u\cdot\nabla u\|_{L_t^{\tilde{q}}L_x^{r^*}(Q^+)}+\|\nabla\phi\|_{L^\infty_{t,x}(Q^+)}\|n\|_{L_t^{\tilde{q}}L_x^{r^*}(Q^+)}\right)\nonumber\\
  			\leq& C(A_u(1)+E_u(1))+C\|\nabla\phi\|_{L^\infty_{t,x}(Q^+)}(A_{\sqrt{n}}(1)+E_{\sqrt{n}}(1)). \label{E1}
  		\end{align}
  		The Sobolev-Poincar\'{e} inequality then yields
  		\begin{align}
  			&\|p_1\|_{L_t^{\tilde{q}}L_x^{\tilde{r}}(\tilde{Q})}\leq C\|\nabla p_1\|_{L_t^{\tilde{q}}L_x^{r^*}(\tilde{Q})}\nonumber\\ 
  			\leq&C(A_u(1)+E_u(1))+C\|\nabla\phi\|_{L^\infty_{t,x}(Q^+)}(A_{\sqrt{n}}(1)+E_{\sqrt{n}}(1)). \label{p_1}
  		\end{align}
  		Setting $w=u-v$ and $p_2=p-p_1$, the system for $(w, p_2)$ is
  		\begin{align*}
  			\begin{cases}
  				\partial_tw-\Delta w+\nabla p_2=0\quad &\text{in}~\tilde{Q},\\
  				\nabla\cdot w=0 &\text{in}~\tilde{Q},\\
  				w=0 &\text{on}~[-1,0)\times(\partial\tilde{B}\cap \{x_3=0\}).
  			\end{cases}
  		\end{align*}
  		Since $\tilde{q}\in (1,2)$, using \eqref{E1}, a simple calculation verifies that 
  		\[
  		w \in W^{0,1}_{\tilde{q}, 1}(\tilde{Q}) \quad \text{and}\; p_2 \in L^{\tilde{q}}_t L^1_x(\tilde{Q}).
  		\]
  		By H\"{o}lder's inequality, Sobolev-Poincar\'{e} inequality, and Lemma \ref{Seregin}, we obtain
  		\begin{align}
  			&\frac{1}{\gamma}\|p_2-[p_2]_\gamma\|_{L_t^{\tilde{q}}L_x^{\tilde{r}}(Q^+(\gamma))}\leq C\|\nabla p_2\|_{L_t^{\tilde{q}}L_x^{\tilde{r}}(Q^+(\gamma))}\nonumber\\
  			\leq & C\gamma^{3(\frac{1}{\tilde{r}}-\frac{1}{r'})}\|\nabla p_2\|_{L_t^{\tilde{q}}L_x^{r'}(Q^+(\gamma))}\leq C\gamma^{3(\frac{1}{\tilde{r}}-\frac{1}{r'})}\|\nabla p_2\|_{L_t^{\tilde{q}}L_x^{r'}(Q^+(\frac12))}\nonumber\\
  			\leq &C\gamma^{3(\frac{1}{\tilde{r}}-\frac{1}{r'})}\left(\|p_2-[p_2]_{\tilde{B}}\|_{L_t^{\tilde{q}}L_x^1(\tilde{Q})}+\|w\|_{L_t^{\tilde{q}}L_x^1(\tilde{Q})}\right)\nonumber\\
  			\leq & C\gamma^{3(\frac{1}{\tilde{r}}-\frac{1}{r'})}\left(D_{\tilde{q},1}(1)+\|p_1\|_{L_t^{\tilde{q}}L_x^1(\tilde{Q})}+\|v\|_{L_t^{\tilde{q}}L_x^1(\tilde{Q})}+\|u\|_{L_t^{\tilde{q}}L_x^1(Q^+)}\right)\nonumber\\
  			\leq &C\gamma^{3(\frac{1}{\tilde{r}}-\frac{1}{r'})}\Big[D_{\tilde{q},1}(1)+A_u(1)+E_u(1)+M_\phi\big(A_{\sqrt{n}}(1)+E_{\sqrt{n}}(1)\big)\nonumber\\
  			&\qquad\qquad\quad+ A_u(1)^{(\frac{1}{2\tilde{r}}-\frac{1}{6\tilde{q}})}E_u(1)^{\frac{1}{2\tilde{q}}}\Big], \label{p_2}
  		\end{align}
  		where $r'>\tilde{r}$ is arbitrary. Combining (\ref{p_1}) and (\ref{p_2}) yields
  		\begin{align*}
  			D(\gamma)\leq& C\gamma^{3(\frac{1}{\tilde{r}}-\frac{1}{r'})}\Big[D_{\tilde{q},1}(1)+A_u(1)+E_u(1)+M_\phi\big(A_{\sqrt{n}}(1)+E_{\sqrt{n}}(1)\big)\\
  			& \qquad\qquad\quad+\left(A_u(1)+E_u(1)\right)^\frac12\Big]\\
  			&+C\gamma^{-1}\Big[A_u(1)+E_u(1)+M_\phi\big(A_{\sqrt{n}}(1)+E_{\sqrt{n}}(1)\big)\Big].
  		\end{align*}
  		Rescaling this inequality gives
  		\begin{align*}
  			D(\gamma\rho)\leq& C\gamma^{3(\frac{1}{\tilde{r}}-\frac{1}{r'})}\Big[D_{\tilde{q},1}(\rho)+A_u(\rho)+E_u(\rho)+\rho M_\phi\big(A_{\sqrt{n}}(\rho)+E_{\sqrt{n}}(\rho)\big)\\
  			&\qquad\qquad\quad+\left(A_u(\rho )+E_u(\rho)\right)^\frac12\Big]\\
  			&+C\gamma^{-1}\left[A_u(\rho)+E_u(\rho)+\rho M_\phi\big(A_{\sqrt{n}}(\rho)+E_{\sqrt{n}}(\rho)\big)\right].
  		\end{align*}
  		Since $\tilde{r}<3$ and $r'$ can be chosen arbitrarily large, $\epsilon_1$ may range in $(0, \frac{3}{\tilde{r}} - 1)$. The force terms in the decaying bracket retain the
factor $\rho M_\phi$. To obtain the stated form of \eqref{D_1},
absorb these terms into the last bracket using
$\gamma^{1+\epsilon_1}\rho M_\phi
\leq\gamma^{-1}\rho M_\phi$; the additional nonnegative density
term in the first bracket may then be kept.
Finally, \eqref{D_2} follows from \eqref{D_1} and H\"older's inequality.

        We now consider the interior case. Fix an interior point
$z_0=(x_0,t_0)$ and an admissible radius $\rho$.
By translation and parabolic rescaling, we may assume that
$z_0=(0,0)$ and $\rho=1$, and retain the same notation for
the rescaled functions. We work on the full cylinder
$Q=(-1,0)\times B$, where $B=B_1$, using the notation
introduced in Section~\ref{subsec:notation-scaling}. In particular, all the
scale-invariant quantities below are evaluated on full
balls and cylinders.

For any spatially integrable function $f$, write
\[
[f]_r(t):=\frac{1}{|B_r|}\int_{B_r}f(x,t)\,dx.
\]
Let $\chi\in C_c^\infty(B)$ satisfy
\[
0\le\chi\le1,\qquad \chi\equiv1\quad\text{in }B_{3/4},
\]
and set $\bar u(t):=[u]_1(t)$ for almost every $t\in(-1,0)$.
Taking the divergence of the momentum equation and using
$\nabla\cdot u=0$, we have in the sense of distributions
\[
    -\Delta p
    =
    \partial_i\partial_j(u_i u_j)
    +
    \nabla\cdot(n\nabla\phi).
\]
Since $\bar u(t)$ is independent of $x$ and $\nabla\cdot u=0$,
\[
    \partial_i\partial_j
    \big((u_i-\bar u_i)(u_j-\bar u_j)\big)
    =
    \partial_i\partial_j(u_i u_j).
\]

For almost every fixed $t$, define on $\mathbb R^3$
\[
\begin{aligned}
p_1
={}&
(-\Delta)^{-1}\partial_i\partial_j
\Big[
\chi
(u_i-\bar u_i)(u_j-\bar u_j)
\Big]+
(-\Delta)^{-1}\nabla\cdot
\big(\chi n\nabla\phi\big),
\end{aligned}
\]
and put
\[
    p_2=p-p_1.
\]
Since $\chi\equiv1$ in $B_{3/4}$, comparison with the pressure
equation above yields
\[
    \Delta p_2=0
    \qquad\text{in }B_{3/4}
\]
for almost every $t\in(-1,0)$.

We first estimate $p_1$. Write
\[
    p_1=p_1^u+p_1^n,
\]
where $p_1^u$ and $p_1^n$ denote respectively the contributions of the
convective term and the external force. By the Calder\'on--Zygmund singular-integral estimate
(see Stein~\cite{Stein}),
\[
\begin{aligned}
\|p_1^u\|_
{L_t^{\tilde q}L_x^{\tilde r}
((-1,0)\times\mathbb R^3)}
&\leq
C
\|(u-\bar u)\otimes(u-\bar u)\|_
{L_t^{\tilde q}L_x^{\tilde r}(Q)}
=
C
\|u-\bar u\|_
{L_t^{2\tilde q}L_x^{2\tilde r}(Q)}^2.
\end{aligned}
\]
Since
\[
    \frac{2}{2\tilde q}
    +\frac{3}{2\tilde r}
    =\frac32,
\]
Lemma \ref{interpolation}, together with the Poincar\'e inequality, gives
\[
    \|u-\bar u\|_
    {L_t^{2\tilde q}L_x^{2\tilde r}(Q)}^2
    \leq
    C\big(A_u(1)+E_u(1)\big).
\]
Hence
\[
    \|p_1^u\|_
    {L_t^{\tilde q}L_x^{\tilde r}
    ((-1,0)\times\mathbb R^3)}
    \leq
    C\big(A_u(1)+E_u(1)\big).
\]

For the contribution generated by $n\nabla\phi$, define
$\tilde r_*$ by
\[
    \frac1{\tilde r_*}
    =
    \frac1{\tilde r}+\frac13.
\]
Since $3/2<\tilde r<3$, one has
\[
    1<\tilde r_*<\frac32.
\]
The Hardy--Littlewood--Sobolev inequality
(see~\cite{Stein}) therefore yields
\[
\begin{aligned}
\|p_1^n\|_
{L_t^{\tilde q}L_x^{\tilde r}
((-1,0)\times\mathbb R^3)}
&\leq
C
\|\chi n\nabla\phi\|_
{L_t^{\tilde q}L_x^{\tilde r_*}(Q)}
\leq
CM_\phi
\|n\|_
{L_t^{\tilde q}L_x^{\tilde r_*}(Q)}.
\end{aligned}
\]
Since $B$ has finite measure and
$\tilde r_*<\tilde r$,
\[
    \|n\|_
    {L_t^{\tilde q}L_x^{\tilde r_*}(Q)}
    \leq
    C
    \|n\|_
    {L_t^{\tilde q}L_x^{\tilde r}(Q)}
    =
    C\|\sqrt n\|_
    {L_t^{2\tilde q}L_x^{2\tilde r}(Q)}^2
    \leq
    C\big(A_{\sqrt n}(1)+E_{\sqrt n}(1)\big).
\]
Consequently,
\[
    \|p_1^n\|_
    {L_t^{\tilde q}L_x^{\tilde r}
    ((-1,0)\times\mathbb R^3)}
    \leq
    CM_\phi
    \big(A_{\sqrt n}(1)+E_{\sqrt n}(1)\big).
\]
Combining the preceding two estimates, we obtain
\[
\begin{aligned}
\|p_1\|_
{L_t^{\tilde q}L_x^{\tilde r}
((-1,0)\times\mathbb R^3)}
\leq
C\Big[
A_u(1)+E_u(1)+
M_\phi\big(A_{\sqrt n}(1)+E_{\sqrt n}(1)\big)
\Big].
\end{aligned}
\]

It remains to estimate the harmonic part $p_2$. Since
$p_2(\cdot,t)$ is harmonic in $B_{3/4}$, the standard interior
estimate for harmonic functions implies
\[
    \|\nabla p_2(\cdot,t)\|_{L^\infty(B_{1/2})}
    \leq
    C
    \|p_2(\cdot,t)-[p_2]_{3/4}(t)\|_
    {L^1(B_{3/4})}.
\]
Hence, for every $\gamma\in(0,1/2]$,
the Poincar\'e inequality gives
\[
\begin{aligned}
&\|p_2-[p_2]_{\gamma}(t)\|_
{L^{\tilde r}(B_\gamma)}
\leq C\gamma\|\nabla p_2(\cdot,t)\|_{L^{\tilde{r}}(B_{\gamma})}\leq C\gamma^{1+\frac3{\tilde r}}\|\nabla p_2(\cdot,t)\|_{L^\infty(B_{1/2})}\\&\qquad\qquad\qquad\qquad\quad\leq
C\gamma^{1+\frac3{\tilde r}}
\|p_2-[p_2]_{3/4}(t)\|_
{L^1(B_{3/4})}.
\end{aligned}
\]
Taking the $L^{\tilde q}$ norm in time, and using
$(-\gamma^2,0)\subset(-1,0)$, we obtain
\[
\begin{aligned}
\|p_2-[p_2]_{\gamma}(t)\|_
{L_t^{\tilde q}L_x^{\tilde r}(Q_\gamma)}
\leq
C\gamma^{1+\frac3{\tilde r}}
\|p_2-[p_2]_{3/4}(t)\|_
{L_t^{\tilde q}L_x^1
((-1,0)\times B_{3/4})}.
\end{aligned}
\]

Since $p_2=p-p_1$, the elementary estimate
\[
    \|f-[f]_{B_{3/4}}\|_{L^1(B_{3/4})}
    \leq
    2\|f-c(t)\|_{L^1(B_{3/4})}
\]
valid for every function $c(t)$, applied with
$c(t)=[p]_1(t)$, gives
\[
\begin{aligned}
\|p_2-[p_2]_{3/4}(t)\|_
{L_t^{\tilde q}L_x^1
((-1,0)\times B_{3/4})}
\leq
C D_{\tilde q,1}(1)
+
C
\|p_1\|_
{L_t^{\tilde q}L_x^{\tilde r}
((-1,0)\times\mathbb R^3)}.
\end{aligned}
\]
Therefore,
\[
\begin{aligned}
&\frac1\gamma
\|p_2-[p_2]_{\gamma}(t)\|_
{L_t^{\tilde q}L_x^{\tilde r}(Q_\gamma)}
\\&\leq{}
C\gamma^{\frac3{\tilde r}}
D_{\tilde q,1}(1)
+
C\gamma^{\frac3{\tilde r}}
\Big[
A_u(1)+E_u(1)
+
M_\phi\big(A_{\sqrt n}(1)+E_{\sqrt n}(1)\big)
\Big].
\end{aligned}
\]

For the localized part $p_1$, we simply use
\[
\begin{aligned}
\frac1\gamma
\|p_1-[p_1]_{\gamma}(t)\|_
{L_t^{\tilde q}L_x^{\tilde r}(Q_\gamma)}
&\leq
C\gamma^{-1}
\|p_1\|_
{L_t^{\tilde q}L_x^{\tilde r}
((-1,0)\times\mathbb R^3)}
\\
&\leq
C\gamma^{-1}
\Big[
A_u(1)+E_u(1)
+
M_\phi\big(A_{\sqrt n}(1)+E_{\sqrt n}(1)\big)
\Big].
\end{aligned}
\]
Since $\gamma^{3/\tilde r}\leq\gamma^{-1}$ for
$0<\gamma\leq1/2$, combining the estimates for $p_1$ and $p_2$
gives
\[
\begin{aligned}
D(\gamma)
\leq{}
C\gamma^{\frac3{\tilde r}}
D_{\tilde q,1}(1)
+
C\gamma^{-1}
\Big[
A_u(1)+E_u(1)
+
M_\phi\big(A_{\sqrt n}(1)+E_{\sqrt n}(1)\big)
\Big].
\end{aligned}
\]

A standard parabolic rescaling now yields
\[
\begin{aligned}
D(\gamma\rho)
\leq{}
C\gamma^{\frac3{\tilde r}}
D_{\tilde q,1}(\rho)
+
C\gamma^{-1}
\Big[
A_u(\rho)+E_u(\rho)
+
\rho M_\phi
\big(A_{\sqrt n}(\rho)+E_{\sqrt n}(\rho)\big)
\Big].
\end{aligned}
\]
The factor $\rho$ in front of $\|\nabla\phi\|_{L^\infty}$ follows from
the rescaling
\[
    \phi_\rho(x)=\phi(\rho x),
    \qquad
    \|\nabla\phi_\rho\|_{L^\infty}
    =
    \rho\|\nabla\phi\|_{L^\infty}.
\]

Finally, H\"older's inequality in space and
$ \frac{2}{\tilde q}+\frac{3}{\tilde r}=3$
   give
\[
    D_{\tilde q,1}(\rho)
    \leq
    C D(\rho).
\]
Thus
\[
\begin{aligned}
D(\gamma\rho)
\leq{}
C\gamma^{\frac3{\tilde r}}D(\rho)
+
C\gamma^{-1}
\Big[
A_u(\rho)+E_u(\rho)
+
\rho M_\phi
\big(A_{\sqrt n}(\rho)+E_{\sqrt n}(\rho)\big)
\Big].
\end{aligned}
\]
Since $\tilde r<3$, one has $3/\tilde r>1$, and hence for
every $ 0<\epsilon_1<\frac3{\tilde r}-1$
and every $\gamma\in(0,1/2]$, $\gamma^{3/\tilde r}
    \leq
    \gamma^{1+\epsilon_1}$.
This proves the asserted interior pressure estimate.
  	\end{proof}
  \end{lemma}
  
  \begin{corollary}\label{corollary2}
  	Let $(n,c,u,p)$ be a suitable weak solution of \eqref{che}--\eqref{bound}. For any $\gamma\in(0,\frac12]$, $0<\rho\leq1$, and $\epsilon_1\in(0,\frac{3}{\tilde{r}}-1)$, if
  	\begin{align*}
  		A_{u,\sqrt{n}}(\rho)+E_{u,\sqrt{n}}(\rho)\leq 1,
  	\end{align*}
  	then
  	\begin{align*}
  		&D(\gamma\rho)\leq C_2\Big[\gamma^{1+\epsilon_1}\Big(D_{\tilde{q},1}(\rho)+\big(A_{u,\sqrt{n}}(\rho)+E_{u,\sqrt{n}}(\rho)\big)^\frac12 \Big)\\&\qquad\qquad\qquad+\gamma^{-1}\Big(A_{u,\sqrt{n}}(\rho)+E_{u,\sqrt{n}}(\rho)\Big)\Big],
  	\end{align*}
  	where $C_2=C(\tilde{q},\tilde{r},\epsilon_1,M_\phi)$.
  \end{corollary}
  \begin{corollary}[Pressure estimate at the exponent $5/3$]
\label{cor:pressure-53}
Let $(n,c,u,p)$ be a suitable weak solution of
\eqref{che}--\eqref{bound}, Set
\[
\bar P(r)
:=
r^{-5/3}
\int_{\mathcal Q_r}
|p-[p]_r(t)|^{5/3}\,dx\,dt.
\]
Then, for every $0<4r\leq\rho\leq1$,
\begin{align}
\bar P(r)
&\le
C\left(\frac r\rho\right)^{5/3}\bar P(\rho)+
C\left(\frac r\rho\right)^{5/3}
\left[
A_{u,\sqrt n}(\rho)^{2/3}E_{u,\sqrt n}(\rho)
+A_u(\rho)^{1/3}E_u(\rho)^{1/2}
\right]
\nonumber\\
&\quad+
C\left(\frac{\rho}{r}\right)^{5/3}
A_{u,\sqrt n}(\rho)^{2/3}E_{u,\sqrt n}(\rho)+
C\left(\frac{\rho}{r}\right)^{5/3}
\rho ^{5/3}A_{\sqrt n}(\rho)^{5/3}.\label{eq:pressure-53-common}
\end{align}
where $C=C(M_\phi)>0$.
\end{corollary}

\begin{proof}
By translation and parabolic rescaling, we first assume
$\rho=1$ and write $\gamma=r/\rho\le1/4$.
The local Sobolev inequality and interpolation give
\[
\|n\|_{L^{5/3}(\mathcal{Q})}
=
\|\sqrt n\|_{L^{10/3}(\mathcal{Q})}^2
\le
C\left[
A_{\sqrt n}(1)^{2/5}E_{\sqrt n}(1)^{3/5}
+A_{\sqrt n}(1)
\right].
\]
The last term is necessary because $\sqrt n$ has neither
zero trace nor zero spatial mean.

For a boundary cylinder, use the Stokes decomposition
from Lemma~\ref{p} with
\[
\tilde q=\tilde r=\frac53,
\qquad r_*=\frac{15}{14},
\qquad r'=\frac{15}{4}.
\]
The auxiliary fields $v,p_1$ are defined on $\tilde Q$ from
Lemma~\ref{p}; their global norms below are taken on this
domain. The original fields and forcing terms are estimated
on $\mathcal Q=Q_1^+$, and the smaller cylinders are contained
in $\tilde Q$.
Since $u=0$ on the flat boundary,
Sobolev--Poincar\'e interpolation yields
\[
\|u\cdot\nabla u\|_{L^{5/3}_tL^{15/14}_x(\mathcal{Q})}
\le C A_u(1)^{2/5}E_u(1)^{3/5},
\qquad
\|u\|_{L^{5/3}_tL^1_x(\mathcal{Q})}
\le C A_u(1)^{1/5}E_u(1)^{3/10}.
\]
Moreover,
\[
\|n\nabla\phi\|_{L^{5/3}_tL^{15/14}_x(\mathcal{Q})}
\le CM_\phi
\left[
A_{\sqrt n}(1)^{2/5}E_{\sqrt n}(1)^{3/5}
+A_{\sqrt n}(1)
\right].
\]
Applying the Stokes estimate to the localized problem,
with the preceding bounds for its forcing terms, gives
\begin{align*}
&\|p_1\|_{L^{5/3}(\tilde Q)}
+\|v\|_{L^{5/3}_tL^1_x(\tilde Q)}
\\
&\quad\le
C A_u(1)^{2/5}E_u(1)^{3/5}
+CM_\phi
\left[
A_{\sqrt n}(1)^{2/5}E_{\sqrt n}(1)^{3/5}
+A_{\sqrt n}(1)
\right].
\end{align*}
For the homogeneous remainder $(w,p_2)=(u-v,p-p_1)$,
the boundary Stokes estimate with $r'=15/4$ yields
\begin{align*}
\gamma^{-1}
\|p_2-[p_2]_\gamma\|_{L^{5/3}(\mathcal{Q}_{\gamma})}
\le C\gamma
\left[
\bar P(1)^{3/5}
+\|p_1\|_{L^{5/3}(\tilde Q)}
+\|v\|_{L^{5/3}_tL^1_x(\tilde Q)}
+\|u\|_{L^{5/3}_tL^1_x(\mathcal{Q})}
\right].
\end{align*}
Combining these estimates with
\[
\gamma^{-1}
\|p_1-[p_1]_\gamma\|_{L^{5/3}(\mathcal{Q}_{\gamma})}
\le
C\gamma^{-1}\|p_1\|_{L^{5/3}(\tilde Q)}
\]
and raising to the power $5/3$ proves the asserted
unit-scale estimate. Here the pure density term with
coefficient $\gamma^{5/3}$ is absorbed into its counterpart
with coefficient $\gamma^{-5/3}$.

For an interior cylinder, use the decomposition
$p=p_1^u+p_1^n+p_2$ from Lemma~\ref{p}.
The Calder\'on--Zygmund inequality and interpolation
for $u-[u]_1$ give
\[
\|p_1^u\|_{L^{5/3}(\mathcal{Q})}
\le C A_u(1)^{2/5}E_u(1)^{3/5}.
\]
The Hardy--Littlewood--Sobolev inequality gives
\begin{align*}
\|p_1^n\|_{L^{5/3}(\mathcal{Q})}
\le CM_\phi
\|n\|_{L^{5/3}_tL^{15/14}_x(\mathcal{Q})}
\le CM_\phi
\left[
A_{\sqrt n}(1)^{2/5}E_{\sqrt n}(1)^{3/5}
+A_{\sqrt n}(1)
\right].
\end{align*}
Since $p_2$ is harmonic in $B_{3/4}$, its interior
oscillation estimate implies
\begin{align*}
\bar P(\gamma)
\le C\gamma^3\bar P(1)
+C\gamma^{-5/3}
\Bigl[
A_u(1)^{2/3}E_u(1)
+M_\phi^{5/3}
\bigl(
A_{\sqrt n}(1)^{2/3}E_{\sqrt n}(1)
+A_{\sqrt n}(1)^{5/3}
\bigr)
\Bigr].
\end{align*}
This implies the asserted unit-scale estimate because
$\gamma^3\le\gamma^{5/3}$.

Rescaling back replaces $M_\phi$ by $\rho M_\phi$,
while $A$, $E$, and $\bar P$ remain invariant.
Using $\rho\le1$ and
\[
A_u^{2/3}E_u+A_{\sqrt n}^{2/3}E_{\sqrt n}
\le A_{u,\sqrt n}^{2/3}E_{u,\sqrt n}
\]
at radius $\rho$ completes the proof.
\end{proof}

\subsection{Decay of the scale-invariant quantities}
\label{subsec:decay-estimates}  	
  \begin{lemma}\label{key} 
There exists a constant 
$\varepsilon_2=\varepsilon_2(\tilde q,M_c,M_\phi)>0$ such that if 
\begin{equation}\label{4.9} 
A(1)+E(1)+D_{\tilde q,1}(1)\leq\varepsilon_2,
\end{equation} 
then there exist constants $\eta_1,\eta_2>0$ and $\rho_0>0$ such that 
\begin{equation}\label{4.10} 
A(\rho)+E(\rho)\leq\rho^{2-\eta_1},
\qquad 
D(\rho)\leq\rho^{1+\eta_2}
\end{equation} 
for every $\rho\in(0,\rho_0]$.
\end{lemma} 
 
\begin{proof} 
Choose 
\[
\epsilon_1
=
\min\left\{
\frac1{100},
\frac12\left(\frac3{\tilde r}-1\right)
\right\},
\qquad 
\eta=\epsilon_1^2,
\qquad 
\beta=\frac{\eta}{4},
\qquad 
\delta=\frac{\eta^2}{16},
\]
and set 
\[
l=\left\lceil\frac{12}{\epsilon_1}\right\rceil,
\qquad 
\tilde\beta=(1+\beta)^{l+1}-1.
\]
These parameters depend only on $\tilde q$, since 
\[
\frac2{\tilde q}+\frac3{\tilde r}=3.
\]
Moreover, Bernoulli's inequality and
\[
e^x-1
\leq\sum_{j=1}^{\infty}\frac{x^j}{j!}
\leq\sum_{j=1}^{\infty}x^j
=\frac{x}{1-x},
\qquad 0\leq x<1,
\]
give
\[
3\epsilon_1
\leq(l+1)\beta
\leq\tilde\beta
\leq e^{3\epsilon_1+\epsilon_1^2/2}-1
\leq4\epsilon_1.
\]

For an initial radius $\rho_0\in(0,1/2)$ to be chosen below, define
\[
\rho_k=\rho_0^{(1+\beta)^k},
\qquad k\geq0.
\]
We shall prove the stronger discrete estimates
\begin{equation}\label{4.11} 
A(\rho_k)+E(\rho_k)
\leq\frac14\rho_k^{2-\eta},
\qquad 
D(\rho_k)\leq\frac14\rho_k^{1+\eta}.
\end{equation} 
We first derive the inductive estimates, then choose $\rho_0$ and 
initialize the induction.

\par\medskip
\noindent
\textbf{Step 1. Inductive estimates.}\quad
Suppose that \eqref{4.11} holds for all indices
$0\le j\le k$, where $k\ge l$.
We take $\rho_0$ sufficiently small that the smallness
conditions in Corollaries~\ref{corollary1}
and~\ref{corollary2} hold at these scales and
$\rho_0^\beta\le1/2$.
The precise choice will be made in Step~2.

For the energy estimate, take $\gamma=\rho_k^\beta$.
Then $\gamma\rho_k=\rho_{k+1}$, and
Corollary~\ref{corollary1} gives
\begin{align*}
A(\rho_{k+1})+E(\rho_{k+1})
&\le
C_1\rho_k^{2\beta}
[A(\rho_k)+E(\rho_k)]^{1-\delta}
\\
&\quad+
C_1\rho_k^{-2\beta}
[A(\rho_k)+E(\rho_k)]^{3/2-\delta}
\\
&\quad+
C_1\rho_k^{-2\beta}
[A(\rho_k)+E(\rho_k)]^{1/2}D(\rho_k)
\\
&\le
C\left[
\rho_k^{2\beta+(2-\eta)(1-\delta)}
+
\rho_k^{-2\beta+(2-\eta)(3/2-\delta)}
+
\rho_k^{-2\beta+2+\eta/2}
\right].
\end{align*}
Here and below, the constants in the inductive estimates
are independent of $k$, $\rho_0$, and $\varepsilon_2$.
The required strict inequality between exponents is
\begin{equation}\label{4.12}
\begin{split}
\min\Bigl\{
&2\beta+(2-\eta)(1-\delta),\,
-2\beta+(2-\eta)(3/2-\delta),\\
&-2\beta+2+\eta/2
\Bigr\}
>
(2-\eta)(1+\beta).
\end{split}
\end{equation}
Indeed, subtracting the right-hand side from the three
exponents on the left gives, respectively,
\[
\beta\eta-(2-\eta)\delta
=\frac{\eta^2(2+\eta)}{16}>0,
\]
\[
(2-\eta)(1/2-\delta)-(4-\eta)\beta
=
1-\frac32\eta+\frac18\eta^2+\frac1{16}\eta^3>0,
\]
and
\[
\frac32\eta-(4-\eta)\beta
=
\frac12\eta+\frac14\eta^2>0.
\]

For the pressure estimate, we use the delayed scale
$\rho_{k-l}$, as in the iteration scheme of
\cite[Lemma 3.4]{DW2020}.
The definition of $\tilde\beta$ ensures that
\[
\rho_{k-l}^{1+\tilde\beta}=\rho_{k+1}.
\]
Thus we apply Corollary~\ref{corollary2} with
$\rho=\rho_{k-l}$ and
$\gamma=\rho_{k-l}^{\tilde\beta}$.
Since $\tilde\beta>\beta$, this scaling factor is also
at most $1/2$.
Using
$D_{\tilde q,1}(\rho)\le CD(\rho)$ and \eqref{4.11},
we obtain
\begin{equation}\label{eq:D_estimate}
D(\rho_{k+1})
\le
C\left[
\rho_{k-l}^{(1+\epsilon_1)\tilde\beta+1+\eta}
+
\rho_{k-l}^{(1+\epsilon_1)\tilde\beta+1-\eta/2}
+
\rho_{k-l}^{-\tilde\beta+2-\eta}
\right].
\end{equation}
The required pressure exponent condition is
\begin{equation}\label{eq:exponent_condition}
\begin{split}
\min\Bigl\{
&(1+\epsilon_1)\tilde\beta+1+\eta,\,
(1+\epsilon_1)\tilde\beta+1-\eta/2,\\
&-\tilde\beta+2-\eta
\Bigr\}
>
(1+\eta)(1+\tilde\beta).
\end{split}
\end{equation}
The corresponding exponent differences are
\[
(\epsilon_1-\eta)\tilde\beta,
\qquad
(\epsilon_1-\eta)\tilde\beta-\frac32\eta,
\qquad
1-2\eta-(2+\eta)\tilde\beta.
\]
They are all positive. In fact,
\[
(\epsilon_1-\eta)\tilde\beta-\frac32\eta
\ge
\left(\frac32-3\epsilon_1\right)\epsilon_1^2>0,
\]
and
\[
1-2\eta-(2+\eta)\tilde\beta
\ge
1-2\epsilon_1^2
-4\epsilon_1(2+\epsilon_1^2)>0.
\]
Here we used
$3\epsilon_1\le\tilde\beta\le4\epsilon_1$
and $0<\epsilon_1\le1/100$.

The strict exponent inequalities
\eqref{4.12} and \eqref{eq:exponent_condition} therefore
give, for some $\xi>0$ and $C_*>0$ independent of
$k$, $\rho_0$, and $\varepsilon_2$,
\[
A(\rho_{k+1})+E(\rho_{k+1})
\le C_*\rho_{k+1}^{2-\eta+\xi},
\qquad
D(\rho_{k+1})
\le C_*\rho_{k+1}^{1+\eta+\xi}.
\]

\par\medskip
\noindent
\textbf{Step 2. Choice of radius and initialization.}\quad
Fix $\rho_0\in(0,1/2)$ sufficiently small that
\[
\rho_0^\beta\le\frac12,
\qquad
C_*\rho_0^\xi\le\frac14,
\]
and that $\frac14\rho_0^{2-\eta}$ lies below the
smallness thresholds required in
Corollaries~\ref{corollary1} and~\ref{corollary2}.
This fixes all the radii $\rho_k$ before
$\varepsilon_2$ is chosen.

We now initialize \eqref{4.11} for $0\le j\le l$.
By \eqref{4.9}, domain inclusion gives
\[
A(1/2)+E(1/2)\le2\varepsilon_2.
\]
Applying Corollary~\ref{corollary2} at scale one
with $\gamma=1/2$, and taking $\varepsilon_2\le1$,
yields
\[
D(1/2)\le C_0\varepsilon_2^{1/2}.
\]
Consequently, for $0<r\le1/2$, domain inclusion
and comparison of the pressure averages give
\[
A(r)+E(r)\le\frac{\varepsilon_2}{r},
\qquad
D(r)\le\frac{C_0\varepsilon_2^{1/2}}{r},
\]
after enlarging $C_0$ if necessary.

Since $\rho_l\le\rho_j$ for $0\le j\le l$,
it suffices to choose $\varepsilon_2>0$ so small that
\[
\varepsilon_2\le\frac14\rho_l^{3-\eta},
\qquad
C_0\varepsilon_2^{1/2}
\le\frac14\rho_l^{2+\eta}.
\]
These inequalities imply \eqref{4.11} for
$j=0,\ldots,l$.

The estimates in Step~1 and
$C_*\rho_0^\xi\le1/4$ then show that, whenever
\eqref{4.11} holds up to an index $k\ge l$,
\[
A(\rho_{k+1})+E(\rho_{k+1})
\le\frac14\rho_{k+1}^{2-\eta},
\qquad
D(\rho_{k+1})
\le\frac14\rho_{k+1}^{1+\eta}.
\]
Thus \eqref{4.11} holds for every $k\ge0$.
All parameter choices depend only on
$\tilde q$, $M_c$, and $M_\phi$, so the same is true
of the resulting threshold $\varepsilon_2$.

\par\medskip
\noindent
\textbf{Step 3. Passage to arbitrary radii.}\quad
Let $0<r\le\rho_0$, and choose $k\ge0$ such that
\[
\rho_{k+1}<r\le\rho_k.
\]
By the definitions and domain inclusion,
\[
A(r)+E(r)
\le
\frac{\rho_k}{r}
\bigl(A(\rho_k)+E(\rho_k)\bigr)
\le
\frac{\rho_k^{3-\eta}}{4r}.
\]
For the pressure, comparison of spatial averages gives,
for almost every time,
\[
\|p-[p]_r\|_{L^{\tilde r}(B_r)}
\le
2\|p-[p]_{\rho_k}\|_{L^{\tilde r}(B_r)}
\le
2\|p-[p]_{\rho_k}\|_{L^{\tilde r}(B_{\rho_k})}.
\]
Restricting the time interval accordingly, we obtain
\[
D(r)
\le
2\frac{\rho_k}{r}D(\rho_k)
\le
\frac{\rho_k^{2+\eta}}{2r}.
\]
Since
$r>\rho_{k+1}=\rho_k^{1+\beta}$, it follows that
\[
A(r)+E(r)
\le
\frac14 r^{(3-\eta)/(1+\beta)-1}
\le r^{2-\eta_1},
\]
and
\[
D(r)
\le
\frac12 r^{(2+\eta)/(1+\beta)-1}
\le r^{1+\eta_2},
\]
where
\[
\eta_1=\frac{\eta+3\beta}{1+\beta}>0,
\qquad
\eta_2=\frac{\eta-2\beta}{1+\beta}>0.
\]
The positivity of $\eta_2$ follows from
$\beta=\eta/4$.
This proves \eqref{4.10}.
\end{proof}
  	
\subsection{H\"older regularity and completion of the proof}
\label{subsec:holder-regularity}

   With the algebraic decay estimates for the scale-invariant quantities established in Lemma \ref{key}, we are now equipped to prove the full regularity of the suitable weak solution. We proceed to the proof of Proposition \ref{main1}.
   
	   	\begin{proof}[Proof of Proposition \ref{main1}] 

We use the normalized reference cylinders
$\mathcal Q_r=(-r^2,0)\times\mathcal B_r$
introduced in Section~\ref{subsec:notation-scaling}, and write $\mathcal Q=\mathcal Q_1$.
After a spatial translation and a time shift, we take the center
to be $(0,0)$; in the boundary case, the spatial translation
is tangential to $\Gamma$.

We first give the preliminary estimates common to both cases.
We then establish the local boundary estimates, treat the
interior case, and finally complete the boundary case
by a covering argument. For any integrable scalar- or vector-valued function $f$, set
\[
[f]_\rho:=\fint_{\mathcal Q_\rho}f(x,t)\,dx\,dt,
\]
where vector-valued averages are taken componentwise.
For the pressure, we retain the spatial-average notation
\[
[p]_\rho(t):=\fint_{\mathcal B_\rho}p(x,t)\,dx.
\]
Whenever the argument $t$ is omitted, $[p]_\rho$ still denotes
this time-dependent spatial average.
	   		 If the exponent in the statement satisfies $\tilde q\geq2$, set $q_*=3/2$.  Since the time interval has length one,
	   		 \[
	   		 \|p-[p]_1(t)\|_{L_t^{q_*}L_x^1(\mathcal{Q})}\leq
	   		 \|p-[p]_1(t)\|_{L_t^{\tilde q}L_x^1(\mathcal{Q})}.
	   		 \]
	   		 Thus the result for $q_*$ implies the stated result after decreasing the threshold.  It is therefore enough to prove the proposition for $1<\tilde q<2$, which is precisely the range used in the preceding pressure lemmas; $\tilde r\in(3/2,3)$ is then determined by $2/\tilde q+3/\tilde r=3$.
	   		 First, we observe that the smallness condition \eqref{14} in Proposition \ref{main1} implies the initial condition \eqref{4.9} of Lemma \ref{key}. Indeed, by the definitions of the scale-invariant quantities, the total energy at the unit scale satisfies
   		 \begin{align*}
   		 A_{u,\sqrt{n},\nabla \sqrt{\tilde{c}}}(1) + E_{u,\sqrt{n},\nabla \sqrt{\tilde{c}}}(1) + D_{\tilde{q},1}(1) \leq C \left( \varepsilon_0 + \varepsilon_0^{1/\tilde{q}} \right)
   		 \end{align*}
	   		 where $C=C(\tilde q,M_c,M_\phi)>1$. Therefore, by choosing $\varepsilon_0 = \varepsilon_0(\tilde{q},M_c,M_\phi) > 0$ sufficiently small such that $C ( \varepsilon_0 + \varepsilon_0^{1/\tilde{q}} ) \leq \varepsilon_2$, the condition \eqref{4.9} holds.

             We now consider the boundary case, so that
$\mathcal Q_r=Q_r^+$ and $\mathcal B_r=B_r^+$.
   		 Consequently, applying Lemma \ref{key}, we have the following estimates for $\rho \in (0, \rho_0]$:
   		\begin{align}\label{rho-estimate}
   			&\|p-[p]_{\rho}\|_{L_t^{\tilde{q}}L_x^{\tilde{r}}(Q^+_\rho)} \leq \rho^{2+\eta_2}, \quad \||u|^2\|_{L_t^{\tilde{q}}L_x^{\tilde{r}}(Q^+_\rho)} \leq C\rho^{3-\eta_1}, \nonumber\\
&\||\nabla\sqrt{\tilde{c}}|^2\|_{L_t^{\tilde{q}}L_x^{\tilde{r}}(Q^+_\rho)} \leq C\rho^{3-\eta_1}, \quad \|n\|_{L_t^{\tilde{q}}L_x^{\tilde{r}}(Q^+_\rho)} \leq C\rho^{3-\eta_1}.
   		\end{align}
The last three bounds also use Lemma~\ref{interpolation};
their constant $C$ depends only on the fixed parameters and is
independent of $\rho$.
    
    For $\rho\in(0,\rho_0]$ and each $i\in\{1,2,3\}$, put
$f^i=-n\partial_i\phi$ on $Q_\rho^+$.
Throughout the following construction, a superscript $*$,
placed after the component index, denotes the corresponding
function on the full ball or cylinder.

For almost every $t\in(-\rho^2,0)$, let $f^{i,*}$ be the odd
extension of $f^i$ across $\{x_3=0\}$, and let
$h^i\in W^{2,\tilde r}(B_\rho)\cap W_0^{1,\tilde r}(B_\rho)$
solve
\begin{equation}\label{eq:poisson}
    \Delta h^i=f^{i,*}\quad\text{in }B_\rho,
    \qquad h^i=0\quad\text{on }\partial B_\rho.
\end{equation}
Uniqueness and symmetry imply that $h^i$ is odd in $x_3$.
Set
\[
G^{i,*}:=\nabla h^i,
\qquad
G^i:=G^{i,*}|_{Q_\rho^+},
\qquad
G:=(G_j^i)_{i,j=1}^3.
\]
Then $\operatorname{div}G^i=f^i$ in $Q_\rho^+$.
Moreover, $G_1^{i,*}$ and $G_2^{i,*}$ are odd in $x_3$,
whereas $G_3^{i,*}$ is even.
The elliptic estimate, followed by integration in time, gives
\begin{equation}\label{eq:force-divergence-bound}
    \|G\|_{L_t^{\tilde q}L_x^{\tilde r}(Q_\rho^+)}
    \leq C\rho\|\nabla\phi\|_{L^\infty(Q_\rho^+)}
    \|n\|_{L_t^{\tilde q}L_x^{\tilde r}(Q_\rho^+)}
    \leq C\|n\|_{L_t^{\tilde q}L_x^{\tilde r}(Q_\rho^+)}.
\end{equation}

For each $i\in\{1,2,3\}$, define the vector field
$\mathcal F^i=(\mathcal F_j^i)_{j=1}^3$ on $Q_\rho^+$ by
\[
\mathcal F_j^i
=-u^iu^j-\delta_{ij}(p-[p]_\rho)+G_j^i.
\]
Let $\mathcal F^{i,*}$ be its extension to $Q_\rho$ obtained
by extending $\mathcal F_1^i,\mathcal F_2^i$ oddly and
$\mathcal F_3^i$ evenly across $\{x_3=0\}$.
Then $\operatorname{div}\mathcal F^{i,*}$ is odd in $x_3$
in the sense of distributions, and
\[
\|\mathcal F^{i,*}\|_{L_t^{\tilde q}L_x^{\tilde r}(Q_\rho)}
\leq C
\|\mathcal F^i\|_{L_t^{\tilde q}L_x^{\tilde r}(Q_\rho^+)}.
\]

Extend $\mathcal F^{i,*}$ by zero outside $B_\rho$ and define
\[
\omega^{i,*}:=\mathcal T_\rho\mathcal F^{i,*},
\qquad \omega^i:=\omega^{i,*}|_{Q_\rho^+},
\qquad \omega=(\omega^1,\omega^2,\omega^3),
\]
using Lemma~\ref{lem:heat-potentials}.
Then $\omega^{i,*}$ is odd and
\begin{equation}\label{eq:omega}
\partial_t\omega^i-\Delta\omega^i
=-\partial_j(u^iu^j)-n\partial_i\phi-\partial_i(p-[p]_\rho)
\quad\hbox{in }Q_\rho^+.
\end{equation}
In particular, \eqref{eq:heat-potential-div} gives directly
\begin{align}\label{omega}
\|\omega\|_{L_t^{\tilde q}L_x^{\tilde r}(Q_\rho^+)}
&\leq C\rho\left(
 \||u|^2\|_{L_t^{\tilde q}L_x^{\tilde r}(Q_\rho^+)}
 +\|p-[p]_\rho\|_{L_t^{\tilde q}L_x^{\tilde r}(Q_\rho^+)}
 +\|G\|_{L_t^{\tilde q}L_x^{\tilde r}(Q_\rho^+)}\right)\notag\\
&\leq C\rho(\rho^{3-\eta_1}+\rho^{2+\eta_2})
\leq C\rho^{3+\eta_2},
\end{align}
where $\eta_1+\eta_2<1$ and $\rho\leq1$.

Let $u^{i,*}$ be the odd extension of $u^i$, component by
component; no solenoidal property of this extension is required.
Its scalar momentum equation has right-hand side
$\operatorname{div}\mathcal F^{i,*}$ in $Q_\rho$.
Indeed, test the half-cylinder equation with the odd part of a
full-cylinder test function. That part vanishes on $\Gamma$;
the zero trace of $u^i$ removes the diffusion boundary term, and
the prescribed parities give exactly the reflected divergence.
The identity follows for the present weak solutions by approximation
of such zero-trace tests and the local integrability of
$\nabla u,u\otimes u,p,n$.
Thus $v^{i,*}=u^{i,*}-\omega^{i,*}$ is odd and caloric in $Q_\rho$.
In particular, its smooth interior representative vanishes on
the flat boundary. Writing $v=u-\omega$ and using
Lemma~\ref{lem:caloric-parity-decay}(i) with
$2/\tilde q+3/\tilde r=3$, we have
\begin{equation}\label{v}
\|v\|_{L_t^{\tilde q}L_x^{\tilde r}(Q_\gamma^+)}
\leq C(\gamma/\rho)^4
\|v\|_{L_t^{\tilde q}L_x^{\tilde r}(Q_\rho^+)},
\qquad 0<\gamma<\rho.
\end{equation}
This estimate uses the zero boundary value itself, so no subtraction
of an incompatible constant at $\Gamma$ is involved.

Combining \eqref{omega} and \eqref{v} gives
\[
\|u\|_{L_t^{\tilde q}L_x^{\tilde r}(Q_\gamma^+)}
\leq C(\gamma/\rho)^4
\|u\|_{L_t^{\tilde q}L_x^{\tilde r}(Q_\rho^+)}
+C\rho^{3+\eta_2}.
\]
The norm $\|u\|_{L_t^{\tilde q}L_x^{\tilde r}(Q_{\rho_0}^+)}$
is bounded by \eqref{rho-estimate} and H\"older's inequality.
Lemma~\ref{it1}, since $3+\eta_2<4$, yields
\[
\|u\|_{L_t^{\tilde q}L_x^{\tilde r}(Q_\gamma^+)}
\leq C\gamma^{3+\eta_2}.
\]

The mixed-norm bound for the averaging operator then implies
\begin{equation}\label{eq:u-campanato-decay}
\|u-[u]_\gamma\|_{L^1(Q_\gamma^+)}
\leq C\gamma^2\|u-[u]_\gamma\|_{L_t^{\tilde q}L_x^{\tilde r}(Q_\gamma^+)}
\leq C\gamma^{5+\eta_2}.
\end{equation}

We next estimate $\nabla\tilde c=\nabla c$.
On $Q_\rho^+$ set
\[
g=-n\tilde c+n-u\cdot\nabla\tilde c.
\]
Extend $g$ evenly across $\Gamma$ and by zero outside $B_\rho$;
denote the resulting function by $g^*$. Define
\[
c_1^*=\mathcal S_\rho g^*,\qquad
c_1=c_1^*|_{Q_\rho^+},\qquad c_2=\tilde c-c_1.
\]
These definitions specify the auxiliary functions uniquely, without
requiring zero data on the artificial sphere. The even extension
$\tilde c^*$ satisfies
$(\partial_t-\Delta)\tilde c^*=g^*$ inside $Q_\rho$.
This follows by testing the weak scalar equation with the even
part of a full-cylinder test function. The weak Neumann condition
and $u=0$ on $\Gamma$ remove the boundary flux.
Consequently $c_2^*:=\tilde c^*-c_1^*$ is even and caloric,
and hence $\partial_3c_2^*=0$ on the flat boundary.

By \eqref{rho-estimate}, $1\leq\tilde c\leq M_c$, and
$\nabla\tilde c=2\sqrt{\tilde c}\nabla\sqrt{\tilde c}$,
\[
\begin{split}
\|g\|_{L_t^{\tilde q}L_x^{\tilde r}(Q_\rho^+)}
\leq C(M_c)\left(
\|n\|_{L_t^{\tilde q}L_x^{\tilde r}(Q_\rho^+)}
 +\||u|^2\|_{L_t^{\tilde q}L_x^{\tilde r}(Q_\rho^+)}^{1/2}
  \||\nabla\sqrt{\tilde c}|^2\|_{L_t^{\tilde q}L_x^{\tilde r}(Q_\rho^+)}^{1/2}
\right)\leq C\rho^{3-\eta_1}.
\end{split}
\]
The explicit gradient estimate \eqref{eq:heat-potential-scalar}
therefore gives
\[
\|\nabla c_1\|_{L_t^{\tilde q}L_x^{\tilde r}(Q_\rho^+)}
\leq C\rho\|g\|_{L_t^{\tilde q}L_x^{\tilde r}(Q_\rho^+)}
\leq C\rho^{4-\eta_1}.
\]

To preserve the boundary symmetry when subtracting constants, set
\[
\Phi_c(r)=\inf_{a\in\mathbb R^3,\ a_3=0}
 \|\nabla\tilde c-a\|_{L_t^{\tilde q}L_x^{\tilde r}(Q_r^+)},
\qquad
\Phi_{c_2}(r)=\inf_{a\in\mathbb R^3,\ a_3=0}
 \|\nabla c_2-a\|_{L_t^{\tilde q}L_x^{\tilde r}(Q_r^+)}.
\]
Lemma~\ref{lem:caloric-parity-decay}(ii) gives
\begin{equation}\label{eq:c2-gradient-decay}
\Phi_{c_2}(\gamma)
\leq C(\gamma/\rho)^4\Phi_{c_2}(\rho),
\qquad 0<\gamma<\rho.
\end{equation}
The triangle inequality, taking the infimum over $a_3=0$,
implies
\[
\Phi_c(\gamma)\leq\Phi_{c_2}(\gamma)+C\rho^{4-\eta_1},
\qquad
\Phi_{c_2}(\rho)\leq\Phi_c(\rho)+C\rho^{4-\eta_1}.
\]
Hence
\begin{equation}\label{eq:grad-c-campanato-iteration}
\Phi_c(\gamma)
\leq C(\gamma/\rho)^4\Phi_c(\rho)+C\rho^{4-\eta_1}.
\end{equation}
The initial value $\Phi_c(\rho_0)$ is bounded by taking $a=0$.
Since $4-\eta_1<4$, Lemma~\ref{it1} yields
$\Phi_c(\gamma)\leq C\gamma^{4-\eta_1}$.
For every admissible constant $a$ the ordinary averaging estimate gives
$\|\nabla\tilde c-[\nabla\tilde c]_\gamma\|
\leq2\|\nabla\tilde c-a\|$ in the same mixed norm.
Thus
\begin{equation}\label{eq:grad-c-mixed-campanato}
\|\nabla\tilde c-[\nabla\tilde c]_\gamma\|_
 {L_t^{\tilde q}L_x^{\tilde r}(Q_\gamma^+)}
\leq2\Phi_c(\gamma)\leq C\gamma^{4-\eta_1},
\qquad 0<\gamma\leq\rho_0.
\end{equation}

Since
\[
	\frac{2}{\tilde q}+\frac{3}{\tilde r}=3,
\]
H\"older's inequality yields
\begin{align}
	\|\nabla\tilde c-[\nabla\tilde c]_{\gamma}\|_
	{L^1(Q_\gamma^+)}
	&\leq
	C\gamma^2
	\|\nabla\tilde c-[\nabla\tilde c]_{\gamma}\|_
	{L_t^{\tilde q}L_x^{\tilde r}(Q_\gamma^+)}
	\nonumber\\
	&\leq C\gamma^{6-\eta_1}.
	\label{ccc}
\end{align}

We justify the uniformity in the center before applying
Campanato's characterization.
Recall that the reference cylinder in the present
boundary case is $Q_1^+(0,0)$.
For every boundary center
\[
z_b=((x_1,x_2,0),t),
\qquad |(x_1,x_2)|<1/8,\quad -1/64<t<0,
\]
we have
\[
Q_{1/2}^+(z_b)\subset Q_1^+(0,0).
\]
Write
\[
X(z,r):=
A_{u,\sqrt n,\nabla\sqrt{\tilde c}}(z,r)
+
E_{u,\sqrt n,\nabla\sqrt{\tilde c}}(z,r).
\]
By restriction to $Q_{1/2}^+(z_b)$ and comparison
of the spatial pressure averages, hypothesis \eqref{14} gives
\[
X(z_b,1/2)+D_{\tilde q,1}(z_b,1/2)
\leq C\bigl(\varepsilon_0+\varepsilon_0^{1/\tilde q}\bigr),
\]
with $C$ independent of $z_b$.
Choosing $\varepsilon_0$ sufficiently small, we may
therefore rescale each $Q_{1/2}^+(z_b)$ to a unit
half-cylinder and apply Lemma~\ref{key}, with the same
parameters and the same choice of $\rho_0$ for all
these centers.

Returning to the present coordinates, we obtain
\[
X(z_b,r)\leq Cr^{2-\eta_1},
\qquad
D(z_b,r)\leq Cr^{1+\eta_2},
\qquad 0<r\leq \rho_0/2,
\]
uniformly in $z_b$.

Repeating the preceding arguments gives, uniformly for the same
boundary centers and radii,
\[
\|u\|_{L_t^{\tilde q}L_x^{\tilde r}(Q_r^+(z_b))}
\leq Cr^{3+\eta_2},\qquad
\inf_{a_3=0}\|\nabla\tilde c-a\|_
 {L_t^{\tilde q}L_x^{\tilde r}(Q_r^+(z_b))}
\leq Cr^{4-\eta_1}.
\]
In particular, \eqref{eq:u-campanato-decay} and \eqref{ccc}
hold on these translated half-cylinders with common constants.

Fix
\[
0<\rho\leq \min\{1/8,\rho_0/6\},
\]
to be decreased further below if necessary.
For an interior center $z=(x,t)\in Q_\rho^+(0,0)$,
put
\[
d=x_3,\qquad s=d/2,\qquad
z_b=((x_1,x_2,0),t).
\]
Here $d$ is the distance to the flat boundary,
so $Q_s(z)$ is a full cylinder contained in the
half-space. Moreover,
\[
Q_s(z)\subset Q_{3s}^+(z_b),
\qquad
3s<3\rho/2\leq \rho_0/4.
\]
Thus the uniform boundary estimates apply at
the center $z_b$ and radius $3s$.
Comparison of the energy quantities and the
spatial pressure averages gives
\[
X(z,s)\leq Ks^{2-\eta_1},
\qquad
D(z,s)\leq Ks^{1+\eta_2},
\]
where $K$ is independent of the choice of
$z\in Q_\rho^+(0,0)$.

For $0<r\leq s$, all cylinders $Q_r(z)$ are full
interior cylinders.
Corollary~\ref{corollary1} applies whenever its
smallness condition holds, while the interior
pressure estimate in Lemma~\ref{p} gives
\[
D(z,\theta r)
\leq C\theta^{3/\tilde r}D(z,r)
+C\theta^{-1}X(z,r).
\]
The boundary comparison provides decay only at the
starting radius $s$.
To propagate this decay to $r<s$ with constants
independent of $s=d/2$, we repeat the induction
argument of Lemma~\ref{key}, retaining the above
powers of $s$ in the initial bounds.
Choose $\sigma>0$ sufficiently small that
\[
2\sigma<\eta_2,\qquad
\sigma(2+\eta_2)<1-\eta_1-\eta_2,
\]
and
\[
\sigma(4-\eta_1)
<
\min\left\{
\frac{2-\eta_1}{4},
\frac{\eta_1}{2}+\eta_2
\right\}.
\]
Then choose
\[
0<\delta<
\min\left\{
\frac14,
\frac{\sigma\eta_1}{2(2-\eta_1)}
\right\}.
\]
These choices are independent of the center $z$.
Since
\[
\eta_2<3/\tilde r-1,
\qquad \eta_1+\eta_2<1,
\]
the scale induction used in Lemma~\ref{key}, now with
$s_k=s^{(1+\sigma)^k}$ and the above interior pressure
estimate, yields
\[
X(z,s_k)\leq Ks_k^{2-\eta_1},
\qquad D(z,s_k)\leq Ks_k^{1+\eta_2}.
\]
Decrease $\rho$ so that the initial energy bounds
lie below the threshold in Corollary~\ref{corollary1},
$s^\sigma\leq 1/2$, and the positive powers of $s_k$
absorb the fixed constants in the induction.
This choice is independent of the center $z$.

Comparison of consecutive scales then gives, for $r\leq s$,
\[
X(z,r)\leq Cr^{2-\widehat\eta_1},
\qquad D(z,r)\leq Cr^{1+\widehat\eta_2},
\]
where
\[
\widehat\eta_1=\frac{\eta_1+3\sigma}{1+\sigma},
\qquad
\widehat\eta_2=\frac{\eta_2-2\sigma}{1+\sigma}.
\]
These exponents satisfy
$0<\widehat\eta_2<\eta_2$,
$\eta_1<\widehat\eta_1<1$, and
$\widehat\eta_1+\widehat\eta_2<1$.

For the interior cylinders $Q_r(z)$, $0<r\leq s$, use the same
heat potentials, with the data extended by zero outside the spatial
ball and without any reflection. For the force, solve the full-ball
Poisson problem in \eqref{eq:poisson} without taking an odd extension.
The remainders $v$ and $c_2$ are caloric in the full cylinder.
Define
\[
\Psi_f(z,r)=\left\|f-\fint_{Q_r(z)}f\right\|_
 {L_t^{\tilde q}L_x^{\tilde r}(Q_r(z))}.
\]
The just-proved energy and pressure decay, Lemma~\ref{lem:heat-potentials},
and Lemma~\ref{lem:caloric-parity-decay}(iii), applied also to
the caloric vector $\nabla c_2$, imply
\[
\Psi_f(z,\gamma)
\leq C(\gamma/r)^4\Psi_f(z,r)+Cr^{3+\alpha},
\qquad 0<\gamma<r\leq s,
\]
where $(f,\alpha)=(u,\widehat\eta_2)$ or
$(\nabla\tilde c,1-\widehat\eta_1)$.
The initial bound at the moving radius $s=d/2$ is uniform:
restriction from $Q_{3s}^+(z_b)$ and the boundary mixed-norm estimates
give
\[
\Psi_f(z,s)\leq Cs^{3+\alpha_b},\qquad
\alpha_b=\eta_2\ \hbox{for }f=u,
\quad \alpha_b=1-\eta_1\ \hbox{for }f=\nabla\tilde c.
\]
In both cases $\alpha_b>\alpha$. Applying Lemma~\ref{it1} with
starting radius $s$ therefore yields
\[
\begin{split}
\Psi_f(z,r)
&\leq C(r/s)^{3+\alpha}\Psi_f(z,s)+Cr^{3+\alpha}\\
&\leq C(1+s^{\alpha_b-\alpha})r^{3+\alpha}
\leq Cr^{3+\alpha},\qquad 0<r\leq s<1.
\end{split}
\]
This explicitly proves that the oscillation constants do not
deteriorate as $d\downarrow0$.

For $r\geq d/2$, the same bound on
\[
\widehat Q_r(z)
:=(B_r(x)\cap\mathbb H^3)\times(t-r^2,t)
\]
follows from
$\widehat Q_r(z)\subset Q_{3r}^+(z_b)$
and the boundary oscillation estimates.
Fix $0<\rho_1<\rho/8$. On $Q_{\rho_1}^+$,
H\"older's inequality gives
\[
\int_{\widehat Q_r(z)}
\left|f-\fint_{\widehat Q_r(z)}f\right|\,dx\,dt
\leq Cr^{5+\alpha},
\]
uniformly in the center and in sufficiently small radii.
The parabolic Campanato characterization implies
\[
u\in C^{\widehat\eta_2,\widehat\eta_2/2}
(\overline{Q_{\rho_1}^+}),
\qquad
\nabla c=\nabla\tilde c
\in C^{1-\widehat\eta_1,(1-\widehat\eta_1)/2}
(\overline{Q_{\rho_1}^+}).
\]
The chosen $\rho_1$ depends only on the fixed parameters,
and the preceding estimates hold uniformly for
centers in $Q_{\rho_1}^+$ and all sufficiently
small radii.
For the remainder of the proof, we retain the notation
$\eta_1,\eta_2$ for the weaker exponents
$\widehat\eta_1,\widehat\eta_2$. Thus we have 
   		 \begin{align}\label{yj}
   		 \|u\|_{L^\infty_{t,x}(Q^+_{\rho_1})} + \|\nabla\tilde{c}\|_{L^\infty_{t,x}(Q^+_{\rho_1})} \leq C.
   		 \end{align}
   		We next estimate $n$.
Recall that \eqref{yj} gives uniform bounds for
$u$ and $\nabla\tilde c$ on $Q_{\rho_1}^+$.
Here and below, $\eta_1$ denotes the exponent retained
after the preceding interior estimates.

For $0<\rho\leq\rho_1/2$, put
$F=-n(u+\nabla\tilde c)$ on $Q_\rho^+$.
Extend $F_1,F_2$ evenly and $F_3$ oddly across $\Gamma$, then
extend the resulting field $F^*$ by zero outside $B_\rho$.
Set
\[
n_1^*=\mathcal T_\rho F^*,\qquad
n_1=n_1^*|_{Q_\rho^+},\qquad n_2=n-n_1.
\]
The even extension $n^*$ satisfies
$(\partial_t-\Delta)n^*=\operatorname{div}F^*$ in $Q_\rho$:
this is exactly the weak no-flux identity tested with even functions.
One may move the derivative of $n$ onto the test function because
its even part has zero normal derivative on $\Gamma$.
Thus $n_2^*:=n^*-n_1^*$ is even and caloric. This construction
uses no classical normal trace of $n$ or $F$ and imposes no
artificial spherical boundary data.
By \eqref{yj}, the density bound following from the energy decay,
and \eqref{eq:heat-potential-div},
\[
\begin{split}
\|n_1\|_{L_t^{\tilde q}L_x^{\tilde r}(Q_\rho^+)}
\leq C\rho\|n(u+\nabla\tilde c)\|_{L_t^{\tilde q}L_x^{\tilde r}(Q_\rho^+)}
\leq C\rho\|n\|_{L_t^{\tilde q}L_x^{\tilde r}(Q_\rho^+)}
\leq C\rho^{4-\eta_1}.
\end{split}
\]

The even extension of $n_2$ solves the homogeneous
heat equation on the full cylinder $Q_\rho$.
Lemma~\ref{lem:caloric-parity-decay}(iii) therefore implies, for
$0<\gamma\leq\rho/2$,
\[
\|n_2-[n_2]_\gamma\|_
 {L_t^{\tilde q}L_x^{\tilde r}(Q_\gamma^+)}
\leq
C\left(\frac{\gamma}{\rho}\right)^4
\|n_2-[n]_\rho\|_
 {L_t^{\tilde q}L_x^{\tilde r}(Q_\rho^+)}.
\]
Using $n_2=n-n_1$ and the estimate for $n_1$, we obtain
\begin{align*}
\|n-[n]_\gamma\|_
 {L_t^{\tilde q}L_x^{\tilde r}(Q_\gamma^+)}
&\leq
\|n_2-[n_2]_\gamma\|_
 {L_t^{\tilde q}L_x^{\tilde r}(Q_\gamma^+)}
+2\|n_1\|_
 {L_t^{\tilde q}L_x^{\tilde r}(Q_\gamma^+)}\\
&\leq
C\left(\frac{\gamma}{\rho}\right)^4
\|n-[n]_\rho\|_
 {L_t^{\tilde q}L_x^{\tilde r}(Q_\rho^+)}
+C\rho^{4-\eta_1}.
\end{align*}
Applying Lemma~\ref{it1} with the fixed starting radius
$\rho_1/2$, and noting that $4-\eta_1<4$, yields
\[
\|n-[n]_\gamma\|_
 {L_t^{\tilde q}L_x^{\tilde r}(Q_\gamma^+)}
\leq C\gamma^{4-\eta_1},
\qquad 0<\gamma\leq\rho_1/2.
\]
The constant may depend on the fixed radius $\rho_1$,
but is independent of $\gamma$.

Finally, since $2/\tilde q+3/\tilde r=3$,
H\"older's inequality gives
\begin{align}
\|n-[n]_\gamma\|_{L^1(Q_\gamma^+)}
&\leq
C\gamma^2
\|n-[n]_\gamma\|_
 {L_t^{\tilde q}L_x^{\tilde r}(Q_\gamma^+)}
\leq C\gamma^{6-\eta_1}.
\label{nnn}
\end{align}

We spell out the uniformity for $n$ as well. For boundary centers
with $|x'|<\rho_1/8$ and $-\rho_1^2/64<t<0$, fix
$0<R_n\leq\rho_1/16$ small enough that the preceding uniform
energy bounds apply at all radii up to $R_n$. Their larger
cylinders lie in $Q_{\rho_1}^+$, where \eqref{yj} holds.
Repeating the preceding heat-potential argument gives the boundary
mixed-norm oscillation bound $Cr^{4-\eta_1}$ for $r\leq R_n$
with a common constant.
For an interior center $z$ in a sufficiently small fixed
half-cylinder, put $s=d/2$ and $a=\min\{s,R_n/3\}$.
If $s\leq R_n/3$, comparison with the boundary cylinder of radius
$3s\leq R_n$ gives $\Psi_n(z,a)\leq Ca^{4-\eta_1}$.
If $s>R_n/3$, the starting radius $a=R_n/3$ is fixed;
the uniform density norm then gives the same bound after changing
the constant, which may depend on $R_n$ but not on $z$.
On scales $r\leq a$, the full-cylinder heat-potential construction
and Lemma~\ref{lem:caloric-parity-decay}(iii) give
\[
\Psi_n(z,\gamma)
\leq C(\gamma/r)^4\Psi_n(z,r)+Cr^{4-\eta_1}.
\]
Here $\|n\|_{L_t^{\tilde q}L_x^{\tilde r}(Q_r(z))}
\leq Cr^{3-\eta_1}$ follows from the uniform energy decay above.
Lemma~\ref{it1} thus gives $\Psi_n(z,r)\leq Cr^{4-\eta_1}$
with a constant independent of $s$. For $s\leq r\leq R_n/3$,
restriction from the boundary cylinder of radius $3r$ gives the
same bound on the truncated cylinder. These two cases cover every
$0<r\leq R_n/3$. H\"older's inequality and Campanato's characterization
therefore give H\"older continuity of $n$ up to the boundary
on a fixed smaller half-cylinder.

 For an interior reference cylinder $Q_1$, we have
$Q_{1/4}(z)\subset Q_1$ for every $z\in Q_{1/2}$.
After decreasing $\varepsilon_0$ by a fixed factor,
Lemma~\ref{key} and the preceding full-cylinder
estimates give uniform oscillation decay at these
centers. Campanato's characterization therefore yields
H\"older continuity of $(n,\nabla c,u)$ in $Q_{1/2}$.

In the boundary case, cover $\overline{Q_{1/2}^+}$
by the closures of finitely many smaller boundary
half-cylinders and interior full cylinders, with
their larger cylinders contained in $Q_1^+$.
Fix this covering and decrease $\varepsilon_0$
so that, after rescaling, Lemma~\ref{key} applies
on every larger cylinder.
The local boundary and interior results proved above
then yield H\"older continuity of $(n,\nabla c,u)$
on $\overline{Q_{1/2}^+}$.
   	\end{proof}

\section{Lower-order regularity criteria}
\label{sec:lower-order-criteria}

We first prove a lower-order criterion with small density mass
in mixed space--time norms, valid in both full and half-cylinders.
We then use its energy estimates and a caloric decomposition to
prove Proposition~\ref{main2}, which permits an arbitrary finite
density mass bound in the equal-exponent case.

	\begin{theorem}[Small-mass criterion in mixed norms]\label{main3}
	For any triple of exponents $(q,r,\tilde{q})$ satisfying
	\begin{equation*}
		\frac2q+\frac3r<2,\quad 2<q<\infty,\quad \frac32<r<\infty,
	\end{equation*}
		and
		\begin{equation}\label{qq}
			\frac{1}{\tilde q}<F(q,r):=
\begin{cases}
1-\dfrac{1}{q}\dfrac{\frac{1}{r}-\frac{1}{2}}{\frac{1}{r}-\frac{1}{6}},
& 1\le r<2,\\[1.2ex]
1,
& r\ge 2.
\end{cases},
\end{equation}
		 Then
			there exists
			\[
			\varepsilon_1=\varepsilon_1(q,r,\tilde{q},M_c,M_\phi)>0
			\]
			such that if $(n,c,u,p)$ is a suitable weak solution to \eqref{che}--\eqref{bound}
in the sense of Definition~\ref{def:suitable-half-space} satisfying
	\begin{equation}
			C^{q,r}_{u,\nabla\sqrt{\tilde{c}}}(1)+\|n\|_{L^1_{t,x}(\mathcal{Q})}+D_{\tilde{q},1}(1)<\varepsilon_1,
	\end{equation}
		then $(n,\nabla c,u)$ is regular in $\overline{\mathcal Q_{1/2}}$.
\end{theorem}
		
			\begin{proof}[Proof of Theorem \ref{main3}]
				We first prove the boundary case, Define
				\[
				L(q,r):=\max\left\{\frac78,1-\frac1q,\frac12+\frac1q\right\}.
				\]
				Under the hypotheses $2/q+3/r<2$ and $r>3/2$, one checks that
				$L(q,r)<F(q,r)$.  If the pressure exponent in the statement does not
				already satisfy $L(q,r)<1/\tilde q$, choose
				$\tilde q_0\in(1,\tilde q)$ such that
				\[
				L(q,r)<\frac1{\tilde q_0}<F(q,r).
				\]
				Because the time interval has finite length,
				$L_t^{\tilde q}L_x^1\hookrightarrow L_t^{\tilde q_0}L_x^1$ and the
				corresponding pressure smallness is preserved up to a fixed constant.
				We relabel $\tilde q_0$ as $\tilde q$ below.  This justifies all later
				uses of $L(q,r)<1/\tilde q<F(q,r)$ without strengthening the theorem's
				hypothesis.
				Let $\rho_k=1-2^{-k}$, $B_k^+=B^+(\rho_k)$, and $Q_k^+=Q^+(\rho_k)$ for any integer $k\geq 1$. 

                For each $k\geq0$, choose a smooth domain
$\tilde B_k\subset\mathbb H^3$ such that, for some universal
$\vartheta\in(0,1/8)$,
\[
B^+_{\rho_{k+1}+\vartheta 2^{-k}}
\subset
\tilde B_k
\subset
B^+_{\rho_{k+2}},
\]
and such that the flat part of $\partial\tilde B_k$ coincides
with a portion of $\Gamma$.  We set
\[
\tilde Q_k
:=
(-\rho_{k+2}^2,0)\times\tilde B_k.
\]
For each $k\ge1$, choose
\[
\psi_k\in C_c^\infty
\bigl(B_{\rho_{k+1}}\times(-\rho_{k+1}^2,1)\bigr),
\qquad 0\le\psi_k\le1,
\]
such that $\psi_k$ is even in $x_3$ and
\[
\psi_k\equiv1
\quad\text{on }
\overline{B_{\rho_k}}\times[-\rho_k^2,0].
\]
We use its restriction to the closed half-space, still denoted
by $\psi_k$. It satisfies
\[
|\partial_t\psi_k|\le C2^{2k},
\qquad
|\nabla\psi_k|\le C2^k,
\qquad
|\nabla^2\psi_k|\le C2^{2k},
\qquad
\partial_3\psi_k=0\quad\text{on }\Gamma.
\]
Thus $\psi_k$ vanishes near the initial time
$t=-\rho_{k+1}^2$ and near the spherical lateral boundary,
but may be nonzero on the flat boundary and at $t=0$.

Set
\[
J_k:=
\int_{Q_{k+1}^+}
n\bigl(1+|u|+|\nabla c|\bigr)\,dx\,ds.
\]
For almost every $t\in(-\rho_k^2,0)$, testing the cell
equation with $\psi_k$ up to time $t$ gives
\[
\begin{aligned}
\int_{B_k^+}n(x,t)\,dx
&\leq \int_{B^+}n(x,t)\psi_k(x,t)\,dx\\
&=
\int_{-\rho_{k+1}^2}^{t}\int_{B^+}
\left[
n(\partial_s\psi_k+\Delta\psi_k)
+n(u+\nabla c)\cdot\nabla\psi_k
\right]\,dx\,ds\leq C2^{2k}J_k.
\end{aligned}
\]

For $0<\delta<1$, weighted H\"older's inequality yields
\[
\begin{aligned}
\int_{B_k^+}|n\log n|(x,t)\,dx
&\leq
\int_{B^+}n\log n\,\psi_k(x,t)\,dx\\
&\quad+
C_\delta
\left(\int_{B^+}n\psi_k(x,t)\,dx\right)^{1-\delta}
\left(\int_{B^+}\psi_k(x,t)\,dx\right)^\delta\\
&\leq
\int_{B^+}n\log n\,\psi_k(x,t)\,dx
+C_\delta 2^{2k(1-\delta)}J_k^{1-\delta}.
\end{aligned}
\]

We combine these estimates with the local energy inequality,
estimating each term on its right-hand side in absolute value.
We then take essential suprema for the terminal energy terms
and let admissible terminal times increase to zero for the
nonnegative dissipation terms. Using the bounds for $\psi_k$
and discarding the remaining nonnegative terms gives

			\begin{align*}
				&\operatorname{ess\,sup}_{-\rho_k^2\leq t\leq 0}\int_{B_k^+}\left(n+|u|^2+|\nabla\sqrt{\tilde{c}}|^2\right)(\cdot,t)\,dx\\&\qquad\qquad\qquad\qquad+\int_{Q_k^+}\left(|\nabla\sqrt{n}|^2+|\nabla^2\sqrt{\tilde{c}}|^2+|\nabla u|^2\right)dxdt\\
				&\leq C2^{2k}\int_{Q_{k+1}^+}\left(n^{1+\delta}+n^{1-\delta}+|\nabla\sqrt{\tilde{c}}|^2+n+|u|^2\right)dxdt\\
				&\quad +C\|\nabla\phi\|_{L^\infty_{t,x}(Q_{k+1}^+)}\int_{Q_{k+1}^+}|n||u|\,dxdt +C2^k\int_{Q_{k+1}^+}\left(n^{1+\delta}+n^{1-\delta}\right)\left(|u|+|\nabla c|\right)\,dxdt\\
				&\quad +C2^k\int_{Q_{k+1}^+}\left(|n||\nabla c|+|\nabla\sqrt{\tilde{c}}|^2|u|+|u|^3+|p||u|+|\nabla\sqrt{\tilde{c}}|^3\right)dxdt\\
				&\quad +C2^{2k}\int_{Q_{k+1}^+}(|n|+|n||u|+|n||\nabla c|)\,dxdt\\
&\quad +C2^{2k}\left(\int_{Q_{k+1}^+}(|n|+|n||u|+|n||\nabla c|)\,dxdt\right)^{1-\delta}\\
				&:=\text{RHS}.
			\end{align*}
			Since $\rho_k\in[\frac{1}{2},1]$, we define $A_k=A_{u,\sqrt{n},\nabla\sqrt{\tilde{c}}}(\rho_k)$ and $E_k=E_{u,\sqrt{n},\nabla\sqrt{\tilde{c}}}(\rho_k)$. Thus, the above inequality can be rewritten as
			\begin{equation}\label{eq:section5-rhs}
				A_k+E_k\leq \text{RHS}.
			\end{equation}
			Without loss of generality, we replace $p-[p]_1(t)$ with $p$. By hypothesis, there exists $\tau^*>0$ such that $\frac{2}{q}+\frac{3}{r}=2-\tau^*$. 

        Consequently, we choose $\hat r\in(3/2,r)$ and $\tau_1>0$ sufficiently small so that $\frac2q+\frac3{\hat r}=2-\tau_1$, and
\[
0<\tau_1<
\min\left\{
\frac13,\,
\frac1{\hat r},\,
\frac1{\hat r'}-\frac13
\right\}.
\]
Then
\[
\frac1q+\frac2{\hat r}>1,
\qquad
\frac1q+\frac1{\hat r}>\frac23,
\]
and the interpolation exponents introduced below lie in the required
ranges.
			
			We first estimate the integral term involving $|u|^3$:
			\begin{align}
				\int_{Q^+_{k+1}}|u|^3\,dxdt&\leq \left(\|u\|_{L_t^{q_u}L_x^{r_u}(Q^+_{k+1})}^{\theta_u}\|u\|_{L_t^qL_x^{\hat{r}}(Q^+_{k+1})}^{1-\theta_u}\right)^3\nonumber\\
				&\leq C(A_{k+1}+E_{k+1})^\frac{3\theta_u}{2}\|u\|_{L_t^qL_x^{\hat{r}}(Q^+_{k+1})}^{3(1-\theta_u)}\nonumber\\&\leq C(A_{k+1}+E_{k+1})^\frac{3\theta_u}{2}\varepsilon_1^{3(1-\theta_u)},\label{eq:u-L3-estimate}
			\end{align}
			where the parameters satisfy
			\begin{equation*}
				\frac{\theta_u}{q_u}+\frac{1-\theta_u}{q}=\frac{\theta_u}{r_u}+\frac{1-\theta_u}{\hat{r}}=\frac13,\quad \theta_u\in[0,1].
			\end{equation*}
			To ensure the existence of exponents $q_u\in[2,\infty]$ and $r_u\in[2,6]$ satisfying $\frac{2}{q_u}+\frac{3}{r_u}=\frac{3}{2}$ and $\frac{3\theta_u}{2}<1$, we verify this via the relation
			\begin{equation*}
				(1-\theta_u)\left(\frac2q+\frac{3}{\hat{r}}\right)+\theta_u\left(\frac{2}{q_u}+\frac{3}{r_u}\right)=(2-\tau_1)(1-\theta_u)+\frac{3}{2}\theta_u=\frac{5}{3} \implies \theta_u=\frac{\frac13-\tau_1}{\frac12-\tau_1}<\frac23,
			\end{equation*}
			which yields
			\begin{equation*}
				q_u=\frac{\theta_u}{\frac13-\frac{1-\theta_u}{q}};\quad r_u=\frac{\theta_u}{\frac13-\frac{1-\theta_u}{\hat{r}}}.
			\end{equation*}
			Choosing $\hat{r}\in(\frac{3}{2},r)$ such that $\frac1q+\frac{2}{\hat{r}}>1$ and $\frac1q+\frac{1}{\hat{r}}>\frac23$ guarantees that $q_u\in[2,\infty]$ and $r_u\in[2,6]$.
			
			Similarly, the integral term involving $|\nabla\sqrt{\tilde{c}}|^3$ can be bounded as
			\begin{align}\label{eq:grad-c-L3-estimate}
				\int_{Q^+_{k+1}}|\nabla\sqrt{\tilde{c}}|^3\,dxdt\leq C(A_{k+1}+E_{k+1})^\frac{3\theta_u}{2}\varepsilon_1^{3(1-\theta_u)}.
			\end{align}
			Applying H\"{o}lder's inequality, we also have
			\begin{align*}
				\int_{Q_{k+1}^+}|u|^2\,dxdt&\leq C\left(\int_{Q^+_{k+1}}|u|^3\,dxdt\right)^\frac23\leq C\left((A_{k+1}+E_{k+1})^\frac{3\theta_u}{2}\varepsilon_1^{3(1-\theta_u)}\right)^\frac23,\\
				\int_{Q_{k+1}^+}|\nabla\sqrt{\tilde{c}}|^2\,dxdt&\leq C\left(\int_{Q^+_{k+1}}|\nabla\sqrt{\tilde{c}}|^3\,dxdt\right)^\frac23\leq C\left((A_{k+1}+E_{k+1})^\frac{3\theta_u}{2}\varepsilon_1^{3(1-\theta_u)}\right)^\frac23,\\
				\int_{Q_{k+1}^+}|\nabla\sqrt{\tilde{c}}|^2|u|\,dxdt&\leq \left(\int_{Q^+_{k+1}}|u|^3\,dxdt\right)^\frac13\left(\int_{Q^+_{k+1}}|\nabla\sqrt{\tilde{c}}|^3\,dxdt\right)^\frac23\\
				&\leq C(A_{k+1}+E_{k+1})^\frac{3\theta_u}{2}\varepsilon_1^{3(1-\theta_u)}.\alabel{cu23}
			\end{align*}     
			For the terms involving $n$, we obtain
			\begin{align*}
				&\int_{Q_{k+1}^+}|n|\,dxdt\leq \|n\|_{L^1_{t,x}(Q^+)}\leq \varepsilon_1,\\
				&\int_{Q_{k+1}^+}|n|^{1+\delta}\,dxdt\leq C\left(\|n\|_{L^1_{t,x}(Q_{k+1}^+)}^{\frac{2-3\delta}{2+2\delta}}\|n\|_{L_{t,x}^{5/3}(Q_{k+1}^+)}^{\frac{5\delta}{2+2\delta}}\right)^{1+\delta}\leq C(A_{k+1}+E_{k+1})^\frac{5\delta}{2}\varepsilon_1^{1-\frac{3\delta}{2}},\\
				&\int_{Q_{k+1}^+}|n|^{1-\delta}\,dxdt\leq C\left(\int_{Q_{k+1}^+}|n|\,dxdt\right)^{1-\delta}\leq C\varepsilon_1^{1-\delta}. \alabel{n1+-}
			\end{align*}
		
For the coupling term $|n||u|$, we first apply H\"{o}lder's inequality followed by the interpolation inequality on $Q_{k+1}^+$ to obtain:
\begin{align}
	\int_{Q_{k+1}^+}|n||u|\,dxdt &\leq \|u\|_{L_t^qL_x^{\hat{r}}(Q_{k+1}^+)}\|n\|_{L_t^{q'}L_x^{\hat{r}'}(Q_{k+1}^+)} \nonumber \\
	&\leq \|u\|_{L_t^qL_x^{\hat{r}}(Q_{k+1}^+)}\|n\|_{L_t^{1}L_x^{1}(Q_{k+1}^+)}^{\theta_n}\|n\|_{L_t^{q_n}L_x^{r_n}(Q_{k+1}^+)}^{1-\theta_n}, \label{nu}
\end{align}
where the conjugate exponents and the interpolation parameters satisfy:
\begin{align*}
	\frac{1}{q}+\frac{1}{q'}=1,\quad \frac{1}{\hat{r}}+\frac{1}{\hat{r}'}=1,\quad \frac{2}{q_n}+\frac{3}{r_n}=3,
\end{align*}
and
\begin{align*}
	\frac{1-\theta_n}{q_n}+\theta_n=\frac{1}{q'},\quad \frac{1-\theta_n}{r_n}+\theta_n=\frac{1}{\hat{r}'}.
\end{align*}
Recall from the definition of $\hat{r}$ that $\frac{2}{q} + \frac{3}{\hat{r}} = 2 - \tau_1$. By using the relations of conjugate exponents, this is equivalent to $\frac{2}{q'} + \frac{3}{\hat{r}'} = 3 + \tau_1$. On the other hand, taking a linear combination of the interpolation relations yields:
\begin{align*}
	\frac{2}{q'}+\frac{3}{\hat{r}'} = 2\left(\frac{1-\theta_n}{q_n}+\theta_n\right) + 3\left(\frac{1-\theta_n}{r_n}+\theta_n\right) = 3(1-\theta_n)+5\theta_n.
\end{align*}
Comparing these two expressions leads directly to:
\begin{align*}
	\theta_n=\frac{\tau_1}{2},\quad q_n=\frac{1-\theta_n}{\frac{1}{q'}-\theta_n},\quad \text{and}\quad r_n=\frac{1-\theta_n}{\frac{1}{\hat{r}'}-\theta_n}.
\end{align*}
Since $q\in(2,\infty)$ and $r\in(\frac{3}{2},\infty)$, we have $q'\in(1,2)$ and $\hat{r}'\in(r',3)\subset(1,3)$. To ensure $q_n\in(1,\infty)$ and $r_n\in(1,3)$, we restrict the positive parameter $\tau_1$ to the range:
\begin{equation*}
	0 < \tau_1 < \min\left\{ \frac{2}{q'}, \, \frac{2}{\hat{r}'}, \, \frac{3}{\hat{r}'} - 1 \right\}.
\end{equation*}
Consequently, utilizing the bounds $\|u\|_{L_t^qL_x^{\hat{r}}} \leq C\varepsilon_1$, $\|n\|_{L_t^1L_x^1} \leq \varepsilon_1$, and applying Lemma \ref{interpolation} to bound the $L_t^{q_n}L_x^{r_n}$ norm of $n$ by $A_{k+1}+E_{k+1}$, we arrive at the desired estimate:
\begin{equation}\label{eq:interpolation-n} 
	\int_{Q_{k+1}^+}|n||u|\,dxdt \leq C\varepsilon_1^{1+\frac{\tau_1}{2}}(A_{k+1}+E_{k+1})^{1-\frac{\tau_1}{2}}.
\end{equation}
					
					By a similar argument, we have
					\begin{equation*}
						\int_{Q_{k+1}^+}|n||\nabla c|\,dxdt\leq CM_c\varepsilon_1^{1+\frac{\tau_1}{2}}(A_{k+1}+E_{k+1})^{1-\frac{\tau_1}{2}}.
					\end{equation*}
					
                    For the term involving $n^{1+\delta}$, introduce new exponents
\[
\theta_{n,+}:=\frac{\tau_1-3\delta}{2(1+\delta)},
\]
\[
\frac1{q_{n,+}}
=\frac{\frac1{q'(1+\delta)}-\theta_{n,+}}
{1-\theta_{n,+}},\qquad
\frac1{r_{n,+}}
=\frac{\frac1{\hat r'(1+\delta)}-\theta_{n,+}}
{1-\theta_{n,+}}.
\]
Then $2/q_{n,+}+3/r_{n,+}=3$.
For sufficiently small $\delta>0$, these exponents remain in the
admissible range, and interpolation holds separately in time and space:

\begin{align*}
\int_{Q_{k+1}^+}|n|^{1+\delta}|u|\,dxdt&\leq \|u\|_{L_t^qL_x^{\hat{r}}(Q_{k+1}^+)}\|n\|_{L_t^{q'(1+\delta)}L_x^{\hat{r}'(1+\delta)}(Q_{k+1}^+)}^{1+\delta}\\
&\leq \|u\|_{L_t^qL_x^{\hat{r}}(Q_{k+1}^+)}\left(\|n\|_{L_t^1L_x^1(Q_{k+1}^+)}^{\theta_{n,+}}
\|n\|_{L_t^{q_{n,+}}L_x^{r_{n,+}}(Q_{k+1}^+)}^{1-\theta_{n,+}}\right)^{1+\delta}\\
&=\|u\|_{L_t^qL_x^{\hat{r}}(Q_{k+1}^+)}\|n\|_{L_t^{1}L_x^{1}(Q_{k+1}^+)}^{\frac{\tau_1-3\delta}{2}}\|n\|_{L_t^{q_{n,+}}L_x^{r_{n,+}}(Q_{k+1}^+)}^{1+\frac{5\delta}{2}-\frac{\tau_1}{2}}\\
&\leq C\varepsilon_1^{1+\frac{\tau_1-3\delta}{2}}(A_{k+1}+E_{k+1})^{1+\frac{5\delta}{2}-\frac{\tau_1}{2}},\alabel{nu1+}
\end{align*}
and
\begin{align*}	
\int_{Q_{k+1}^+}|n|^{1-\delta}|u|\,dxdt&\leq \|u\|_{L_t^qL_x^{\hat{r}}(Q_{k+1}^+)}\|n\|_{L_t^{q'(1-\delta)}L_x^{\hat{r}'(1-\delta)}(Q_{k+1}^+)}^{1-\delta}\\
&\leq C\|u\|_{L_t^qL_x^{\hat{r}}(Q_{k+1}^+)}\|n\|_{L_t^{q'}L_x^{\hat{r}'}(Q_{k+1}^+)}^{1-\delta}\\
&\leq C\|u\|_{L_t^qL_x^{\hat{r}}(Q_{k+1}^+)}\left(\|n\|_{L_t^{1}L_x^{1}(Q_{k+1}^+)}^{\theta_n}\|n\|_{L_t^{q_n}L_x^{r_n}(Q_{k+1}^+)}^{1-\theta_n}\right)^{1-\delta}\\
&\leq C\varepsilon_1^{1+\frac{\tau_1}{2}(1-\delta)}(A_{k+1}+E_{k+1})^{(1-\frac{\tau_1}{2})(1-\delta)}.\alabel{nu1-}
\end{align*}
Analogously, we obtain
\begin{align*}
&\int_{Q_{k+1}^+}|n|^{1+\delta}|\nabla c|\,dxdt\leq CM_c^{\frac12}\varepsilon_1^{1+\frac{\tau_1-3\delta}{2}}(A_{k+1}+E_{k+1})^{1+\frac{5\delta}{2}-\frac{\tau_1}{2}},\\
&\int_{Q_{k+1}^+}|n|^{1-\delta}|\nabla c|\,dxdt\leq CM_c^{\frac12}\varepsilon_1^{1+\frac{\tau_1}{2}(1-\delta)}(A_{k+1}+E_{k+1})^{(1-\frac{\tau_1}{2})(1-\delta)}. \alabel{nc1+-}
\end{align*}
Finally, we address the pressure term $|p||u|$. For a suitable pair of exponents $(q', d)$ where $d$ reflects the spatial regularity of the Stokes flow, there exists a unique solution pair $(v_k, p_k) \in W^{1,2}_{q', d}(\tilde{Q}_k) \times W^{0,1}_{q', d}(\tilde{Q}_k)$ to the following initial-boundary value problem:
\begin{equation}\label{p_k}
\begin{cases}	
\partial_t v_k - \Delta v_k + \nabla p_k = -u \cdot \nabla u - n \nabla \phi \quad & \text{in } \tilde{Q}_k, \\
									\nabla \cdot v_k = 0 & \text{in } \tilde{Q}_k, \\
									[p_k]_{\tilde{B}_k}(t) = 0 & \text{for a.e. } t \in [-\rho_{k+2}^2, 0], \\
									v_k = 0 & \text{on } \partial_p \tilde{Q}_k,
								\end{cases}
							\end{equation}
							where $\tilde{B}_k \in \mathbb{R}^3$ is a fixed domain with smooth boundary satisfying $B_{k+1}^{+} \subset \tilde{B}_k \subset B_{k+2}^+$, the $C^2$ norm of $\partial\tilde{B}_k$ is bounded by $C2^k$, and $\tilde{Q}_{k} = (-\rho_{k+2}^2, 0) \times \tilde{B}_k$. 
							
							Setting $w_k = u - v_k$ and $h_k = p - p_k$, we observe that $w_k$ and $h_k$ satisfy:
							\begin{equation}
								\begin{cases}
									\partial_t w_k - \Delta w_k + \nabla h_k = 0 \quad & \text{in } \tilde{Q}_k, \\
									\nabla \cdot w_k = 0 & \text{in } \tilde{Q}_k, \\
									w_k = 0 & \text{on } [-\rho_{k+2}^2, 0) \times (\partial\tilde{B}_k \cap \{x_3 = 0\}).
								\end{cases}
							\end{equation}
							
							Before estimating the pressure term $\int_{Q_{k+1}^+} |u||p_k| \, dxdt$, we first establish the local integrability of the non-linear convective term $u \cdot \nabla u$. By H\"{o}lder's inequality, the $L_t^{q'} L_x^d$-norm of $u \cdot \nabla u$ on the larger domain $Q_{k+2}^+$ is bounded by:
							\begin{equation}
								\|u \cdot \nabla u\|_{L_t^{q'} L_x^d(Q_{k+2}^+)} \le \|\nabla u\|_{L_{t,x}^2(Q_{k+2}^+)} \|u\|_{L_t^{a_1} L_x^{b_1}(Q_{k+2}^+)},
								\label{eq:nonlinear-bound}
							\end{equation}
							where $d$ is chosen such that $\hat{r}' = \frac{3d}{3-d}$. Since $\hat r\in(3/2,\min\{r,3\})$, we have $\hat r'\in(\max\{r',3/2\},3)$. Hence $d=3\hat r'/(3+\hat r')\in(1,3/2)$ and $d<\hat r'$. In addition, the exponents satisfy:
							\begin{equation}\label{frac1}
								\frac{1}{q'} = \frac{1}{2} + \frac{1}{a_1}, \quad \frac{1}{d} = \frac{1}{2} + \frac{1}{b_1}.
							\end{equation}
							
							Direct calculation yields:
							\begin{equation*}
								\frac{3}{\hat{r}'} = \frac{3-d}{d} = \frac{3}{d} - 1 \implies \frac{2}{q'} + \frac{3}{d} = \frac{2}{q'} + \frac{3}{\hat{r}'} + 1 = 4 + \tau_1.
							\end{equation*}
							From \eqref{frac1}, since $1 < q' < 2$ and $1 < d < \frac{3}{2}$, we have:
							\begin{equation*}
								\frac{2}{q'} + \frac{3}{d} = \frac{2}{a_1} + \frac{3}{b_1} + \frac{5}{2} = 4 + \tau_1 \implies \frac{2}{a_1} + \frac{3}{b_1} = \frac{3}{2} + \tau_1,
							\end{equation*}
							with $a_1 \in (2, \infty)$ and $b_1 \in (2, 6)$.  The interpolation must hold separately in time and space.  Choose
							\begin{equation}\label{eq:theta-p-corrected}
								\theta_p=\frac{\tau_1}{\frac12-\tau_1},\qquad 0<\tau_1<\frac14,
							\end{equation}
							and define $\hat a_1,\hat b_1$ by
							\begin{equation}\label{eq:perturbation}
								\frac1{\hat a_1}=
								\frac{\frac1{a_1}-\frac{\theta_p}{q}}{1-\theta_p},
								\qquad
								\frac1{\hat b_1}=
								\frac{\frac1{b_1}-\frac{\theta_p}{\hat r}}{1-\theta_p}.
							\end{equation}
							Then
							\[
							\frac1{a_1}=\frac{1-\theta_p}{\hat a_1}+\frac{\theta_p}{q},
							\qquad
							\frac1{b_1}=\frac{1-\theta_p}{\hat b_1}+\frac{\theta_p}{\hat r},
							\]
							and, using $2/q+3/\hat r=2-\tau_1$,
							\[
							\frac2{\hat a_1}+\frac3{\hat b_1}
							=\frac{\frac32+\tau_1-\theta_p(2-\tau_1)}{1-\theta_p}
							=\frac32.
							\]
							After decreasing $\tau_1$ if necessary, $\hat a_1\in(2,\infty)$ and $\hat b_1\in(2,6)$.  Hence the mixed-norm interpolation and Lemma~\ref{interpolation} give
							\begin{equation}\label{eq:u-corrected-interpolation}
								\|u\|_{L_t^{a_1}L_x^{b_1}(Q_{k+2}^+)}
								\leq C\|u\|_{L_t^{\hat a_1}L_x^{\hat b_1}}^{1-\theta_p}
								\|u\|_{L_t^qL_x^{\hat r}}^{\theta_p}
								\leq C\varepsilon_1^{\theta_p}
								(A_{k+2}+E_{k+2})^{\frac{1-\theta_p}{2}}.
							\end{equation}
							Therefore,
							\begin{equation}\label{pk1}
								\|u \cdot \nabla u\|_{L_t^{q'} L_x^{d}(Q_{k+2}^+)} \le C\varepsilon_1^{\theta_p} (A_{k+2} + E_{k+2})^{1 - \frac{\theta_p}{2}}.
							\end{equation}
							
							Similarly to \eqref{eq:interpolation-n} and applying H\"{o}lder's inequality, we obtain:
							\begin{equation}\label{nk2}
								\|n\|_{L_t^{q'} L_x^{d}(Q_{k+2}^+)} \le C \|n\|_{L_t^{q'} L_x^{\hat{r}'}(Q_{k+2}^+)} \le C \varepsilon_1^{\frac{\tau_1}{2}} (A_{k+2} + E_{k+2})^{1 - \frac{\tau_1}{2}}.
							\end{equation}
							
							Next, we consider the linear Stokes system \eqref{p_k}. In view of the maximal regularity estimate in Lemma \ref{stokes}, and noting the domain inclusion $Q_{k+1}^+ \subset \tilde{Q}_k \subset Q_{k+2}^+$, the pressure term $p_k$ satisfies:
							\begin{align*}
								\|p_k\|_{L_t^{q'} W_{d,x}^{0,1}(Q_{k+1}^+)} &\le \|p_k\|_{W_{q', d}^{0,1}(\tilde{Q}_k)} \\
								&\le C 2^{bk} \left( \|u \cdot \nabla u\|_{L_t^{q'} L_x^d(\tilde{Q}_k)} + \|n \nabla \phi\|_{L_t^{q'} L_x^d(\tilde{Q}_k)} \right).\alabel{L3.1}
							\end{align*}
							Since $\tilde{Q}_k \subset Q_{k+2}^+$ and $\nabla \phi \in L^\infty_{t,x}(Q^+)$, the above inequality can be further estimated as:
							\begin{equation}
								\|p_k\|_{L_t^{q'} W_{d,x}^{0,1}(Q_{k+1}^+)} \le C 2^{bk} \left( \|u \cdot \nabla u\|_{L_t^{q'} L_x^d(Q_{k+2}^+)} + \|\nabla \phi\|_{L_{t,x}^\infty(Q^+)} \|n\|_{L_t^{q'} L_x^d(Q_{k+2}^+)} \right).
								\label{eq:pressure-estimate-intermediate}
							\end{equation}
							
							With these preparations, by applying H\"{o}lder's inequality and the Sobolev embedding theorem (where the spatial exponents satisfy $\hat{r}' = \frac{3d}{3-d}$ with $1 < d < \frac{3}{2}$), we obtain the estimate for the target pressure product term on $Q_{k+1}^+$:
							\begin{align*}
								\int_{Q^+_{k+1}} |u||p_k| \, \mathrm{d}x\mathrm{d}t &\le \|u\|_{L_t^q L_x^{\hat{r}}(Q_{k+1}^+)} \|p_k\|_{L_t^{q'} L_x^{\hat{r}'}(Q_{k+1}^+)} \\
								&\le \|u\|_{L_t^q L_x^{\hat{r}}(Q_{k+1}^+)} \|p_k\|_{ W_{q',d}^{0,1}(Q_{k+1}^+)} \\
								&\le C \|u\|_{L_t^q L_x^{\hat{r}}(Q_{k+1}^+)} 2^{bk} \left[ \|u \cdot \nabla u\|_{L_t^{q'} L_x^d(Q_{k+2}^+)} +  \|n\|_{L_t^{q'} L_x^d(Q_{k+2}^+)} \right].
							\end{align*}
							Thus, combining \eqref{pk1}--\eqref{eq:pressure-estimate-intermediate}, we conclude that:
							\begin{align}\label{pkk}
								\int_{Q^+_{k+1}} |u||p_k| \, \mathrm{d}x\mathrm{d}t
								&\le C2^{bk}\left[\varepsilon_1^{1+\theta_p}(A_{k+2}+E_{k+2})^{1-\frac{\theta_p}{2}}+\varepsilon_1^{1+\frac{\tau_1}{2}}(A_{k+2}+E_{k+2})^{1-\frac{\tau_1}{2}}\right].
						\end{align}        
We now estimate the convective coupling term $\int_{Q^+_{k+1}}|u||h_k|\,dxdt$. Applying H\"{o}lder's inequality, we obtain
\begin{align}\label{eq:u-hk-holder}
	\int_{Q^+_{k+1}}|u||h_k|\,dxdt \leq \|u\|_{L_t^{q_2}L_x^{r_2}(Q_{k+1}^+)}^{1-\theta_1}\|u\|_{L_t^qL_x^r(Q_{k+1}^+)}^{\theta_1}\|h_k\|_{L_t^{\tilde{q}}L_x^{\tilde{s}}(Q_{k+1}^+)},
\end{align}
where the exponents satisfy the scaling and interpolation relations
\begin{align}\label{eq:exponent-relations}
	\frac{2}{q_2}+\frac{3}{r_2}=\frac{3}{2},\quad \frac{1}{\tilde{q}}+\frac{1-\theta_1}{q_2}+\frac{\theta_1}{q}=1,\quad \frac{1}{\tilde{s}}+\frac{1-\theta_1}{r_2}+\frac{\theta_1}{r}=1.
\end{align}
Since the parameters $\tilde{q}$, $q$, and $r$ are given, selecting $q_2$ uniquely determines the exponents $r_2$ and $\theta_1$, which subsequently fixes the value of $\tilde{s}$. To apply the interpolation inequalities from Lemma \ref{interpolation}, we require $2\leq q_2\leq \infty$, which is equivalent to
\begin{align*}
	\frac{1}{\tilde{q}}+\frac{\theta_1}{q}\leq 1 \leq \frac{1}{\tilde{q}}+\frac{\theta_1}{q}+\frac{1-\theta_1}{2}.
\end{align*}
	By the preliminary reduction of the pressure exponent, $F(q,r)>\frac{1}{\tilde{q}}>1-\frac{1}{q}$. We choose $\theta_1=q(1-\frac{1}{\tilde{q}})$, so $\theta_1\in (0,1)$.

To apply    Lemma \ref{Seregin} to the pair $(w_k, h_k)$ on the domain $\tilde{Q}_k$, we must verify the regularities $w_k \in W^{0,1}_{\tilde{q}, 1}(\tilde{Q}_k)$ and $h_k \in L_t^{\tilde{q}} L_x^1(\tilde{Q}_k)$:
\begin{enumerate}
	\item \textbf{Energy bounds of $u$:}
	The velocity field satisfies the standard energy bounds:
	\[
	u \in L^\infty_t L^2_x(Q^+) \quad \text{and}\; \nabla u \in L^2_{t,x}(Q^+).
	\]
	According to \eqref{qq}, we may assume $\tilde{q} \in (1, 2)$.  H\"older's inequality on the bounded cylinder implies that $u \in W^{0,1}_{\tilde{q}, 1}(\tilde{Q}_k)$.
	
	\item \textbf{Regularity of the auxiliary velocity $v_k$:}
	The auxiliary velocity field $v_k$ is obtained from Lemma \ref{stokes} with $v_k \in W^{1,2}_{q', d}(\tilde{Q}_k)$. Utilizing the assumption $\frac{1}{\tilde{q}}>1-\frac{1}{q}$ combined with H\"older's inequality yields $v_k \in W^{0,1}_{\tilde{q}, 1}(\tilde{Q}_k)$. Combining these estimates, we obtain $w_k = u - v_k \in W^{0,1}_{\tilde{q}, 1}(\tilde{Q}_k)$.
	
	\item \textbf{Pressure bounds of $p$:}
By the pressure smallness assumption, after the normalization
$p\mapsto p-[p]_1(t)$ made above, we have
\[
\|p\|_{L_t^{\tilde q}L_x^1(Q^+)}\le\varepsilon_1.
\]
Since $\tilde Q_k\subset Q^+$, it follows that
\[
\|p\|_{L_t^{\tilde q}L_x^1(\tilde Q_k)}
\le
\|p\|_{L_t^{\tilde q}L_x^1(Q^+)}
\le\varepsilon_1.
\]
In particular, $p\in L_t^{\tilde q}L_x^1(\tilde Q_k)$.
	
	\item \textbf{Regularity of the auxiliary pressure $p_k$:}
	From \eqref{L3.1}, we have $p_k \in W^{0,1}_{q', d}(\tilde{Q}_k)$, which guarantees that $\nabla p_k \in L^t_{q'} L^x_{d  }(\tilde{Q}_k)$. Since $p_k$ satisfies the zero-mean condition $[p_k]_{\tilde{B}_k}(t) = 0$, the Sobolev--Poincar\'e inequality implies $p_k \in L^t_{q'} L^x_{d}(\tilde{Q}_k)$. Consequently, we have $h_k = p - p_k \in L^t_{\tilde{q}} L^x_1(\tilde{Q}_k)$.
\end{enumerate}
Since $(w_k,h_k)$ solves the homogeneous Stokes system in
$\tilde Q_k$, it also solves it in the standard half-cylinder
\[
Q^+\bigl(\rho_{k+1}+\vartheta2^{-k}\bigr)
\Subset \tilde Q_k.
\]
Applying the rescaled version of Lemma~\ref{Seregin} from
$Q^+(\rho_{k+1}+\vartheta2^{-k})$ to $Q^+_{\rho_{k+1}}$ yields
\begin{align*}
\|h_k\|_{L_t^{\tilde q}L_x^{\tilde s}(Q_{k+1}^+)}
&\leq
C2^{bk}
\left(
\|w_k\|_{L_t^{\tilde q}L_x^1(\tilde Q_k)}
+
\|h_k\|_{L_t^{\tilde q}L_x^1(\tilde Q_k)}
\right)\\
&\leq C2^{bk}\left[\||u|+|p|\|_{L_t^{\tilde{q}}L_x^{1}(\tilde{Q}_k)}+\||v_k|+|p_k|\|_{L_t^{\tilde{q}}L_x^{1}(\tilde{Q}_k)}\right].
\end{align*}
Applying H\"{o}lder's inequality, we obtain
\begin{align*}
	\||u|+|p|\|_{L_t^{\tilde{q}}L_x^{1}(\tilde{Q}_k)}\leq \|u\|_{L_t^{q}L_x^{r}(\tilde{Q}_k)}+\|p\|_{L_t^{\tilde{q}}L_x^{1}(\tilde{Q}_k)}\leq \varepsilon_1.
\end{align*}
For any $\hat{s}>1$, the maximal regularity of the Stokes system (Lemma \ref{stokes}) yields
\begin{align*}
	\||v_k|+|p_k|\|_{L_t^{\tilde{q}}L_x^{1}(\tilde{Q}_k)}&\leq C\||v_k|+|p_k|\|_{L_t^{\tilde{q}}L_x^{\hat{s}}({\tilde{Q}}_{k})}\\
	&\leq C2^{bk}\left[\|u\cdot\nabla u\|_{L_t^{\tilde{q}}L_x^{\hat{s}}({\tilde{Q}}_{k})}+\|\nabla\phi\|_{L^\infty_{t,x}(Q^+)}\|n\|_{L_t^{\tilde{q}}L_x^{\hat{s}}({\tilde{Q}}_{k})}\right].
\end{align*}

We next verify that the external force and convective terms satisfy the necessary integrability conditions. For any $1<\hat{s}<\tilde{q}<\frac53$, H\"{o}lder's inequality gives
\begin{align*}
	\|n\|_{L_t^{\tilde{q}}L_x^{\hat{s}}({\tilde{Q}}_{k})}\leq \|n\|_{L_{t,x}^{\tilde{q}}({\tilde{Q}}_{k})}\leq  \|n\|_{L^\frac53_{t,x}({\tilde{Q}}_{k})}^{\theta_2}\|n\|_{L^1_{t,x}({\tilde{Q}}_{k})}^{1-\theta_2},
\end{align*}
where $\theta_2 \in (0,1)$ is determined by the relation
\begin{align*}
	1-\theta_2+\frac{3}{5}\theta_2=\frac{1}{\tilde{q}}\implies \theta_2=\frac{5}{2}\left(1-\frac{1}{\tilde{q}}\right).
\end{align*}
This yields the bound
\begin{align*}
	\|n\|_{L_t^{\tilde{q}}L_x^{\hat{s}}({\tilde{Q}}_{k})}\leq C(A_{k+2}+E_{k+2})^{\theta_2}\varepsilon_1^{1-\theta_2}.
\end{align*}
We choose $\frac{1}{\tilde{q}}>\frac45$ to guarantee that $\frac12-\frac{\theta_1}{2}+\theta_2<1$.

For the convective term $u\cdot\nabla u$, H\"{o}lder's inequality yields
\begin{align*}
	\|u\cdot\nabla u\|_{L_t^{\tilde{q}}L_x^{\hat{s}}({\tilde{Q}}_{k})}\leq \|\nabla u\|_{L^2_{t,x}(\tilde{Q}_{k})}\|u\|_{L_t^qL_x^r(\tilde{Q}_{k})}^{1-\theta_3}\|u\|_{L_t^{q_3}L_x^{r_3}(\tilde{Q}_{k})}^{\theta_3},
\end{align*}
where the exponents satisfy $2\leq q_3\leq \infty$, $\frac{2}{q_3}+\frac{3}{r_3}=\frac{3}{2}$, and
\begin{align*}
	\frac{1}{2}+\frac{1-\theta_3}{q}+\frac{\theta_3}{q_3}\leq \frac{1}{\tilde{q}},\quad \frac{1}{2}+\frac{1-\theta_3}{r}+\frac{\theta_3}{r_3}\leq \frac{1}{\hat{s}}. 
\end{align*}
To ensure $\hat{s}>1$, we choose a suitable exponent $q_3$ and its associated $r_3$ such that
\begin{align}\label{theta3}
	\frac{1-\theta_3}{r}+\frac{\theta_3}{r_3}< \frac12.
\end{align}
To optimize the admissible range of $\tilde{q}$, we minimize the combination $1+\theta_3-\theta_1$ by choosing the smallest possible $\theta_3$ and the largest possible $\theta_1$, namely, $\theta_1=q(1-\frac{1}{\tilde{q}})$. We analyze the choice of $\theta_3$ by considering two cases:

\textit{Case i):} If $r>2$, we select $\theta_3=0$, which implies $\frac{1}{\tilde{q}}>\frac12+\frac1q$. Under this condition, we obtain $\hat{s}>1$ and
\begin{align*}
	\|u\cdot\nabla u\|_{L_t^{\tilde{q}}L_x^{\hat{s}}({\tilde{Q}}_{k})}\leq \|\nabla u\|_{L^2_{t,x}(\tilde{Q}_{k})}\varepsilon_1\leq C(A_{k+2}+E_{k+2})^\frac12\varepsilon_1.
\end{align*}

\textit{Case ii):} If $r\in(\frac32,2]$, we have $q>4$ since $\frac2q+\frac3r<2$. Assuming $1<\hat{s}\leq \min\{\frac{6}{5},\tilde{q}\}$, we have $(\frac{1}{\hat{s}}-\frac12)^{-1}\in (2,3)$. Choosing $q_3 = 2$ and $r_3 = 6$ to minimize $\theta_3$, we obtain from \eqref{theta3}:
\begin{align*}
	\theta_3>\frac{\frac1r-\frac12}{\frac1r-\frac16}.
\end{align*} 
We set $\frac{1}{\tilde{q}}>\frac78$ to satisfy the inequality above since
\begin{align*}
	\frac{1}{2}+\frac{1-\theta_3}{q}+\frac{\theta_3}{q_3}\leq \frac{1}{2}+\frac{1-\theta_3}{4}+\frac{\theta_3}{2}<\frac78,\quad \text{when}\quad q>4, \,\theta_3<\frac12,
\end{align*} 
and 
\begin{align*}
	0\leq\frac{\frac1r-\frac12}{\frac1r-\frac16}<\frac13, \quad\text{when}\quad r\in\left(\frac32,2\right].
\end{align*} 

Since $1/\tilde q<F(q,r)$, the lower bound above is also
strictly less than $\theta_1=q(1-1/\tilde q)$. We therefore choose
\[
\frac{\frac1r-\frac12}{\frac1r-\frac16}
<\theta_3<\min\left\{\frac12,\theta_1\right\}.
\]
This simultaneous choice ensures both the required mixed-norm
estimate and $1+\theta_3/2-\theta_1/2<1$, as needed for Young's
inequality in the energy iteration.

Under these conditions, the estimate for $h_k$ becomes
\begin{align*}
	\|h_k\|_{L_t^{\tilde{q}}L_x^{\tilde{s}}(Q_{k+1}^+)}\leq C2^{2bk}\Big(\varepsilon_1+&\|\nabla\phi\|_{L^\infty_{t,x}(Q^+)}(A_{k+2}+E_{k+2})^{\theta_2}\varepsilon_1^{1-\theta_2}\\
	&+(A_{k+2}+E_{k+2})^{\frac12+\frac12\theta_3}{\varepsilon_1}^{1-\theta_3}\Big).
\end{align*}
 Combining the exponents yields
\begin{align}\label{hkk}
	\int_{Q^+_{k+1}}|u||h_k|\,dxdt &\leq C2^{2bk}\Big(\varepsilon_1+(A_{k+2}+E_{k+2})^{\theta_2}\varepsilon_1^{1-\theta_2}\nonumber\\
	&\quad +(A_{k+2}+E_{k+2})^{\frac12+\frac12\theta_3}{\varepsilon_1}^{1-\theta_3}\Big)\|u\|_{L_t^{q_2}L_x^{r_2}(Q_{k+1}^+)}^{1-\theta_1}\|u\|_{L_t^qL_x^r(Q_{k+1}^+)}^{\theta_1}\nonumber\\
	&\leq C2^{2bk}\Big(\varepsilon_1^{1+\theta_1}(A_{k+2}+E_{k+2})^{\frac{1-\theta_1}{2}}+(A_{k+2}+E_{k+2})^{\frac12-\frac12\theta_1+\theta_2}\varepsilon_1^{1+\theta_1-\theta_2}\nonumber\\
	&\quad +(A_{k+2}+E_{k+2})^{1+\frac12\theta_3-\frac12\theta_1}{\varepsilon_1}^{1-\theta_3+\theta_1}\Big).
\end{align}
Finally, substituting \eqref{eq:u-L3-estimate}--\eqref{nc1+-}, \eqref{pkk} and \eqref{hkk} back into \eqref{eq:section5-rhs}, and assuming $\max\left\{\frac78, 1-\frac{1}{q}, \frac12+\frac1q\right\} < \frac{1}{\tilde{q}} < F(q,r)$, $\min\{\frac{2}{q'}, \frac{1}{\hat{r}'},\frac{1}{\hat{r}'} -\frac13 \}>\tau_1>5\delta$, we obtain
\begin{align*}
	A_k + E_k &\le \Big\{C(A_{k+1}+E_{k+1})^\frac{3\theta_u}{2}\varepsilon_1^{3(1-\theta_u)}+C(A_{k+1}+E_{k+1})^{\theta_u}\varepsilon_1^{2(1-\theta_u)}\\&\quad+\varepsilon_1+C(A_{k+1}+E_{k+1})^\frac{5\delta}{2}\varepsilon_1^{1-\frac{3\delta}{2}}+C\varepsilon_1^{1-\delta}\\&\quad+C\varepsilon_1^{1+\frac{\tau_1}{2}}(A_{k+1}+E_{k+1})^{1-\frac{\tau_1}{2}}+C\varepsilon_1^{1+\frac{\tau_1-3\delta}{2}}(A_{k+1}+E_{k+1})^{1+\frac{5\delta}{2}-\frac{\tau_1}{2}}\\&\quad+C\varepsilon_1^{(1+\frac{\tau_1}{2})(1-\delta)}(A_{k+1}+E_{k+1})^{(1-\frac{\tau_1}{2})(1-\delta)}\\&\quad+C2^{bk}\left[\varepsilon_1^{1+\theta_p}(A_{k+2}+E_{k+2})^{1-\frac{\theta_p}{2}}+\varepsilon_1^{1+\frac{\tau_1}{2}}(A_{k+2}+E_{k+2})^{1-\frac{\tau_1}{2}}\right]\\&\quad+C2^{2bk}\Big[\varepsilon_1^{1+\theta_1}(A_{k+2}+E_{k+2})^{\frac{1-\theta_1}{2}}+(A_{k+2}+E_{k+2})^{\frac12-\frac12\theta_1+\theta_2}\varepsilon_1^{1+\theta_1-\theta_2}\\
	&\qquad +(A_{k+2}+E_{k+2})^{1+\frac12\theta_3-\frac12\theta_1}{\varepsilon_1}^{1-\theta_3+\theta_1}\Big]\Big\}\times 2^{2k}.
\end{align*}
By Young's inequality, there exist
$B=B(q,r,\tilde q)>0$ and
$\beta=\beta(q,r,\tilde q)\in(0,1)$ such that,
for every $\delta_1>0$,
\begin{equation}
\label{eq:section5-recursion}
A_k+E_k\leq\delta_1^2(A_{k+2}+E_{k+2})+\tilde C\,2^{Bk}\varepsilon_1^\beta.
\end{equation}
where
$\tilde C=
\tilde C(\delta_1,q,r,\tilde q,M_c,M_\phi)>0$
is independent of $k$ and $\varepsilon_1$.

Choosing $\delta_1 = 3^{-B}$ ensures that the second term on the right-hand side of \eqref{eq:section5-recursion} is summable. Multiplying both sides of \eqref{eq:section5-recursion} by $\delta_1^k$ and summing over $k$ from $1$ to infinity, we obtain
\begin{align*}
	\sum_{k=1}^\infty \delta_1^k(A_k + E_k) &\le \sum_{k=1}^\infty \delta_1^{k+2}(A_{k+2} + E_{k+2}) + \tilde{C} \varepsilon_1^\beta \sum_{k=1}^\infty \left(\frac23\right)^{Bk} \\
	&\le \sum_{k=3}^\infty \delta_1^k(A_k + E_k) + \tilde{C} \varepsilon_1^\beta.
\end{align*}

Since $\rho_k\geq1/2$, domain inclusion gives
\[
\sup_{k\geq1}(A_k+E_k)
\leq
2\bigl(A_{u,\sqrt n,\nabla\sqrt{\tilde c}}(1)
      +E_{u,\sqrt n,\nabla\sqrt{\tilde c}}(1)\bigr)
<\infty.
\]
Thus all weighted series below converge, and cancellation
of their common tails is legitimate.
Since all terms are non-negative, this implies
\[
A_1 + E_1 \le \tilde{C} \varepsilon_1^\beta.
\]

Together with $D_{\tilde q,1}(1)<\varepsilon_1$ and the fixed-ratio
pressure bound $D_{\tilde q,1}(1/2)\leq C D_{\tilde q,1}(1)$,
this permits application of Proposition~\ref{main1} at radius $1/2$.
For sufficiently small $\varepsilon_1$, it first gives regularity in
$\overline{Q_{1/4}^+}$. After proving the analogous interior assertion
below, we extend this conclusion to $\overline{Q_{1/2}^+}$ by a
uniform localization argument.

We record the corresponding estimate of the pressure term when the
center is an interior point. After translation and parabolic rescaling,
we may work in a full cylinder $Q_1$.
Let
\[
B_k=B_{\rho_k},\qquad
Q_k=(-\rho_k^2,0)\times B_k,
\qquad
\rho_k=1-2^{-k}.
\]
Choose a smooth domain $\tilde B_k$ such that
\[
B_{k+1}\Subset\tilde B_k\Subset B_{k+2},
\qquad
\operatorname{dist}
(B_{k+1},\partial\tilde B_k)
\geq c2^{-k},
\]
and set
\[
\tilde Q_k
=
(-\rho_{k+2}^2,0)\times\tilde B_k.
\]

As in the boundary argument, we may replace $p$ by
$p-[p]_{B_1}(t)$ and still denote it by $p$. Indeed, since
$\nabla\cdot u=0$ and the cutoff functions are compactly supported
in the interior,
\[
\int c(t)u\cdot\nabla\psi_k\,dx=0
\]
for every function $c=c(t)$. Hence
\[
\|p\|_{L_t^{\tilde q}L_x^1(Q_1)}
\leq\varepsilon_1.
\]

Consider the auxiliary Stokes problem
\[
\left\{
\begin{aligned}
\partial_t v_k-\Delta v_k+\nabla p_k
&=-u\cdot\nabla u-n\nabla\phi
&&\text{in }\tilde Q_k,\\
\nabla\cdot v_k&=0
&&\text{in }\tilde Q_k,\\
[p_k]_{\tilde B_k}(t)&=0
&&\text{for a.e. }t,\\
v_k&=0
&&\text{on }\partial_p\tilde Q_k.
\end{aligned}
\right.
\]
The zero boundary condition above is imposed only on the auxiliary
Stokes problem and does not impose any additional boundary condition
on the original velocity $u$.

Set
\[
w_k=u-v_k,\qquad h_k=p-p_k.
\]
Then
\[
\partial_t w_k-\Delta w_k+\nabla h_k=0,
\qquad
\nabla\cdot w_k=0
\qquad\text{in }\tilde Q_k.
\]
Taking the divergence of this equation gives
\[
\Delta h_k(\cdot,t)=0
\qquad\text{in }\tilde B_k
\]
for almost every $t\in(-\rho_{k+2}^2,0)$.

We first estimate the contribution of $p_k$. The estimates
\eqref{eq:nonlinear-bound}--\eqref{eq:pressure-estimate-intermediate} are unchanged in the full-cylinder setting. In
particular, with $d$ chosen so that
\[
\hat r'=\frac{3d}{3-d},
\]
the maximal regularity of the Stokes system and the Sobolev embedding
$W^{1,d}\hookrightarrow L^{\hat r'}$ give
\[
\begin{aligned}
\|p_k\|_
{L_t^{q'}L_x^{\hat r'}(Q_{k+1})}
\leq C2^{bk}\Big[
\varepsilon_1^{\theta_p}
(A_{k+2}+E_{k+2})^{1-\frac{\theta_p}{2}}
+
\varepsilon_1^{\frac{\tau_1}{2}}
(A_{k+2}+E_{k+2})^{1-\frac{\tau_1}{2}}
\Big].
\end{aligned}
\]
Consequently,
\[
\begin{aligned}
\int_{Q_{k+1}}|u||p_k|\,dxdt
&\leq
\|u\|_{L_t^qL_x^{\hat r}(Q_{k+1})}
\|p_k\|_{L_t^{q'}L_x^{\hat r'}(Q_{k+1})}
\\
&\leq
C2^{bk}\Big[
\varepsilon_1^{1+\theta_p}
(A_{k+2}+E_{k+2})^{1-\frac{\theta_p}{2}}
+
\varepsilon_1^{1+\frac{\tau_1}{2}}
(A_{k+2}+E_{k+2})^{1-\frac{\tau_1}{2}}
\Big].
\end{aligned}
\]
Thus the estimate corresponding to~\eqref{pkk} is the same as in the
boundary case.

The estimate of $h_k$ is simpler. Since $h_k(\cdot,t)$ is harmonic
in $\tilde B_k$ and
$B_{k+1}\Subset\tilde B_k$, the standard interior estimate for
harmonic functions yields, for every fixed
$1<\tilde s<\infty$,
\[
\|h_k(\cdot,t)\|_{L^{\tilde s}(B_{k+1})}
\leq
C2^{bk}
\|h_k(\cdot,t)\|_{L^1(\tilde B_k)}.
\]
Taking the $L^{\tilde q}$ norm in time gives
\[
\|h_k\|_
{L_t^{\tilde q}L_x^{\tilde s}(Q_{k+1})}
\leq
C2^{bk}
\|h_k\|_
{L_t^{\tilde q}L_x^1(\tilde Q_k)}.
\]
Since $h_k=p-p_k$,
\[
\|h_k\|_
{L_t^{\tilde q}L_x^1(\tilde Q_k)}
\leq
\varepsilon_1
+
\|p_k\|_
{L_t^{\tilde q}L_x^1(\tilde Q_k)}.
\]

To estimate the last term, we use the same lower-integrability
Stokes estimate as in the boundary proof. For a suitable
$\hat s>1$,
\[
\begin{aligned}
\|p_k\|_
{L_t^{\tilde q}L_x^1(\tilde Q_k)}
\leq
C2^{bk}\Big(
\|u\cdot\nabla u\|_
{L_t^{\tilde q}L_x^{\hat s}(\tilde Q_k)}+
\|\nabla\phi\|_{L^\infty(Q_1)}
\|n\|_
{L_t^{\tilde q}L_x^{\hat s}(\tilde Q_k)}
\Big).
\end{aligned}
\]
The interpolation estimates used above give
\[
\|n\|_
{L_t^{\tilde q}L_x^{\hat s}(\tilde Q_k)}
\leq
C
(A_{k+2}+E_{k+2})^{\theta_2}
\varepsilon_1^{1-\theta_2},
\]
and
\[
\|u\cdot\nabla u\|_
{L_t^{\tilde q}L_x^{\hat s}(\tilde Q_k)}
\leq
C
(A_{k+2}+E_{k+2})^{\frac{1+\theta_3}{2}}
\varepsilon_1^{1-\theta_3}.
\]
Therefore, after enlarging $b$ if necessary,
\[
\begin{aligned}
\|h_k\|_
{L_t^{\tilde q}L_x^{\tilde s}(Q_{k+1})}
\leq
C2^{2bk}\Big[
\varepsilon_1
+
(A_{k+2}+E_{k+2})^{\theta_2}
\varepsilon_1^{1-\theta_2}
+
(A_{k+2}+E_{k+2})^{\frac{1+\theta_3}{2}}
\varepsilon_1^{1-\theta_3}
\Big].
\end{aligned}
\]

We now use the same exponents
$q_2,r_2,\theta_1,\tilde s$ as in~\eqref{eq:u-hk-holder}--\eqref{eq:exponent-relations}. Thus
\[
\frac2{q_2}+\frac3{r_2}=\frac32,
\]
and
\[
\frac1{\tilde q}
+\frac{1-\theta_1}{q_2}
+\frac{\theta_1}{q}
=1,
\qquad
\frac1{\tilde s}
+\frac{1-\theta_1}{r_2}
+\frac{\theta_1}{r}
=1.
\]
H\"older's inequality, Lemma~\ref{interpolation}, and the smallness assumption give
\[
\begin{aligned}
\int_{Q_{k+1}}|u||h_k|\,dxdt
\leq{}&
\|u\|_
{L_t^{q_2}L_x^{r_2}(Q_{k+1})}^{1-\theta_1}
\|u\|_
{L_t^qL_x^r(Q_{k+1})}^{\theta_1}
\|h_k\|_
{L_t^{\tilde q}L_x^{\tilde s}(Q_{k+1})}
\\
\leq{}&
C2^{2bk}\Big[
\varepsilon_1^{1+\theta_1}
(A_{k+2}+E_{k+2})^{\frac{1-\theta_1}{2}}
\\
&+
\varepsilon_1^{1+\theta_1-\theta_2}
(A_{k+2}+E_{k+2})^{
\frac12-\frac{\theta_1}{2}+\theta_2}
\\
&+
\varepsilon_1^{1-\theta_3+\theta_1}
(A_{k+2}+E_{k+2})^{
1+\frac{\theta_3}{2}-\frac{\theta_1}{2}}
\Big].
\end{aligned}
\]
This is precisely the full-cylinder counterpart of~\eqref{hkk}.

Combining the estimates for $p_k$ and $h_k$, we conclude that
\[
\begin{aligned}
\int_{Q_{k+1}}|u||p|\,dxdt
\leq
C2^{2bk}\Big[
&
\varepsilon_1^{1+\theta_p}
(A_{k+2}+E_{k+2})^{1-\frac{\theta_p}{2}}
\\
&+
\varepsilon_1^{1+\frac{\tau_1}{2}}
(A_{k+2}+E_{k+2})^{1-\frac{\tau_1}{2}}
\\
&+
\varepsilon_1^{1+\theta_1}
(A_{k+2}+E_{k+2})^{\frac{1-\theta_1}{2}}
\\
&+
\varepsilon_1^{1+\theta_1-\theta_2}
(A_{k+2}+E_{k+2})^{
\frac12-\frac{\theta_1}{2}+\theta_2}
\\
&+
\varepsilon_1^{1-\theta_3+\theta_1}
(A_{k+2}+E_{k+2})^{
1+\frac{\theta_3}{2}-\frac{\theta_1}{2}}
\Big].
\end{aligned}
\]
Under the same restrictions on
$\theta_p,\theta_1,\theta_2,\theta_3$ as in the boundary argument,
all powers of $A_{k+2}+E_{k+2}$ on the right-hand side are strictly
less than one. Hence Young's inequality gives, for every
$\delta_1>0$,
\[
2^k\int_{Q_{k+1}}|u||p|\,dxdt
\leq
\delta_1^2(A_{k+2}+E_{k+2})
+
C_{\delta_1}2^{Bk}\varepsilon_1^\beta,
\]
for some $B>0$ and $\beta=\beta(q,r,\tilde{q})\in(0,1)$.

Thus the pressure contribution in the interior satisfies exactly the
same type of estimate needed to derive~\eqref{eq:section5-recursion}. The only essential
difference from the boundary argument is that the estimate of $h_k$
uses the interior harmonic regularity of
$\Delta h_k=0$, rather than the boundary Stokes estimate of
Lemma~\ref{Seregin}; in particular, no separate regularity estimate for $w_k$
is required.

All the remaining terms in the localized energy inequality are
estimated exactly as in the boundary case, with $Q_k^+$ and $B_k^+$
replaced by $Q_k$ and $B_k$, respectively. In fact, the interior
cutoff functions are compactly supported in the full cylinder, so
no flat-boundary terms occur. Combining these estimates with the
interior pressure estimate above yields the same recursion
\eqref{eq:section5-recursion}. The same summation argument therefore
gives
\[A_1+E_1\leq C\varepsilon_1^\beta.\]

Choosing $\varepsilon_1$ sufficiently small and applying the interior
counterpart of Proposition~\ref{main1} at radius $1/2$ gives regularity in
$\overline{Q_{1/4}}$.

We now recover the claimed region in both cases. Denote the small
threshold for these quarter-cylinder conclusions by $\eta$. If a
boundary or interior cylinder of radius $R$ is contained in the
reference cylinder, restriction, change of pressure mean, and rescaling
bound its normalized lower-order hypothesis by
\[
C\left(R^{1-2/q-3/r}+R^{-3}+R^{-1-2/\tilde q}\right)\varepsilon_1.
\]
For the pressure, use
$\int_{\mathcal B_R}|p-[p]_{\mathcal B_R}|dx
\leq2\int_{\mathcal B_R}|p-[p]_{\mathcal B_1}|dx$.
For a boundary reference cylinder, take $R_b=1/8$. At each
$z=(x,t)\in\overline{Q_{1/2}^+}$ with $x_3\leq1/64$, the cylinder
$Q_{R_b}^+(x',0,t)$ lies in $Q_1^+$ and its quarter-cylinder contains
$z$. If $x_3>1/64$, the full cylinder $Q_{R_i}(z)$ with $R_i=1/128$
is contained in $Q_1^+$. For an interior reference cylinder, the same
fixed radius $R_i$ works for every $z\in\overline{Q_{1/2}}$.
Choose $\varepsilon_1$ smaller than $\eta$ divided by the maximum
of the two fixed rescaling constants. The quarter-cylinder result
then applies at every center just described. Translating the terminal
times slightly forward when $t<0$, and using the terminal closure
when $t=0$, gives neighborhoods of every point in the claimed
closed half-cylinder. A finite subcover yields regularity throughout
$\overline{\mathcal Q_{1/2}}$.

\end{proof}

\begin{proof}[Proof of Proposition~\ref{main2}]
Translate the normalized center to the origin. All constants
below may depend on $q,\tilde q,M_c,M_\phi$.
Write $\varepsilon$ for the left-hand side of
\eqref{eq:main2-smallness} and initially assume
$\varepsilon\leq1$. We use $C_M$ for constants that also depend
on the prescribed bound $M$, but not on the solution's entropy
or energy norms. They may be chosen nondecreasing in $M$ by
using $1+M$ in their upper bounds.

\medskip\noindent
\textit{Step 1: a mass-dependent local energy bound.}
Set
\[
\rho_k=\frac78-2^{-k-3},\qquad
X_k=A_{u,\sqrt n,\nabla\sqrt{\tilde c}}(\rho_k)
    +E_{u,\sqrt n,\nabla\sqrt{\tilde c}}(\rho_k),
\qquad k\geq0.
\]
In particular, $\rho_0=3/4$ and $\rho_k\uparrow7/8$.
Repeat the cutoff and pressure estimates in the proof of
Theorem~\ref{main3} on these nested cylinders, with its spatial
exponent equal to $q$. The pressure-exponent reduction made
there remains valid, since $F(q,q)=1$.
In those estimates, retain the bound $M$ for each density
$L^1$ factor and the bound $\varepsilon$ for the velocity,
signal-gradient, and pressure factors.

Here is the resulting control of the terms for which the
density smallness was previously used. With the interpolation
parameter $\tau_1$ of that proof and sufficiently small
$0<\alpha<\min\{2/5,\tau_1/5\}$, the estimates on each
$\mathcal Q_{\rho_{k+1}}$ give
\begin{align*}
\int n&\leq M,\qquad
\int n^{1-\alpha}\leq C M^{1-\alpha},\\
\int n^{1+\alpha}
&\leq C M^{1-3\alpha/2}X_{k+1}^{5\alpha/2},\\
\int n|u|
&\leq C\varepsilon M^{\tau_1/2}
X_{k+1}^{1-\tau_1/2},\\
\int n^{1+\alpha}|u|
&\leq C\varepsilon M^{(\tau_1-3\alpha)/2}
X_{k+1}^{1+5\alpha/2-\tau_1/2},\\
\int n^{1-\alpha}|u|
&\leq C\varepsilon M^{\tau_1(1-\alpha)/2}
X_{k+1}^{(1-\tau_1/2)(1-\alpha)}.
\end{align*}
All integrals in this display are space--time integrals.
The same estimates hold with $|u|$ replaced by $|\nabla c|$,
because
$|\nabla c|\leq2\sqrt{M_c}|\nabla\sqrt{\tilde c}|$.
The fractional power of the mass-flux integral in
\eqref{eq:section5-rhs} is controlled by the same bounds and
$(a+b)^{1-\alpha}\leq a^{1-\alpha}+b^{1-\alpha}$.

For completeness, the density contributions to the two pressure
products in~\eqref{pkk} and~\eqref{hkk} are now bounded by
\[
C2^{bk}\varepsilon M^{\tau_1/2}
X_{k+2}^{1-\tau_1/2},
\qquad
C2^{2bk}\varepsilon^{\theta_1}M^{1-\theta_2}
X_{k+2}^{(1-\theta_1)/2+\theta_2},
\]
respectively. The energy exponents remain strictly less than
one, by the choices made in that proof. All the remaining
terms are independent of density mass smallness and have the
same sublinear energy exponents as before. Since the radii
stay in $[3/4,7/8]$, inclusion of cylinders bounds $X_{k+1}$
by a fixed multiple of $X_{k+2}$. Consequently, for some
$\theta\in(0,1)$ and $b>0$,
\[
X_k\leq C_M2^{bk}(1+X_{k+2}^{\theta}).
\]
Young's inequality therefore gives, for every $a>0$,
\begin{equation}
X_k\leq aX_{k+2}+C_{M,a}2^{Bk}
\label{eq:bounded-mass-energy-recursion}
\end{equation}
with $B>0$. Choose $a2^{2B}<1$. Suitability and cylinder
inclusion imply $\sup_kX_k<\infty$; this finiteness is used
only to discard the tail in the iteration
\[
X_0\leq a^jX_{2j}
+C_{M,a}\sum_{\ell=0}^{j-1}(a2^{2B})^\ell.
\]
Letting $j\to\infty$ and using the energy interpolation
inequality yields
\begin{equation}
A_{u,\sqrt n,\nabla\sqrt{\tilde c}}(3/4)
+E_{u,\sqrt n,\nabla\sqrt{\tilde c}}(3/4)
+\|n\|_{L^{5/3}(\mathcal Q_{3/4})}
\leq C_M.
\label{eq:bounded-mass-local-energy}
\end{equation}
This is a boundedness estimate, not a small-energy estimate.

\medskip\noindent
\textit{Step 2: a caloric decomposition of the density.}
The density equation is
\[
\partial_t n-\Delta n=-\nabla\cdot F,
\qquad F=n(u+\nabla c).
\]
Since $q'=q/(q-1)<5/3$, H\"older's inequality and
\eqref{eq:bounded-mass-local-energy} give
\begin{equation}
\|F\|_{L^1(\mathcal Q_{3/4})}
\leq\|n\|_{L^{q'}(\mathcal Q_{3/4})}
\|u+\nabla c\|_{L^q(\mathcal Q_{3/4})}
\leq C_M\varepsilon.
\label{eq:bounded-mass-small-flux}
\end{equation}
In the boundary case, extend $n$ evenly across $x_3=0$,
extend $F_1,F_2$ evenly and $F_3$ oddly, and denote the
extensions by $N,F^{\mathrm e}$ on $Q_{3/4}$.
The weak Neumann identity gives
\[
\partial_tN-\Delta N=-\nabla\cdot F^{\mathrm e}
\quad\text{in }\mathcal D'(Q_{3/4}).
\]
Indeed, the sum of a test function and its reflection has
zero normal derivative on the plane and is an admissible
test in that identity. Thus no classical normal trace is
being assumed. In the interior case put $N=n$ and
$F^{\mathrm e}=F$.

Let $a_0=-9/16$, and extend $F^{\mathrm e}$ by zero from
$Q_{3/4}$ to $\mathbb R^3\times(a_0,0)$. Denote the result
by $\widehat F$. With the heat kernel
\[
G(x,t)=(4\pi t)^{-3/2}\exp(-|x|^2/(4t)),\qquad t>0,
\]
define
\[
w(x,t)=-\int_{a_0}^t\int_{\mathbb R^3}
\nabla_xG(x-y,t-s)\cdot\widehat F(y,s)\,dy\,ds.
\]
Then $(\partial_t-\Delta)w=-\nabla\cdot\widehat F$ in
distributions. Since
$\|\nabla G(\cdot,t)\|_{L^1(\mathbb R^3)}\leq Ct^{-1/2}$,
Tonelli's theorem gives
\begin{align}
\int_{a_0}^0\|w(\cdot,t)\|_{L^1(\mathbb R^3)}\,dt
&\leq C\int_{a_0}^0\int_{a_0}^t
(t-s)^{-1/2}\|\widehat F(\cdot,s)\|_{L^1}\,ds\,dt
\nonumber\\
&\leq C\|\widehat F\|_{L^1(\mathbb R^3\times(a_0,0))}
\leq C_M\varepsilon.
\label{eq:bounded-mass-heat-potential}
\end{align}
The function $v=N-w$ is caloric on $Q_{3/4}$. The local
$L^1$-to-$L^\infty$ estimate for caloric functions therefore
implies
\begin{equation}
\|v\|_{L^\infty(Q_{5/8})}
\leq C\|v\|_{L^1(Q_{3/4})}
\leq C M+C_M\varepsilon.
\label{eq:bounded-mass-caloric-part}
\end{equation}
Only the flux has been extended by zero at the spherical
boundary; no artificial boundary condition is imposed on $n$.

\medskip\noindent
\textit{Step 3: small rescaled density mass.}
For every full or boundary half-cylinder
$\mathcal Q_\rho(z)\subset\mathcal Q_{5/8}$, with
$0<\rho\leq1$, the preceding decomposition and $n\geq0$
give
\begin{equation}
\rho^{-3}\int_{\mathcal Q_\rho(z)}n\,dx\,dt
\leq C M\rho^2+C_M\varepsilon\rho^{-3}.
\label{eq:bounded-mass-density-decay}
\end{equation}
Here the $C_M\varepsilon\rho^2$ contribution from
\eqref{eq:bounded-mass-caloric-part} has been absorbed into
the last term. This estimate is uniform in the center.

Let $\eta>0$ be the threshold in Theorem~\ref{main3} for
the equal exponents $(q,q)$ and pressure exponent $\tilde q$.
It is admissible because $5/q<2$ and $F(q,q)=1$.
For scales at most one we use the original upper bound
$M_\phi$ for the rescaled potential gradient.
Choose $0<r\leq1/64$, depending only on $M$ and the fixed
parameters, so that
\[
C M r^2\leq\eta/6.
\]
The constants can be enlarged independently of $r$ so that,
at either radius $\rho=r$ or $\rho=r/8$, the error term in
\eqref{eq:bounded-mass-density-decay} is bounded by
$C_M\varepsilon r^{-3}$.
For the same cylinders, restriction, rescaling, and change of
pressure mean give
\begin{align*}
C^{q,q}_{u,\nabla\sqrt{\tilde c}}(\rho;z)
&\leq C r^{1-5/q}\varepsilon,\\
D_{\tilde q,1}(\rho;z)
&\leq C r^{-1-2/\tilde q}\varepsilon.
\end{align*}
In the second inequality we used, for almost every time,
\[
\int_{\mathcal B_\rho(z)}|p-[p]_{\mathcal B_\rho(z)}|\,dx
\leq2\int_{\mathcal B_\rho(z)}|p-[p]_1|\,dx.
\]
After fixing $r$, choose
\begin{equation}
0<\varepsilon_1\leq
\min\left\{1,\frac{\eta r^3}{6C_M},
\frac{\eta r^{5/q-1}}{3C},
\frac{\eta r^{1+2/\tilde q}}{3C}\right\}.
\label{eq:bounded-mass-threshold-choice}
\end{equation}
For $\varepsilon\leq\varepsilon_1$, every such cylinder of
either selected radius satisfies
\[
C^{q,q}_{u,\nabla\sqrt{\tilde c}}(\rho;z)
+\rho^{-3}\int_{\mathcal Q_\rho(z)}n\,dx\,dt
+D_{\tilde q,1}(\rho;z)\leq\eta.
\]
Thus Theorem~\ref{main3} applies after rescaling.
In particular, it is the small density mass on this smaller
cylinder, rather than the original bound $M$, that enters
the small-energy iteration leading to Proposition~\ref{main1}.

\medskip\noindent
\textit{Step 4: coverage of the asserted region.}
In the boundary case, take any
$z=(x,t)\in\overline{Q_{1/2}^+}$. If $x_3\leq r/4$,
use the boundary cylinder $Q_r^+(x',0,t)$; its half-cylinder
contains $z$ in its spatial interior relative to the closed
half-space. If $x_3>r/4$, use the full cylinder $Q_{r/8}(z)$,
which is contained in the half-space. These cylinders lie
in $Q_{5/8}^+$. In the interior reference case, use
$Q_{r/8}(z)\subset Q_{5/8}$ for every
$z\in\overline{Q_{1/2}}$.
The preceding estimates are uniform for all these centers.
As in the proof of Theorem~\ref{main3}, move terminal times
slightly forward when $t<0$ and use the terminal closure at
$t=0$. The resulting neighborhoods cover the asserted closed
half-cylinder or full cylinder, and a finite subcover gives
the claimed regularity.
Finally, choose the intermediate radius nonincreasing in $M$
and replace $r^{5/q-1}$ in~\eqref{eq:bounded-mass-threshold-choice}
by $r^{\max\{5/q-1,0\}}$ if necessary. Together with the
nondecreasing choice of $C_M$, this makes $\varepsilon_1$
nonincreasing in $M$.
\end{proof}

   \section{Dissipation criterion and the singular set}
\label{sec:dissipation-singular-set}
\subsection{Auxiliary regularity criteria}
\label{subsec:auxiliary-regularity-criteria}
Throughout this section, we use the normalized notation
$\mathcal B_r$ and $\mathcal Q_r$ introduced in Section~\ref{subsec:notation-scaling}.
In the interior case, the initial physical radius is always chosen
sufficiently small so that the corresponding cylinder is compactly
contained in $\mathbb H^3\times(0,T)$.

For the dissipation iteration we use the small-mass
criterion of Theorem~\ref{main3} with both mixed-norm exponents
equal to $10/3$. This keeps its threshold independent of a
prescribed density mass bound and gives the following corollary.

\begin{corollary}\label{6.1}
Let $(n,c,u,p)$ be a normalized local restriction of a suitable weak
solution in the sense described above. There exists
\[
\varepsilon_4=\varepsilon_4(M_c,M_\phi)>0
\]
such that, if
\begin{align}\label{coro2}
\|u\|_{L_{t,x}^{10/3}(\mathcal Q)}
+\|\nabla\sqrt{\tilde c}\|_{L_{t,x}^{10/3}(\mathcal Q)}
+\|n\|_{L_{t,x}^{5/3}(\mathcal Q)}
+\|p-[p]_1(t)\|_{L_t^{5/3}L_x^1(\mathcal Q)}
\leq\varepsilon_4,
\end{align}
then $(n,\nabla c,u)$ is regular in
$\overline{\mathcal Q_{1/2}}$.
\end{corollary}

\begin{proof}

Apply Theorem~\ref{main3} with $q=r=10/3$ and $\tilde q=5/3$.
On the unit cylinder,
$\|n\|_{L^1}\leq |\mathcal Q|^{2/5}\|n\|_{L^{5/3}}$.
Choosing $\varepsilon_4$ sufficiently small in terms of the
threshold in that theorem proves the assertion. Its only
nonuniversal dependencies are $M_c$ and $M_\phi$.

\end{proof}

For the remainder of this section, define the pressure quantities
\begin{equation}\label{eq:pressure-P-definition}
	P(r):=r^{-1}\|p-[p]_r(t)\|_{L_t^{\frac53}L_x^{\frac53}(\mathcal{Q}_r)},
	\qquad
	\bar P(r):=P(r)^{\frac53}
	=r^{-\frac53}\int_{\mathcal{Q}_r}|p-[p]_r(t)|^{\frac53}\,dxdt.
\end{equation}
Both are invariant under the natural scaling, and unlike the pressure itself
they are unchanged when a function of time is added to $p$.

\begin{corollary}\label{coro6.2}
	Let $\rho_1 \in (0, 1)$. There exists a constant $\varepsilon_5 = \varepsilon_5(\varepsilon_4) > 0$ determined by the threshold $\varepsilon_4$ in Corollary \ref{6.1} such that if
	\begin{align}\label{6.2}
		C^{\frac{10}{3}, \frac{10}{3}}_{\sqrt{n}, u, \nabla\sqrt{\tilde{c}}}(\rho_1) + P(\rho_1) \leq \varepsilon_5,
	\end{align}
	then the scaled solution $(n_{\rho_1}, \nabla c_{\rho_1}, u_{\rho_1})$ is regular in $\overline{\mathcal Q_{1/2}}$.
\end{corollary}

\begin{proof}
	By scale invariance,
	\begin{align*}
		\|u_{\rho_1}\|_{L^{\frac{10}{3}}_{t,x}(\mathcal{Q})}
		&=C^{\frac{10}{3},\frac{10}{3}}_u(\rho_1),\\
		\|\nabla\sqrt{\tilde c_{\rho_1}}\|_{L^{\frac{10}{3}}_{t,x}(\mathcal{Q})}
		&=C^{\frac{10}{3},\frac{10}{3}}_{\nabla\sqrt{\tilde c}}(\rho_1),\\
		\|n_{\rho_1}\|_{L^{\frac53}_{t,x}(\mathcal{Q})}
		&=\left(C^{\frac{10}{3},\frac{10}{3}}_{\sqrt n}(\rho_1)\right)^2.
	\end{align*}
	Moreover, H\"older's inequality in space gives
	$D_{\frac53,1}(\rho_1)\leq C P(\rho_1)$.  Therefore the left-hand side of \eqref{coro2} for the scaled solution is bounded by $C(\varepsilon_5+\varepsilon_5^2)$.  Choosing
	\[
	\varepsilon_5\leq\min\left\{1,\frac{\varepsilon_4}{3C}\right\}
	\]
	allows us to apply Corollary~\ref{6.1}, proving the claim.
\end{proof}

\subsection{Proof of the dissipation criterion}
\label{subsec:dissipation-criterion}
Having established Corollary \ref{6.1} and \ref{coro6.2}, we now focus on the proof of Theorem \ref{main4}. To this end, we analyze the behavior of the scaled energy-like quantities and the pressure term at small scales. The estimates for $A(r)$ and $E(r)$ follow from the local energy inequality. The subsequent estimates for $C(r)$ and $\bar P(r)$ are established to formulate an iterative scheme, which serves as the key mechanism for proving the unified interior-boundary regularity criterion asserted in Theorem \ref{main4}.

\begin{proof}[Proof of Theorem \ref{main4}]
Fix
\[
z_0=(x_0,t_0)
\in
\overline{\mathbb H^3}\times(0,\infty).
\]
Choose $r_0>0$ sufficiently small so that
\[
4r_0^2<t_0.
\]
If $x_0\in\mathbb H^3$, we additionally require
\[
2r_0<\operatorname{dist}(x_0,\Gamma).
\]
Hence the cylinder of scale $2r_0$ is either a boundary half-cylinder
or an interior full cylinder.

After translating $z_0$ to the origin and parabolically rescaling by
$r_0$, the reference unit cylinder is $Q_1^+$ in the boundary case
and $Q_1$ in the interior case. In either case we denote it by
$\mathcal Q$ and keep the same notation for the rescaled solution.

We first initialize the stronger pressure quantity in \eqref{6.2}.
Since
\[
p\in L^{5/3}(\mathcal{Q}),
\]
Jensen's inequality for the
spatial average and the triangle inequality give
\begin{align}\label{eq:pressure-gain-initialization}
P(\rho)
&=\rho^{-1}
\|p-[p]_\rho(t)\|_{L^{5/3}(\mathcal{Q}_\rho)}
\leq 2\rho^{-1}\|p\|_{L^{5/3}(\mathcal{Q}_\rho)}
<\infty
\end{align}
for every fixed admissible radius $0<\rho<1$. Therefore the pressure iteration below is well defined at every fixed admissible starting scale ($0<\rho<1$).
In particular, no additional pressure integrability assumption is required.
\par\textbf{Estimate of $C^{\frac{10}{3},\frac{10}{3}}_{\sqrt{n},u,\nabla\sqrt{\tilde{c}}}(r)$:}
For any $0 < 4r < \rho<\rho_0 < 1$, we begin by decomposing the velocity term as follows:
\begin{align*}
	\left(C^{\frac{10}{3},\frac{10}{3}}_{u}(r)\right)^\frac{10}{3}=r^{-\frac{5}{3}}\int_{\mathcal{Q}_r}|u|^\frac{10}{3}dxdt\leq C r^{-\frac{5}{3}}\left(\int_{\mathcal{Q}_r}|u-u_\rho|^\frac{10}{3}dxdt + \int_{\mathcal{Q}_r}|u_\rho|^\frac{10}{3}dxdt\right).
\end{align*}
For almost every $t$, denote the spatial mean by
\[
u_\rho(t)
:=
\fint_{\mathcal B_\rho}u(x,t)\,dx.
\] we obtain:
\begin{align*}
	C r^{-\frac{5}{3}}\int_{\mathcal{Q}_r}|u_\rho|^\frac{10}{3}dxdt &\leq Cr^\frac{4}{3}\int_{-\rho^2}^{0}\rho^{-10}\left(\int_{\mathcal{B}_\rho}|u|dx\right)^\frac{10}{3}dt\\
	&\leq Cr^\frac{4}{3}\rho^{-3}\int_{\mathcal{Q}_\rho}|u|^\frac{10}{3}dxdt = C\left(\frac{r}{\rho}\right)^\frac{4}{3}\left(C^{\frac{10}{3},\frac{10}{3}}_{u}(\rho)\right)^\frac{10}{3}.
\end{align*}
Applying the Sobolev embedding theorem along with H\"{o}lder's inequality in space yields:
\begin{align*}
	\|u-u_\rho\|_{L^\frac{10}{3}(\mathcal{B}_\rho)}\leq \|u-u_\rho\|^\frac{2}{5}_{L^2(\mathcal{B}_\rho)}\|u-u_\rho\|^\frac{3}{5}_{L^6(\mathcal{B}_\rho)}\leq C\|u\|^\frac{2}{5}_{L^2(\mathcal{B}_\rho)}\|\nabla u\|^\frac{3}{5}_{L^2(\mathcal{B}_\rho)}.
\end{align*}
Consequently, we have:
\begin{align*}
	Cr^{-\frac{5}{3}}\int_{\mathcal{Q}_r}|u-u_\rho|^\frac{10}{3}dxdt\leq Cr^{-\frac{5}{3}}\int_{-\rho^2}^{0}\|u\|^\frac{4}{3}_{L^2(\mathcal{B}_\rho)}\|\nabla u\|^2_{L^2(\mathcal{B}_\rho)}dt\leq C\left(\frac{\rho}{r}\right)^\frac{5}{3}A_u^\frac{2}{3}(\rho)E_u(\rho).
\end{align*}
Since the estimates for $C^{\frac{10}{3},\frac{10}{3}}_{\sqrt{n}}(r)$ and $C^{\frac{10}{3},\frac{10}{3}}_{\nabla\sqrt{\tilde{c}}}(r)$ can be handled in an analogous manner, their details are omitted for brevity. Finally, we define
\begin{align*}
	\mathscr{C}(\rho):=\rho^{-\frac{5}{3}}\int_{\mathcal{Q}_\rho}\left(|u|^\frac{10}{3}+|\sqrt{n}|^\frac{10}{3}+|\nabla\sqrt{\tilde{c}}|^\frac{10}{3}\right)dxdt,
\end{align*}
which leads to the following bound:
\begin{align}\label{C}
	\mathscr{C}(r)\leq C\left(\frac{r}{\rho}\right)^\frac{4}{3}\mathscr{C}(\rho)+C\left(\frac{\rho}{r}\right)^\frac{5}{3}A_{\sqrt{n},u,\nabla\sqrt{\tilde{c}}}^\frac{2}{3}(\rho)E_{\sqrt{n},u,\nabla\sqrt{\tilde{c}}}(\rho).
\end{align}

\par\textbf{Estimate of $\bar P(r)$:}
By Corollary~\ref{cor:pressure-53}, we have:
\begin{align}\label{D}
	\bar P(r)\leq& C\left(\frac{r}{\rho}\right)^\frac{5}{3}\bar P(\rho)+C\left(\frac{r}{\rho}\right)^\frac{5}{3}\left(A_{u,\sqrt{n}}^\frac{2}{3}(\rho)E_{u,\sqrt{n}}(\rho)+A_u(\rho)^{\frac13}E_u(\rho)^{\frac12}\right)\nonumber\\
	&+C\left(\frac{\rho}{r}\right)^\frac{5}{3}A_{u,\sqrt{n}}^\frac{2}{3}(\rho)E_{u,\sqrt{n}}(\rho)+C\left(\frac{\rho}{r}\right)^\frac{5}{3}\rho^{\frac53} A_{\sqrt{n}}^\frac{5}{3}(\rho).
\end{align}

\par\textbf{Estimate of $A(r)+E(r)$:}

Use the abbreviations $A=A_{u,\sqrt n,\nabla\sqrt{\tilde c}}$ and
$E=E_{u,\sqrt n,\nabla\sqrt{\tilde c}}$. For $0<r<1/2$, normalize
the cylinder of radius $2r$ and use the preceding heat-kernel proof
with $\gamma=1/2$ and, for example, $\alpha=1/15$.

The normalized norms obey
\[
(2r)^{-1}\|n\|_{L^{5/3}(\mathcal Q_{2r})}
\leq C\mathscr C(2r)^{3/5}
\]
and
\[
(2r)^{-1/2}\left(\|u\|_{L^{10/3}(\mathcal Q_{2r})}
+\|\nabla\sqrt{\tilde c}\|_{L^{10/3}(\mathcal Q_{2r})}\right)
\leq C\mathscr C(2r)^{3/10}.
\]
The weak mass identity bounds $M_\psi$ by
$C\bigl(\mathscr C(2r)^{3/5}+\mathscr C(2r)^{9/10}\bigr)$.
Estimate the pressure contribution directly by
$C\mathscr C(2r)^{3/10}\bar P(2r)^{3/5}$,
using $u\in L^{5/2}$ on the unit cylinder.
Including \eqref{eq:weighted-negative-entropy} therefore gives
\begin{equation}\label{aerr}
\begin{aligned}
A(r)+E(r)\leq C_\alpha\bigl[&\mathscr C(2r)^{3/5}
+\mathscr C(2r)^{3(1+\alpha)/5}
+\mathscr C(2r)^{3(1-\alpha)/5}\\
&+\mathscr C(2r)^{9/10}
+\mathscr C(2r)^{9/10+3\alpha/5}
+\mathscr C(2r)^{9/10-3\alpha/5}\\
&+\mathscr C(2r)^{3/10}\bar P(2r)^{3/5}
+\bigl(\mathscr C(2r)^{3/5}+\mathscr C(2r)^{9/10}\bigr)^{1-\alpha}\bigr]\\
\leq{}&C\bigl(\mathscr C(2r)+\bar P(2r)+1\bigr).
\end{aligned}
\end{equation}
All pure powers of $\mathscr C(2r)$ displayed here are less than one;
the mixed pressure term is bounded by
$C\bigl(\mathscr C(2r)+\bar P(2r)+1\bigr)$ by Young's inequality.
Thus $C=C(M_c,M_\phi)$ after fixing $\alpha$.

\par\textbf{Iteration argument:} 
We define the combined quantity:
\begin{align}
	G(r):=\bar P(r)+\mathscr{C}(r).
\end{align}

Combining \eqref{C} and \eqref{D} at radius $\rho/2$ is valid for
$0<8r\leq\rho$. Put $r=\theta\rho$, $0<\theta<1/8$.
Since $G(\rho/2)\leq CG(\rho)$, \eqref{aerr} and
\[
A^{2/3}E\leq(A+E)E^{2/3},\qquad
A_u^{1/3}E_u^{1/2}\leq(A+1)E^{1/2}
\]
give
\begin{equation}\label{eq:corrected-G-iteration}
\begin{aligned}
G(\theta\rho)\leq{}&C\left[\theta^{4/3}
+\theta^{-5/3}E(\rho/2)^{2/3}
+\theta^{5/3}E(\rho/2)^{1/2}\right]G(\rho)\\
&+C\left[\theta^{-5/3}E(\rho/2)^{2/3}
+\theta^{5/3}E(\rho/2)^{1/2}
+\theta^{-5/3}\bigl(\rho A_{\sqrt n}(\rho/2)\bigr)^{5/3}\right].
\end{aligned}
\end{equation}
Here $E$ is the full dissipation, including both scalar terms, at every
occurrence.

Choose $\theta_0\in(0,1/8)$ with $C\theta_0^{4/3}\leq1/8$.
Choose $0<\varepsilon_6\leq1$ sufficiently small that
$C(\varepsilon_6^{3/10}+\varepsilon_6^{3/5})<\varepsilon_5$,
where $\varepsilon_5$ is the threshold in Corollary~\ref{coro6.2}.
Now choose $\varepsilon_3>0$ such that
\[
C\theta_0^{-5/3}(2\varepsilon_3)^{2/3}\leq\varepsilon_6/16,
\qquad
C\theta_0^{5/3}(2\varepsilon_3)^{1/2}\leq\varepsilon_6/16.
\]
The limsup assumption in Theorem~\ref{main4} implies that, for some
$\rho_2>0$,
\[
E(\rho/2)\leq2\varepsilon_3\qquad(0<\rho\leq\rho_2).
\]
The slack factor $2$ makes this a uniform statement at all subsequent
scales, not merely at one selected radius.

The local entropy bound also implies $rA_{\sqrt n}(r)\to0$ as
$r\downarrow0$. Indeed, for a fixed bounded domain $D$ containing
the balls and any $K>1$,
\[
rA_{\sqrt n}(r)
\leq K|\mathcal B_r|
+\frac1{\log K}\operatorname*{ess\,sup}_{-1<s<0}
\int_D n(x,s)\log^+n(x,s)\,dx.
\]
The supremum is finite by Definition~\ref{def:suitable-half-space}(iii).
Let $r\downarrow0$ first and then $K\to\infty$. Shrinking $\rho_2$
therefore ensures
\[
C\theta_0^{-5/3}\bigl(\rho A_{\sqrt n}(\rho/2)\bigr)^{5/3}
\leq\varepsilon_6/8\qquad(0<\rho\leq\rho_2).
\]
Equation~\eqref{eq:corrected-G-iteration} now yields, at every such scale,
\[
G(\theta_0\rho)\leq\tfrac12G(\rho)+\tfrac14\varepsilon_6,
\qquad
G(\theta_0^k\rho_2)\leq2^{-k}G(\rho_2)+\tfrac12\varepsilon_6.
\]
The initial value $G(\rho_2)$ is finite. For sufficiently large $k$,
$G(\theta_0^k\rho_2)\leq3\varepsilon_6/4$, which implies
\eqref{6.2}. Corollary~\ref{coro6.2} proves regularity at the center.
The regularity threshold $\varepsilon_3$ depends only on
$M_c$, $M_\phi$, and fixed geometric constants.
The starting radius $\rho_2$ and the number of iterations
needed to reach the regularity threshold may depend on
the solution and on the point under consideration.

\end{proof}

\subsection{Parabolic Hausdorff measure of the singular set}
\label{subsec:singular-set-measure}
With the unified boundary interior criterion of
Theorem~\ref{main4} established, we now estimate the full singular
set $\mathcal S$.

\begin{proof}[Proof of Corollary~\ref{cor:full-singular-set}]

Recall the dissipation density $e$ and the parabolic metric $d_P$
introduced in Section~\ref{subsec:main-regularity-results}. We denote the corresponding
parabolic balls by
\[
P_r(z):=B_r(x)\times(t-r^2,t+r^2),
\qquad z=(x,t).
\]

For each integer $m\ge2$, define
\[
K_m
:=
\left\{
(x,t)\in
\overline{\mathbb H^3}\times(0,\infty):
|x|\le m,\quad
\frac1m\le t\le m
\right\},
\]
and set
\[
\mathcal S_m:=\mathcal S\cap K_m.
\]
Since
\[
\overline{\mathbb H^3}\times(0,\infty)
=
\bigcup_{m=2}^{\infty}K_m,
\]
we have
\[
\mathcal S
=
\bigcup_{m=2}^{\infty}\mathcal S_m.
\]
Thus it suffices to prove
\[
\mathcal P^1(\mathcal S_m)=0
\]
for every fixed $m\ge2$.

Fix $m\ge2$. Choose a bounded open set
\[
U_m\Subset
\mathbb R^3\times(0,\infty)
\]
such that
\[
K_m\subset U_m.
\]
For instance, one may take
\[
U_m
=
\left\{
(x,t):
|x|<m+1,\quad
\frac1{2m}<t<m+1
\right\}.
\]
By the local integrability of the dissipation density,
\[
e\in
L^1\bigl(
U_m\cap(\mathbb H^3\times(0,\infty))
\bigr).
\]

Extend $e$ by zero outside the half-space and set
\[
\widetilde e(x,t)
:=
e(x,t)
\mathbf 1_{\mathbb H^3\times(0,\infty)}(x,t).
\]
Then
\[
\widetilde e\in L^1(U_m).
\]

We first show that $\mathcal S_m$ has zero four-dimensional
Lebesgue measure. By the Lebesgue differentiation theorem for
parabolic balls, for almost every $z\in U_m$,
\[
\fint_{P_r(z)}
\widetilde e\,dx\,dt
\longrightarrow
\widetilde e(z)
\qquad\text{as }r\downarrow0.
\]
Let such a Lebesgue point satisfy
\[
z=(x,t)\in\mathbb H^3\times(0,\infty).
\]
For all sufficiently small admissible radii $r$,
\[
\mathcal Q_r(z)\subset P_r(z),
\]
and hence
\[
\frac1r
\int_{\mathcal Q_r(z)}e\,dx\,dt
\le
\frac1r
\int_{P_r(z)}\widetilde e\,dx\,dt.
\]
Since
\[
|P_r(z)|=C r^5,
\]
we obtain
\[
\frac1r
\int_{\mathcal Q_r(z)}e\,dx\,dt
\le
C r^4
\fint_{P_r(z)}
\widetilde e\,dx\,dt
\longrightarrow0.
\]
Therefore Theorem~\ref{main4} implies that every such interior
Lebesgue point is regular. Hence the interior part of $\mathcal S_m$
has zero four-dimensional Lebesgue measure. The boundary part is
contained in
\[
\Gamma\times(0,\infty),
\]
which itself has zero four-dimensional Lebesgue measure. Consequently,
\[
|\mathcal S_m|=0.
\]

Let $\bar\eta>0$.
Since $|\mathcal S_m|=0$, there exists an open set
containing $\mathcal S_m$ with arbitrarily small
Lebesgue measure; this is the outer regularity of
Lebesgue measure.
Intersecting this set with a fixed open neighborhood
of $K_m$ compactly contained in $U_m$, we obtain
an open set $J$ satisfying
\[
\mathcal S_m\subset J\Subset U_m.
\]
Since $\widetilde e\in L^1(U_m)$, the absolute
continuity of the Lebesgue integral ensures that,
by choosing $|J|$ sufficiently small,
\begin{align}\label{6.9}
\int_J\widetilde e\,dx\,dt<\bar\eta.
\end{align}

Let $z\in\mathcal S_m$. Since $z$ is singular, the contrapositive of
Theorem~\ref{main4} gives
\[
\limsup_{r\downarrow0}
\frac1r
\int_{\mathcal Q_r(z)}
e\,dx\,dt
>
\varepsilon_3,
\]
where the limit is taken along admissible radii. Hence there exist
arbitrarily small admissible radii $r>0$ such that
\begin{align}\label{6.11}
\int_{\mathcal Q_r(z)}
e\,dx\,dt
>
\varepsilon_3 r.
\end{align}

Fix $\delta_2>0$. Since $J$ is open and $z\in J$,
\[
d_P(z,J^c)>0.
\]
Because the admissible radii satisfying \eqref{6.11} can be chosen
arbitrarily small, for every $z=(x_z,t_z)\in\mathcal S_m$ we may
choose one such radius $r_z>0$ satisfying
\[
r_z
<
\min\left\{
\frac{\delta_2}{10},
\frac{d_P(z,J^c)}5,
\frac{\sqrt{t_z}}5
\right\}.
\tag{6.12}
\]
Consider the family of parabolic balls
\[
\mathcal F
:=
\{P_{r_z}(z):z\in\mathcal S_m\}.
\]
By the $5r$-covering lemma, there exists a countable pairwise disjoint
subfamily
\[
\{P_{r_j}(z_j)\}_{j=1}^{\infty}
\subset\mathcal F
\]
such that
\[
\mathcal S_m
\subset
\bigcup_{j=1}^{\infty}
P_{5r_j}(z_j)
\subset J.
\tag{6.13}
\]
Indeed, the last inclusion follows from
\[
5r_j<d_P(z_j,J^c).
\]

Since
\[
\mathcal Q_{r_j}(z_j)
\subset
P_{r_j}(z_j),
\]
and the balls $P_{r_j}(z_j)$ are pairwise disjoint, the family
\[
\{\mathcal Q_{r_j}(z_j)\}_{j=1}^{\infty}
\]
is also pairwise disjoint. Using \eqref{6.9} and \eqref{6.11}, we obtain
\[
\varepsilon_3
\sum_{j=1}^{\infty}r_j
<
\sum_{j=1}^{\infty}
\int_{\mathcal Q_{r_j}(z_j)}
e\,dx\,dt
\le
\int_{J\cap(\mathbb H^3\times(0,\infty))}
e\,dx\,dt
<
\bar{\eta}.
\tag{6.14}
\]
Consequently,
\[
\sum_{j=1}^{\infty}r_j
<
\frac{\bar{\eta}}{\varepsilon_3}.
\]

Since
\[
\operatorname{diam}_P P_{5r_j}(z_j)
=
10r_j,
\]
and $r_j<\delta_2/10$, the family
\[
\{P_{5r_j}(z_j)\}_{j=1}^{\infty}
\]
is an admissible $\delta_2$-cover of $\mathcal S_m$. Hence
\[
\mathcal P^1_{\delta_2}(\mathcal S_m)
\le
10\sum_{j=1}^{\infty}r_j
<
\frac{10\bar{\eta}}{\varepsilon_3}.
\]
Since $\bar{\eta}>0$ is arbitrary,
\[
\mathcal P^1_{\delta_2}(\mathcal S_m)=0.
\]
Letting $\delta_2\downarrow0$, we obtain
\[
\mathcal P^1(\mathcal S_m)=0.
\]

Finally,
\[
\mathcal S
=
\bigcup_{m=2}^{\infty}\mathcal S_m.
\]
Therefore, by countable subadditivity,
\[
\mathcal P^1(\mathcal S)
\le
\sum_{m=2}^{\infty}
\mathcal P^1(\mathcal S_m)
=
0.
\]
Thus
\[
\boxed{\mathcal P^1(\mathcal S)=0.}
\]
\end{proof}

\section{Eventual regularity of the global solution}
\label{sec:eventual-regularity}

We prove Theorem~\ref{thm:eventual-regularity} for a suitable weak
solution obtained from the approximation in Section~\ref{sec:existence}.
The proof begins with a signal estimate that is independent of the
transport velocity. We then select a limit that retains the global
kinetic-energy inequality and permits the weighted scalar tests.
These tests give density decay and an integrable effective force.
The velocity can therefore be restarted as a global strong solution,
after which parabolic estimates yield uniform bounds and regularity
up to the flat boundary.

Throughout this section,
$M=\|n_0\|_{L^1(\mathbb H^3)}$ and
$L=\|c_0\|_{L^1(\mathbb H^3)}$; spatial norms and integrals are over
$\mathbb H^3$ unless stated otherwise. We retain the Helmholtz
projection $\mathbb P_+$ and the no-slip Stokes operator
$A=-\mathbb P_+\Delta$ from Section~\ref{sec:existence}.
The first three subsections require only the standing bound on
$\nabla\phi$. Boundedness of $\phi$ enters in
Lemma~\ref{lem:ev-force-tail}, and H\"older continuity of
$\nabla\phi$ enters only in the final Schauder argument.
Constants may depend on the initial quantities in
\eqref{eq:initial-data-half-space} and on $M_\phi$, with the
additional dependence on $\|\phi\|_{L^\infty}$ and
$\|\nabla\phi\|_{C^{0,\beta_0}}$ at the respective steps.
No second derivative of $\phi$ is used.

The limits are taken in two stages: first $\varepsilon\to0$ with
$\tau$ fixed, giving $(n^\tau,c^\tau,u^\tau,p^\tau)$, and then
$\tau\to0$, giving $(n,c,u,p)$. We indicate explicitly when
superscripts are suppressed. Bounds that depend on $\tau$ are used
only to justify tests before the second limit.

\subsection{A transport-independent signal estimate}

\begin{lemma}[A drift-independent signal estimate]
\label{lem:ev-signal}
The approximate solutions in Section~\ref{sec:existence} satisfy
\begin{equation}\label{eq:ev-signal-decay}
\|c^{\varepsilon,\tau}(t)\|_{L^\infty}
\leq \min\{\|c_0\|_{L^\infty},\,C L t^{-3/2}\},
\qquad t>0,
\end{equation}
with a numerical constant independent of both parameters.
The same estimate holds, almost everywhere in time, for each
limit selected below. In addition, for $2\leq q<\infty$,
\begin{equation}\label{eq:ev-signal-energy}
\frac{d}{dt}\|c^{\varepsilon,\tau}\|_{L^q}^q
+\frac{4(q-1)}q\|\nabla(c^{\varepsilon,\tau})^{q/2}\|_{L^2}^2
\leq0.
\end{equation}
\end{lemma}

\begin{proof}

Fix $\varepsilon,\tau>0$ and a finite time interval $[0,T]$.
For $0\leq s<t\leq T$, let $U(t,s)$ be the solution operator for
the linear Neumann problem with generator
\[
 \Delta-u^{\varepsilon,\tau}(t)\cdot\nabla
       -h_\tau(n^{\varepsilon,\tau}(t)).
\]
The coefficients are fixed throughout this argument. In particular,
$U(t,s)c^{\varepsilon,\tau}(s)=c^{\varepsilon,\tau}(t)$.
We first use smooth boundary-compatible initial data
$f\in L^1\cap L^2$, allowing either sign, and then pass by density.

The transport term vanishes in the scalar energy estimates because
$u^{\varepsilon,\tau}$ is divergence free and has zero boundary trace.
The absorption term is nonnegative. Testing with regularized signs
and with $U(t,s)f$, respectively, gives
\[
 \|U(t,s)f\|_{L^1}\leq\|f\|_{L^1},
 \qquad
 \frac{d}{dt}\|U(t,s)f\|_{L^2}^2
       +2\|\nabla U(t,s)f\|_{L^2}^2\leq0.
\]
These tests use the boundary-compatible cutoffs from
Section~\ref{sec:existence}. The half-space Nash inequality,
obtained by even extension, reads
\begin{equation}\label{eq:ev-nash}
 \|U(t,s)f\|_{L^2}^{10/3}
 \leq C\|U(t,s)f\|_{L^1}^{4/3}
          \|\nabla U(t,s)f\|_{L^2}^2.
\end{equation}
For $f\neq0$, it follows that
\[
 \frac{d}{dt}\|U(t,s)f\|_{L^2}^2
 +\frac{2}{C\|f\|_{L^1}^{4/3}}
       \bigl(\|U(t,s)f\|_{L^2}^2\bigr)^{5/3}\leq0.
\]
Integration gives
\[
 \|U(t,s)f\|_{L^2}\leq C(t-s)^{-3/4}\|f\|_{L^1},
 \qquad
 \|U(t,s)\|_{L^1\to L^2}\leq C(t-s)^{-3/4}.
\]
The case $f=0$ is immediate.

Integration by parts with the Neumann boundary condition gives
\[
\bigl(\Delta-u^{\varepsilon,\tau}\cdot\nabla
-h_\tau(n^{\varepsilon,\tau})\bigr)^*
=\Delta+u^{\varepsilon,\tau}\cdot\nabla
-h_\tau(n^{\varepsilon,\tau}).
\]
For a fixed terminal time $t$, reversing time in the backward
adjoint equation therefore replaces the transport velocity by
$-u^{\varepsilon,\tau}$, with both coefficients evaluated at the
reversed time. This velocity is still divergence free and vanishes
on $\Gamma$, and the absorption coefficient is still nonnegative.
The same $L^1$ contraction and Nash argument apply to the adjoint
solution operator $U(t,s)^*$. Hence, by duality,
\[
\|U(t,s)\|_{L^2\to L^\infty}
=\|U(t,s)^*\|_{L^1\to L^2}
\leq C(t-s)^{-3/4}.
\]
The evolution property
$U(t,s)=U(t,(s+t)/2)U((s+t)/2,s)$ now yields
\[
\|U(t,s)\|_{L^1\to L^\infty}\leq C(t-s)^{-3/2}.
\]
Taking $s=0$ and $f=c_0^\varepsilon$, we recover
$c^{\varepsilon,\tau}(t)=U(t,0)c_0^\varepsilon$ and use
$\|c_0^\varepsilon\|_{L^1}\leq L$ to obtain the decay bound in
\eqref{eq:ev-signal-decay}. The maximum principle and
$\|c_0^\varepsilon\|_{L^\infty}\leq\|c_0\|_{L^\infty}$ give
its other bound.

Finally, testing the regularized concentration equation by
$(c^{\varepsilon,\tau})^{q-1}$, for $2\leq q<\infty$, gives
\[
\begin{aligned}
\frac{d}{dt}\|c^{\varepsilon,\tau}\|_{L^q}^q
+\frac{4(q-1)}q
\|\nabla(c^{\varepsilon,\tau})^{q/2}\|_{L^2}^2
+q\int_{\mathbb H^3}h_\tau(n^{\varepsilon,\tau})
(c^{\varepsilon,\tau})^q\,dx=0.
\end{aligned}
\]
Dropping the nonnegative absorption term proves
\eqref{eq:ev-signal-energy}. None of these bounds uses a norm of
the transport velocity. Almost-everywhere convergence on an
exhaustion of $\mathbb H^3\times(0,\infty)$ passes
\eqref{eq:ev-signal-decay} to the successive limits.

\end{proof}

\subsection{A finite-energy selection and admissible scalar tests}
\label{subsec:eventual-finite-energy-selection}

The signal estimate is uniform in both approximation parameters.
The next step is to retain the additional properties needed to use it
in the density and velocity equations. In particular, the global
bounds in Definition~\ref{def:suitable-half-space}(i) must be
supplemented by a kinetic-energy inequality starting at almost every
positive time. The selection below provides this inequality and
justifies the later nonlinear scalar tests. We assume $M>0$ in the
density estimates; when $M=0$, nonnegativity gives $n=0$.

\begin{lemma}[Finite-energy selection]\label{lem:ev-selection}
The approximation in Section~\ref{sec:existence} admits a diagonal
subsequence selected by first letting $\varepsilon\to0$ for each fixed
$\tau>0$, and then letting $\tau\to0$, whose limit has all the properties
asserted in Theorem~\ref{thm:half-space-existence}.
In addition, the velocity satisfies
\begin{equation}\label{eq:eventual-global-kinetic}
 \frac12\|u(t)\|_{L^2}^2+\int_s^t\|\nabla u\|_{L^2}^2
 \le \frac12\|u(s)\|_{L^2}^2
       -\int_s^t\int_{\mathbb H^3} n\nabla\phi\cdot u
\end{equation}
for $s=0$ and for almost every $s>0$, and every $t\ge s$.

At the intermediate, fixed-$\tau$ level, the scalar equations have the
same transport velocity $u^\tau$, and, for every $0<\delta<T<\infty$,
\begin{equation}\label{eq:eventual-fixed-tau-scalar-class}
\begin{aligned}
 &n^\tau\in L^\infty((\delta,T);L^1\cap L^4)
       \cap L^2((\delta,T);H^1),\\
 &c^\tau\in L^\infty(\mathbb H^3\times(0,T)),\qquad
 \nabla c^\tau\in L^4(\mathbb H^3\times(0,T)).
\end{aligned}
\end{equation}
The constants in this assertion may depend on $\delta,T,\tau$.
The density bounds in \eqref{eq:eventual-fixed-tau-scalar-class}
are asserted for $n^\tau$ at fixed $\tau$, not as uniform bounds for
the final density $n$.
\end{lemma}

\begin{proof}
The local convergence and localized energy inequality follow from
Section~\ref{sec:existence}. We explain the additional estimates
in three steps: fixed-$\tau$ scalar regularity, density tightness,
and global velocity compactness. These retain the scalar tests
and the kinetic-energy inequality for the same selected limit.

\emph{Step 1: scalar regularity before removing the truncation.}
Fix $\tau>0$ and retain the superscript $(\varepsilon,\tau)$ throughout
this step. All estimates below are uniform in $\varepsilon$ at this
fixed $\tau$; the density bounds may depend on $\tau$.
Testing the concentration equation with $c^{\varepsilon,\tau}$ gives
\[
 \sup_{t<T}\|c^{\varepsilon,\tau}(t)\|_{L^2}^2
 +2\int_0^T\|\nabla c^{\varepsilon,\tau}\|_{L^2}^2
 \le\|c_0\|_{L^2}^2.
\]
Since $a_\tau\le\tau^{-1}$, testing the density equation with
$n^{\varepsilon,\tau}$ gives
\[
 \frac{d}{dt}\|n^{\varepsilon,\tau}\|_{L^2}^2
 +\|\nabla n^{\varepsilon,\tau}\|_{L^2}^2
 \le\tau^{-2}\|\nabla c^{\varepsilon,\tau}\|_{L^2}^2.
\]
The half-space Nash inequality and
$\|n^{\varepsilon,\tau}(t)\|_{L^1}=M$ imply
\[
 \frac{d}{dt}\|n^{\varepsilon,\tau}\|_{L^2}^2
 +c_M\|n^{\varepsilon,\tau}\|_{L^2}^{10/3}
 \le\tau^{-2}\|\nabla c^{\varepsilon,\tau}\|_{L^2}^2.
\]
Here and below $c_M>0$ depends only on $M$ and dimensional constants.
Compare the squared $L^2$ norm with the solution of the
corresponding unforced differential equation plus the time integral
of the nonnegative forcing. This sum is a supersolution, and
integration of the unforced equation gives
\begin{equation}\label{eq:eventual-forced-ode-comparison}
 \|n^{\varepsilon,\tau}(t)\|_{L^2}^2
 \le C_Mt^{-3/2}+C\tau^{-2}\|c_0\|_{L^2}^2.
\end{equation}
The estimate is independent of the approximate initial $L^2$ norm.
Integrating the preceding density-energy inequality from a positive
time gives, for $0<\delta<T$,
\[
 \sup_{\delta<t<T}\|n^{\varepsilon,\tau}(t)\|_{L^2}^2
 +\int_\delta^T\|\nabla n^{\varepsilon,\tau}\|_{L^2}^2
 \le C_{\delta,T,\tau}.
\]
No uniform bound for $\|n_0^\varepsilon\|_{L^2}$ is used.

For the additional $L^4$ regularity, testing with
$(n^{\varepsilon,\tau})^3$ gives
\[
\begin{aligned}
 \frac{d}{dt}\|n^{\varepsilon,\tau}\|_{L^4}^4
 +\frac32\|\nabla(n^{\varepsilon,\tau})^2\|_{L^2}^2
 \le
 6\tau^{-2}\|n^{\varepsilon,\tau}\|_{L^4}^2
       \|\nabla c^{\varepsilon,\tau}\|_{L^4}^2.
\end{aligned}
\]
Mass interpolation and the Sobolev inequality give
\[
 \|\nabla(n^{\varepsilon,\tau})^2\|_{L^2}^2
 \ge c_M\|n^{\varepsilon,\tau}\|_{L^4}^{44/9}.
\]
Dividing the resulting inequality by
$2\|n^{\varepsilon,\tau}\|_{L^4}^2$, with positive regularization
at its zeros, gives
\[
 \frac{d}{dt}\|n^{\varepsilon,\tau}\|_{L^4}^2
 +c_M\|n^{\varepsilon,\tau}\|_{L^4}^{26/9}
 \le C\tau^{-2}\|\nabla c^{\varepsilon,\tau}\|_{L^4}^2.
\]
Proposition~\ref{prop:uniform-energy} gives
\[
 \int_0^T\|\nabla c^{\varepsilon,\tau}\|_{L^4}^2
 \le T^{1/2}
       \left(\int_0^T\|\nabla c^{\varepsilon,\tau}\|_{L^4}^4\right)^{1/2}
 \le C_T,
\]
with $C_T$ independent of both parameters. Applying the same
differential-inequality comparison as in
\eqref{eq:eventual-forced-ode-comparison} to the squared $L^4$ norm yields
\[
 \|n^{\varepsilon,\tau}(t)\|_{L^4}^2
 \le C_Mt^{-9/4}+C_T\tau^{-2},\qquad 0<t<T.
\]
In particular,
\[
 \sup_{0<\varepsilon<1}
 \|n^{\varepsilon,\tau}\|_{L^\infty((\delta,T);L^4)}
 \le C_{\delta,T,\tau}.
\]
We now let $\varepsilon\to0$ at this fixed $\tau$. Weak lower
semicontinuity and the local strong convergence from
Section~\ref{sec:existence} give
\eqref{eq:eventual-fixed-tau-scalar-class} for
$(n^\tau,c^\tau,u^\tau)$.
The approximation property of $\mathcal J_\varepsilon$ gives
\[
 \|\mathcal J_\varepsilon u^{\varepsilon,\tau}
             -u^{\varepsilon,\tau}\|_{L^2(0,T;L^2)}
 \le C\varepsilon^{1/2}
       \|\nabla u^{\varepsilon,\tau}\|_{L^2(0,T;L^2)}
 \longrightarrow0.
\]
Consequently the limiting scalar equations are
\[
\begin{aligned}
 \partial_t n^\tau+u^\tau\cdot\nabla n^\tau-\Delta n^\tau
   &=-\nabla\cdot\bigl(a_\tau(n^\tau)\nabla c^\tau\bigr),\\
 \partial_t c^\tau+u^\tau\cdot\nabla c^\tau-\Delta c^\tau
   &=-h_\tau(n^\tau)c^\tau.
\end{aligned}
\]
The $L^4$ and $L^2H^1$ density bounds in this step are only
fixed-$\tau$ bounds. They justify the later scalar tests at that
level; they are not passed directly to the final density $n$
as $\tau\to0$.

\emph{Step 2: tightness of density and convergence of the force.}

For the cutoff estimates in this step, suppress the superscripts
$(\varepsilon,\tau)$. The bounds from Section~\ref{sec:existence}
are uniform in both parameters on each finite time interval:
\[
 n\in L^\infty_tL^1_x\cap L^{5/3}_{t,x},\qquad
 u,\mathcal J_\varepsilon u,\nabla c\in L^2_tL^6_x.
\]
We use these bounds here, rather than the fixed-$\tau$ estimates
from Step~1.
For the last assertion, note explicitly that
\[
 D^2c=2\sqrt{1+c}\,D^2\sqrt{1+c}
       +2\nabla\sqrt{1+c}\otimes\nabla\sqrt{1+c}
       \in L^2_{t,x},
\]
and use the half-space Sobolev inequality for $\nabla c$.
Mass interpolation also gives $n\in L^2_tL^{6/5}_x$.
Take a radial exterior cutoff $\eta_R$, equal to zero on $B_R$ and
one outside $B_{2R}$, with $|D^j\eta_R|\le CR^{-j}$ and
$\partial_\nu\eta_R=0$ on $\Gamma$.  The scalar equation and
$a_\tau(n)\le n$ yield
\begin{align*}
 \sup_{0<t<T}\int_{\mathbb H^3} n(t)\eta_R
 &\le \int_{\mathbb H^3} n_0^\varepsilon\eta_R+\frac{CMT}{R^2}
       +\frac C R\int_0^T\int_{\mathbb H^3}n(|\mathcal J_\varepsilon u|+|\nabla c|)\\
 &\le \int_{\mathbb H^3} n_0^\varepsilon\eta_R+\frac{C_T}{R}.
\end{align*}
Thus mass is uniformly tight in space.
Combining this fact with local strong convergence gives
$n^{\varepsilon_j,\tau}\to n^\tau$ at fixed $\tau$, and subsequently
$n^{\tau_j}\to n$, strongly in $L^r(\mathbb H^3\times(0,T))$
for every $1\le r<5/3$.
To include times near zero, interpolate the uniformly bounded
$L^{5/3}_{t,x}$ norm with the $L^1$ norm on $\mathbb H^3\times(0,\delta)$,
which is at most $M\delta$; this makes each such $L^r$ norm
uniformly small as $\delta\downarrow0$.
More precisely, interpolation
on a measurable set $E\subset\mathbb H^3$ gives
\[
 \int_0^T\|n\mathbf1_E\|_{L^{6/5}}^2
 \le C_T\bigl(\sup_{t<T}\|n(t)\mathbf1_E\|_{L^1}\bigr)^{7/6}
             \left(\int_0^T\int_{\mathbb H^3} n^{5/3}\right)^{1/2}.
\]

On bounded sets, the same convergence is strong in
$L^2_tL^{6/5}_x$. Indeed, the interpolation inequality
\[
 \|f(t)\|_{L^{6/5}}^2
 \leq\|f(t)\|_{L^1}\|f(t)\|_{L^{3/2}},
\]
applied to the difference of a density and its limit, combines the
uniform $L^\infty_tL^1_x$ bound with strong $L^{3/2}_{t,x}$
convergence. The preceding tail estimate extends this convergence
to the whole half-space at each of the two limiting stages.

Thus, with all convergences strong in
$L^2((0,T);L^{6/5}(\mathbb H^3))$,
\begin{equation}\label{eq:eventual-global-force-compactness}
\begin{aligned}
 n^{\varepsilon_j,\tau}&\longrightarrow n^\tau,&
 f^{\varepsilon_j,\tau}&\longrightarrow n^\tau\nabla\phi
 &&(\tau\text{ fixed}),\\
 n^{\tau_j}&\longrightarrow n,&
 n^{\tau_j}\nabla\phi&\longrightarrow n\nabla\phi
 &&(\tau_j\downarrow0).
\end{aligned}
\end{equation}
For the first force convergence, use
$f^{\varepsilon_j,\tau}=\mathcal K_{\varepsilon_j}
(n^{\varepsilon_j,\tau}\nabla\phi)$, boundedness and strong continuity
of $\mathcal K_\varepsilon$ on $L^{6/5}$, and
$\nabla\phi\in L^\infty$. The second passage uses the same uniform
density-tail estimates and has no force-smoothing operator.

\emph{Step 3: global velocity compactness, including the initial time.}

At the first passage, fix $\tau$ and write
\[
 (n_j,u_j,b_j,p_j,f_j)
 =(n^{\varepsilon_j,\tau},u^{\varepsilon_j,\tau},
   \mathcal J_{\varepsilon_j}u^{\varepsilon_j,\tau},
   p^{\varepsilon_j,\tau},f^{\varepsilon_j,\tau}).
\]
At the second passage the corresponding sequence is
\[
 (n_j,u_j,b_j,p_j,f_j)
 =(n^{\tau_j},u^{\tau_j},u^{\tau_j},p^{\tau_j},
   n^{\tau_j}\nabla\phi).
\]
Only within Step~3, unindexed solution variables denote the limit of
the sequence currently under consideration: the fixed-$\tau$ limit
in the first passage and the final selected limit in the second.
The initial velocity is $u_0^{\varepsilon_j}$ in the first passage
and exactly $u_0$ in the second. All constants used below are
uniform in the remaining approximation parameters.

We include this step because local strong convergence of velocity alone
would not suffice to pass the kinetic energy at almost every starting
time.  The global Stokes decomposition in the proof of
Proposition~\ref{prop:uniform-pressure} permits the normalization
\[
 p_j=p_{1j}+p_{2j},\qquad
 \|p_{1j}\|_{L^{5/3}(\mathbb H^3\times(0,T))}\le C_T.
\]
Indeed, $\nabla p_{1j}\in L^{5/3}_tL^{15/14}_x$, and the homogeneous
Sobolev inequality on $\mathbb H^3$ selects the unique representative in
$L^{5/3}(\mathbb H^3)$.  At the first limit, the second part is the pressure generated by
$e^{-tA}u_0^{\varepsilon_j}$; choose its representative in
$L^6(\mathbb H^3)$. At the second limit it is the same pressure
generated by $e^{-tA}u_0$ for every $j$. Thus, for every $0<\delta<T$,
the first limit satisfies
\[
 \|p_{2j}(t)\|_{L^6}\le C t^{-1}\|u_0^{\varepsilon_j}\|_{L^2},
 \qquad
 p_{2j}\longrightarrow p_2
       \quad\hbox{in }L^1((\delta,T);L^6(\mathbb H^3)).
\]
The latter convergence follows from linearity and
$u_0^{\varepsilon_j}\to u_0$ in $L^2$.

There is uniform strong continuity at the genuine initial time along
the approximation sequence.  In fact, interpolation gives
\[
 \int_0^t\|n_j\|_{L^{6/5}}^2
 \le C t^{1/2}
      \left(\int_0^t\int_{\mathbb H^3}n_j^{5/3}\right)^{1/2}
 \le C t^{1/2}\qquad(0<t<1).
\]
The kinetic-energy identity, or the inequality retained after the
first limit, and Young's inequality imply
\[
 \|u_j(t)\|_{L^2}^2\le\|u_j(0)\|_{L^2}^2+C t^{1/2}.
\]
The equations also bound $\partial_tu_j$ uniformly in
\[
 L^{4/3}\bigl(0,T;
 (H^1_0(\mathbb H^3)\cap L^2_\sigma(\mathbb H^3))'\bigr).
\]
The convection term is controlled by
$b_j,u_j\in L^{8/3}_tL^4_x$, and the force by
\eqref{eq:eventual-global-force-compactness}. Testing against fixed
solenoidal $H^1_0$ functions and using their density in
$L^2_\sigma$ proves uniform weak continuity at zero.  Together with the last
upper bound on the norm, this proves
\begin{equation}\label{eq:eventual-uniform-initial-continuity}
 \lim_{\delta\downarrow0}\limsup_{j\to\infty}
       \sup_{0<t<\delta}\|u_j(t)-u_0\|_{L^2}=0.
\end{equation}

For the exterior kinetic-energy test, first multiply $\eta_R$ by
$1-\eta_{2^mR}$ and then let $m\to\infty$. The energy and convection
bounds and the global pressure bounds just established remove the
outer cutoff errors; the $p_{2j}$ error is controlled by its
$L^1((\delta,T);L^6)$ spatial tail, as below.
Apply the approximate local kinetic-energy identity with the exterior
cutoff $\eta_R$ on $(\delta,t)$, choosing $\delta$ from the
full-measure set of admissible initial times when working with a limit.  Diffusion and convection cutoff errors
are bounded by
\[
 CR^{-2}\int_\delta^T\|u_j\|_{L^2}^2
 +CR^{-1}\int_\delta^T\int_{\mathbb H^3} |b_j||u_j|^2=o_R(1)
\]
uniformly in $j$.  The $p_{1j}$ contribution is at most
\[
 CR^{-1}\|p_{1j}\|_{L^{5/3}_{t,x}}
           \|u_j\|_{L^{5/2}_{t,x}}=o_R(1).
\]
Since the cutoff annulus has volume at most $CR^3$, the remaining pressure
term is bounded by
\[
 C\|u_j\|_{L^\infty_tL^2_x}
    \int_\delta^T\|p_{2j}(t)\|_{L^6(\mathbb H^3\setminus B_R)}\,dt=o_R(1).
\]
Here uniformity follows from the strong convergence of $p_{2j}$ on
$(\delta,T)$.  Finally the force term is bounded by
\[
 C\|f_j\|_{L^2((\delta,T);L^{6/5}(\mathbb H^3\setminus B_R))}
     \|u_j\|_{L^2((\delta,T);L^6)}=o_R(1)
\]
by \eqref{eq:eventual-global-force-compactness}.  Consequently,
\[
 \sup_{\delta<t<T}\int_{\mathbb H^3\setminus B_{2R}}|u_j(t)|^2
 \le C\int_{\mathbb H^3}|u_j(\delta)|^2\eta_R+o_R(1).
\]
For each fixed admissible $\delta>0$, take first $\limsup_{j\to\infty}$
and then $R\to\infty$. The initial tail is bounded by
$C\limsup_j\|u_j(\delta)-u_0\|_{L^2}^2$, since $u_0\in L^2$.
The part $0<t<\delta$ is controlled directly by
\eqref{eq:eventual-uniform-initial-continuity}. Letting $\delta\downarrow0$
through common admissible times proves the precise tightness statement
\[
 \lim_{R\to\infty}\limsup_{j\to\infty}
 \operatorname*{ess\,sup}_{0<t<T}
 \int_{\mathbb H^3\setminus B_R}|u_j(t)|^2=0.
\]
At the second limit, the homogeneous pressure is already the same
pressure generated by $e^{-tA}u_0$ for every $\tau$.
The time-derivative bound also selects a representative in
$C_{\rm w}([0,T];L^2_\sigma(\mathbb H^3))$ for each limit.
Local compactness therefore improves to global strong convergence in
$L^2(\mathbb H^3\times(0,T))$.  A subsequence converges strongly in $L^2(\mathbb H^3)$ at
almost every time.

We may now pass to the kinetic-energy inequality from each such starting
time $s$: the force work converges by
\eqref{eq:eventual-global-force-compactness} and weak convergence of
$u_j$ in $L^2_tL^6_x$, the energy at $s$ converges strongly, and the
terminal energy and dissipation are lower semicontinuous.  Weak continuity
of the limit extends the inequality to every terminal time $t\ge s$.
At $s=0$ use strong convergence of the initial data.  This argument applies
first at fixed $\tau$ and then to a countable sequence $\tau\downarrow0$.
The same passage also retains the local kinetic-energy inequality for
the fixed-$\tau$ velocity: the pressure--velocity product converges by
weak $L^{5/3}$ convergence of pressure and strong local $L^{5/2}$
convergence of velocity.  Thus the exterior-cutoff argument remains
available at the second limit, even though its input is then an
inequality instead of the smooth identity.  Intersecting the
corresponding full-measure sets of starting times proves
\eqref{eq:eventual-global-kinetic} for the final limit.

Finally, realize the iterated limit $(n,c,u,p)$ by a diagonal
sequence of the original smooth approximants.
For a convergent sequence $\tau_j\downarrow0$, choose
$\varepsilon_j<1/j$ so that the original approximant is within $1/j$
of the fixed-$\tau_j$ limit in the global strong norms above and in
the strong local norms of Section~\ref{subsec:compactness-local-energy}
on the first $j$ members of a fixed compact exhaustion. Impose the
same accuracy for pressure pairings against the first $j$ members
of countable dense families of dual tests. These finitely many
requirements can be met simultaneously. This diagonal sequence has
the same limit in all the topologies used in that section, so its
original localized entropy--energy argument proves suitability of
this precise limit. The additional global and scalar estimates are
retained for the same selection.
\end{proof}

\subsection{Weighted density estimates and dissipation tails}
\label{subsec:eventual-scalar-decay}

The signal becomes uniformly small after a fixed time. We use this
smallness in a weighted density test, performed first on the
fixed-$\tau$ system selected above. The resulting bounds are uniform
in $\tau$ and therefore pass to the final solution.

\begin{lemma}[Eventual scalar estimates]\label{lem:ev-scalar-decay}
For the selection in Lemma~\ref{lem:ev-selection},
\[
 \|c(t)\|_{L^{\infty}}\le C t^{-3/2}\|c_0\|_{L^1}.
\]
There exists $T_0<\infty$ such that, for $q=2,6$ and almost every
$t\ge T_0+1$,
\begin{equation}\label{eq:ev-density-decay}
 \|n(t)\|_{L^q}\le C_q M(t-T_0)^{-\frac32(1-1/q)}.
\end{equation}
Moreover, for every $t\geq T_0+1$,
\begin{equation}\label{eq:ev-density-tail}
 \int_t^\infty\|\nabla n(s)\|_{L^2}^2\,ds
       \le C M^2(t-T_0)^{-3/2}.
\end{equation}
\end{lemma}

\begin{proof}
Lemma~\ref{lem:ev-signal} supplies the stated signal estimate,
uniformly at every approximation level.

For $q\in\{2,6\}$ set
\[
 d_q^2=\frac1{10q^2(q-1)},\qquad
 \rho_q(z)=(1-z/d_q)^{-1/(10q)}.
\]
Choose $T_0$ so that all the fixed-$\tau$ limits satisfy
$0\le c^\tau\le\frac12\min\{d_2,d_6\}$ for $t\ge T_0$.
Then $1\le\rho_q(c^\tau)\le2^{1/(10q)}$.
We justify the joint test on the fixed-$\tau$ system, where the two scalar
transport velocities coincide.  This order is essential: applying it
directly before $\varepsilon\to0$ would leave a term containing
$(\mathcal J_\varepsilon u^{\varepsilon,\tau}
-u^{\varepsilon,\tau})\cdot\nabla c^{\varepsilon,\tau}$ multiplied by
a high power of $n^{\varepsilon,\tau}$, for which no uniform estimate
has been proved.

For $\lambda>0$ put
\[
 \Psi_{q,\lambda}(s)
   =q\int_0^s\frac{z^{q-1}}{1+\lambda z^{q-9/5}}\,dz.
\]
For fixed $q$ and $\lambda$, write $\Psi=\Psi_{q,\lambda}$ in the
following scalar identities. These functions increase to $s^q$ as
$\lambda\downarrow0$, have growth at infinity of order $s^{9/5}$,
and satisfy
\[
 \frac{(\Psi')^2}{\Psi\Psi''}\le\frac{5q}{4},\qquad
 \frac{s^2\Psi''}{\Psi}\le q(q-1).
\]
These are the elementary one-variable estimates of
Winkler~\cite[Lemma~6.1]{Winkler2017}.  Since
\[
 \frac{\rho_q'^2}{\rho_q\rho_q''}
    =\frac1{10q+1},\qquad
 \frac{\rho_q}{\rho_q''}
    =\frac{100q^2(d_q-z)^2}{10q+1},
\]
we have, for $s>0$ and $0\le z\le d_q/2$,
\begin{equation}\label{eq:eventual-weight-structural}
 4(\Psi')^2(\rho_q')^2+s^2(\Psi'')^2\rho_q^2
 \le2\Psi\Psi''\rho_q\rho_q''.
\end{equation}
Indeed, after division by $\Psi\Psi''\rho_q\rho_q''$, the left side is
at most $15q/(10q+1)<2$.

For the joint-test calculation and its justification, up to the
removal of the spatial cutoffs below, fix $\tau>0$ and abbreviate
\[
 (n,c,u)=(n^\tau,c^\tau,u^\tau).
\]
In this part, the integrability bounds used to justify testing may
depend on $\tau$ and $\lambda$. The constants in the structural
inequality and the resulting smoothing bound do not.
The formal joint identity is

\begin{align*}
 \frac{d}{dt}\int_{\mathbb H^3}\Psi(n)\rho_q(c)
 ={}&-\int_{\mathbb H^3}\Psi''\rho_q|\nabla n|^2
     +\int_{\mathbb H^3}(a_\tau(n)\Psi''\rho_q-2\Psi'\rho_q')\nabla n\cdot\nabla c\\
    &-\int_{\mathbb H^3}(\Psi\rho_q''-a_\tau(n)\Psi'\rho_q')|\nabla c|^2
     -\int_{\mathbb H^3}h_\tau(n)c\Psi\rho_q'.
\end{align*}
Young's inequality, $0\le a_\tau(n)\le n$, and
\eqref{eq:eventual-weight-structural} leave one half of the first
dissipation on the left.  More explicitly, the coefficient produced
by absorbing the mixed term is bounded by
\[
 \frac{(a_\tau(n)\Psi''\rho_q-2\Psi'\rho_q')^2}{2\Psi''\rho_q}
       +a_\tau(n)\Psi'\rho_q'
 \le \frac{n^2\Psi''\rho_q}{2}
       +\frac{2(\Psi')^2(\rho_q')^2}{\Psi''\rho_q}
 \le\Psi\rho_q''.
\]

The resulting time-integrated inequality is legitimate under
\eqref{eq:eventual-fixed-tau-scalar-class}. One uses paired Steklov
time averages in the two scalar equations and then lets the averaging
length tend to zero, as in \cite[Lemma~5.1]{Winkler2017}. The required
products have the same integrability: $n,\nabla c\in L^4_{t,x}$,
$\nabla n\in L^2_{t,x}$, and $u\in L^{10/3}_{t,x}$ on finite positive
time intervals, so that $n h_\tau'(n)\nabla c\in L^2$,
$u\cdot\nabla n\in L^{5/4}$, and $u\cdot\nabla c\in L^{20/11}$.
For fixed $\lambda$, $\Psi''$ is bounded, and the test functions satisfy
$\Psi'(n)\in L^2\cap L^4\cap L^5$,
$a_\tau(n)\Psi''(n)\in L^4\cap L^5$, and
$\Psi(n)\in L^1\cap L^2\cap L^{20/9}$ on these space--time slabs.
In particular, $\nabla(\Psi'(n)\rho_q(c))\in L^2$, while the
concentration-test gradients converge in $L^{4/3}$ under time averaging.
For example, the terms $\Psi'(n)\nabla n\cdot\nabla c$,
$u\Psi'(n)\cdot\nabla n$, and $u\Psi(n)\cdot\nabla c$ are
integrable by H\"older's inequality with reciprocal exponents
$(1/4,1/2,1/4)$, $(3/10,1/5,1/2)$, and
$(3/10,9/20,1/4)$, respectively. Also,
$\Psi(n)|\nabla c|^2\in L^1$, and the reaction term is integrable
since $h_\tau(n)\le n$ and $n\in L^1\cap L^4$.
At zeros of the density, the tests are justified first by smooth
nonnegative regularization on bounded cylinders and then by
dominated convergence; no nonzero constant is integrated over $\mathbb H^3$.

The passage from bounded spatial cutoffs to $\mathbb H^3$ introduces no extra
term.  The common transport combines into
$\nabla\cdot(u\Psi(n)\rho_q(c))$; its flux belongs to $L^1$ on every
finite positive time interval because $u\in L^\infty_tL^2_x$,
$n\in L^\infty_t(L^1\cap L^4)$, and
$\Psi(n)\in L^\infty_tL^2_x$.  The other cutoff fluxes are integrable
for the same reason, using $\nabla n,\nabla c\in L^2_{t,x}$ and the
growth of $\Psi$ and $\Psi'$.  Their cutoff derivatives tend to zero.
The physical-boundary terms vanish by the weak Neumann conditions and
the zero velocity trace. We now restore the superscript $\tau$.
For the truncated weighted integral
\[
 F_{q,\lambda}^{\tau}(t)
 :=\int_{\mathbb H^3}\Psi_{q,\lambda}(n^\tau(t))\rho_q(c^\tau(t)),
\]
the justified calculation gives
\[
 F_{q,\lambda}^{\tau}(t)+\frac12\int_s^t\int_{\mathbb H^3}
       \Psi_{q,\lambda}''(n^\tau)\rho_q(c^\tau)|\nabla n^\tau|^2
 \le F_{q,\lambda}^{\tau}(s).
\]
To obtain smoothing, we first estimate the truncated integral.
Its finite starting value will not need a uniform bound.

At almost every $s_\tau\in(T_0,T_0+1/4)$ the density belongs to
$L^1\cap L^2$, so that $F_{q,\lambda}^{\tau}(s_\tau)<\infty$ for every fixed
$\lambda>0$.  Choose one such time for a countable sequence
$\lambda\downarrow0$ and both $q=2,6$.  The derivative-ratio estimate gives
\[
 \|\nabla\sqrt{\Psi_{q,\lambda}(n^\tau)}\|_{L^2}^2
 \le\frac{5q}{16}\int_{\mathbb H^3}\Psi_{q,\lambda}''(n^\tau)|\nabla n^\tau|^2,
 \qquad
 \int_{\mathbb H^3} \Psi_{q,\lambda}(n^\tau)^{1/q}\le\int_{\mathbb H^3} n^\tau\le M.
\]
Only ordinary H\"older and Sobolev inequalities are needed; no
Banach-space property below exponent one is used. Indeed,
\[
\begin{aligned}
 \int_{\mathbb H^3}\Psi_{q,\lambda}(n^\tau)
 &\le\left(\int_{\mathbb H^3}\Psi_{q,\lambda}(n^\tau)^{1/q}\right)^{2q/(3q-1)}
       \left(\int_{\mathbb H^3}\Psi_{q,\lambda}(n^\tau)^3\right)^{(q-1)/(3q-1)}\\
 &\le C M^{2q/(3q-1)}
       \|\nabla\sqrt{\Psi_{q,\lambda}(n^\tau)}\|_{L^2}^{6(q-1)/(3q-1)}.
\end{aligned}
\]
Consequently the truncated joint inequality gives
\[
 (F_{q,\lambda}^{\tau})'+c_q M^{-2q/[3(q-1)]}
       (F_{q,\lambda}^{\tau})^{1+2/[3(q-1)]}\le0.
\]
This holds in distributions, and hence for the nonincreasing
representative of $F_{q,\lambda}^{\tau}$; downward jumps only strengthen the
comparison.  It follows that
\[
 F_{q,\lambda}^{\tau}(t)\le C_q M^q(t-s_\tau)^{-3(q-1)/2},\qquad t>s_\tau,
\]
independently of $\lambda,\tau$ and of the size of
$F_{q,\lambda}^{\tau}(s_\tau)$.  Now monotone convergence removes $\lambda$.

For $t\ge T_0+1$, $t-s_\tau\ge\frac34(t-T_0)$, giving
\[
 \|n^\tau(t)\|_{L^q}
 \le C_qM(t-T_0)^{-\frac32(1-1/q)},\qquad q=2,6,
\]
with $C_q$ and $T_0$ independent of $\tau$. These bounds will pass to the final density $n$.

The weighted starting integrals are now finite at almost every
late time. Monotone convergence and Fatou's lemma therefore pass the
truncated energy inequality to
\begin{equation}\label{eq:eventual-weighted-density}
 \begin{aligned}
 &\int_{\mathbb H^3}(n^\tau(t))^q\rho_q(c^\tau(t))
 +\frac{2(q-1)}q\int_s^t\|\nabla(n^\tau)^{q/2}\|_{L^2}^2\\
 &\qquad\leq
 \int_{\mathbb H^3}(n^\tau(s))^q\rho_q(c^\tau(s)).
 \end{aligned}
\end{equation}
This also holds for almost every pair $T_0<s<t$ whenever its
right-hand side is finite. In particular, the $q=2$ inequality gives
\begin{equation}\label{eq:eventual-density-tail}
 \int_t^\infty\|\nabla n^\tau\|_{L^2}^2
 \leq\int_{\mathbb H^3}(n^\tau(t))^2\rho_2(c^\tau(t))
 \le C M^2(t-T_0)^{-3/2}.
\end{equation}

Finally, let $\tau_j\downarrow0$ along the selection of
Lemma~\ref{lem:ev-selection}, with $n^{\tau_j}\to n$ locally strongly.
Global Fatou lower semicontinuity and weak lower semicontinuity
of the gradient pass these $\tau$-uniform estimates to the final
limit, proving \eqref{eq:ev-density-decay} and
\eqref{eq:ev-density-tail} for $n$.

The tail estimate, first obtained for almost every starting time,
extends to every $t\geq T_0+1$ by continuity of the integral on the
left and of its upper bound. No convergence of a possibly unbounded
starting value is required.
The compactness, energy inequality, and suitability from
Lemma~\ref{lem:ev-selection} are retained along this same diagonal
subsequence.
\end{proof}

\subsection{An integrable effective force on the half-space}

The density dissipation tail has two consequences for the velocity:
it makes the effective force integrable in time, and it bounds the
kinetic energy on an infinite time interval. Here boundedness of
$\phi$ is needed for the first time. Together with the standing
gradient bound, we now assume
$\phi\in W^{1,\infty}(\mathbb H^3)$.

\begin{lemma}[Time-integrable density gradient and effective force]
\label{lem:ev-force-tail}
Assume additionally that $\phi\in L^\infty(\mathbb H^3)$, and let
$(n,c,u,p)$ be the solution selected in
Lemma~\ref{lem:ev-selection}. For $R\geq1$,
\begin{equation}\label{eq:ev-gradient-L1-tail}
\int_{T_0+R}^{\infty}\|\nabla n(t)\|_{L^2}\,dt
\leq C M R^{-1/4}.
\end{equation}
In particular, with
\begin{equation}\label{eq:ev-effective-force}
g:=\mathbb P_+(\phi\nabla n),
\end{equation}
the velocity equation on $(T_0+1,\infty)$ can be written as
\begin{equation}\label{eq:ev-projected-NS}
u_t+Au+\mathbb P_+(u\cdot\nabla u)=g,
\end{equation}
where
\begin{equation}\label{eq:ev-force-integrability}
g\in L^1(T_0+1,\infty;L^2_\sigma(\mathbb H^3))
\cap L^2(T_0+1,\infty;L^2_\sigma(\mathbb H^3)).
\end{equation}
There is a good time $S>T_0+1$ such that
\begin{equation}\label{eq:ev-velocity-finite-energy}
\operatorname*{ess\,sup}_{t\geq S}\|u(t)\|_{L^2}^2
+\int_S^\infty\|\nabla u(t)\|_{L^2}^2\,dt<\infty.
\end{equation}
\end{lemma}

\begin{proof}
By \eqref{eq:ev-density-tail} and Cauchy--Schwarz on the successive
intervals $(T_0+2^jR,T_0+2^{j+1}R)$,
\[
\begin{aligned}
\int_{T_0+R}^{\infty}\|\nabla n\|_{L^2}
&\leq\sum_{j=0}^\infty (2^jR)^{1/2}
\left(\int_{T_0+2^jR}^\infty\|\nabla n\|_{L^2}^2\right)^{1/2}\\
&\leq C M R^{-1/4}\sum_{j=0}^\infty2^{-j/4}.
\end{aligned}
\]
This proves \eqref{eq:ev-gradient-L1-tail}. The identity
\[
-n\nabla\phi=-\nabla(n\phi)+\phi\nabla n
\]
holds in distributions. Since $n(t)\in H^1(\mathbb H^3)$ for almost every
late time and $\phi\in W^{1,\infty}(\mathbb H^3)$, we have $n(t)\phi\in H^1(\mathbb H^3)$.
The Helmholtz projection therefore annihilates $\nabla(n\phi)$.
Its $L^2$ boundedness gives
\[
\|g(t)\|_{L^2}\leq\|\phi\|_{L^\infty}\|\nabla n(t)\|_{L^2}.
\]
Together with \eqref{eq:ev-density-tail} this proves
\eqref{eq:ev-force-integrability}. The gradient $\nabla(n\phi)$ disappears only in the projected
equation; the pressure in the suitable quadruple remains $p$.

For almost every pair of admissible times the global kinetic
energy inequality from Lemma~\ref{lem:ev-selection} becomes
\begin{equation}\label{eq:ev-kinetic-g}
\frac12\|u(t)\|_{L^2}^2+
\int_s^t\|\nabla u\|_{L^2}^2
\leq\frac12\|u(s)\|_{L^2}^2+\int_s^t(g,u)_{L^2}.
\end{equation}
Indeed $u\in H^1_0\cap L^2_\sigma$ implies
$-\int n\nabla\phi\cdot u=\int\phi\nabla n\cdot u$.
There is no boundary term, and a cutoff at infinity is justified
by $n,u\in L^2$ and $\phi\in W^{1,\infty}$.
Fix a good $S>T_0+1$. Since $g\in L^1_tL^2_x$, the energy
inequality on any finite interval yields
\[
\operatorname*{ess\,sup}_{S<t<B}\|u(t)\|_{L^2}
\leq \|u(S)\|_{L^2}+2\int_S^\infty\|g(t)\|_{L^2}\,dt.
\]
This follows by solving the quadratic inequality for the essential
supremum of the $L^2$ norm.
Return to \eqref{eq:ev-kinetic-g} and let $B\to\infty$ to obtain
\eqref{eq:ev-velocity-finite-energy}.
This step uses the bounded potential, not a Poincar\'e inequality
on the unbounded domain.
\end{proof}

\subsection{Restarting the velocity as a global strong solution}

\begin{lemma}[Eventual strong regularity of the velocity]
\label{lem:ev-velocity}
For the selected solution there is $t_*>S$ such that
\begin{equation}\label{eq:ev-velocity-strong}
\begin{aligned}
&u\in C([t_*,\infty);H^1_0(\mathbb H^3)\cap L^2_\sigma(\mathbb H^3)),\\
\sup_{t\geq t_*}&\|\nabla u(t)\|_{L^2}^2
+\int_{t_*}^\infty\|Au(t)\|_{L^2}^2\,dt<\infty.
\end{aligned}
\end{equation}
In particular $u\in L^\infty(t_*,\infty;L^6(\mathbb H^3))$.
\end{lemma}

\begin{proof}
Keep the force $g$ from \eqref{eq:ev-effective-force} fixed.
For a good restart time $s\geq S$ with
$u(s)\in H^1_0\cap L^2_\sigma$, let $v$ be the local strong solution of
\[
 v_t+Av+\mathbb P_+(v\cdot\nabla v)=g,\qquad v(s)=u(s).
\]
The energy argument in \eqref{eq:ev-kinetic-g} gives
\[
 \|v(t)\|_{L^2}
 \leq\operatorname*{ess\,sup}_{r\geq S}\|u(r)\|_{L^2}
       +2\int_S^\infty\|g(r)\|_{L^2}\,dr\leq C
\]
throughout its existence interval, with a constant independent of
the restart time.
Local strong solvability and the $H^1$ continuation criterion
follow from the half-space Stokes estimates
\cite{DHP2001,DHP2003}; here $g\in L^2_{\rm loc}L^2$.

The homogeneous Sobolev and
Stokes elliptic estimates give
\[
\|v\|_{L^3}\leq C\|v\|_{L^2}^{1/2}
\|\nabla v\|_{L^2}^{1/2},
\qquad
\|\nabla v\|_{L^6}\leq C\|\nabla^2v\|_{L^2}
\leq C\|Av\|_{L^2}.
\]
Testing by $Av$ thus gives
\begin{equation}\label{eq:ev-enstrophy-bootstrap}
\begin{aligned}
\frac12\frac{d}{dt}\|\nabla v\|_{L^2}^2
+\|Av\|_{L^2}^2
\leq C\|\nabla v\|_{L^2}^{1/2}\|Av\|_{L^2}^2+\frac14\|Av\|_{L^2}^2+\|g\|_{L^2}^2.
\end{aligned}
\end{equation}
Choose $\eta>0$ so small that $C\eta^{1/4}\leq1/4$.
Since the two integrals in
\eqref{eq:ev-force-integrability} and
\eqref{eq:ev-velocity-finite-energy} have vanishing tails,
one can select a good time $t_*>S$ for the kinetic energy
inequality such that $u(t_*)\in H^1_0\cap L^2_\sigma$ and
\begin{equation}\label{eq:ev-restart-smallness}
\|\nabla u(t_*)\|_{L^2}^2
+2\int_{t_*}^\infty\|g(t)\|_{L^2}^2\,dt\leq\frac\eta2.
\end{equation}

As long as $\|\nabla v(t)\|_{L^2}^2\leq\eta$,
\eqref{eq:ev-enstrophy-bootstrap} implies
\[
\|\nabla v(t)\|_{L^2}^2+\int_{t_*}^t\|Av\|_{L^2}^2
\leq\|\nabla v(t_*)\|_{L^2}^2
+2\int_{t_*}^t\|g\|_{L^2}^2\leq\frac\eta2.
\]
Continuity of $t\mapsto\|\nabla v(t)\|_{L^2}^2$ excludes a first exit
from $\|\nabla v(t)\|_{L^2}^2<\eta$.

The continuation criterion then yields a global strong solution
$v$ on $[t_*,\infty)$.

To identify $v$ with the selected weak velocity, test the difference of their equations.
The global energy inequality starting at the good time $t_*$,
the weak momentum identity and the strong equation for $v$
give the usual difference estimate, with the identical force
$g$ canceling. Its nonlinear term satisfies
\[
\begin{aligned}
\left|\int_{\mathbb H^3}((u-v)\cdot\nabla)v\cdot(u-v)\right|
&=\left|\int_{\mathbb H^3}((u-v)\cdot\nabla)(u-v)\cdot v\right|\\
&\leq C\|v\|_{L^6}\|u-v\|_{L^2}^{1/2}
\|\nabla(u-v)\|_{L^2}^{3/2}\\
&\leq\frac12\|\nabla(u-v)\|_{L^2}^2
+C\|v\|_{L^6}^4\|u-v\|_{L^2}^2.
\end{aligned}
\]
These tests are justified by time regularization and approximation
of the strong solution in the solenoidal energy space.
Since $v\in L^\infty L^6$ on each finite interval, Gronwall's
inequality and $u(t_*)=v(t_*)$ imply $u=v$.
This use of weak--strong uniqueness is for a forced
Navier--Stokes equation with a fixed force and does not assume
weak--strong uniqueness of the coupled scalar--fluid system.
The preceding estimates prove \eqref{eq:ev-velocity-strong}.
\end{proof}

\subsection{Uniform bounds and regularity up to the flat boundary}

\begin{lemma}[The final scalar--fluid regularity bootstrap]
\label{lem:ev-boundary-bootstrap}
There is $T_*>t_*$ such that
\begin{equation}\label{eq:ev-uniform-boundedness}
\sup_{t\geq T_*}
\bigl(\|n(t)\|_{L^\infty}
+\|c(t)\|_{W^{1,\infty}}+\|u(t)\|_{L^\infty}\bigr)<\infty.
\end{equation}
If, in addition,
$\nabla\phi\in C^{0,\beta_0}(\overline{\mathbb H^3})$ for some
$\beta_0\in(0,1)$, then after a further increase of $T_*$, for some
$\alpha\in(0,\beta_0)$,
\begin{equation}\label{eq:ev-classical-tail}
n,c,u\in C^{2+\alpha,1+\alpha/2}_{\rm loc}
\bigl(\overline{\mathbb H^3}\times(T_*,\infty)\bigr).
\end{equation}
The original equations and boundary conditions hold classically
there, with the corresponding local regularity of $\nabla p$.
\end{lemma}

\begin{proof}
Start after $\max\{t_*,T_0+1\}$. The preceding lemmas give
$u,n\in L^\infty_tL^6_x$. We first obtain global bounds for $u$,
then for $\nabla c$, and finally for $n$. All estimates use sliding
time intervals of a fixed length, so their waiting times are
independent of the spatial center. Local boundary estimates then
give the asserted classical regularity.

\emph{Step 1: bounded velocity.}
For $1<a\leq b\leq\infty$, in the ranges used below, the
half-space Stokes estimates are
\begin{align*}
\|e^{-rA}\mathbb P_+F\|_{L^b}
&\leq C r^{-\frac32(\frac1a-\frac1b)}\|F\|_{L^a},\\
\|e^{-rA}\mathbb P_+\nabla\cdot\mathcal F\|_{L^b}
&\leq C r^{-\frac12-\frac32(\frac1a-\frac1b)}
\|\mathcal F\|_{L^a},\qquad 0<r\leq1;
\end{align*}
the second estimate is understood by the Stokes divergence
solution operator. Its $L^a\to L^a$ bound follows by duality from
the gradient estimate for the Stokes semigroup; a semigroup
factorization and the homogeneous Sobolev estimates give the displayed
$L^a\to L^b$ bounds, including the exponents used here
\cite{DHP2001}.
Use the mild formula with the original force $-n\nabla\phi$.
On a sliding interval of length one, the mapping from
$u\otimes u\in L^\infty_tL^3_x$ to $L^{12}$ has kernel
$r^{-7/8}$, and the mapping from $n\nabla\phi\in L^\infty_tL^6_x$
has kernel $r^{-1/8}$. Therefore, after a unit delay,
\[
\begin{aligned}
\|u(s+1)\|_{L^{12}}
\leq C\bigl(&\|u\|_{L^\infty(s,s+1;L^6)}
+\|u\|_{L^\infty(s,s+1;L^6)}^2\\
&+\|\nabla\phi\|_{L^\infty}
\|n\|_{L^\infty(s,s+1;L^6)}\bigr).
\end{aligned}
\]
Apply the mild formula once more. Now $u\otimes u\in L^\infty L^6$,
and the two kernels for the $L^\infty$ estimate are $r^{-3/4}$
and $r^{-1/4}$. Both are integrable, so
\begin{equation}\label{eq:ev-u-infinity}
\begin{aligned}
\|u(s+2)\|_{L^\infty}
\leq C\bigl(&\|u\|_{L^\infty(s+1,s+2;L^{12})}
+\|u\|_{L^\infty(s+1,s+2;L^{12})}^2\\
&+\|\nabla\phi\|_{L^\infty}
\|n\|_{L^\infty(s+1,s+2;L^6)}\bigr).
\end{aligned}
\end{equation}
The two unit delays and the resulting bounds are uniform in $s$.

\emph{Step 2: bounded signal gradient.}
Use $\|c(t)\|_{L^\infty}\leq\|c_0\|_{L^\infty}$.
With $S_N$ the Neumann heat semigroup, the scalar mild formula on
$(s,s+h)$ gives
\[
c(s+r)=S_N(r)c(s)
-\int_0^r S_N(r-\sigma)
\bigl[u\cdot\nabla c+nc\bigr](s+\sigma)\,d\sigma.
\]
The gradient kernels for the two terms on the right are
$(r-\sigma)^{-1/2}$ on $L^\infty$ and
$(r-\sigma)^{-3/4}$ from $L^6$ to $L^\infty$.
Multiplying by $r^{1/2}$ and taking the supremum gives
\[
\begin{aligned}
 \sup_{0<r<h}r^{1/2}\|\nabla c(s+r)\|_{L^\infty}
 &\leq C\|c_0\|_{L^\infty}
 +Ch^{3/4}\|n\|_{L^\infty(s,s+h;L^6)}\|c_0\|_{L^\infty}\\
 &\quad+Ch^{1/2}\|u\|_{L^\infty(s,s+h;L^\infty)}
       \sup_{0<r<h}r^{1/2}\|\nabla c(s+r)\|_{L^\infty}.
\end{aligned}
\]
For completeness this estimate can first be applied to the
linear mild solution with bounded drift $u$ and given source
$-nc$. Its construction in the displayed weighted gradient
space follows by contraction for small $h$. The construction
is carried out simultaneously in
$C([s,s+h];L^2(\mathbb H^3))\cap
L^2(s,s+h;H^1(\mathbb H^3))$: here $c(s)\in L^2$ and
$nc\in L^\infty(s,s+h;L^2)$, and the $L^2$ energy estimate
is uniform along smooth approximations of the bounded drift.
Its equality with the weak concentration then follows from
the $L^2$ energy estimate for their difference. This justifies the
weighted gradient estimate without assuming it in advance for the
weak concentration. By the uniform velocity bound, choose
$h\in(0,1)$ so that the coefficient of the last supremum is at most
$1/2$ for every sufficiently late starting time $s$. Use sliding intervals of this
same length. This proves a uniform
$L^\infty$ bound for $\nabla c$ after another delay.

\emph{Step 3: bounded density.}
Write the density equation in flux form,
\[
n_t-\Delta n=-\nabla\cdot\bigl[n(u+\nabla c)\bigr].
\]
The flux has zero normal component, in the weak Neumann sense.
The Neumann divergence operator from
\eqref{eq:neumann-divergence-operator} maps $L^6$ to $L^\infty$
with kernel $r^{-3/4}$. Hence, on sliding intervals of the fixed length $h$,
\[
\begin{aligned}
\|n(s+h)\|_{L^\infty}
&\leq C h^{-1/4}\|n\|_{L^\infty(s,s+h;L^6)}\\
&\quad+C h^{1/4}\|n\|_{L^\infty(s,s+h;L^6)}
\sup_{s<t<s+h}\|u(t)+\nabla c(t)\|_{L^\infty}.
\end{aligned}
\]
The mild identity follows from the weak Neumann equation by
testing against the backward Neumann heat kernel; all fluxes
are in $L^\infty_tL^6_x$ on these intervals. This proves
\eqref{eq:ev-uniform-boundedness}.

\emph{Step 4: boundary H\"older regularity and classical equations.}
The concentration equation now has bounded right-hand side
$-u\cdot\nabla c-nc$. Local Neumann parabolic estimates give
$c\in W^{2,1}_r$ on each later cylinder meeting the boundary,
for every finite $r>1$. In particular, taking $r>5$ gives
H\"older continuity of $\nabla c$ on smaller closed cylinders.
These statements follow as well by even reflection of $c$.

For $n$, reflect the scalar evenly and the flux
$n(u+\nabla c)$ with even tangential and odd normal components.
The reflected equation has no measure supported on $\Gamma$,
by the weak no-flux identity. Its flux is bounded, and the heat
potential estimates imply local $C^{\alpha,\alpha/2}$ regularity
of $n$ for any fixed $\alpha\in(0,1)$ after shrinking cylinders.
For example, the spatial H\"older estimate for the gradient heat
potential uses the integrable kernel
$r^{-(1+\alpha)/2}$; the corresponding time-difference estimate
gives the exponent $\alpha/2$.

For the velocity, the strong estimates and boundedness of $u$ give
$u_t\in L^2$ on each finite strip, since $\nabla u,n\in L^2$ there.
Thus $u\in W^{2,1}_2$ and $\nabla u\in L^{10/3}$.
The source $-u\cdot\nabla u-n\nabla\phi$ consequently belongs to
$L^{10/3}$. After a temporal cutoff, Stokes maximal regularity gives
$u\in W^{2,1}_{10/3}$ on a smaller strip, hence $\nabla u\in L^{10}$.
Interpolation with the existing $L^2$ bound yields
\[
 \|\nabla u\|_{L^8}
 \le \|\nabla u\|_{L^2}^{1/16}
      \|\nabla u\|_{L^{10}}^{15/16}.
\]
A further Stokes estimate gives $u\in W^{2,1}_8$ and $\nabla p\in L^8$
\cite{DHP2001}. All norms here are on finite strips over the entire
half-space. The cutoff sources have the required integrability because
$u\in L^2\cap L^\infty$ and $n\in L^1\cap L^\infty$.
Parabolic embedding now gives H\"older continuity of $u$ and
$\nabla u$ up to $\Gamma$, with bounds independent of the spatial center.

Up to this point, only $\phi\in W^{1,\infty}$ has been used.
For the final Stokes Schauder step, now assume additionally that
$\nabla\phi\in C^{0,\beta_0}(\overline{\mathbb H^3})$.
Fix a compact time interval $I$ and a slightly larger interval $J$
in the bounded time tail.
Choose $\chi\in C_c^\infty(J)$ equal to one on $I$.
The pair $(\chi u,\chi p)$ solves the global half-space Stokes problem
with zero initial and boundary data and source
\[
 F=\chi(-u\cdot\nabla u-n\nabla\phi)+\chi'u
 \in L^\infty(J;L^2(\mathbb H^3))
       \cap C^{\beta,\beta/2}(\overline{\mathbb H^3}\times J)
\]
for some $\beta\in(0,\beta_0)$. Choose $\beta$ also smaller than
 the H\"older exponent already obtained for $n,u,\nabla u$. Then
\[
[n\nabla\phi]_{C^{\beta,\beta/2}}
\leq C\|\nabla\phi\|_{L^\infty}[n]_{C^{\beta,\beta/2}}
   +C\|n\|_{L^\infty}[\nabla\phi]_{C^{0,\beta}},
\]
where the norms are taken on the preceding space--time strip
(and spatially for $\nabla\phi$). Time independence of $\phi$
handles the time increments. This is the only step requiring
H\"older continuity of its gradient: boundedness of $\nabla\phi$
alone would not control the second term.
The $L^\infty_tL^2_x$ bound follows from boundedness of $u$ and the
uniform $L^2$ bounds for $u,\nabla u,n$.

The local Neumann elliptic Schauder estimate, combined with the global
$L^4$ Helmholtz estimate, gives, for $0<\alpha<1$,
\[
 \|\mathbb P_+F\|_{C^\alpha(\overline{\mathbb H^3})}
 +\|(I-\mathbb P_+)F\|_{C^\alpha(\overline{\mathbb H^3})}
 \le C_\alpha\bigl(\|F\|_{C^\alpha(\overline{\mathbb H^3})}
                       +\|F\|_{L^4}\bigr).
\]
Indeed, on unit half-balls the gradient of the Neumann potential is
controlled by the $C^\alpha$ norm of $F$ and its own local $L^4$ norm;
the latter is bounded by the global Helmholtz estimate. The same
argument applies on interior balls, with uniform constants.
For time differences, interpolation gives
\[
\begin{aligned}
 \|F(t)-F(s)\|_{C^\alpha}
 &\le C|t-s|^{(\beta-\alpha)/2},\\
 \|F(t)-F(s)\|_{L^4}
 &\le \|F(t)-F(s)\|_{L^\infty}^{1/2}
       \|F(t)-F(s)\|_{L^2}^{1/2}
 \le C|t-s|^{\beta/4}.
\end{aligned}
\]
Hence both projected and gradient parts belong to
$C^{\alpha,\alpha/2}$ if $0<\alpha<\beta/2$.
The projected source is divergence free and has zero normal trace.

Apply the global half-space Stokes H\"older estimate
\cite[Theorem~1.1]{Solonnikov2003Half} to $\mathbb P_+F$.
To justify its use without a prescribed decay rate at infinity,
approximate $F$ by smooth spatially compactly supported functions,
still zero near the initial endpoint, uniformly bounded in
$L^\infty_tL^2_x\cap C^{\beta,\beta/2}$ and converging in $L^4_{t,x}$.
Project these approximations after truncation; they then have zero
normal trace and algebraic spatial decay. The preceding estimates
bound their H\"older norms uniformly. With zero initial and boundary
data, the cited Schauder estimate contains only the source H\"older
seminorms, with no decay constant. Maximal $L^4$ regularity controls
the lower-order norms and identifies the limit with $\chi u$.
Local compactness, with a smaller positive exponent if necessary,
passes the Schauder bounds to this same solution. Adding back the
gradient part of $F$ recovers the original pressure gradient. Thus
\[
 u\in C^{2+\alpha,1+\alpha/2}_{\rm loc},
 \qquad \nabla p\in C^{\alpha,\alpha/2}_{\rm loc}
 \quad\text{on }\overline{\mathbb H^3}\times I.
\]
Using the global Stokes solution class here controls the
possible time-dependent harmonic pressure component.

The Neumann Schauder estimates applied to $c$ next give
$c\in C^{2+\alpha,1+\alpha/2}_{\rm loc}$. Finally apply the
scalar boundary estimates to
\[
n_t-\Delta n+(u+\nabla c)\cdot\nabla n=-n\Delta c,
\qquad\partial_\nu n=0.
\]
Its coefficients and zeroth-order right-hand side have the
required H\"older regularity. This gives the same regularity
for $n$. These are local estimates on time intervals separated
from their initial endpoints; no compatibility at the original
time $t=0$ is needed. The preceding uniform global bounds ensure
that a single waiting time works for every interior and boundary
point. This proves \eqref{eq:ev-classical-tail}.
\end{proof}

\begin{proof}[Proof of Theorem~\ref{thm:eventual-regularity}]
When $M>0$, Lemmas~\ref{lem:ev-selection} and
\ref{lem:ev-scalar-decay} select a global suitable weak solution of
Theorem~\ref{thm:half-space-existence} with the global kinetic-energy
inequality and the scalar decay estimates. Lemmas~\ref{lem:ev-force-tail} and \ref{lem:ev-velocity}
give a global strong velocity after a finite time.
Lemma~\ref{lem:ev-boundary-bootstrap} then proves the asserted
uniform bounds and, using the additional H\"older assumption in
\eqref{eq:eventual-potential-assumptions}, classical regularity up to
$\Gamma$. Increase the waiting time once more if necessary so that all
statements hold also at the left endpoint of the displayed
closed time interval.
If $M=0$, mass conservation and nonnegativity give $n=0$.
The selected velocity is then a Leray--Hopf solution of the
unforced half-space Navier--Stokes equations. The proof of
Lemma~\ref{lem:ev-velocity} applies with $g=0$, and
Lemma~\ref{lem:ev-boundary-bootstrap} applies to the remaining
linear concentration equation with $n=0$.
Finally, classical regularity implies regularity in the sense
of Definition~\ref{def:boundary-regular-point}; hence the full
singular set has no point at or after the resulting waiting time.
\end{proof}

\section{Large-time decay and Gaussian asymptotics}
\label{sec:large-time-asymptotics}

We prove Theorem~\ref{thm:large-time-asymptotics} for the solution
selected in Section~\ref{sec:eventual-regularity}. The argument has
three parts. First, we pass the scalar mass identities to this
solution and show that a nonzero signal retains positive mass at
every finite time. Next, the force estimates from the preceding
section give velocity decay, which in turn yields the heat-scale
scalar bounds. Finally, the mild formulas identify the Gaussian
profiles: the divergence terms are lower order, while absorption
changes the limiting mass of the signal.

We retain $M,L$ and the effective force $g$ from the preceding
section, and $G_N,L_\infty$ from
Theorem~\ref{thm:large-time-asymptotics}.
For the repeatedly used heat exponent, write
$\alpha_p=\frac32(1-1/p)$, with $1/\infty=0$.
Spatial norms and integrals are over $\mathbb H^3$ unless stated
otherwise. Constants may depend on the initial data and on $\phi$;
decay estimates are asserted after a sufficiently large fixed time.
No additional smallness or spatial moment assumption is needed.

\subsection{Global mass identities and survival of the signal}
\label{subsec:asym-mass}

\begin{lemma}[Mass identities and a positive signal mass]
\label{lem:asym-mass}
Let $(n,c,u,p)$ be the solution selected in
Lemma~\ref{lem:ev-selection}.
For the weakly $L^1$-continuous representatives of the scalar
variables, and every $t\geq0$,
\begin{equation}\label{eq:asym-mass-identities}
 \int_{\mathbb H^3}n(t)=M,\qquad
 \int_{\mathbb H^3}c(t)+\int_0^t\int_{\mathbb H^3}nc=L.
\end{equation}
In particular, $t\mapsto\int_{\mathbb H^3}c(t)$ is nonincreasing and locally absolutely
continuous, and
\begin{equation}\label{eq:asym-signal-mass-limit}
 L_\infty:=\lim_{t\to\infty}\int_{\mathbb H^3}c(t)
 =L-\int_0^\infty\int_{\mathbb H^3}nc\in[0,L].
\end{equation}
If $L>0$, then no finite-time loss of the entire signal mass
can occur: more precisely,
\begin{equation}\label{eq:asym-finite-time-positive-mass}
 \int_{\mathbb H^3}c(t)\geq L\exp\left(-\frac{M\|c_0\|_{L^\infty}}{L}t\right)>0,
 \qquad 0\leq t<\infty.
\end{equation}
If $L=0$, then $c=0$.
\end{lemma}

\begin{proof}
\emph{Step 1: no mass is lost at spatial infinity in the
approximation limit.}
We first record this point for the concentration, in addition to
the density tightness established in the proof of
Lemma~\ref{lem:ev-selection}.  Let $\eta_R$ be a smooth radial
exterior cutoff, zero on $B_R$ and equal to one outside $B_{2R}$,
with $|D^j\eta_R|\leq CR^{-j}$ and
$\partial_\nu\eta_R=0$ on $\Gamma$.

Suppress the superscripts $(\varepsilon,\tau)$ in the next
cutoff estimate. Testing the regularized concentration equation and
dropping the nonnegative absorption term gives

\begin{align}
 \sup_{0<t<T}\int_{\mathbb H^3}c(t)\eta_R
 &\leq\int_{\mathbb H^3}c_0^\varepsilon\eta_R
       +\frac{CLT}{R^2}
       +\frac C R\int_0^T\int_{\mathbb H^3} c|u|\nonumber\\
 &\leq\int_{\mathbb H^3}c_0^\varepsilon\eta_R+\frac{C_T}{R}.
 \label{eq:asym-c-tightness}
\end{align}
Indeed, $\|c(t)\|_{L^{6/5}}\leq L^{5/6}\|c_0\|_{L^\infty}^{1/6}$ and
$u\in L^2(0,T;L^6(\mathbb H^3))$ uniformly in both parameters.  The initial
concentrations converge strongly in $L^1(\mathbb H^3)$, since
\[
 c_0^\varepsilon=(S_N(\varepsilon)\sqrt{c_0})^2,
 \qquad S_N(\varepsilon)\sqrt{c_0}\longrightarrow\sqrt{c_0}
             \quad\hbox{in }L^2(\mathbb H^3).
\]
Thus their $L^1$ tails are uniformly small, and
\eqref{eq:asym-c-tightness} is uniform spatial tightness on
every finite time interval.

Restoring the superscripts, local almost-everywhere convergence,
the uniform concentration bound, and spatial tightness give
\[
 c^{\varepsilon_j,\tau}\to c^\tau\quad(\tau\text{ fixed}),
 \qquad c^{\tau_j}\to c\quad(\tau_j\downarrow0)
 \quad\text{in }L^1(\mathbb H^3\times(0,T)).
\]
  The corresponding strong convergences
$n^{\varepsilon_j,\tau}\to n^\tau$ and $n^{\tau_j}\to n$ in
$L^1(\mathbb H^3\times(0,T))$ follow from the density tightness and
uniform integrability proved in Lemma~\ref{lem:ev-selection}; see also
\eqref{eq:eventual-global-force-compactness}.

The absorption terms converge strongly on the whole finite
space--time slab.  At fixed $\tau$, the $1$-Lipschitz property of $h_\tau$,
the global strong $L^1$ convergence $n^{\varepsilon_j,\tau}\to n^\tau$,
and bounded almost-everywhere convergence
$c^{\varepsilon_j,\tau}\to c^\tau$ give

\[
 h_\tau(n^{\varepsilon_j,\tau})c^{\varepsilon_j,\tau}
 \longrightarrow h_\tau(n^\tau)c^\tau
 \quad\text{in }L^1(\mathbb H^3\times(0,T)).
\]
For the product with $h_\tau(n^\tau)$, use dominated convergence
and $h_\tau(n^\tau)\le n^\tau$.
At the subsequent limit $\tau_j\to0$, the same conclusion follows from
\[
 \|h_{\tau_j}(n^{\tau_j})-n\|_{L^1_{t,x}}
 \leq\|n^{\tau_j}-n\|_{L^1_{t,x}}
       +\|h_{\tau_j}(n)-n\|_{L^1_{t,x}}\longrightarrow0,
\]
and
\[
 \|h_{\tau_j}(n^{\tau_j})c^{\tau_j}-nc\|_{L^1_{t,x}}
 \leq \|c_0\|_{L^\infty}\|h_{\tau_j}(n^{\tau_j})-n\|_{L^1_{t,x}}
       +\int_0^T\int_{\mathbb H^3}n|c^{\tau_j}-c|\longrightarrow0.
\]
Here $0\leq h_\tau(n)\leq n$ and $n\in L^1(\mathbb H^3\times(0,T))$ justify
both dominated-convergence arguments.  Hence there is no defect
measure in the total absorption.

On the smooth approximation,
\[
 \int_{\mathbb H^3}n^{\varepsilon,\tau}(t)=M,\qquad
 \int_{\mathbb H^3}c^{\varepsilon,\tau}(t)
 +\int_0^t\int_{\mathbb H^3}h_\tau(n^{\varepsilon,\tau})c^{\varepsilon,\tau}
 =\|c_0^\varepsilon\|_{L^1}\longrightarrow L.
\]
These identities pass to both limits for almost every time.
To obtain representatives at every time, observe that the weak
scalar equations give locally absolutely continuous pairings with
smooth compactly supported spatial tests. The entropy bound for
$n$, boundedness of $c$, and the uniform spatial tails give weak
$L^1$ compactness on each finite time interval. Thus any sequence
of admissible times approaching a fixed time has a weak $L^1$
cluster point, uniquely determined by those continuous pairings.

The convergence of the pairings extends to bounded measurable
tests. First discard the uniformly small spatial tails; on the
remaining bounded set, Lusin approximation and uniform integrability
control the error in replacing such a test by a continuous one.
Smooth approximation then applies. This gives weakly
$L^1$-continuous representatives, including at $t=0$, for which
\eqref{eq:asym-mass-identities} holds at every time. On the classical
time interval these are the classical representatives. The signal
identity gives local absolute continuity of its total mass, and
monotone convergence gives \eqref{eq:asym-signal-mass-limit}.

\emph{Step 2: positivity of the signal mass at finite times.}
Assume $L>0$. Fix the approximation parameters and, within this
step, suppress the superscripts on $(n,c,u)$. Let $q$ solve the
passive scalar problem
\[
 q_t+u\cdot\nabla q=\Delta q,\qquad
 \partial_\nu q=0,\qquad q(0)=c_0^\varepsilon.
\]
Incompressibility, the boundary conditions, and the maximum
principle give
\[
 \int_{\mathbb H^3}q(t)=\|c_0^\varepsilon\|_{L^1}>0,\qquad
 0<q(t)\leq\|c_0\|_{L^\infty}.
\]
The initial regularization is strictly positive by positivity of
the Neumann kernel. Comparison with the concentration equation
therefore gives
\[
 e^{-\int_0^t\|h_\tau(n(s))\|_{L^\infty}\,ds}q(t)
 \leq c(t)\leq q(t).
\]
In particular, the relative logarithmic density is well defined and
\[
 0\leq q(t)\log\frac{q(t)}{c(t)}
 \leq q(t)\int_0^t\|h_\tau(n(s))\|_{L^\infty}\,ds.
\]
Using the common transport velocity, direct differentiation gives
\begin{equation}\label{eq:asym-relative-logarithmic-identity}
 \begin{aligned}
 (\partial_t+u\cdot\nabla-\Delta)\left(q\log\frac qc\right)
 &=-q\left|\nabla\log\frac qc\right|^2+h_\tau(n)q,\\
 \partial_\nu\left(q\log\frac qc\right)&=0.
 \end{aligned}
\end{equation}
This identity can be integrated over the half-space. Indeed,
insert the cutoff $1-\eta_R$ from Step~1. On any fixed interval
$(0,T)$, the diffusion and transport errors are bounded by
\[
 \begin{aligned}
 C\|c_0^\varepsilon\|_{L^1}
 \left(R^{-2}+R^{-1}\|u\|_{L^\infty(\mathbb H^3\times(0,T))}\right)
 \int_0^T\int_0^t\|h_\tau(n(s))\|_{L^\infty}\,ds\,dt
 \longrightarrow0.
 \end{aligned}
\]
The coefficients are bounded for fixed approximation parameters.
Fatou's lemma handles the nonnegative gradient term, and the
initial logarithmic density is zero. Since $h_\tau(n)\leq n$,
we obtain the parameter-independent estimate
\begin{equation}\label{eq:asym-relative-entropy-bound}
 \int_{\mathbb H^3}q(t)\log\frac{q(t)}{c(t)}
 \leq\int_0^t\int_{\mathbb H^3}h_\tau(n)q
 \leq M\|c_0\|_{L^\infty}t.
\end{equation}
Jensen's inequality for the probability measure
$q(t,x)\,dx/\|c_0^\varepsilon\|_{L^1}$ gives
\[
 \begin{aligned}
 -\frac1{\|c_0^\varepsilon\|_{L^1}}
       \int_{\mathbb H^3}q(t)\log\frac{q(t)}{c(t)}
 &\leq
 \log\left(\frac{\|c(t)\|_{L^1}}{\|c_0^\varepsilon\|_{L^1}}\right).
 \end{aligned}
\]
Restoring the superscripts, it follows that
\[
 \|c^{\varepsilon,\tau}(t)\|_{L^1}
 \geq\|c_0^\varepsilon\|_{L^1}
 \exp\left(-\frac{M\|c_0\|_{L^\infty}}
                  {\|c_0^\varepsilon\|_{L^1}}\,t\right).
\]
The global strong $L^1$ convergence from Step~1 passes this bound
to the selected solution for almost every time, first as
$\varepsilon\to0$ and then as $\tau\to0$. Continuity of the total
signal mass extends it to every finite time, proving
\eqref{eq:asym-finite-time-positive-mass}. If $L=0$, nonnegativity
and \eqref{eq:asym-mass-identities} give $c=0$.
\end{proof}

\subsection{Decay of the velocity in the finite-energy class}
\label{subsec:asym-velocity}

Finite total dissipation alone does not give kinetic-energy decay
on the half-space, where no Poincar\'e inequality is available.
We first compare $u$ with its homogeneous Stokes evolution.
Averaging the gradient energy inequality then gives $L^6$ decay,
and two applications of Stokes smoothing yield the endpoint
$L^\infty$ rate.

\begin{lemma}[Decay of the velocity]
\label{lem:asym-velocity}
For the global solution selected in Section~\ref{sec:eventual-regularity},
\begin{equation}\label{eq:asym-velocity-L2-H1}
 \lim_{t\to\infty}\|u(t)\|_{L^2}=0,
 \qquad
 \lim_{t\to\infty}t^{1/2}\|\nabla u(t)\|_{L^2}=0.
\end{equation}
Moreover, for every $2<p\le\infty$,
\begin{equation}\label{eq:asym-velocity-Lp}
 \lim_{t\to\infty}
 t^{\frac32(\frac12-\frac1p)}\|u(t)\|_{L^p}=0,
\end{equation}
where $1/\infty=0$.  In particular, $u(t)\to0$ strongly in
$H^1_0(\mathbb H^3)$ and uniformly on $\overline{\mathbb H^3}$.
\end{lemma}

\begin{proof}
Choose $S\ge1$ sufficiently large that the strong-solution and
boundedness conclusions of Section~\ref{sec:eventual-regularity}
hold on $[S,\infty)$.  Recall that
$g=\mathbb P_+(\phi\nabla n)$.
By \eqref{eq:ev-force-integrability}, \eqref{eq:ev-density-tail},
and Lemma~\ref{lem:ev-velocity},
\begin{equation}\label{eq:asym-velocity-tail-inputs}
\begin{gathered}
 \sup_{t\ge S}\|u(t)\|_{L^2}<\infty,\qquad
 \int_S^\infty\|\nabla u(t)\|_{L^2}^2\,dt<\infty,\qquad
 g\in L^1((S,\infty);L^2_\sigma(\mathbb H^3)),\\
 \int_t^\infty\|g(s)\|_{L^2}^2\,ds
 \le \|\phi\|_{L^{\infty}}^2
       \int_t^\infty\|\nabla n(s)\|_{L^2}^2\,ds
 \le C(1+t)^{-3/2}\qquad(t\ge S).
\end{gathered}
\end{equation}
The last estimate is understood with a larger constant if necessary
to replace $t-T_0$ by $1+t$.  The velocity energy relation and the
closed bootstrap in \eqref{eq:ev-enstrophy-bootstrap} give, for
almost every $t\ge S$,
\begin{equation}\label{eq:asym-velocity-differential-energies}
\begin{aligned}
 \frac{d}{dt}\|u(t)\|_{L^2}^2
 +2\|\nabla u(t)\|_{L^2}^2
 &\le 2(g(t),u(t))_{L^2},\\
 \frac{d}{dt}\|\nabla u(t)\|_{L^2}^2
 +\|Au(t)\|_{L^2}^2&\le2\|g(t)\|_{L^2}^2.
\end{aligned}
\end{equation}
All time-dependent estimates below use the continuous strong-solution
representative of Lemma~\ref{lem:ev-velocity}.

\emph{Step 1: decay of the kinetic energy.}
The eventual $H^1$ bound and the finite total dissipation imply
\[
 \int_S^\infty\|u(s)\|_{L^4}^4\,ds
 \le C\sup_{s\ge S}\|u(s)\|_{L^2}
       \sup_{s\ge S}\|\nabla u(s)\|_{L^2}
       \int_S^\infty\|\nabla u(s)\|_{L^2}^2\,ds<\infty.
\]
Fix $T\ge S$ and put $z(t)=u(t)-e^{-(t-T)A}u(T)$.
Then $z(T)=0$ and
$z_t+Az=-\mathbb P_+\operatorname{div}(u\otimes u)+g$.
Testing by $z$ is justified by the eventual strong regularity.
Young's inequality for the convection term gives, for every finite $B>T$,
\[
 \sup_{T\le t\le B}\|z(t)\|_{L^2}^2
 \le \int_T^B\|u(s)\|_{L^4}^4\,ds
   +2\left(\sup_{T\le t\le B}\|z(t)\|_{L^2}\right)
       \int_T^B\|g(s)\|_{L^2}\,ds.
\]
Solving this quadratic inequality and then letting $B\to\infty$ yields
\[
 \sup_{t\ge T}\|z(t)\|_{L^2}
 \le \left(\int_T^\infty\|u(s)\|_{L^4}^4\,ds\right)^{1/2}
       +2\int_T^\infty\|g(s)\|_{L^2}\,ds.
\]
The nonnegative self-adjoint Dirichlet Stokes operator has trivial
kernel: $Av=0$ implies $\|\nabla v\|_{L^2}=0$, and hence $v=0$ in
$L^2(\mathbb H^3)$. The spectral theorem therefore gives
$\|e^{-(t-T)A}u(T)\|_{L^2}\to0$ for fixed $T$.
Thus $\limsup_{t\to\infty}\|u(t)\|_{L^2}$ is bounded by the last
display's right-hand side. Letting $T\to\infty$, using the first
display and $g\in L^1((S,\infty);L^2)$, proves the claim.

\emph{Step 2: decay of the velocity gradient.}
Integrating the second inequality in
\eqref{eq:asym-velocity-differential-energies} from $s$ to $t$ and
then averaging over $s\in[t/2,t]$, for $t\ge2S$, yields
\begin{align}
 \|\nabla u(t)\|_{L^2}^2
 &\le\frac2t\int_{t/2}^t\|\nabla u(s)\|_{L^2}^2\,ds
       +2\int_{t/2}^\infty\|g(s)\|_{L^2}^2\,ds\nonumber\\
 &=o(t^{-1})+O(t^{-3/2}).
 \label{eq:asym-gradient-rate}
\end{align}
Here the first term is $o(t^{-1})$ because
\[
\int_S^\infty\|\nabla u(s)\|_{L^2}^2\,ds<\infty.
\]
Consequently $t^{1/2}\|\nabla u(t)\|_{L^2}\to0$, and the homogeneous
Sobolev inequality gives
\begin{equation}\label{eq:asym-velocity-L6-preliminary}
 \|u(t)\|_{L^6}=o(t^{-1/2}).
\end{equation}

\emph{Step 3: a preliminary uniform decay estimate.}
We will use the half-space Stokes estimates, for the exponents
appearing below,
\begin{align}
 \|e^{-rA}\mathbb P_+F\|_{L^b}
 &\le C r^{-\frac32(\frac1a-\frac1b)}\|F\|_{L^a},
 \nonumber\\
 \|e^{-rA}\mathbb P_+\operatorname{div}\mathcal F\|_{L^b}
 &\le C r^{-\frac12-\frac32(\frac1a-\frac1b)}
             \|\mathcal F\|_{L^a},
 \qquad r>0.
 \label{eq:asym-stokes-smoothing}
\end{align}
Here the input exponents used below satisfy $2\le a<\infty$,
and $a\le b\le\infty$. We give the endpoint justification.
Use the same notation $A$ for the consistent Stokes operators on
$L^a_\sigma$. For $1<a<\infty$, analyticity and the homogeneous
elliptic estimate give \cite{DHP2001}
\[
\|e^{-rA}v\|_{L^a}+r\|\nabla^2e^{-rA}v\|_{L^a}\le C_a\|v\|_{L^a},
\qquad
\|\nabla e^{-rA}v\|_{L^a}\le C_a r^{-1/2}\|v\|_{L^a}.
\]
Duality with the gradient estimate at the conjugate exponent
defines the divergence operator and yields
\[
\|e^{-rA}\mathbb P_+\operatorname{div}\mathcal F\|_{L^a}
\le C_a r^{-1/2}\|\mathcal F\|_{L^a}.
\]
For $a\ge2$, the homogeneous Gagliardo--Nirenberg inequality
\[
\|w\|_{L^b}
\le C\|\nabla^2w\|_{L^a}^{\frac32(\frac1a-\frac1b)}
       \|w\|_{L^a}^{1-\frac32(\frac1a-\frac1b)},
\]
has derivative exponent less than one, includes $b=\infty$,
and gives the first estimate in
\eqref{eq:asym-stokes-smoothing}. Factoring the semigroup at
$r/2$ and applying the preceding same-exponent divergence bound
gives the second one. Consistency on intersections identifies
these operators for all exponents used. The constants are valid
for every $r>0$: the dilation $x=r^{1/2}y$ preserves the half-space
and the no-slip Stokes problem. No $L^1$ or $L^\infty$ input
endpoint is used.

Equation~\eqref{eq:asym-velocity-L6-preliminary} and
\eqref{eq:ev-density-decay} imply
\[
 \|u\|_{L^\infty(t,\infty;L^6)}=o(t^{-1/2}),
 \qquad \|n\|_{L^\infty(t,\infty;L^6)}=O(t^{-5/4}).
\]
The first bound follows directly from the definition of the
pointwise small-$o$ estimate. Apply the mild formula on sliding
intervals of length one,
using the original force $-n\nabla\phi$. The $L^3\to L^{12}$
divergence kernel is $r^{-7/8}$, and the $L^6\to L^{12}$ force
kernel is $r^{-1/8}$. Therefore
\[
\begin{aligned}
 \|u\|_{L^\infty(t+1,\infty;L^{12})}
 \le C\bigl(\|u\|_{L^\infty(t,\infty;L^6)}
             +\|u\|_{L^\infty(t,\infty;L^6)}^2
+\|n\|_{L^\infty(t,\infty;L^6)}\bigr)
 =o(t^{-1/2}).
\end{aligned}
\]
The constants here absorb $\|\nabla\phi\|_{L^\infty}$.
On another unit interval the nonlinear tensor belongs to $L^6$,
and the relevant kernels are $r^{-3/4}$ and $r^{-1/4}$. Hence
\begin{equation}\label{eq:asym-velocity-infinity-preliminary}
\begin{aligned}
 \|u\|_{L^\infty(t+2,\infty;L^\infty)}
 &\le C\bigl(\|u\|_{L^\infty(t+1,\infty;L^{12})}
             +\|u\|_{L^\infty(t+1,\infty;L^{12})}^2\\
 &\qquad\quad+\|n\|_{L^\infty(t+1,\infty;L^6)}\bigr)
 =o(t^{-1/2}).
\end{aligned}
\end{equation}
Thus the preliminary decay is uniform on all sufficiently late
sliding intervals.

\emph{Step 4: the $L^\infty$ decay rate.}
Use the mild formula on the longer interval $[t/2,t]$, where
$t\ge\max\{2S,4\}$, and write $r=t-s$:
\begin{align}
 u(t)
 ={}&e^{-(t/2)A}u(t/2)
 -\int_0^{t/2}e^{-rA}\mathbb P_+
       \operatorname{div}(u\otimes u)(t-r)\,dr
 +\int_0^{t/2}e^{-rA}g(t-r)\,dr.
 \label{eq:asym-long-window-mild}
\end{align}
By Step~1 and \eqref{eq:asym-stokes-smoothing}, the first term satisfies
\[
 \|e^{-(t/2)A}u(t/2)\|_{L^{\infty}}
 \le Ct^{-3/4}\|u(t/2)\|_{L^2}=o(t^{-3/4}).
\]
For the nonlinear term split its integral at $r=1$.  On $(0,1)$,
use
$\|u\otimes u\|_{L^6}\le\|u\|_{L^{\infty}}\|u\|_{L^6}$ and the
$L^6\to L^\infty$ divergence estimate.  On $(1,t/2)$, use
$\|u\otimes u\|_{L^3}\le\|u\|_{L^6}^2$ and the
$L^3\to L^\infty$ estimate.  This gives
\begin{equation}\label{eq:asym-long-window-nonlinearity}
\begin{aligned}
 &\left\|\int_0^{t/2}e^{-rA}\mathbb P_+
       \operatorname{div}(u\otimes u)(t-r)\,dr\right\|_{L^{\infty}}\\
 &\quad\le C\|u\|_{L^\infty(t/2,\infty;L^\infty)}
                   \|u\|_{L^\infty(t/2,\infty;L^6)}
          \int_0^1r^{-3/4}\,dr\\
 &\qquad+C\|u\|_{L^\infty(t/2,\infty;L^6)}^2
          \int_1^{t/2}r^{-1}\,dr\\
 &\quad=o(t^{-1})+o(t^{-1}\log t)=o(t^{-3/4}).
\end{aligned}
\end{equation}

For the force, two representations are useful.  At almost every
late time, $n\phi\in H^1(\mathbb H^3)$ and the Helmholtz projection annihilates
$\nabla(n\phi)$, so that
\begin{equation}\label{eq:asym-force-two-representations}
 g=\mathbb P_+(\phi\nabla n)
   =-\mathbb P_+(n\nabla\phi).
\end{equation}
The Helmholtz projections and Stokes semigroups in $L^2$ and
$L^6$ are consistent on their intersection.

Since $\sqrt t\le t/2$, on $0<r<\sqrt t$ use the second
representation, \eqref{eq:ev-density-decay}, and the
$L^6\to L^\infty$ estimate to obtain
\begin{equation}\label{eq:asym-recent-force}
\begin{aligned}
 \left\|\int_0^{\sqrt t}e^{-rA}g(t-r)\,dr\right\|_{L^{\infty}}
 &\le C\|n\|_{L^\infty(t/2,\infty;L^6)}
          \int_0^{\sqrt t}r^{-1/4}\,dr\\
 &\le Ct^{-5/4}t^{3/8}=Ct^{-7/8}.
\end{aligned}
\end{equation}
On $\sqrt t<r<t/2$, use the $L^2\to L^\infty$ estimate,
Cauchy--Schwarz in time, and \eqref{eq:asym-velocity-tail-inputs}:
\begin{align}
 \left\|\int_{\sqrt t}^{t/2}e^{-rA}g(t-r)\,dr\right\|_{L^{\infty}}
 &\le C\left(\int_{\sqrt t}^\infty r^{-3/2}\,dr\right)^{1/2}
         \left(\int_{t/2}^\infty\|g(s)\|_{L^2}^2\,ds\right)^{1/2}
 \nonumber\\
 &\le C t^{-1/8}t^{-3/4}=Ct^{-7/8}.
 \label{eq:asym-earlier-force}
\end{align}

Thus the free evolution is $o(t^{-3/4})$, the convection term
has the same bound, and both force contributions are $O(t^{-7/8})$.
Substituting these estimates in \eqref{eq:asym-long-window-mild}
proves
\begin{equation}\label{eq:asym-velocity-infinity}
 \|u(t)\|_{L^{\infty}}=o(t^{-3/4}).
\end{equation}
Finally, for $2<p<\infty$,
\[
 \|u(t)\|_{L^p}
 \le\|u(t)\|_{L^2}^{2/p}\|u(t)\|_{L^{\infty}}^{1-2/p},
\]
which, together with Step~1 and
\eqref{eq:asym-velocity-infinity}, yields
\eqref{eq:asym-velocity-Lp}.  The $H^1$ convergence follows from
\eqref{eq:asym-velocity-L2-H1}.
\end{proof}

\subsection{Heat-scale decay of the scalar variables}

The velocity decay now allows the scalar equations to be treated
as perturbations of the Neumann heat equation. We record the kernel
bounds first, since both endpoint norms will also be needed for the
Gaussian profiles.
Write $y^*=(y_1,y_2,-y_3)$ and
\[
 K_N(r,x,y)=(4\pi r)^{-3/2}
 \left(e^{-|x-y|^2/(4r)}+e^{-|x-y^*|^2/(4r)}\right).
\]
Thus $S_N(r)f(x)=\int_{\mathbb H^3} K_N(r,x,y)f(y)\,dy$.
For a vector field $F$, the weak Neumann divergence propagator is
\[
 \mathcal D_N(r)F(x):=-\int_{\mathbb H^3}\nabla_yK_N(r,x,y)\cdot F(y)\,dy.
\]
When $F\cdot\nu=0$, this is $S_N(r)\nabla\cdot F$; the kernel
formula is also its definition without a classical normal trace.
The Gaussian formula and Young's inequality give, for
$1\leq a\leq b\leq\infty$ and every $r>0$,
\begin{equation}\label{eq:asym-neumann-bounds}
\begin{aligned}
 \|S_N(r)f\|_{L^b}
 &\leq C r^{-\frac32(\frac1a-\frac1b)}\|f\|_{L^a},\\
 \|\nabla S_N(r)f\|_{L^b}+\|\mathcal D_N(r)F\|_{L^b}
 &\leq C r^{-\frac12-\frac32(\frac1a-\frac1b)}
       (\|f\|_{L^a}+\|F\|_{L^a}).
\end{aligned}
\end{equation}
In particular the bounds for $a=b=1$ and $a=b=\infty$ hold.
Equivalently, scalar functions are extended evenly across the flat
boundary, and vector fluxes have even tangential and odd normal
components. This specifies the boundary operator used in all the
following mild formulas.

\begin{lemma}[Scalar decay at the heat scale]\label{lem:asym-scalars}
For all sufficiently large $t$ and every $1\leq p\leq\infty$,
\begin{equation}\label{eq:asym-scalar-rates}
 \|n(t)\|_{L^p}+\|c(t)\|_{L^p}\leq C_p t^{-\alpha_p}.
\end{equation}
Moreover,
\begin{equation}\label{eq:asym-signal-gradient}
 \|\nabla c(t)\|_{L^{\infty}}\leq Ct^{-2}.
\end{equation}
\end{lemma}

\begin{proof}
All intervals in this proof lie inside the classical time interval
provided by Theorem~\ref{thm:eventual-regularity}. The scalar bounds
initially valid almost everywhere in time extend to every such time
by local continuity and Fatou's lemma.

For large $s$, Lemma~\ref{lem:asym-velocity} and the previously established scalar
estimates imply
\[
\begin{gathered}
 \|u\|_{L^\infty(s,2s;L^6)}=o(s^{-1/2}),\qquad
 \|nc\|_{L^\infty(s,2s;L^6)}=O(s^{-11/4}),\\
 \|c(s)\|_{L^{\infty}}=O(s^{-3/2}).
\end{gathered}
\]
The concentration mild formula and \eqref{eq:asym-neumann-bounds} give,
for $0<r\leq s$,
\begin{align*}
 \|\nabla c(s+r)\|_{L^{\infty}}
 &\leq C r^{-1/2}\|c(s)\|_{L^{\infty}}
   +C \|u\|_{L^\infty(s,2s;L^6)}\int_0^r(r-\sigma)^{-3/4}
                    \|\nabla c(s+\sigma)\|_{L^{\infty}}\,d\sigma\\
 &\qquad+C \|nc\|_{L^\infty(s,2s;L^6)}r^{1/4}.
\end{align*}
Multiply by $r^{1/2}$ and take the supremum over $0<r\leq s$.
The time-kernel estimate
\[
 r^{1/2}\int_0^r(r-\sigma)^{-3/4}\sigma^{-1/2}\,d\sigma
 \leq Cr^{1/4}
\]
gives
\[
 \begin{aligned}
 \sup_{0<r\leq s}r^{1/2}\|\nabla c(s+r)\|_{L^\infty}
 &\leq C\|c(s)\|_{L^\infty}
       +Cs^{3/4}\|nc\|_{L^\infty(s,2s;L^6)}\\
 &\quad+Cs^{1/4}\|u\|_{L^\infty(s,2s;L^6)}
       \sup_{0<r\leq s}r^{1/2}\|\nabla c(s+r)\|_{L^\infty}.
 \end{aligned}
\]
The supremum is finite by the eventual gradient bound already
proved. The coefficient of the last term tends to zero, while
$s^{3/4}\|nc\|_{L^\infty(s,2s;L^6)}=O(s^{-2})$.
Absorption therefore gives
\[
 \sup_{0<r\leq s}r^{1/2}\|\nabla c(s+r)\|_{L^\infty}
 \leq Cs^{-3/2}.
\]
Taking $r=s$ proves \eqref{eq:asym-signal-gradient}.

For the density equation, the combined transport satisfies
\begin{equation}\label{eq:asym-small-rescaled-drift}
 \sqrt s\bigl(
 \|u\|_{L^\infty(\mathbb H^3\times[s,2s])}
 +\|\nabla c\|_{L^\infty(\mathbb H^3\times[s,2s])}
 \bigr)\longrightarrow0.
\end{equation}
Its mild formula is
\[
 n(s+r)=S_N(r)n(s)
 -\int_0^r\mathcal D_N(r-\sigma)
       [n(u+\nabla c)](s+\sigma)\,d\sigma.
\]
The flux has zero normal component, so this is the mild form of the
classical no-flux equation. The $L^2\to L^\infty$ heat estimate,
the $L^\infty\to L^\infty$ divergence estimate, and
\[
 r^{3/4}\int_0^r(r-\sigma)^{-1/2}\sigma^{-3/4}\,d\sigma
 \leq Cr^{1/2}
\]
yield
\[
 \begin{aligned}
 \sup_{0<r\leq s}r^{3/4}\|n(s+r)\|_{L^\infty}
 &\leq C\|n(s)\|_{L^2}\\
 &\quad+C\sqrt s\bigl(
       \|u\|_{L^\infty(\mathbb H^3\times[s,2s])}
       +\|\nabla c\|_{L^\infty(\mathbb H^3\times[s,2s])}\bigr)\\
 &\qquad\qquad\times
       \sup_{0<r\leq s}r^{3/4}\|n(s+r)\|_{L^\infty}.
 \end{aligned}
\]
Again the supremum is finite by eventual boundedness. Absorbing
the last term using \eqref{eq:asym-small-rescaled-drift} and taking
$r=s$ gives
\[
 \|n(2s)\|_{L^\infty}
 \leq Cs^{-3/4}\|n(s)\|_{L^2}\leq Cs^{-3/2},
\]
where \eqref{eq:ev-density-decay} was used in the last step.
The calculation can first be made at admissible starting times;
continuity on the classical interval and Fatou's lemma extend
the bound to every sufficiently large terminal time.
Interpolation with $\|n(t)\|_{L^1}=M$ proves the density bounds in
\eqref{eq:asym-scalar-rates}. For $c$, interpolate its mass bound
with \eqref{eq:ev-signal-decay}.
\end{proof}

The density bound just proved is integrable in time in
$L^\infty$. Together with finite-time positivity from
Lemma~\ref{lem:asym-mass}, this ensures that a nonzero signal also
has positive limiting mass.

\begin{lemma}[Positive remaining mass]\label{lem:asym-residual-mass}
The remaining mass in \eqref{eq:asym-signal-mass-limit} satisfies
\begin{equation}\label{eq:asym-residual-mass}
 0<L_\infty\leq L\qquad\text{if }L>0.
\end{equation}
If $L=0$, then $c\equiv0$.
\end{lemma}

\begin{proof}
If $L>0$, choose any sufficiently large time $T$.
Lemma~\ref{lem:asym-mass} gives $\int_{\mathbb H^3}c(T)>0$.
The mass identity and Lemma~\ref{lem:asym-scalars} give
\[
 \frac{d}{dt}\int_{\mathbb H^3}c(t)
 =-\int_{\mathbb H^3}nc
 \geq-\|n(t)\|_{L^\infty}\int_{\mathbb H^3}c(t),
 \qquad
 \int_T^\infty\|n(t)\|_{L^\infty}\,dt<\infty.
\]
Integrating this differential inequality yields
\[
 L_\infty\geq
 \left(\int_{\mathbb H^3}c(T)\right)
 \exp\left(-\int_T^\infty\|n(t)\|_{L^\infty}\,dt\right)>0.
\]
For $L=0$, nonnegativity and the mass inequality force $c=0$.
\end{proof}

\subsection{Gaussian profiles, including the endpoint norms}

The remaining task is to identify the leading profiles. We first
separate two heat-kernel facts: an $L^1$ datum contributes its total
mass times the Gaussian, and a sufficiently decaying divergence
flux contributes only a lower-order term. The reaction $nc$ in
the signal equation is then handled through its total time integral.

\begin{lemma}[Neumann heat asymptotics and divergence remainders]
\label{lem:asym-heat-profile}
Let $G_N$ be the unit-mass Neumann Gaussian defined in
Theorem~\ref{thm:large-time-asymptotics}. For $f\in L^1(\mathbb H^3)$, a fixed $T\geq0$, and
$1\leq p\leq\infty$,
\begin{equation}\label{eq:asym-linear-profile}
  t^{\alpha_p}
 \left\|S_N(t-T)f-\left(\int_{\mathbb H^3}f\right)G_N(t)\right\|_{L^p}
 \longrightarrow0.
\end{equation}
Suppose a vector field $F$ belongs to
$L^1((T,B);L^1(\mathbb H^3))$ for every $B>T$ and satisfies, for this $p$,
\[
 \|F(s)\|_{L^1}=o(s^{-1/2}),\qquad
 \|F(s)\|_{L^p}=o(s^{-\alpha_p-1/2}).
\]
Then, whenever the following integral is defined (as it is for the
classical fluxes below),
\begin{equation}\label{eq:asym-flux-remainder}
 t^{\alpha_p}
 \left\|\int_T^t\mathcal D_N(t-s)F(s)\,ds\right\|_{L^p}
 \longrightarrow0.
\end{equation}
\end{lemma}

\begin{proof}
Extend $f$ evenly to $f^e$ on $\mathbb R^3$. Its integral is $2\int_{\mathbb H^3}f$,
and $S_N(r)f=(G(\cdot,r)*f^e)|_{\mathbb H^3}$, where $G$ is the whole-space
Gaussian. For every fixed $y\in\mathbb R^3$,
\[
 r^{\alpha_p}\|G(\cdot-y,r)-G(\cdot,r)\|_{L^p(\mathbb R^3)}
 =\|G(\cdot-y/\sqrt r,1)-G(\cdot,1)\|_{L^p(\mathbb R^3)}\longrightarrow0.
\]
This follows from translation continuity for finite $p$, and from
uniform continuity of the Gaussian for $p=\infty$.
The expression is bounded by $2\|G(\cdot,1)\|_{L^p(\mathbb R^3)}$.
Minkowski's inequality and dominated convergence against $|f^e(y)|$
prove the assertion with $r=t-T$ and $G_N(t-T)$.
Replacing $t-T$ by $t$ in the Gaussian changes the normalized norm
by a quantity tending to zero, by dilation continuity of the Gaussian
and $(t-T)/t\to1$. This proves \eqref{eq:asym-linear-profile}
without any spatial moment assumption.

For the second assertion, split the time integral at $t/2$.
The early part satisfies
\[
 \left\|\int_T^{t/2}\mathcal D_N(t-s)F(s)\,ds\right\|_{L^p}
 \leq C t^{-\alpha_p-1/2}\int_T^{t/2}\|F(s)\|_{L^1}\,ds
 =o(t^{-\alpha_p}).
\]
The last equality follows by separating a fixed initial interval and
using $\|F(s)\|_{L^1}=o(s^{-1/2})$ on the rest.
For the late part, use the $L^p\to L^p$ divergence bound:
\[
 \begin{aligned}
 \left\|\int_{t/2}^t\mathcal D_N(t-s)F(s)\,ds\right\|_{L^p}
 &\leq C\left(\operatorname*{ess\,sup}_{t/2<s<t}\|F(s)\|_{L^p}\right)
       \int_{t/2}^t(t-s)^{-1/2}\,ds\\
 &=o(t^{-\alpha_p-1/2})\,O(t^{1/2})
 =o(t^{-\alpha_p}).
 \end{aligned}
\]
Both estimates use \eqref{eq:asym-neumann-bounds}, so they include
$p=1$ and $p=\infty$.
\end{proof}

\begin{lemma}[The scalar Gaussian profiles]\label{lem:asym-gaussian}
For every $1\leq p\leq\infty$,
\begin{equation}\label{eq:asym-gaussian-profiles}
 \lim_{t\to\infty}t^{\alpha_p}
 \left(\|n(t)-MG_N(t)\|_{L^p}
       +\|c(t)-L_\infty G_N(t)\|_{L^p}\right)=0.
\end{equation}
\end{lemma}

\begin{proof}
Fix $T$ after the waiting times of the preceding lemmas.
For every $1\leq p\leq\infty$, the density flux satisfies
\[
 \|n(t)(u(t)+\nabla c(t))\|_{L^p}
 \leq\|n(t)\|_{L^p}
       \bigl(\|u(t)\|_{L^\infty}+\|\nabla c(t)\|_{L^\infty}\bigr)
 =o(t^{-\alpha_p-1/2}).
\]
For $p=1$ this is $o(t^{-1/2})$, and the flux is integrable on
every finite interval starting at $T$. Apply
Lemma~\ref{lem:asym-heat-profile} to
\[
 n(t)=S_N(t-T)n(T)
       -\int_T^t\mathcal D_N(t-s)[n(u+\nabla c)](s)\,ds.
\]
Since $\int_{\mathbb H^3}n(T)=M$, this gives the density profile.

The concentration mild formula is
\begin{equation}\label{eq:asym-c-mild-profile}
 \begin{aligned}
 c(t)={}&S_N(t-T)c(T)
 -\int_T^t\mathcal D_N(t-s)(uc)(s)\,ds\\
 &-\int_T^tS_N(t-s)(nc)(s)\,ds.
 \end{aligned}
\end{equation}
Here incompressibility and the zero boundary trace of $u$ justify
the divergence form of the transport term. The flux $uc$ satisfies
the same small-$o$ bounds as the density flux, so its contribution
is $o(t^{-\alpha_p})$. The absorption term satisfies
\begin{equation}\label{eq:asym-reaction-decay}
 \begin{aligned}
 \|(nc)(s)\|_{L^p}
 &\leq\|n(s)\|_{L^\infty}\|c(s)\|_{L^p}
 \leq C_p s^{-\alpha_p-3/2},\\
 nc&\in L^1((T,\infty)\times\mathbb H^3).
 \end{aligned}
\end{equation}
By the mass identity, its total integral from $T$ onward equals
$\int_{\mathbb H^3}c(T)-L_\infty$. We claim that
\begin{equation}\label{eq:asym-reaction-profile}
 t^{\alpha_p}
 \left\|\int_T^tS_N(t-s)(nc)(s)\,ds
 -\left(\int_{\mathbb H^3}c(T)-L_\infty\right)G_N(t)\right\|_{L^p}
 \longrightarrow0.
\end{equation}

For almost every fixed $s$, the linear profile estimate gives
\[
 t^{\alpha_p}
 \left\|S_N(t-s)(nc)(s)
       -\left(\int_{\mathbb H^3}(nc)(s)\right)G_N(t)\right\|_{L^p}
 \longrightarrow0.
\]
On $T<s<t/2$, the expression is bounded by
$C_p\|(nc)(s)\|_{L^1}$, which is integrable in $s$.
Dominated convergence handles this early part. The omitted mass
$\int_{t/2}^\infty\int_{\mathbb H^3}nc$ tends to zero, and
$t^{\alpha_p}\|G_N(t)\|_{L^p}$ is constant. For the late part,
contractivity and \eqref{eq:asym-reaction-decay} give
\[
 \left\|\int_{t/2}^tS_N(t-s)(nc)(s)\,ds\right\|_{L^p}
 \leq C_p\int_{t/2}^t s^{-\alpha_p-3/2}\,ds
 \leq C_p t^{-\alpha_p-1/2}.
\]
These estimates include $p=1$ and $p=\infty$, and prove
\eqref{eq:asym-reaction-profile}.

Finally, the linear term in \eqref{eq:asym-c-mild-profile}
contributes $\bigl(\int_{\mathbb H^3}c(T)\bigr)G_N(t)$.
Subtracting the absorption profile leaves exactly
$L_\infty G_N(t)$. This proves \eqref{eq:asym-gaussian-profiles}.
\end{proof}

\begin{proof}[Proof of Theorem~\ref{thm:large-time-asymptotics}]
Use the same selected solution as in
Theorem~\ref{thm:eventual-regularity}; no further choice or change of
initial data is involved. Lemma~\ref{lem:asym-mass} gives the global
mass identities. Lemma~\ref{lem:asym-velocity} gives the velocity
limits and rates, Lemma~\ref{lem:asym-scalars} gives the scalar rates,
and Lemmas~\ref{lem:asym-residual-mass} and \ref{lem:asym-gaussian}
give the remaining signal mass and both Gaussian profiles.
If $M=0$, then $n=0$ and $g=0$; the velocity argument remains valid
for the unforced equation and the scalar proofs apply with the density
terms omitted. If $L=0$, then $c=0$ and its profile is identically
zero. Thus the conclusions also include both degenerate cases.
\end{proof}

\numberwithin{equation}{section} 
\numberwithin{theorem}{section}  

	\noindent {\bf Acknowledgments.}
	W. Wang was supported by National Key R\&D Program of China (No. 2023YFA1009200), NSFC under grants 12071054 and 12471219.\\	
	\noindent {\bf Declaration of competing interest.}
	The authors state that there is no conflict of interest.\\	
	\noindent {\bf Data availability.}
	No data was used in this paper.

\end{document}